\documentclass[10pt,reqno]{amsart}
\usepackage{amssymb}
\usepackage{amsmath}
\usepackage{amsrefs}
\usepackage[colorlinks=true,linktocpage]{hyperref}
\usepackage{cleveref}
\usepackage{amscd}
\usepackage{amsthm}
\usepackage{mathtools}
\usepackage{tikz-cd}
\usepackage{mathrsfs}
\usepackage{bm}
\usepackage[margin=1in]{geometry}

\usepackage{xcolor}
\hypersetup{
    colorlinks,
    linkcolor={red!50!black},
    citecolor={green!40!black},
    urlcolor={blue!80!black}
}

\allowdisplaybreaks

\newtheorem{theorem}{Theorem}[section]
\newtheorem{lemma}[theorem]{Lemma}
\newtheorem{proposition}[theorem]{Proposition}
\newtheorem{corollary}[theorem]{Corollary}

\newtheorem{conjecture}[theorem]{Conjecture}
\newtheorem{question}[theorem]{Question}

\newtheorem{theorem-star}[theorem]{Theorem*}
\newtheorem{lemma-star}[theorem]{Lemma*}
\newtheorem{proposition-star}[theorem]{Proposition*}
\newtheorem{corollary-star}[theorem]{Corollary*}

\theoremstyle{definition}
\newtheorem{definition}[theorem]{Definition}
\newtheorem{notation}[theorem]{Notation}
\newtheorem{example}[theorem]{Example}
\newtheorem{construction}[theorem]{Construction}

\newtheorem{assumption}[theorem]{Assumption}
\newtheorem*{setup1}{Setup I}
\newtheorem*{setup2}{Setup II}
\newtheorem*{setup3}{Setup III}
\newtheorem*{setup4}{Setup IV}
\newtheorem*{setup1star}{Setup I*}
\newtheorem*{setup2star}{Setup II*}
\newtheorem*{setup3star}{Setup III*}
\newtheorem*{setup4star}{Setup IV*}

\newtheorem{definition-star}[theorem]{Definition*}
\newtheorem{definition-assumption-star}[theorem]{Definition/Assumption*}
\newtheorem{theorem-assumption-star}[theorem]{Theorem/Assumption*}

\theoremstyle{remark}
\newtheorem{remark}[theorem]{Remark}

\numberwithin{equation}{section}

\newcommand{\R}{\mathbb{R}}

\newcommand{\Sch}{\mathcal{S}}

\newcommand{\N}{\mathbb{N}}
\newcommand{\Z}{\mathbb{Z}}

\newcommand{\SO}{\mathrm{SO}}
\newcommand{\alm}{\mathfrak{alm}}

\renewcommand{\L}{\mathcal{L}}

\newcommand{\M}{\mathcal{M}}

\newcommand{\lt}{\left}
\newcommand{\rt}{\right}
\newcommand{\tms}{\times}

\newcommand{\rmk}{\begin{remark}}
\newcommand{\ermk}{\end{remark}}
\newcommand{\cor}{\begin{corollary}}
\newcommand{\ecor}{\end{corollary}}
\newcommand{\eq}{\begin{equation}}
\newcommand{\eeq}{\end{equation}}
\newcommand{\eqs}{\begin{equation*}}
\newcommand{\eeqs}{\end{equation*}}
\newcommand{\prop}{\begin{proposition}}
\newcommand{\eprop}{\end{proposition}}
\newcommand{\thm}{\begin{theorem}}
\newcommand{\ethm}{\end{theorem}}
\newcommand{\conj}{\begin{conjecture}}
\newcommand{\econj}{\end{conjecture}}
\newcommand{\lem}{\begin{lemma}}
\newcommand{\elem}{\end{lemma}}
\newcommand{\defi}{\begin{definition}}
\newcommand{\edefi}{\end{definition}}
\newcommand{\ex}{\begin{example}}
\newcommand{\eex}{\end{example}}
\newcommand{\alis}{\begin{align*}}
\newcommand{\ealis}{\end{align*}}
\newcommand{\pf}{\begin{proof}}
\newcommand{\epf}{\end{proof}}
\newcommand{\ali}{\begin{align}}
\newcommand{\eali}{\end{align}}
\newcommand{\qus}{\begin{question}}
\newcommand{\equs}{\end{question}}

\newcommand{\mc}{\mathcal}
\renewcommand{\bf}{\textbf}

\newcommand{\C}{\mathbb{C}}

\newcommand{\sub}{\subset}

\newcommand{\ov}{\overline}

\newcommand{\bb}{\mathbb}

\newcommand{\op}{\operatorname}

\renewcommand{\a}{\alpha}
\renewcommand{\b}{\beta}
\renewcommand{\d}{\partial}
\newcommand{\e}{\epsilon}

\newcommand{\g}{\gamma}

\newcommand{\s}{\sigma}
\renewcommand{\t}{\theta}

\renewcommand{\l}{\lambda}
\renewcommand{\o}{\omega}
\newcommand{\z}{\zeta}

\newcommand{\fk}{\mathfrak}

\newcommand{\G}{\Gamma}
\renewcommand{\L}{\Lambda}
\renewcommand{\O}{\Omega}

\renewcommand{\S}{\Sigma}
\newcommand{\T}{\Theta}
\newcommand{\cP}{\mathcal{P}}

\newcommand{\Q}{\mathbb{Q}}

\newcommand{\CZ}{\operatorname{CZ}}
\newcommand{\Diff}{\operatorname{Diff}}

\newcommand{\Sign}{\operatorname{Sign}}
\newcommand{\Sym}{\operatorname{Sym}}

\newcommand{\good}{\mathrm{good}}
\newcommand{\oo}{\mathfrak o}
\newcommand{\Aut}{\operatorname{Aut}}
\newcommand{\codim}{\operatorname{codim}}

\newcommand{\Mbar}{\overline{\mathcal M}}

\newcommand{\vdim}{\operatorname{vdim}}
\newcommand{\vir}{\mathrm{vir}}
\newcommand{\I}{\mathrm I}
\newcommand{\II}{\mathrm{II}}
\newcommand{\III}{\mathrm{III}}
\newcommand{\IV}{\mathrm{IV}}

\newcommand{\bu}{\bullet}

\renewcommand{\ov}{\overline}

\begin{document}
\title{Homological invariants of codimension 2 contact submanifolds}
\author[Laurent C\^ot\'e]{Laurent C\^ot\'e}
\address{Department of Mathematics, Harvard University, 1 Oxford Street, Cambridge, MA 02138}
\email{lcote@math.harvard.edu}

\author[Fran\c{c}ois-Simon Fauteux-Chapleau]{Fran\c{c}ois-Simon Fauteux-Chapleau}
\address{Department of Mathematics, Stanford University, 450 Serra Mall, Stanford, CA 94305}
\email{fsimon@stanford.edu}

\date{\today}  
\begin{abstract} 
Codimension $2$ contact submanifolds are the natural generalization of transverse knots to contact manifolds of arbitrary dimension. In this paper, we construct new invariants of codimension $2$ contact submanifolds. Our main invariant can be viewed as a deformation of the contact homology algebra of the ambient manifold. We describe various applications of these invariants to contact topology. In particular, we exhibit examples of codimension $2$ contact embeddings into overtwisted and tight contact manifolds which are formally isotopic but fail to be isotopic through contact embeddings. We also give new obstructions to certain relative symplectic and Lagrangian cobordisms. 
\end{abstract}
\maketitle

\setcounter{tocdepth}{1}

\tableofcontents

\section{Introduction}

\subsection{Overview}
The purpose of this paper is to introduce new invariants of codimension $2$ contact submanifolds. Given a closed, co-oriented contact manifold $(Y, \xi)$ and a codimension $2$ contact submanifold $V$ with trivial normal bundle, our main construction produces a unital, $\Z/2$-graded $\Q[U]$-algebra 
\eq \label{equation:intro-main-invariant} CH_\bu(Y, \xi, V; \fk{r}).\eeq 
Here $\fk{r}=(\a_V, \tau, r)$ is a triple consisting of (i) a non-degenerate contact form $\a_V$ on $V$ inducing $\xi_V:= \xi|_V \cap TV$, (ii) a trivialization $\tau$ of $N_{Y/V}$, and (iii) a real number $r>0$. This triple is required to satisfy certain conditions (stated in \Cref{definition:local-datum}), and should be viewed as encoding a choice of Reeb dynamics in an infinitesimally small neighborhood of $V$. The (non-empty) set of all such triples is denoted by $\fk{R}(Y, \xi, V)$. 

The invariant \eqref{equation:intro-main-invariant} can be viewed as a deformation of the contact homology algebra $CH_\bu(Y, \xi)$, as explained in \Cref{remark:deformation} below. In particular, $U$ is a formal variable and there is a natural map 
\eq \op{ev}_{U=1}: CH_\bu(Y, \xi, V; \fk{r}) \to CH_\bu(Y, \xi)\eeq obtained by setting $U=1$.

The algebra $CH_\bu(Y, \xi, V; \fk{r})$ is generated by (good) Reeb orbits for an auxiliary non-degenerate contact form $\l$ on $(Y, \xi)$. See \cite[Def.\ 2.59]{pardon} for the notion of a good Reeb orbit; all non-degenerate Reeb orbits are good except for certain even degree covers of simple Reeb orbits with odd Conley--Zehnder index. The form $\l$ is required to be \emph{adapted} to $\fk{r} \in \fk{R}(Y, \xi, V)$, which means in particular that $V$ is preserved by the Reeb flow of $\l$; see \Cref{definition:adapted-contact-form}.  

The differential is defined as in ordinary contact homology by counting pseudo-holomorphic curves in the symplectization $\hat Y$, where the additional $U$ variable keeps track of the intersection number of curves with the symplectization $\hat{V} \sub \hat{Y}$. More precisely, we fix an almost-complex structure $J: \xi \to \xi$ which is compatible with the symplectic form $d \l$ and preserves $\xi|_V \cap TV \sub TY|_V$. We then consider $\hat{J}$-holomorphic curves in $\hat{Y}$, where $\hat{J}= -\d_t \otimes \l + R_{\l} \otimes dt+ J$. The differential is defined on generators by (roughly)\footnote{Strictly speaking, \eqref{equation:intro-differential} should be refined as follows: (i) one should indicate that the virtual moduli counts depend on a choice of ``virtual perturbation data"; (ii) one should indicate that the counts depend on the order of a certain group of automorphisms of the triple $(\g, \g_1 \sqcup \dots \sqcup \g_l, \beta)$ which acts on $\g$ by the identity; (iii) one must specify signs (or, more invariantly and following \cite{pardon}, orientations lines). See \eqref{equation:twist-differential-I} for the precise formula.} the following formula:
\eq \label{equation:intro-differential} d(\g)= \sum_{\beta \in \pi_2(\hat{Y}, \g \sqcup \g_1 \sqcup \dots \sqcup \g_l)} \#\ov{\mc{M}}(\beta) \, U^{\hat{V} * \beta + \G^-(\beta, V)} \, \g_1\dots \g_l, \eeq
where $\g$ is a positive orbit (i.e.\ associated to the convex end of the symplectization) and $\g_1, \dots, \g_l$ are negative orbits (i.e.\ associated to the concave end). We define $\G^-(\beta, V)= \#\{\g_i \sub V\}$ to be the number of negative orbits of $\beta$ which are contained in $V$.

We denote by $\ov{\mc{M}}(\beta)$ the compactification of the moduli space $\mc{M}(\beta)$ of $\hat{J}$-holomorphic curves in the class $\beta$. Such moduli spaces are in general non-transversally cut out, so the moduli counts $\#\ov{\mc{M}}(\beta)$ appearing in \eqref{equation:intro-differential} are defined as in Pardon's construction of contact homology, via his theory of virtual fundamental cycles \cites{pardon-vfc, pardon}. In particular, $\#\ov{\mc{M}}(\beta)$ is non-zero only for moduli spaces of virtual dimension zero. Finally, the pairing $(- * -)$ is a count of intersections between $\hat{V}$ and $\beta$ which was introduced by Siefring \cite{siefring, moreno-siefring}. 

We also construct a closely related invariant 
\eq \label{equation:intro-reduced-invariant} \widetilde{CH}_\bu(Y, \xi, V; \fk{r}) \eeq
which we call \emph{reduced}. This is a unital, $\Z/2$-graded $\Q$-algebra which is generated by Reeb orbits in the complement of $V$. The differential counts pseudo-holomorphic curves which do not intersect $\hat{V}$. For appropriately chosen pairs $\fk{r}, \fk{r}' \in  \fk{R}(Y, \xi, V)$, we have a morphism of $\Q$-algebras
\eq CH_\bu(Y, \xi, V; \fk{r}') \to \widetilde{CH}_\bu(Y, \xi, V; \fk{r}).\eeq
The invariant $\widetilde{CH}_\bu(-;-)$ is called \emph{reduced} because it carries less information (in particular, it does not involve taking the kernel of an augmentation as in e.g.\ reduced singular homology). However, it is easier to compute.

\rmk \label{remark:deformation} The $\Q[U]$-algebra $CH_\bu(Y, \xi, V; \fk{r})$ can be viewed as a deformation of the contact homology $\Q$-algebra $CH_\bu(Y, \xi)$ in the following way. First, recall that for a ring $R$ and a differential graded $R$-algebra $(A, d)$, a \emph{(formal) deformation} of $(A,d)$ is the data of a differential $d_t:=d+td_1+t^2d_2+ \dots$ on the $R[[t]]$-algebra $A[[t]]$ satisfying the graded Leibnitz rule, where each $d_i$ is an endomorphism of $A$ (see \cite{gerstenhaber-wilkerson}). Now, let us set $U=e^t$ in \eqref{equation:intro-differential} and expand in $t$. We then get
\eq d(\g)= \sum_\b \sum_{k=0}^{\infty} t^k \cdot (\#\ov{\mc{M}}(\beta)) \cdot \frac{(\hat{V}* \beta + \G^-(\beta, V))^k}{k!} \g_1\dots \g_l. \eeq
Thus $CH_\bu(Y, \xi, V; \fk{r})$ is indeed a deformation of ordinary contact homology, which can be recovered by sending $t \to 0$. In the case where $\G^-(\beta, V) = 0$, the coefficient $\#\ov{\mc{M}}(\beta) \cdot \frac{(\hat{V}* \beta + \G^-(\beta, V))^k}{k!} = \#\ov{\mc{M}}(\beta) \cdot \frac{(\hat{V}* \beta)^k}{k!}$ could be interpreted as a count of rigid pseudo-holomorphic curves which send $k$ marked points in the source to the pseudo-holomorphic divisor $\hat{V}$. 
\ermk

\subsection{Energy and positivity of intersection}

In order to ensure that \eqref{equation:intro-differential} defines a differential over $\Q[U]$, we need to ensure that $\hat{V}  * \beta + \G^-(\beta, V) \geq 0$ whenever $\#\ov{\mc{M}}(\beta) \neq 0$.  If $\mc{M}(\beta)$ is nonempty\footnote{Since \cite{pardon} uses virtual techniques to define contact homology without making any transversality assumptions, it is possible for the compactification $\ov{\mc{M}}(\b)$ to be nonempty even if $\mc{M}(\b)$ is empty. Positivity of intersection still holds when this happens, but the proof requires a bit more work. Details can be found in sections \ref{subsection:intersection-buildings} and \ref{subsection:twisting-maps-contact-subman}.} and at least one of the asymptotic orbits of $\beta$ is disjoint from $V$, then this is a consequence of the familiar phenomenon of positivity of intersection. Indeed, in this case, $\beta$ admits a $\hat{J}$-holomorphic representative $u$ which is not contained in $\hat{V}$. Positivity of intersection then implies that $\hat{V}*\beta = \hat{V}*u \geq 0$.

The situation is more complicated when all of the asymptotic orbits of $\beta$ are contained in $V$. Indeed, in this case, the $\hat{J}$-holomorphic representatives of $\beta$ may be contained in $\hat{V}$ and positivity of intersection fails in general. However, one can show that there is a universal lower bound on the intersection number
\eq \label{equation:intro-lower-bound} \hat{V}  * \beta \geq - \G^-(\beta, V). \eeq
This explains the appearance of the correction term $\G^-(\beta, V)$ in \eqref{equation:intro-differential}.

In order to construct $CH_\bu(Y, \xi, V; \fk{r})$, it is not enough to define a differential: one also needs to define continuation maps, composition homotopies, etc. These maps are defined by counting curves in more complicated setups. For example, the continuation map is obtained by counting curves in a suitably marked exact relative symplectic cobordism $(\hat{X},\hat{\l},H)$.\footnote{An exact relative symplectic cobordism is the data of an exact symplectic cobordism $(\hat X, \hat \l)$ which looks like $(\hat Y^\pm, \hat{\l}^\pm)$ near the ends, together with a codimension $2$ symplectic submanifold $H \sub \hat X$ which looks like $\hat V^\pm$ near the ends; see \Cref{definition:relative-sympcobordism} for the details.} More precisely, one obtains an algebra map similar to \eqref{equation:intro-differential} by counting $\hat{J}$-holomorphic curves in $(\hat{X}, \hat \l)$ weighted by their intersection number with $H$, for a compatible almost-complex structure $\hat{J}$ which agrees with $\hat{J}^\pm$ near the ends.

Unfortunately, for an arbitrary relative symplectic cobordism, a lower bound of the type \eqref{equation:intro-lower-bound} fails to hold. A key step in constructing the invariants \eqref{equation:intro-main-invariant} is to identify a sufficiently large class of relative symplectic cobordisms for which such a lower bound does hold. This leads us to introduce notions of energy for exact symplectic cobordisms and almost-complex structures on exact relative symplectic cobordisms. These energy notions are developed in \Cref{section:energy-twisting-maps} and are of central importance in this paper.

We prove that a lower bound as in \eqref{equation:intro-lower-bound} holds under a certain condition which relates the behavior of $\l^\pm$ near $V^\pm$ to the energy of $\hat{J}$. We also prove analogous statements for other related setups. This allows us to prove that $CH_\bu(Y, \xi, V; \fk{r})$ is well-defined, i.e.\ it does not depend on the auxiliary contact form and almost-complex structure. We also prove that an exact relative symplectic cobordism $(\hat X, \hat \l, H)$ induces a map 
\eq \label{equation:intro-cobordism-map} CH_\bu(Y^+, \xi^+, V^+; \fk{r}^+) \to CH_\bu(Y^-, \xi^-, V^-; \fk{r}^-) \eeq
provided that a certain inequality is satisfied, where the inequality involves $\fk{r}^\pm$ and the energy of the (sub)cobordism $H \sub (\hat X, \hat \l)$.

Energy considerations play a similarly central role in our construction of the reduced invariant $\widetilde{CH}_\bu(-;-)$. Although the continuation map for the reduced invariant does not count curves contained in $H$, one needs to ensure that sequences of curves disjoint from $H$ do not degenerate into $H$. This requires hypotheses on the energy of the relevant cobordism. In general, the arguments involved in constructing $CH_\bu(-;-)$ and $\widetilde{CH}_\bu(-;-)$ turn out to be very similar. 

Energy is not in general well-behaved under gluing symplectic cobordisms, unless one of them happens to be a symplectization. As a result, cobordism maps cannot be composed arbitrarily. This lack of functoriality of the invariants \eqref{equation:intro-main-invariant} and \eqref{equation:intro-reduced-invariant} can be remedied by considering variants of these invariants which are obtained by taking certain (co)limits over $\fk{r} \in \fk{R}(Y, \xi, V)$; see \Cref{subsection:asymptotic-invariants}. These variants are fully functorial but also seem harder to compute. 

\rmk
The apparent failure of positivity of intersection in the absence of energy bounds is not a deficiency of our method: if one could define maps \eqref{equation:intro-cobordism-map} without additional hypotheses involving energy and the $\fk{r}^\pm$, then this would imply that \eqref{equation:intro-main-invariant} and \eqref{equation:intro-reduced-invariant} are independent of $\fk{r}$. To see that this cannot hold in general, consider an ``irrational ellipsoid" $E(r_1, r_2)= \{z \in \C^2 \mid \pi |z_1|^2/r_1+ \pi |z_2|^2/r_2 \leq 1\} \sub \R^4$, where $r_1/r_2$ is irrational. Following  \cite[Sec.\ 4.2]{hutchings}, there are exactly two families of Reeb orbits $\{\g_1^k\}_{k \in \mathbb{N}_+}$ and $\{\g_2^k\}_{k \in \mathbb{N}_+}$, where $\g_i^1 \sub \{z_i=0\} \cap \partial E(r_1, r_2)$ and $\g_i^k$ denotes the $k$-fold cover of $\g_i^1$. These orbits generate the (ordinary) contact homology of the $3$-sphere, but the Conley--Zehnder indices of the $\g_i^k$ are highly sensitive to $r_1, r_2$. If we could define the contact homology of the complement of (say) $\g_1^1$, then it would be generated by the $\g_2^k$. One can verify that this is not an invariant, since changing the $r_i$ does not change the contactomorphism type of $S^3-\g_1^1$, but drastically changes the indices of the $\g_2^k$.
\ermk

\subsection{Legendrian invariants and the surgery formula} \label{subsection:intro-surgery}

Contact homology is one of many invariants which can be constructed using the framework of Symplectic Field Theory (SFT). SFT  was first introduced by Eliashberg--Givental--Hofer \cite{sft-egh} and provides (among other things) a conjectural mechanism for constructing invariants in symplectic and contact topology by counting punctured pseudo-holomorphic curves in symplectic manifolds with cylindrical ends. 

In some of the later sections of this paper, we discuss how the invariants \eqref{equation:intro-main-invariant} and \eqref{equation:intro-reduced-invariant} are related to other SFT-type invariants. For computational purposes, it is particularly useful to explore the behavior of the invariants \eqref{equation:intro-main-invariant} and \eqref{equation:intro-reduced-invariant} under Weinstein handle attachment, following the work of Bourgeois--Ekholm--Eliashberg \cite{bee}. 

To this end, we introduce analogs of \eqref{equation:intro-main-invariant} and \eqref{equation:intro-reduced-invariant} for Legendrian submanifolds. With $(Y, \xi, V)$ as above, suppose that $\L \sub (Y-V, \xi)$ is a Legendrian submanifold. We then define (under mild topological assumptions) invariants 
\eq  \label{equation:intro-legendrian} \mc{L}(Y, \xi, V, \L; \fk{r})\text{ and }\widetilde{\mc{L}}(Y, \xi, V, \L; \fk{r}). \eeq The first invariant can be thought of as a deformation of the Chekanov-Eliashberg dg algebra of $\L \sub (Y, \xi)$, while the second invariant is a reduced version. 

We describe a conjectural surgery exact sequence which relates (linearized versions of) the invariants \eqref{equation:intro-main-invariant}, \eqref{equation:intro-reduced-invariant} and \eqref{equation:intro-legendrian} under Weinstein handle attachments. This surgery exact sequence is an analog of the conjectural surgery exact sequence for linearized contact homology of Bourgeois--Eklholm--Eliashberg \cite[Thm.\ 5.2]{bee}. 

\rmk \label{remark:star-convention}
The main invariants of this paper \eqref{equation:intro-main-invariant} and \eqref{equation:intro-reduced-invariant} are constructed fully rigorously using Pardon's virtual perturbation framework \cite{pardon}. However, our discussion of the surgery formula (and of the Legendrian invariants therein) requires certain transversality assumptions which are essentially the same as in \cite{bee}. We fully expect that if \cite{bee} can be made rigorous using Pardon's techniques \cite{pardon}, then extending this to our context should pose no substantial additional difficulties. 

For the reader's convenience, \emph{all statements in this paper which depend on unproved assumptions are labeled by a star.} The proofs of starred statements depend only on a limited set of assumptions which are clearly identified in the text.

One could also attempt to define the invariants in this paper using other perturbation frameworks such as polyfolds \cite{h-w-z1, h-w-z2, h-w-z3} or the techniques of Ekholm \cite{ekholm2}, but do not pursue this here.
\ermk

\subsection{Applications to contact and Legendrian embeddings}
Transverse knots are important objects of study in three dimensional contact topology. The notion of a codimension $2$ contact embedding generalizes transverse knots to contact manifolds of arbitrary dimension. However, until recently, it was not understood whether the high-dimensional theory of codimension $2$ contact embeddings is interesting from the perspective of contact topology, or whether it reduces entirely to differential topology. 

\defi \label{definition:contact embedding}
Given a pair of contact manifolds $(V^{2m-1}, \zeta)$ and $(Y^{2n-1}, \xi)$, a \emph{contact embedding} is a smooth embedding 
\eq i: (V, \zeta) \to (Y, \xi) \eeq
 such that $i^*(\xi_{i(V)} \cap di(TV))= \zeta$. Such a map is also referred to as an isocontact embedding in the literature (see e.g.\ \cite{casals-etnyre, eliashberg-mishachev, pancholi-pandit}), but we will not use this terminology. A contact submanifold $V \sub (Y, \xi)$ is a submanifold with the property that $\xi|_V \cap TV$ is a contact structure. 
\edefi
Observe that if $2n-1=3$ and $2m-1=1$ in the above definition, then we recover the familiar notion of a (parametrized) transverse knot. The following basic examples of codimension $2$ contact embeddings will play an important role in this paper. 

\ex[cf.\ \Cref{definition:girouxform} and \cite{girouxicm}]
Let $\pi: Y-B \to S^1$ be an open book decomposition which supports the contact structure $\xi$ on $Y$. Then the binding $B \sub (Y, \xi)$ is a codimension $2$ contact submanifold.
\eex

\ex[see \Cref{definition:contact-pushoff} and Def.\ 3.1 in \cite{casals-etnyre}]
Let $(Y, \xi)$ be a contact manifold and let $\Lambda \hookrightarrow Y$ be a Legendrian embedding. Then the Weinstein neighborhood theorem furnishes an embedding 
\eq \tau(\L): (\partial(D^*\L), \xi_{\op{std}}) \hookrightarrow (Y, \xi) \eeq 
which is canonical up to isotopy through codimension $2$ contact embeddings. We refer to $\tau(\L)$ as the \emph{contact pushoff} of $\L \hookrightarrow Y$. By abuse of notation, we will routinely identify $\tau(\L)$ with its image. 
\eex

As is customary in contact and symplectic topology, there is a notion of a \emph{formal} contact embedding. This notion encodes certain necessary bundle-theoretic conditions which must be satisfied by any (genuine) contact embedding. It is then natural to seek to understand to what extent the space of genuine contact embeddings of $(V, \zeta)$ into $(Y, \xi)$ differs from the space of formal contact embeddings.

In case $V$ is a closed manifold of codimension at least $4$ with respect to $Y$, or open and of codimension at least $2$, then an h-principle due to Gromov (see \cite[Thm.\ 12.3.1 and Rmk.\ 12.3]{eliashberg-mishachev}) implies that the space of contact embeddings is essentially equivalent to the space of formal contact embeddings. Thus, in these settings, the theory of contact embeddings reduces to differential topology. 

In contrast, a breakthrough result due to Casals and Etnyre \cite{casals-etnyre} shows that this h-principle fails in general for codimension $2$ contact embeddings of closed manifolds. More precisely, for $n \geq 3$, Casals and Etnyre \cite[Thm.\ 1]{casals-etnyre} exhibited a pair of contact embeddings of $(D^*S^{n-1}, \xi)= \d_{\infty} (T^*S^{n-1}, \l_{\op{can}})$ into the standard contact sphere $(S^{2n-1}, \xi_{\op{std}})$ which are formally isotopic but are not isotopic through contact embeddings (here and throughout the paper, $\d_\infty(-)$ denotes the ideal contact boundary). Building on these methods, Zhou \cite{zhou} recently proved that there are in fact infinitely many formally isotopic contact embeddings of $\d_{\infty} (T^*S^{n-1}, \l_{\op{can}})$ into the standard contact sphere which are not isotopic through contact embeddings, provided $n \geq 4$. 

There has also been recent work to establish existence results for codimension $2$ contact embeddings under certain conditions \cite{pancholi-pandit, lazarev}. This culminates in a full existence h-principle for codimension $2$ contact embeddings due to Casals--Pancholi--Presas \cite{casals-pancholi-presas}, which states that any formal codimension $2$ contact embedding is formally isotopic to a genuine contact embedding. 

The invariants constructed in this paper can be used to distinguish pairs of formally isotopic contact embeddings which are not isotopic through contact embeddings. We illustrate two types of applications, applying respectively to contact embeddings into overtwisted contact manifolds and into the standard contact sphere. 

Let us begin with the overtwisted case. In \Cref{construction:ot-infinite-topology}, we describe a procedure for constructing pairs of formally isotopic contact embeddings into certain overtwisted contact manifolds which are not isotopic through contact embeddings. Here is a special case of this construction: let $(Y, \xi)$ be an overtwisted contact manifold and fix an open book decomposition for $(Y, B, \pi)$ which supports $\xi$ (see \Cref{subsection:open-book-decomp}). Let $i: B \to Y$ be the tautological inclusion of the binding. Using the relative $h$-principle for contact structures of Borman--Eliashberg--Murphy \cite[Thm.\ 1.2]{bem}, it can be shown that there exists a codimension $2$ contact embedding $j:B \to Y$ which is formally isotopic to $i$, and such that $(Y-j(B), \xi)$ is overwisted. (\Cref{construction:ot-infinite-topology} is more general, but this is the most important example.)

\thm[see \Cref{theorem:distinguish-infinite-ot}] \label{theorem:intro-ot}
Let $i, j$ and $(Y, \xi)$ be constructed according to \Cref{construction:ot-infinite-topology}, where $(Y, \xi)$ is an overtwisted contact manifold and $i$ and $j$ are (formally isotopic) contact embeddings. Then $i$ and $j$ are not isotopic through contact embeddings. In fact, $i$ is not isotopic to any reparametrization of $j$ in the source, meaning that $i(V)$ and $j(V)$ are not isotopic as codimension $2$ contact submanifolds of $(Y, \xi)$.
\ethm

\Cref{theorem:intro-ot} can be proved using either of the invariants \eqref{equation:intro-main-invariant} or \eqref{equation:intro-reduced-invariant}. To the best of our knowledge, it cannot be proved in general using invariants already in the literature. However, in the special case where $i(B)$ is the binding of an open book decomposition (i.e.\ the example sketched above), then the conclusion of \Cref{theorem:intro-ot} follows from the fact that the complement of the binding of an open book decomposition is tight. This later fact is due to Etnyre and Vela-Vick \cite[Thm.\ 1.2]{etnyre-velavick} in dimension $3$; in higher dimensions, it follows from work of Klukas \cite[Cor.\ 3]{klukas}, who proved (following an outline of Wendl \cite[Rmk.\ 4.1]{wendllocalfilling}) the stronger statement that any local filling obstruction (such as an overtwisted disk) in a closed contact manifold must intersect the binding of any supporting open book. 

In some cases (see \Cref{corollary:leg-distinguish-infinite-ot}), the embeddings $i$ and $j$ in fact coincide with the contact pushoffs of Legendrian embeddings. It is not hard to show that an isotopy of Legendrian embeddings induces an isotopy of their contact pushoffs. Thus the invariants \eqref{equation:intro-main-invariant} and \eqref{equation:intro-reduced-invariant} also distinguish certain Legendrian embeddings in overtwisted contact manifolds. To our knowledge, these embeddings cannot in general be distinguished using invariants already in the literature (see \Cref{remark:intro-lag-cobordisms}).

Our second application concerns codimension $2$ contact embeddings into the standard contact spheres $(S^{4n-1}, \xi_{\op{std}})$. More precisely, we use the reduced invariant \eqref{equation:intro-reduced-invariant} to distinguish formally isotopic contact embeddings of $(S^*S^{2n-1}, \xi)= \d_{\infty}(T^*S^{2n-1}, \l_{\op{can}})$ into $(S^{4n-1}, \xi_{\op{std}})$, thus reproving the main result of Casals and Etnyre \cite[Thm.\ 1]{casals-etnyre} in dimensions $4n-1$ for $n>1$.

\begin{theorem-star}[see \Cref{theorem:tight-examples}] \label{theorem:intro-tight}
Let $(V, \xi)$ be the ideal boundary of $(T^*S^{2n-1}, \l_{\op{can}})$. Then for $n>1$, there exists a pair of formally isotopic contact embeddings 
\eq i_0, i_1: (V, \xi) \to (S^{4n-1}, \xi_{\op{std}}) \eeq 
which are not isotopic through contact embeddings.
\end{theorem-star}

The embeddings we exhibit turn out to coincide with those exhibited by Casals and Etnyre in their proof of \cite[Thm.\ 1.1]{casals-etnyre}, although this fact is not entirely obvious (see \Cref{remark:casals-etnyre-coincide}). However, the methods for distinguishing them are completely different. Casals and Etnyre consider double branched covers along the contact submanifolds $i_0(V)$ and $i_1(V)$. Using symplectic homology, they prove that these branched covers do not admit the same fillings. This implies that $i_0(V)$ and $i_1(V)$ cannot be isotopic, since otherwise they would have contactomorphic branched covers.  

In contrast, our proof of \Cref{theorem:intro-tight} uses the invariant $\widetilde{CH}_\bu(-;-)$ introduced in this paper. Roughly speaking, we prove \Cref{theorem:intro-tight} by partially computing (linearizations of) $\widetilde{CH}_\bu(-;-)$ associated to the two embeddings under consideration, and observing that they do not match. Our computations rely crucially on our version of the surgery formula discussed in \Cref{subsection:intro-surgery} as well as the well-definedness of the invariants therein. This explains why this theorem statement is starred, following the convention stated in \Cref{remark:star-convention}.  We also remark that although \Cref{theorem:intro-tight} only applies to spheres of dimension $4n-1$, we expect that the same invariant also distinguishes embeddings into spheres of dimension $4n-3$. However, proving this would likely require more involved computations than those carried out in this paper.

\subsection{Applications to symplectic and Lagrangian cobordisms}
\label{subsection:intro-cob-applications}
Consider a pair of contact manifolds $(Y^\pm, \xi^\pm)$ and codimension $2$ contact submanifolds $(V^\pm, {\xi^\pm}|_{V^\pm}) \sub (Y^\pm, \xi^\pm)$. An exact relative symplectic cobordism from $(Y^+, \xi^+, V^+)$ to $(Y^-, \xi^-, V^-)$ is a triple $(\hat X, \hat \l, H)$, where $(\hat X, \hat \l)$ is an exact symplectic cobordism from $(Y^+, \xi^+)$ to $(Y^-, \xi^-)$ and $H \sub \hat X$ is a codimension $2$ symplectic submanifold which coincides near the ends with the symplectization of $V^\pm$; see \Cref{definition:relative-sympcobordism}.\footnote{Note that our convention of regarding a cobordim as going \emph{from} the convex end \emph{to} the concave end is consistent with \cite{pardon}, but differs from most of the contact topology literature.} In the special case where $\hat X$ is the symplectization of $Y^\pm$ and $H$ is diffeomorphic to $\R \tms V^\pm$, we speak of a \emph{symplectic concordance} from $V^+$ to $V^-$. These notions were first considered by Bowden in his PhD thesis. Using gauge theory, he exhibited certain restrictions on symplectic cobordisms between transverse links in contact $3$-manifolds \cite[Sec.\ 7]{bowden}. 

The following theorem gives a constraint on exact symplectic cobordisms between certain pairs of codimension $2$ contact submanifolds of an ambient overtwisted manifold. To the best of our knowledge, this is the first negative result in the literature on relative symplectic cobordisms in dimensions greater than three.

\thm \label{theorem:intro-symplectic-cobordism}
Let $V= i(B), V'=j(B)$ be the codimension $2$ contact submanifolds of the overtwisted contact manifold $(Y, \xi)$ as described in \Cref{construction:ot-infinite-topology}. Then there does not exist an exact relative symplectic cobordism $(\hat X, \hat \l, H)$ from $(Y, \xi, V')$ to $(Y, \xi, V)$ with $H^1(H, (-\infty,0] \times V; \Z)= H^2(H, (-\infty,0] \times V; \Z)=0$. In particular, there is no symplectic concordance from $V'$ to $V$.
\ethm

One can similarly consider Lagrangian cobordisms and concordances between Legendrian submanifolds. An exact Lagrangian cobordism from $(Y^+, \xi^+, \L^+)$ to $(Y^-, \xi^-, V^-)$ is a triple $(\hat X, \hat \l, L)$ where $(\hat X, \hat \l)$ is an exact symplectic cobordism from $(Y^+, \xi^+)$ to $(Y^-, \xi^-)$ and $L \sub \hat X$ is a Lagrangian submanifold which coincides near the ends with the Lagrangian cone of $\L^\pm$; see \Cref{definition:lagrangian-cob}. If $\hat X$ is the symplectization of $Y^-$ and $L = \R \tms \L^-$, one speaks of a \emph{Lagrangian concordance} from $\L^+$ to $\L^-$. 

The theory of Lagrangian cobordisms has been extensively developed in the literature from various perspectives (see e.g.\ \cite{c-d-g-g, ekholm, ekholm-honda-kalman, sabloff-traynor}). While much is known in $(\R^{2n+1}, \xi_{\op{std}})$ and certain other tight contact manifolds, we are not aware of any results constraining cobordisms and concordances in overtwisted contact manifolds; see \Cref{remark:intro-lag-cobordisms}. The next theorem provides a first result in this direction. 

\thm \label{theorem:intro-lag-concordance}
Let $\L, \L'$ be the Legendrian submanifolds of the overtwisted contact manifold $(Y, \xi)$ as constructed in \Cref{construction:ot-legendrian}. Then $\L'$ is not concordant to $\L$. 
\ethm

In contrast, a result of Eliashberg and Murphy \cite[Thm.\ 2.2]{eliashberg-murphy-caps} implies that $\L$ is concordant to $\L'$.

\rmk \label{remark:intro-lag-cobordisms}
It is a basic fact that exact Lagrangian cobordisms induce morphisms on Legendrian contact homology which behave well under composition of cobordisms \cite[Sec.\ 5.1]{etnyre-ng}. This leads to a myriad of interesting obstructions to the existence of Lagrangian cobordisms and concordances. One can also obtain many interesting obstructions using finite-dimensional invariants (which are closely related to Legendrian contact homology) coming from generating functions or sheaf theory; see e.g.\ \cite{sabloff-traynor, li}. 

One drawback of these approaches is that they are necessary blind on overtwisted contact manifolds. Indeed, even if Legendrian contact homology could be rigorously defined in full generality following the framework of \cite[Sec.\ 2.8]{sft-egh}, it would provide no information for Legendrians in overtwisted contact manifolds: being a module over the contact homology algebra, it would vanish. In contrast, the invariants developed in this paper do give information about Legendrians even in the overtwisted case.
\ermk

Our final application states that certain Lagrangian concordances cannot be displaced from a codimension $2$ symplectic submanifold. More precisely, let $(Y,\xi)= \op{obd}(T^*S^{n-1}, \op{id})$ and let $V \sub Y$ be the binding of the open book. Let $\L \sub Y$ the zero section of a page and let $\L'$ be obtained by stabilizing $\L$ in the complement of $V$. It can be shown \cite[Prop.\ 2.9]{casals-murphy} that $\L \sub (Y, \xi)$ is a loose Legendrian; hence $\L, \L'$ are Legendrian isotopic in $(Y, \xi)$ and in particular concordant.
	
\begin{theorem-star} \label{theorem:introduction-non-displaceability}
Any Lagrangian concordance from $\L'$ to $\L$ must intersect the symplectization of $V$. 
\end{theorem-star}

In contrast, work of Eliashberg and Murphy  \cite[Thm.\ 2.2]{eliashberg-murphy-caps} implies that there exists a Lagrangian concordance from $\L$ to $\L'$ which is disjoint from the symplectization of $V$. Our proof of \Cref{theorem:introduction-non-displaceability} uses the deformed versions of the Chekanov-Eliashberg dg algebra in \eqref{equation:intro-legendrian}. Hence the statement is starred according to the convention stated in \Cref{subsection:intro-surgery}. 

\subsection{Context and related invariants}

The invariants constructed in this paper, when specialized to contact $3$-manifolds, are related to other invariants in the literature. The most closely related invariant is due to Momin \cite{momin}. Given a contact $3$-manifold $(Y^3, \xi)$, Momin considers the set of pairs $(\l, L)$ where $\l$ is a contact form and $L \sub Y$ is a link of Reeb orbits of $\l$. Two such pairs $(\l, L), (\l', L')$ are said to be equivalent if $L=L'$ and each component orbit (and all its multiple covers) has the same Conley-Zehnder index. Under certain assumptions on $(Y, \l, L)$, Momin defines an invariant which we denote by $CH^{mo}_\bu(Y, [(\l, L)])$. This is a $\Z$-graded $\Q$-vector space which depends only on $Y$ and the equivalence class of $(\l, L)$. 

The invariant constructed by Momin is in general distinct from the invariants described in this paper. In particular, he considers cylindrical contact homology, whereas we work with ordinary contact homology. However, in the special case where $(Y^3, \xi)$ is the standard contact sphere (or more generally a subcritical Stein manifold with $c_1(\xi)=0$) and $L \sub (Y, \xi)$ is a collection of Reeb orbits which bound a symplectic submanifold $H \sub B^4$, then we expect that 
\eq \label{equation:intro-momin} CH_\bu^{mo}(Y, [\l, L])= \widetilde{CH}_\bu^{\tilde{\e}}(Y, \xi, L; \fk{r}), \eeq for suitable $\fk{r}$ which depends on the equivalence class of $(\l, L)$.  Here the right hand side of \eqref{equation:intro-momin} denotes the linearization of $\widetilde{CH}_\bu(Y, \xi, L; \fk{r})$ with respect to an augmentation $\tilde{\e}$ induced by the relative filling $(B^4, \l_{\op{std}}, H)$ (recall that an augmentation of a dg-algebra is a morphism to the ground ring, viewed as a dg algebra concentrated in degree zero). See \Cref{subsection:augmentations-fillings} for details.

Momin's work has led to beautiful applications to Reeb dynamics on contact $3$-manifolds (see e.g.\ \cite{alves-pirnapasov, h-m-s}). It would be interesting to explore whether the invariants developed in this paper can be used in studying Reeb dynamics in higher dimensions.

Another related invariant is Hutchings' ``knot-filtered" embedded contact homology \cite{hutchings}. The setting for this invariant is a contact $3$-manifold $(Y, \xi)$ with $H_1(Y; \Z)=0$. Given a transverse knot $L \sub (Y, \xi)$ and an irrational parameter $\theta \in \R- \Q$, Hutchings defines a filtration on embedded contact homology with values in $\Z + \Z \t$ which is an invariant of $(L, \t)$. The basic idea is to choose a contact form $\xi= \op{ker} \l$ so that $L$ is a Reeb orbit, and to filter the generators of embedded contact homology by their linking number with $L$. Positivity of intersection considerations imply that the differential decreases the linking number for orbits which are disjoint from $L$. However, the situation is more complex when the differential involves $L$, which explains why the filtration is only valued in $\Z + \Z \t$. 

One could presumably carry over Hutchings' construction to the context of (cylindrical) contact homology in dimension $3$. We expect that the resulting invariant would carry related information to the one defined by Momin or to the invariants constructed in this paper. However, we do not have a precise formulation of what this relationship should be.

We remark that the invariants introduced by Momin and Hutchings are built using techniques from $4$-dimensional symplectic topology which cannot be generalized to higher dimensions. In contrast, the invariants introduced in this paper are constructed by a different approach which ultimately relies on Pardon's robust virtual fundamental cycles package \cite{pardon-vfc}.

In a slightly different direction, we also wish to highlight work of Ekholm--Etnyre--Ng--Sullivan \cite{ekholm-etnyre-ng-sullivan} which is similar in spirit to the present work.  Recall that the knot contact homology of a framed link $K \sub \R^3$ is an invariant which can be defined as the Legendrian dg algebra of the conormal lift of $K$ (see \cite{ng, ekholm-etnyre-ng-sullivan1}).  If $K$ is a transverse knot with respect to the standard contact structure, the authors define in \cite{ekholm-etnyre-ng-sullivan} a two-parameter deformation of knot contact homology by weighting holomorphic curve counts by their intersection number with a pair of canonically defined complex submanifolds in the symplectization. The resulting deformed dg algebra is an invariant of the transverse knot type of $K$. Unfortunately, we do not know a precise relationship between this invariant and the ones introduced in this paper.

\subsection{Notation and conventions}
All manifolds in this paper are assumed to be smooth. If $M$ is a manifold, a \emph{ball} $B \sub M$ is an open subset diffeomorphic to the open unit disk and whose closure is embedded and diffeomorphic to the closed unit disk. If $(M, \o)$ is symplectic, a \emph{Darboux ball} $B \sub M$ is a ball which is symplectomorphic to the open unit disk equipped with (some rescaling of) the standard symplectic form.

Unless otherwise specified, all contact manifolds considered in this paper are compact without boundary and co-oriented. Given such a contact manifold $(Y,\xi = \ker\l)$, the Reeb vector field associated to the contact form $\l$ will be denoted by $R_\l$. We will let $\xi_V := \xi|_V \cap TV$ denote the contact structure induced by $\xi$ on a contact submanifold $V \subset (Y,\xi)$.

\subsection{Acknowledgements}
We thank Yasha Eliashberg for suggesting this project, and for many helpful discussions. We also benefited from discussions and correspondence with C\'{e}dric De Groote, Georgios Dimitroglou Rizell, Sheel Ganatra, Oleg Lazarev, Josh Sabloff and Kyler Siegel. We are grateful to John Etnyre for pointing us to multiple useful references. Finally, we wish to thank the anonymous referee for many helpful comments and suggestions. 

The first author was partially supported by a Stanford University Benchmark Graduate Fellowship and by the National Science Foundation under Grant No. DMS-1926686.

\section{Geometric preliminaries}

\subsection{Symplectic cobordisms}
Let $(Y,\xi)$ be a closed co-oriented contact manifold. The \emph{symplectization} of $(Y,\xi)$ is the exact symplectic manifold $(SY,\l_Y)$ where $SY \subset T^*Y$ is the total space of the bundle of positive contact forms on $Y$ (i.e.\ a point $(p,\a) \in T^*Y$ is in $SY$ if and only if $\a : T_p Y \to \R$ vanishes on $\xi_p$ and the induced map $T_p Y / \xi_p \to \R$ is an orientation-preserving isomorphism) and $\l_Y$ is the restriction of the tautological Liouville form on $T^*Y$. Given a choice of positive contact form $\a$ for $(Y,\xi)$, there is a canonical identification
\eq \label{symplectization-marking}
\s_\a : (\R \times Y, e^s \a) \to (SY,\l_Y)
\eeq
given by $\s_\a(s,p) = (p,e^s \a_p)$. We will refer to $(\hat{Y},\hat{\a}) := (\R \times Y, e^s \a)$ as the symplectization of $(Y,\a)$.

A subset $U \subset SY$ will be called a \emph{neighborhood of $+\infty$} (resp. of $-\infty$) if it contains $\s_\a([N,\infty) \times Y)$ (resp. $\s_\a((-\infty,-N] \times \R$) for $N > 0$ sufficiently large (note that this notion doesn't depend on the choice of $\a$).

\defi \label{definition:symplectic-lift}
Given a contactomorphism $f:(Y, \xi) \to (Y, \xi)$, we define its \emph{symplectic lift} 
\begin{align}
\tilde{f}: (SY, \l_Y) &\to (SY, \l_Y)\\ 
 (p, \a) &\mapsto (f(p), \a \circ (df_p)^{-1}).
\end{align} 
One can verify that $\tilde{f}^* \l_Y= \l_Y$, so $\tilde{f}$ is in particular a symplectomorphism. There is a canonical bijection between (i) contact vector fields on $(Y, \xi)$, (ii) sections of $TY/\xi$, (iii) linear Hamiltonians on the symplectization (recall that $H$ is \emph{linear} if $Z H = H$, where $Z$ denotes the Liouville vector field). The correspondence between (i) and (ii) is clear; the correspondence between (ii) and (iii) takes a section $\s$ to the Hamiltonian $H(p,\a) =\a(\s(p))$. In particular, the symplectic lift of a (time-dependent) family of contactomorphisms is induced by a (time-dependent) family of linear Hamiltonians; cf.\ \cite[Prop.\ 2.2]{chantraine}. 
\edefi

\defi \label{definition:sympcobordism}
	Let $(Y^+,\xi^+)$ and $(Y^-,\xi^-)$ be closed co-oriented contact manifolds. An \emph{exact symplectic cobordism} from $(Y^+,\xi^+)$ to $(Y^-,\xi^-)$ is an exact symplectic manifold $(\hat{X},\hat{\l})$ equipped with embeddings
	\begin{align}
		e^+ : SY^+ &\to \hat{X} \label{positive-end}\\
		e^- : SY^- &\to \hat{X} \label{negative-end}
	\end{align}
	satisfying the following properties:
	\begin{itemize}
		\item $(e^\pm)^*\hat{\l} = \l_{Y^\pm}$;
		\item there exists a neighborhood $U^+ \subset SY^+$ of $+\infty$ and a neighborhood $U^- \subset SY^-$ of $-\infty$ such that the restriction of $e^\pm$ to $U^\pm$ is proper, the images $e^+(U^+)$ and $e^-(U^-)$ are disjoint and the complement $\hat{X} \setminus (e^+(U^+) \cup e^-(U^-))$ is compact.
	\end{itemize}
\edefi

\begin{definition}[{cf.\ \cite[Sec.\ 1.3]{pardon}}] \label{definition:strict-sympcobordism}
	Let $(Y^+,\l^+)$ and $(Y^-,\l^-)$ be closed manifolds equipped with contact forms. A \emph{(strict) exact symplectic cobordism} from $(Y^+,\l^+)$ to $(Y^-,\l^-)$ is an exact symplectic manifold $(\hat X,\hat\l)$ equipped with embeddings
	\begin{align}
		e^+ : \R \times Y^+ &\to \hat{X} \label{marked-positive-end}\\
		e^- : \R \times Y^- &\to \hat{X} \label{marked-negative-end}
	\end{align} 
    satisfying the following properties:
	\begin{itemize}
		\item $(e^\pm)^*\hat{\l} = \hat{\l}^\pm$;
		\item there exists an $N \in \R$ such that the restrictions of $e^+$ to $[N,\infty) \times Y^+$ and of $e^-$ to $(-\infty,-N] \times Y^-$ are proper and that the images $e^+([N,\infty) \times Y^+)$ and $e^-((-\infty,-N] \times Y^-)$ are disjoint and together cover a neighborhood of infinity (i.e. the complement of their union is compact).
	\end{itemize}
\end{definition}

\begin{notation}\label{notation:choice-forms}
Let $(\hat{X},\hat{\l})$ be an exact symplectic cobordism from $(Y^+,\xi^+)$ to $(Y^-,\xi^-)$ in the sense of \Cref{definition:sympcobordism}. Given any choice of contact forms $\l^\pm$ on $(Y^\pm,\xi^\pm)$, one can obtain from $\hat{X}$ a cobordism from $(Y^+,\l^+)$ to $(Y^-,\l^-)$ in the sense of \Cref{definition:strict-sympcobordism} by pre-composing the embeddings \eqref{positive-end}--\eqref{negative-end} with the canonical identifications $\R \times Y^\pm \to SY^\pm$ induced by $\l^\pm$. We will denote this cobordism by $(\hat{X},\hat{\l})^{\l^+}_{\l^-}$ or simply by $\hat{X}^{\l^+}_{\l^-}$ when this creates no ambiguity.

Similarly, any cobordism $(\hat{X},\hat{\l})$ in the sense of \Cref{definition:strict-sympcobordism} can be viewed as a cobordism in the sense of \Cref{definition:sympcobordism} as well.
\end{notation}

\begin{remark}
In light of the above discussion, \Cref{definition:sympcobordism} and \Cref{definition:strict-sympcobordism} are essentially equivalent. However, it will be convenient for us to be able to discuss symplectic cobordisms without fixing a particular choice of contact forms on the ends, so we adopt \Cref{definition:sympcobordism} as our main definition moving forward.
\end{remark}

\ex[Symplectizations]  \label{example:symplectization} The symplectization $(SY, \l_Y)$ of a contact manifold $(Y, \xi)$ is canonically endowed with the structure of an exact symplectic cobordism in the sense of \Cref{definition:sympcobordism} by letting $e^+=e^-=\op{id}$.  The additional data of a pair of contact forms $\l^+, \l^-$ for $(Y, \xi)$, endows $(SY, \l_Y)$ with the structure of a strict exact symplectic cobordism in the sense of \Cref{definition:strict-sympcobordism} and we write $(SY, \l_Y)^{\l^+}_{\l^-}$. 

\eex

\begin{definition}\label{definition:gluing-cobordisms}
Let $(\hat{X}^{01},\hat{\l}^{01})$ and $(\hat{X}^{12},\hat{\l}^{12})$ be exact symplectic cobordisms from $(Y^0,\xi^0)$ to $(Y^1,\xi^1)$ and from $(Y^1,\xi^1)$ to $(Y^2,\xi^2)$ respectively. Fix a real number $t \ge 0$ and let $\mu_t : SY^1 \to SY^1$ denote multiplication by $e^t$. The \emph{$t$-gluing} of $\hat{X}^{01}$ and $\hat{X}^{12}$, denoted by $\hat{X}^{01} \#_t \hat{X}^{12}$, is the smooth manifold obtained by gluing $\hat{X}^{01}$ and $\hat{X}^{12}$ along the maps
	\begin{center}
	\begin{tikzcd}
		SY^1 \ar{r}{\eqref{negative-end}} \ar{d}{\mu_t} & \hat{X}^{01} \\
		SY^1 \ar{r}{\eqref{positive-end}} & \hat{X}^{12}
	\end{tikzcd}
	\end{center}
	Since $\mu_t^*\l_{Y^1} = e^t \l_{Y^1}$, there is, for any $s \in \R$, a Liouville form on $\hat{X}^{01} \#_t \hat{X}^{12}$ which agrees with $e^{t + s}\hat{\l}^{01}$ on $\hat{X}^{01}$ and with $e^s \hat{\l}^{12}$ on $\hat{X}^{12}$. We will denote it by $\hat{\l}^{01} \#_{t,s} \hat{\l}^{12}$. Note that $(\hat{X}^{01} \#_t \hat{X}^{12}, \hat{\l}^{01} \#_{t,s} \hat{\l}^{12})$ is canonically equipped with the structure of an exact symplectic cobordism from $(Y^0,\xi^0)$ to $(Y^2,\xi^2)$ via the embeddings
	\begin{center}
	\begin{tikzcd}[row sep = small, column sep = large]
		SY^0 \ar{r}{\mu_{-t-s}} & SY^0 \ar{r}{\eqref{positive-end}} & \hat{X}^{01} \ar{r}{} & \hat{X}^{01} \#_t \hat{X}^{12} \\
		SY^2 \ar{r}{\mu_{-s}} & SY^2 \ar{r}{\eqref{negative-end}} & \hat{X}^{02} \ar{r}{} & \hat{X}^{01} \#_t \hat{X}^{12}
	\end{tikzcd}
	\end{center}
	The precise choice of $s$ doesn't really matter since the forms $\hat{\l}^{01} \#_{t,s} \hat{\l}^{12}$, $s \in \R$, are all constant multiples of each other. When $t = 0$, it is natural to choose $s = 0$, and we will denote the resulting cobordism simply by $(\hat{X}^{01} \# \hat{X}^{12}, \hat{\l}^{01} \# \hat{\l}^{12})$. There is no obvious choice for $t > 0$, but for the sake of definiteness we set $\hat{\l}^{01} \#_t \hat{\l}^{12} := \hat{\l}^{01} \#_{t,-t/2} \hat{\l}^{12}$ and will refer to $(\hat{X}^{01} \#_t \hat{X}^{12}, \hat{\l}^{01} \#_t \hat{\l}^{12})$ as ``the'' $t$-gluing of $(\hat{X}^{01},\hat{\l}^{01})$ and $(\hat{X}^{12},\hat{\l}^{12})$.
\end{definition}

\begin{remark}\label{remark:associativity}
When $t = s = 0$, it follows directly from the definition that the gluing operation is associative: $\bigl((\hat{X}^{01} \# \hat{X}^{12}) \# \hat{X}^{23}, (\hat{\l}^{01} \# \hat{\l}^{12}) \# \hat{\l}^{23} \bigr)$ and $\bigl(\hat{X}^{01} \# (\hat{X}^{12} \# \hat{X}^{23}), \hat{\l}^{01} \# (\hat{\l}^{12} \# \hat{\l}^{23}) \bigr)$ are canonically isomorphic.
\end{remark}

\begin{remark}
Multiplication by $e^t$ on $SY$ corresponds to translation by $t$ in the $\R$ coordinate under the identification $SY \cong \R \times Y$ induced by a choice of contact form on $Y$. \Cref{definition:gluing-cobordisms} is therefore consistent with the notion of ``$t$-gluing'' in \cite[Sec. 1.5]{pardon}.
\end{remark}

\begin{definition}\label{definition:iso-sympcobordism}
	Let $(\hat{X}^1,\hat{\l}^1)$ and $(\hat{X}^2,\hat{\l}^2)$ be cobordisms from $(Y^+,\xi^+)$ to $(Y^-,\xi^-)$. An \emph{isomorphism of exact symplectic cobordisms} $\phi : (\hat{X}^1,\hat{\l}^1) \to (\hat{X}^2,\hat{\l}^2)$ consists of a diffeomorphism $\phi : \hat{X}^1 \to \hat{X}^2$ such that $\phi^*\hat{\l}^2 = \hat{\l}^1$ and which is compatible with the ends in the sense that the following diagram commutes:
	\begin{center}
	\begin{tikzcd}
		& & \hat{X}^1 \arrow{dd}{\phi} & & \\
		SY^+ \ar{rru}{\eqref{positive-end}} \ar{rrd}[swap]{\eqref{positive-end}} & & & & SY^- \ar{llu}[swap]{\eqref{negative-end}} \ar{lld}{\eqref{negative-end}}\\
		& & \hat{X}^2
	\end{tikzcd}
	\end{center}
\end{definition}

\begin{example} \label{example:gluing-trivial-end}
	Let $(\hat{X},\hat{\l})$ be an exact symplectic cobordism from $(Y^+,\xi^+)$ to $(Y^-,\xi^-)$. Then for any $t \ge 0$ and $s \in \R$, the glued cobordisms $(SY^+ \#_t \hat{X}, \l_{Y^+} \#_{t,s} \hat{\l})$ and $(\hat{X} \#_t SY^-, \hat{\l} \#_{t,s} \l_{Y^-})$ are canonically isomorphic to $(\hat{X},\hat{\l})$.
\end{example}

\begin{definition}\label{definition:family-sympcobordisms}
	A \emph{one-parameter family of exact symplectic cobordisms} from $(Y^+,\xi^+)$ to $(Y^-,\xi^-)$ is a manifold $\hat{X}$ equipped with a family of Liouville forms $\{ \hat{\l}^t \}_{t \in I}$ (where $I \subset \R$ is an interval), together with embeddings
	\begin{align}
		e_t^+ : SY^+ &\to \hat{X} \label{family-marked-positive-end}\\
		e_t^- : SY^- &\to \hat{X} \label{family-marked-negative-end}
	\end{align}
	as in \Cref{definition:sympcobordism}. We will always assume that the family is fixed at infinity, meaning that for every compact subinterval $[a,b] \subset I$,
	\begin{itemize}
		\item  $\{ \hat{\l}^t \}_{t \in [a,b]}$ is constant outside of a compact subset of $\hat{X}$;
		\item $\{e_t^+\}_{t \in [a,b]}$ (resp. $\{e_t^-\}_{t \in [a,b]}$) is independent of $t$ on some neighborhood of $+\infty$ in $SY^+$ (resp. of $-\infty$ in $SY^-$).
	\end{itemize}

	Two cobordisms $(\hat{X}^0,\hat{\l}^0)$ and $(\hat{X}^1,\hat{\l}^1)$ are said to be \emph{deformation equivalent} if there exists a one-parameter family $(\hat{W},\hat{\mu}^t)_{t \in [0,1]}$ such that $(\hat{X}^0,\hat{\l}^0)$ is isomorphic to $(\hat{W},\hat{\mu}^0)$ and $(\hat{X}^1,\hat{\l}^1)$ is isomorphic to $(\hat{W},\hat{\mu}^1)$. The deformation class of a cobordism $(\hat{X},\hat{\l})$ will be denoted by $[\hat{X},\hat{\l}]$.
\end{definition}

\begin{lemma} \label{lemma:t-gluing}
	Given two cobordisms $(\hat{X}^{01},\hat{\l}^{01})$ and $(\hat{X}^{12},\hat{\l}^{12})$ as in \Cref{definition:gluing-cobordisms}, the glued cobordisms $(\hat{X}^{01} \#_t \hat{X}^{12}, \hat{\l}^{01} \#_t \hat{\l}^{12})_{t \in{} [0,\infty)}$ form a one-parameter family.\footnote{Strictly speaking, the underlying manifold of $\hat{X}^{01} \#_t \hat{X}^{12}$ depends on $t$, so in order to obtain a family in the sense of \Cref{definition:family-sympcobordisms} one needs to choose suitable diffeomorphisms $\hat{X}^{01} \#_t \hat{X}^{12} \cong \hat{X}^{01} \# \hat{X}^{12}$.} Similarly $(\hat{X}^{01} \#_t \hat{X}^{12}, \hat{\l}^{01} \#_{t,s} \hat{\l}^{12})_{s \in \R}$ is a one-parameter family for any fixed $t \ge 0$. In particular, we have that the deformation class $[\hat{X}^{01} \#_t \hat{X}^{12}, \hat{\l}^{01} \#_{t,s} \hat{\l}^{12}]$ is independent of both $t$ and $s$.
\end{lemma}
\begin{proof}
	We will construct a two-parameter family $\phi_{t,s} : \hat{X}^{01} \# \hat{X}^{12} \to \hat{X}^{01} \#_t \hat{X}^{12}$ of diffeomorphisms, with $\phi_{0,0} = \mathrm{id}$, such that the forms $\phi_{t,s}^*(\hat{\l}^{01} \#_{t,s} \hat{\l}^{12})$ agree with $\hat{\l}^{01} \# \hat{\l}^{12}$ outside of a compact set (depending on $t,s$) and form a smooth family.

	In order to simplify the notation, we fix contact forms $\l^i$ on $(Y^i,\xi^i)$, $i = 0, 1, 2$, so that we can view the symplectization of $Y^i$ as a product $\R \times Y^i$. For $C > 0$ sufficiently large, we can decompose the cobordisms $\hat{X}^{01}$ and $\hat{X}^{12}$ as
	\begin{align}
		\hat{X}^{01} &= (-\infty,1] \times Y^1 \cup \bar{X}^{01} \cup{} [C,\infty) \times Y^0 \\
		\hat{X}^{12} &= (-\infty,-C] \times Y^2 \cup \bar{X}^{12} \cup{} [-1,\infty) \times Y^1
	\end{align}
	where $\bar{X}^{01} \subset \hat{X}^{01}$ is a compact submanifold with boundary $\{1\} \times Y^1 \sqcup \{C\} \times Y^0$, and similarly for $\bar{X}^{12}$. This induces a decomposition of $\hat{X}^{01} \#_t \hat{X}^{12}$ of the form
	\eq
		\hat{X}^{01} \#_t \hat{X}^{12}
		= (-\infty,-C] \times Y^2 \cup \bar{X}^{12} \cup [-1,t + 1] \times Y^1 \cup \bar{X}^{01} \cup{} [C,\infty) \times Y^0
	\eeq
	for any $t \ge 0$. Hence, in order to define $\phi_{t,s}$, it suffices to make a choice of:
	\begin{itemize}
		\item a smooth family of diffeomorphisms $f_t : [-1,1] \to [-1,t+1]$ which coincide with the identity near $-1$ and with translation by $t$ near $1$;
		\item a smooth family of diffeomorphisms $g_{t,s} : [C,\infty) \to{} [C,\infty)$ which coincide with the identity near $C$ and with translation by $-t-s$ at infinity;
		\item a smooth family of diffeomorphisms $h_{t,s} : (-\infty,-C] \to (-\infty,-C]$ which coincide with the identity near $-C$ and with translation by $-s$ at infinity.
	\end{itemize}
	We of course also require that $f_0$, $g_{0,0}$ and $h_{0,0}$ be the identity on their respective domains.
\end{proof}

\begin{proposition}\label{proposition:invariance-gluing-deformation}
	The deformation class of $(\hat{X}^{01} \# \hat{X}^{12}, \hat{\l}^{01} \# \hat{\l}^{12})$ only depends on the deformation classes of $(\hat{X}^{01},\hat{\l}^{01})$ and $(\hat{X}^{12},\hat{\l}^{12})$.
\end{proposition}
\begin{proof}
	Let $(\hat{W}^{01},\hat{\mu}^{01,s})_{s \in [0,1]}$ and $(\hat{W}^{12},\hat{\mu}^{12,s})_{s \in [0,1]}$ be one-parameter families of exact symplectic cobordisms from $(Y^0,\xi^0)$ to $(Y^1,\xi^1)$ and from $(Y^1,\xi^1)$ to $(Y^2,\xi^2)$ respectively. The negative end \eqref{family-marked-negative-end} of $(\hat{W}^{01},\hat{\mu}^{01,s})$ will be denoted by $e_s^- : SY^1 \to \hat{W}^{01}$ and the positive end \eqref{family-marked-positive-end} of $(\hat{W}^{12},\hat{\mu}^{12,s})$ will be denoted by $e_s^+ : SY^1 \to \hat{W}^{12}$. By definition, we can find a neighborhood $U^+ \subset SY^1$ of $+\infty$ and a neighborhood $U^- \subset SY^1$ such that the restriction of $e_s^\pm$ to $U^\pm$ is independent of $s$. This common restriction will be denoted by $e^\pm$.

	Fix a large $t > 0$ so that the intersection $V := \mu_t^{-1}(U^+) \cap U^-$ is nonempty. Let $\hat{W}^{01} \#_V \hat{W}^{12}$ be the space obtained by gluing $\hat{W}^{01} \setminus e^-(U^- \setminus V)$ and $\hat{W}^{12} \setminus e^+(U^+ \setminus \mu^t(V))$ along the maps
	\begin{center}
	\begin{tikzcd}
		V \ar{r}{e^-} \ar{d}{\mu_t} & \hat{W}^{01} \setminus e^-(U^- \setminus V)\\
		U^+ \ar{r}{e^+} & \hat{W}^{12} \setminus e^+(U^+ \setminus \mu^t(V))
	\end{tikzcd}
	\end{center}

	As a smooth manifold, $\hat{W}^{01} \#_V \hat{W}^{12}$ is canonically identified with $\hat{W}^{01} \#_t \hat{W}^{12}$. Thus we can view $\hat{\mu}^{01,s} \#_t \hat{\mu}^{12,s}$ as a Liouville form on $\hat{W}^{01} \#_V \hat{W}^{12}$ for each $s$, and this then makes $(\hat{W}^{01} \#_V \hat{W}^{12}, \hat{\mu}^{01,s} \#_t \hat{\mu}^{12,s})_{s \in [0,1]}$ into a one-parameter family of cobordisms. In particular, it follows that $(\hat{W}^{01} \#_t \hat{W}^{12},\hat{\mu}^{01,0} \#_t \hat{\mu}^{12,0})$ and that $(\hat{W}^{01} \#_t \hat{W}^{12},\hat{\mu}^{01,1} \#_t \hat{\mu}^{12,1})$ are deformation equivalent.
\end{proof}

\begin{corollary}
	There is a well-defined gluing operation on deformation classes of exact symplectic cobordisms given by
	\begin{equation}\label{equation:gluing-deformation-classes}
		[\hat{X}^{01},\hat{\l}^{01}] \# [\hat{X}^{12},\hat{\l}^{12}] = [\hat{X}^{01} \#_t \hat{X}^{12}, \hat{\l}^{01} \#_{t,s} \hat{\l}^{12}]
	\end{equation}
	for any $t \ge 0$ and $s \in \R$.
\end{corollary}

\begin{proposition}
	The gluing operation \eqref{equation:gluing-deformation-classes} is associative.
\end{proposition}
\begin{proof}
	This follows from \Cref{remark:associativity}.
\end{proof}

We now discuss submanifolds in exact symplectic cobordisms. Let $V \sub (Y, \xi)$ be a contact submanifold. There is a canonical embedding
\eq (SV,\l_V) \to (SY,\l_Y) \eeq
which takes a pair $(p, \a_p) \in SV$ to the unique pair $(p, \tilde{\a}_p) \in SY$ such that $\tilde{\a}_p(w)=\a_p(w)$ for some (and hence any) $w \in T_pV \setminus (\xi_p \cap T_pV)$. 

\defi \label{definition:relative-sympcobordism}
	Let $V^+ \subset (Y^+,\xi^+)$ and $V^- \subset (Y^-,\xi^-)$ be contact submanifolds of the same codimension, and let $(\hat{X},\hat{\l})$ be an exact symplectic cobordism from $(Y^+,\xi^+)$ to $(Y^-,\xi^-)$. We say that a smooth submanifold $H \subset \hat{X}$ is \emph{cylindrical} with ends $V^\pm$ if it is closed (as a subset) and there exist neighborhoods $U^\pm \subset SY^\pm$ of $\pm \infty$ such that
	\eq
		(e^\pm)^{-1}(H) \cap U^\pm = SV^\pm \cap U^\pm,
	\eeq
	where $e^\pm : SY^\pm \to \hat{X}$ are the ends \eqref{positive-end}--\eqref{negative-end} of $(\hat{X},\hat{\l})$.

	If $H$ is a \emph{symplectic} cylindrical submanifold of $(\hat{X},\hat{\l})$, then we say that $(\hat{X},\hat{\l},H)$ is an \emph{exact relative symplectic cobordism} from $(Y^+,\xi^+,V^+)$ to $(Y^-,\xi^-,V^-)$.
	Note that in this case, the restrictions of $e^\pm$ to $SV^\pm \cap U^\pm$ endow $(H,\hat{\l}|_H)$ with the structure of an exact symplectic cobordism from $(V^+,\xi^+_{V^+})$ to $(V^-,\xi^-_{V^-})$.
\edefi

\begin{example}
	If $V$ is a contact submanifold of $(Y, \xi)$, then, as noted above, $SV$ can be viewed as a symplectic submanifold of $(SY,\l_Y)$, and $(SY,\l_Y,SV)$ is canonically endowed with the structure of an exact relative symplectic cobordism in the sense of \Cref{definition:relative-sympcobordism} by letting $e^+=e^-=\op{id}$.
\end{example}

\begin{notation} \label{notation:emphasis-relative}
	Let $\hat{X}, \hat{\l}, H$ be as in \Cref{definition:relative-sympcobordism}. As explained in \Cref{notation:choice-forms}, a choice of contact forms $\op{ker} \l^{\pm} = \xi^{\pm}$ endows $(\hat{X}, \hat{\l})$ with the structure of a \emph{strict} relative symplectic cobordism. We analogously speak of a \emph{strict} relative exact symplectic cobordism and write $(\hat{X},\hat{\l},H)_{\l^-}^{\l^+}$ when we wish to emphasize that we are fixing contact forms $\l^{\pm}$ on the ends.
\end{notation}

Let $\L \sub (Y, \xi)$ be a Legendrian submanifold. The \emph{Lagrangian cone} of $\L$ is the Lagrangian submanifold 
\eq L = \{(p, \a) \in SY \sub T^*Y \mid  p \in \L \} \sub (SY, \l_Y). \eeq

\defi \label{definition:lagrangian-cob}
Let $\L^+ \sub (Y^+, \xi^+)$ and $\L^- \sub (Y^-, \xi^-)$ be Legendrian submanifolds and let $(\hat X, \hat \l)$ be an exact symplectic cobordism from $(Y^+, \xi^+)$ to $(Y^-, \xi^-)$. We say that a Lagrangian submanifold $L \sub (\hat X, \hat \l)$ is \emph{cylindrical} with ends $\L^{\pm}$  if it is closed (as a subset) and there exist neighborhoods $U^\pm \subset SY^\pm$ of $\pm \infty$ such that
	\eq
		(e^\pm)^{-1}(L) \cap U^\pm = L^\pm \cap U^\pm,
	\eeq
	where $e^\pm : SY^\pm \to \hat{X}$ are the ends \eqref{positive-end}--\eqref{negative-end} of $(\hat{X},\hat{\l})$ and $L^{\pm}$ are the Lagrangian cones of $\L^{\pm}$.
	
The data of a triple $(\hat X, \hat \l, L)$ is called an \emph{(exact) Lagrangian cobordism} from $(Y^+, \xi^+, \L^+)$ to $(Y^-, \xi^-, \L^-)$. 
\edefi

\begin{definition} \label{definition:homotopy-classes-submanifold}
The set of equivalence classes of cylindrical codimension $2$ submanifolds of 
$\hat X$ with ends $V^\pm$, where two submanifolds are equivalent if they are isotopic via a compactly supported (smooth) isotopy, will be denoted by $\O_{2n-2}(\hat X, V^+ \sqcup V^-)$.
\end{definition}

\defi \label{definition:strong-contact}
A contact submanifold $V \sub (Y,\l)$ is said to be a \emph{strong contact submanifold} if it is (set-wise) invariant under the Reeb flow of $\l$ on $Y$. We will also say that $(\hat{X},\hat{\l},H)_{\l^-}^{\l^+}$ is a \emph{strong} relative exact symplectic cobordism if both $V^+ \subset (Y^+,\l^+)$ and $V^- \subset (Y^-,\l^-)$ are strong contact submanifolds.
\edefi

\begin{definition} \label{definition:gluing-relative-cobordisms}
Let $(\hat{X}^{01},\hat{\l}^{01},H^{01})$ and $(\hat{X}^{12},\hat{\l}^{12},H^{12})$ be exact relative symplectic cobordisms from $(Y^0,\xi^0,V^0)$ to $(Y^1,\xi^1,V^1)$ and from $(Y^1,\xi^1,V^1)$ to $(Y^2,\xi^2,V^2)$ respectively. For any sufficiently large real number $t \ge 0$, $H^{01} \#_t H^{12}$ sits naturally inside $(\hat{X}^{01} \#_t \hat{X}^{12}, \hat{\l}^{01} \#_t \hat{\l}^{12})$ as a symplectic submanifold, and $(\hat{X}^{01} \#_t \hat{X}^{12}, \hat{\l}^{01} \#_t \hat{\l}^{12}, H^{01} \#_t H^{12})$ is a relative cobordism from $(Y^0,\xi^0,V^0)$ to $(Y^2,\xi^2,V^2)$. We will refer to it as the \emph{$t$-gluing} of $(\hat{X}^{01},\hat{\l}^{01},H^{01})$ and $(\hat{X}^{12},\hat{\l}^{12},H^{12})$.
\end{definition}

\begin{definition}
Let $(\hat{X}^1,\hat{\l}^1,H^1)$ and $(\hat{X}^2,\hat{\l}^2,H^2)$ be relative cobordisms from $(Y^+,\xi^+,V^+)$ to $(Y^-,\xi^-,V^-)$. An \emph{isomorphism of exact relative symplectic cobordisms} $\phi : (\hat{X}^1,\hat{\l}^1,H^1) \to (\hat{X}^2,\hat{\l}^2,H^2)$ is an isomorphism $\phi : (\hat{X}^1,\hat{\l}^1) \to (\hat{X}^2,\hat{\l}^2)$ in the sense of \Cref{definition:iso-sympcobordism} which maps $H^1$ diffeomorphically onto $H^2$.
\end{definition}

\begin{example}
Let $(\hat{X},\hat{\l},H)$ be an exact relative symplectic cobordism from $(Y^+,\xi^+,V^+)$ to $(Y^-,\xi^-,V^-)$. Then for any $t \ge 0$, the glued cobordisms $(SY^+, \l_{Y^+}, SV^+) \#_t (\hat{X},\hat{\l},H)$ and $(\hat{X},\hat{\l},H) \#_t (SY^-,\l_{Y^-},SV^-)$ are defined and canonically isomorphic to $(\hat{X},\hat{\l},H)$.
\end{example}

\begin{definition}\label{definition:family-relative-sympcobordisms}
A \emph{one-parameter family of exact relative symplectic cobordisms} from $(Y^+,\xi^+,V^+)$ to $(Y^-,\xi^-,V^-)$ is a manifold $\hat{X}$ equipped with a family of Liouville forms $\{ \hat{\l}^t \}_{t \in I}$, a family of symplectic submanifolds $H^t \subset (\hat{X},\hat{\l}^t)$, and embeddings
	\begin{align}
		e_t^+ : SY^+ &\to \hat{X} \\
		e_t^- : SY^- &\to \hat{X}
	\end{align}
	as in \Cref{definition:relative-sympcobordism}. We will always assume that the family is fixed at infinity, meaning that for every compact subinterval $[a,b] \subset I$,
	\begin{itemize}
		\item  $\{ \hat{\l}^t \}_{t \in [a,b]}$ and $\{ H^t \}_{t \in [a,b]}$ are constant outside of a compact subset of $\hat{X}$;
		\item $\{e_t^+\}_{t \in [a,b]}$ (resp. $\{e_t^-\}_{t \in [a,b]}$) is independent of $t$ on some neighborhood of $+\infty$ in $SY^+$ (resp. of $-\infty$ in $SY^-$).
	\end{itemize}

	Two relative cobordisms $(\hat{X}^0,\hat{\l}^0,H^0)$ and $(\hat{X}^1,\hat{\l}^1,H^1)$ are said to be \emph{deformation equivalent} if there exists a one-parameter family $(\hat{W},\hat{\mu}^t,K^t)_{t \in [0,1]}$ such that $(\hat{X}^0,\hat{\l}^0,H^0)$ is isomorphic to $(\hat{W},\hat{\mu}^0,K^0)$ and $(\hat{X}^1,\hat{\l}^1,H^1)$ is isomorphic to $(\hat{W},\hat{\mu}^1,K^1)$. The deformation class of a cobordism $(\hat{X},\hat{\l},H)$ will be denoted by $[\hat{X},\hat{\l},H]$.
\end{definition}

\begin{example}
Given $(\hat{X}^{01},\hat{\l}^{01},H^{01})$ and $(\hat{X}^{12},\hat{\l}^{12},H^{12})$ as in \Cref{definition:gluing-relative-cobordisms}, the glued cobordisms $(\hat{X}^{01} \#_t \hat{X}^{12}, \hat{\l}^{01} \#_t \hat{\l}^{12}, H^{01} \#_t H^{12})_{t \in{} [N,\infty)}$ form a one-parameter family for $N > 0$ sufficiently large. Similarly, for any fixed $t \gg 0$, $(\hat{X}^{01} \#_t \hat{X}^{12}, \hat{\l}^{01} \#_{t,s} \hat{\l}^{12}, H^{01} \#_t H^{12})_{s \in \R}$ is a one-parameter family. As in \Cref{lemma:t-gluing}, it follows that the deformation class 
\eq [\hat{X}^{01} \#_t \hat{X}^{12}, \hat{\l}^{01} \#_{t,s} \hat{\l}^{12}, H^{01} \#_t H^{12}] \eeq 
is independent of $t \gg 0$ and $s \in \R$.
\end{example}

\begin{proposition}
	The deformation class of $(\hat{X}^{01} \#_t \hat{X}^{12}, \hat{\l}^{01} \#_t \hat{\l}^{12}, H^{01} \#_t H^{12})$ only depends on the deformation classes of $(\hat{X}^{01},\hat{\l}^{01},H^{01})$ and $(\hat{X}^{12},\hat{\l}^{12}, H^{12})$.
\end{proposition}
\begin{proof}
	The proof of \Cref{proposition:invariance-gluing-deformation} also works in the relative case as long as $t > 0$ is chosen large enough.
\end{proof}

\begin{corollary}
	There is a well-defined gluing operation on deformation classes of exact relative symplectic cobordisms given by
	\begin{equation}\label{equation:gluing-relative-deformation-classes}
		[\hat{X}^{01},\hat{\l}^{01},H^{01}] \# [\hat{X}^{12},\hat{\l}^{12},H^{12}] = [\hat{X}^{01} \#_t \hat{X}^{12}, \hat{\l}^{01} \#_{t,s} \hat{\l}^{12}, H^{01} \#_t H^{12}]
	\end{equation}
	for any $t \gg 0$ and $s \in \R$.
\end{corollary}

\begin{proposition}
	The gluing operation \eqref{equation:gluing-relative-deformation-classes} is associative.
\end{proposition}
\begin{proof}
Let $(\hat{X}^{i,i+1},\hat{\l}^{i,i+1},H^{i,i+1})$ be a relative cobordism from $(Y^i,\xi^i,V^i)$ to $(Y^{i+1},\xi^{i+1},V^{i+1})$, $i \in \{0,1,2\}$, and fix $t_1, t_2 \gg 0$. Note that $\bigl( (\hat{X}^{01} \#_{t_1} \hat{X}^{12}) \#_{t_2} \hat{X}^{23}, (H^{01} \#_{t_1} H^{12}) \#_{t_2} H^{23}\bigr)$ and $\bigl( \hat{X}^{01} \#_{t_1} (\hat{X}^{12} \#_{t_2} \hat{X}^{23}), H^{01} \#_{t_1} (H^{12} \#_{t_2} H^{23}) \bigr)$ can be canonically identified as pairs of smooth manifolds. Hence, it suffices to show that there exist $s_1, s_2 \in \R$ such that
	\eq
		(\hat{\l}^{01} \#_{t_1,s_1} \hat{\l}^{12}) \#_{t_2,s_2} \hat{\l}^{23} = \hat{\l}^{01} \#_{t_1,s_1} (\hat{\l}^{12} \#_{t_2,s_2} \hat{\l}^{23}).
	\eeq
	One can easily see from \Cref{definition:gluing-cobordisms} that taking $s_1 = 0$ and $s_2 = -t_2$ works.
\end{proof}

\subsection{Homotopy classes of asymptotically cylindrical maps} \label{subsection:homotopy-classes-maps}

\defi[see Sec.\ 6.1 in \cite{wendlsft}] \label{def-asympt-cyl-map}
Suppose that $(\hat{X}, \hat{\l})$ is an exact symplectic cobordism from $(Y^+, \l^+)$ to $(Y^-, \l^-)$. Given a closed surface $\S$ and finite subsets $\bf{p}^+, \bf{p}^- \subset \S$ (corresponding respectively to positive and negative punctures), a smooth map $u: \S- (\bf{p}^+ \sqcup \bf{p}^-) \to \hat X$ is said to be \emph{asymptotically cylindrical} if it converges exponentially near each puncture $z \in \bf{p}^+ \sqcup \bf{p}^-$ to a trivial cylinder over a Reeb orbit. 
	
More precisely, given any choice of translation invariant metric on $\R \tms Y^{\pm}$, we require that there exists a choice of holomorphic cylindrical coordinates near each $z \in \bf{p}^{\pm}$ such that $u$ takes the form 
\eq\label{equation:asymptotic-formula-example} u(s,t)= \op{exp}_{(P s, \g_z(t))} h(s,t) \eeq 
for $|s|$ large, where $\g_z$ is a Reeb orbit of period $P$ and $h(s,t)$ is a vector field which decays to zero with all its derivatives as $|s| \to \infty$. (i.e.\ these properties hold for $s \gg 0$ if $z \in \bf{p}^+$ and for $s \ll 0$ if $z \in \bf{p}^-$). 
\edefi


\rmk
There is also a notion of an asymptotically cylindrical submanifold which will not be needed in this paper. 
\ermk

\begin{definition}[{cf. \cite[Sec.\ 1.2(I)]{pardon}}]
Let $(\hat{X},\hat{\l})$ be an exact symplectic cobordism from $(Y^+,\l^+)$ to $(Y^-,\l^-)$ and let $\G^\pm$ be a finite set of Reeb orbits in $(Y^\pm,\l^\pm)$.

By truncating the ends of $\hat{X}$, we obtain a compact submanifold $X_0 \subset \hat{X}$ with boundary $\partial X_0 = Y^+ \sqcup Y^-$.
We define the set of homotopy classes $\pi_2(\hat{X}, \G^+ \sqcup \G^-)$ by
\begin{equation} \label{equation:homotopy-classes}
	\pi_2(\hat{X}, \G^+ \sqcup \G^-) := [(S,\partial S),(X_0,\G^+ \sqcup \G^-)]/\Diff(S,\partial S),
\end{equation}
where $S$ is a compact connected oriented surface of genus $0$ equipped with a homeomorphism $\partial S \to \Gamma^+ \sqcup \Gamma^-$, and $\Diff(S,\partial S)$ is the group of diffeomorphisms of $S$ which fix $\partial S$ pointwise. (The notation $[-, -]$ here stands for \emph{homotopy classes} of maps of pairs.)
\end{definition}
\begin{remark}
	The right-hand side of \eqref{equation:homotopy-classes} is independent of the choice of truncation $X_0$ up to canonical bijection.
	In the case where $(\hat{X},\hat{\l}) = (\R \times Y, e^s \lambda)$ is the symplectization of a contact manifold $(Y,\lambda)$, we can take $X_0 = \{0\} \times Y$ and \eqref{equation:homotopy-classes} becomes identical to \cite[(1.2)]{pardon}.
\end{remark}

For any choice of truncation $X_0 \subset \hat{X}$, there is a canonical retraction $\pi : \hat{X} \to X_0$ induced by quotienting by the Liouville flow (more precisely, one should quotient by the backwards Liouville flow at the positive end and by the forwards Liouville flow at the negative end). If $u: \S- (\bf{p}^+ \sqcup \bf{p}^-) \to \hat X$ is an asymptotically cylindrical map, then the composition $\pi \circ u$ can be extended to a map
\begin{equation} \label{equation:compactified-map}
	\bar{u} : (\overline{\S}, \partial \overline{\S}) \to (X_0,\G^+ \sqcup \G^-),
\end{equation}
where $\overline{\S}$ is a compactification of $\S- (\bf{p}^+ \sqcup \bf{p}^-)$ obtained by adding one boundary circle for each puncture. The homotopy class $[u] \in \pi_2(\hat{X}, \G^+ \sqcup \G^-)$ of $u$ is defined to be the equivalence class of \eqref{equation:compactified-map}.

\begin{definition}
Let $(\hat X, \hat \l, L)$ be an exact Lagrangian cobordism from $(Y^+, \l^+, \L^+)$ to $(Y^-, \l^-, \L^-)$ (see \Cref{definition:lagrangian-cob}).

Given a surface with boundary $\S$ and finite subsets $\bf{p}^+, \bf{p}^- \sub \op{int}(\S)$ and $\bf{c}^+, \bf{c}^- \in \d \S$, a smooth map $u: \S- (\bf{p}^{\pm} \cup \bf{c}^{\pm})$ is said to be cylindrical if it converges asymptotically near each interior puncture to a trivial cylinder over a Reeb orbit, and it converges exponentially near each boundary puncture to a trivial strip over a Reeb chord.
\end{definition}

Let $\G^{\pm}$ be a finite set of Reeb orbits in $(Y^\pm, \l^\pm)$ and let $\G_{\L^{\pm}}$ be a finite \emph{ordered} set of Reeb chords of $\L^{\pm} \sub (Y^\pm, \l^\pm)$. We let $\bf{p}^{\pm}$ be a finite set equipped with bijections $\g^{\pm}: \bf{p}^{\pm} \to \G^{\pm}$ and we let $\bf{c}= \bf{c}^+ \sqcup \bf{c}^-$ be a finite ordered set equipped with order-preserving bijections $a^{\pm}: \bf{c}^{\pm} \to \G_{\L^{\pm}}$. Then we let 
\eq \pi_2(\hat X; \G_{\L^+}, \G_{\L^-}, \G^+, \G^- ) \eeq 
be the set of equivalence classes of maps from $\S- (\bf{p}^{\pm} \cup \bf{c}^{\pm})$ to $\hat X$ which are asymptotic to $\g_p^{\pm}$ at $p \in \bf{p}^\pm$ (resp.\ $a_c^{\pm}$ at $c \in \bf{c}^\pm$), where two such maps $u,v$ are equivalent if there exists a compactly supported diffeomorphism $\phi$ of $\S- (\bf{p}^{\pm} \cup \bf{c}^{\pm})$ such that $u$ and $v \circ \phi$ are homotopic (through cylindrical maps).

\section{Reeb dynamics near a codimension $2$ contact submanifold}

\subsection{The Conley-Zehnder index}  \label{subsection:cz-index}

A hermitian vector bundle $(E, J, \o)$ is a vector bundle $E$ together with an almost-complex structure $J$ and a symplectic structure $\o$ such that $J$ is compatible with $\o$. An \emph{asymptotic operator} on a hermitian vector bundle $(E, J, \o)$ over $S^1$ is a real-linear differential operator $\mathbf{A}: \G(E) \to \G(E)$ which, in some (and hence any) unitary trivialization, takes the the form 
\begin{align}
\bf{A}: C^\infty(S^1, \R^{2n}) &\to  C^\infty(S^1, \R^{2n})\\
\eta &\mapsto - J \d_t \eta - S(t) \eta,
\end{align}
where $t \in S^1$ and $S: S^1 \to \op{End}(\R^{2n})$ is a loop of symmetric matrices. The asymptotic operator $\mathbf{A}$ is said to be non-degenerate if $0$ is not an eigenvalue.

Fix a hermitian vector bundle $(E, J, \o)$ over $S^1$ and a unitary trivialization $\tau$. Given a non-degenerate asymptotic operator $\mathbf{A}$, we can obtain a non-degenerate path of symplectic matrices by solving the ordinary differential equation
\eq \label{equation:ode} (-J \d_t - S(t)) \Psi(t)= 0, \; \Psi(0)= \op{id}. \eeq
Conversely, given a non-degenerate path of symplectic matrices, we can recover a non-degenerate asymptotic operator by solving \eqref{equation:ode} for $S(t)$. 

The \emph{Conley-Zehnder index} $\CZ(\mathbf{A})= \CZ(\Psi) \in \Z$ is an integer-valued invariant which can be associated equivalently to a non-degenerate asymptotic operator equipped with a unitary trivialization $\tau$ or to a non-degenerate path of symplectic matrices. It only depends on $\tau$ up to homotopy through unitary trivializations. We refer the reader to \cite[Sec.\ 3.4]{wendlsft} or \cite{guttcz} for a detailed overview of the Conley-Zehnder index.

\defi
Let $(Y, \xi= \op{ker} \l)$ be a contact manifold and let $\g$ be a Reeb orbit of period $P > 0$, parametrized so that $\l(\g') = P$. Given a choice of $d\l$-compatible almost complex structure $J$ on $\xi$, we can define the asymptotic operator $\mathbf{A}_\g : \G(\g^*\xi) \to \G(\g^*\xi)$ by $\mathbf{A}_\g = -J(\nabla_t - P\nabla R_\l)$, where $\nabla$ is some symmetric connection on $Y$.
\edefi

The Conley-Zehnder index of a Reeb orbit $\g$ relative to a trivialization $\tau$ of $\g^*\xi$ will be denoted by $\CZ^\tau(\g):= \CZ^\tau(\mathbf{A}_\g)$. 

Let us now consider a contact manifold $(Y^{2n-1}, \xi= \op{ker} \l)$ and a strong contact submanifold $(V^{2n-3}, \l|_V)$. Observe that the contact distribution splits naturally along $V$ as
\eq
\label{equation:splitting} \xi|_V = \xi|_V^{\top} \oplus \xi|_V^{\perp},
\eeq
where $\xi|_V^{\top}= \xi|_V \cap TV$ and $\xi|_V^{\perp}$ is the symplectic orthogonal complement of $\xi|_V^{\top} \sub \xi|_V$ with respect to $d\l$. Suppose that $J$ is a $d\lambda$-compatible almost complex structure on $\xi$ which respects the splitting \eqref{equation:splitting}. 

Let $\g: S^1 \to V$ be a Reeb orbit. Since $J$ respects the above splitting, then so does the associated asymptotic operator, which we can therefore write as $\mathbf{A}_\g = \mathbf{A}_{\g}^{\top} \oplus \mathbf{A}_{\g}^{\perp}$. If we choose a unitary trivialization $\tau$ of $\xi|_{\g}$ which is also compatible with the splitting, we can define $\CZ^{\tau}_{T}(\g):= \CZ^{\tau}(\mathbf{A}_{\g}^{\top})$ and $\CZ^{\tau}_{N}(\g):= \CZ^{\tau}(\mathbf{A}_{\g}^{\perp})$. We call these respectively the tangential and normal Conley-Zehnder indices of $\g$ with respect to $\tau$. 

We define the integers
\eq
\a_N^{{\tau}; -}(\g) := \lfloor \CZ_N^{\tau}(\g)/2 \rfloor,  \qquad \a_N^{{\tau}; +}(\g) := \lceil \CZ_N^{\tau}(\g)/2 \rceil.
\eeq

Let $p_N(\g)= \a_N^{{\tau}; +}(\g) - \a_N^{{\tau}; -}(\g)\in \{0,1\}$ be the \emph{(normal) parity} of $\g$ and observe that it is independent of the choice of trivialization. We have
\eq \CZ_N^{\tau}(\g)= 2 \a_N^{{\tau}; -}(\g) +p_N(\g) = 2 \a_N^{{\tau}; +}(\g) - p_N(\g) \eeq
from which it also follows that $p_N(\g) \equiv \CZ_N^{\tau}(\g)$ mod $2$. 

There is a canonical isomorphism 
\eq \label{equation:normal-bundle-isom} \xi|_V^{\perp} \xrightarrow{\sim} N_{Y/V}, \eeq
where $N_{Y/V}$ denotes the normal bundle of $V \subset Y$.
If $(V, \xi_V) \sub (Y, \xi)$ are co-oriented and hence oriented, then $N_{Y/V}$ is also oriented and \eqref{equation:normal-bundle-isom} is orientation preserving. If we assume that $N_{Y/V}$ is trivial, then it follows that \eqref{equation:normal-bundle-isom} induces a bijection between homotopy classes of trivializations of $N_{Y/V}$ compatible with the orientation and homotopy classes of unitary trivializations on $\xi|_V^{\perp}$. Since the Conley-Zehnder index only depends on the homotopy class of unitary trivializations, we may define $\CZ^{\tau}_N(\g)$ with respect to any homotopy class of trivializations $\tau$ of $N_{Y/V}$ compatible with the orientation.

\subsection{Normal dynamics and adapted contact forms} \label{subsection:adapted-forms}

We now state some important definitions which will be used throughout this paper.

\defi \label{definition:contact-pair}
A \emph{trivial-normal contact pair} (or just \emph{TN contact pair}) is a datum $(Y, \xi, V)$ consisting of a closed co-oriented contact manifold $(Y, \xi)$ and a co-oriented codimension $2$ contact submanifold $V \sub (Y, \xi)$ with trivial normal bundle $N_{Y/V}$. 
\edefi

An important example of a TN contact pair is the binding of a contact open book decomposition (see \Cref{definition:openbook}).  A choice of (homotopy class of) trivialization $\tau$ on $N_{Y/V}$ is called a \emph{framing} and we say that $(V, \tau) \sub Y$ is a \emph{framed} codimension $2$ submanifold. Note that we do not assume in \Cref{definition:contact-pair} that $V$ and $Y$ are non-empty. For future reference, we let $\xi_{\emptyset}$ denote the unique contact structure on the empty set. 

\defi \label{definition:local-datum}
Given a TN contact pair $(Y, \xi, V)$ with $V$ non-empty, let $\fk{R}(Y, \xi, V)$ be the set of triples $\fk{r}=(\a_V, \tau, r)$ where: 
\begin{itemize}
\item $\a_V \in \O^1(V)$ is a non-degenerate contact form for $(V,\xi_V)$;
\item $\tau$ is a homotopy class of trivializations of $N_{Y/V}$ (i.e.\ a framing) which is compatible with the orientation;
\item $r>0$ is a strictly positive real number, and we have $(1/r) \Z \cap \mc{S}(\a_V)= \emptyset$, where $\mc{S}(\a_V)$ is the action spectrum of $\a_V$.
\end{itemize}

If $V=\emptyset$ (with $Y$ possibly also empty), we define $\fk{R}(Y, \xi, \emptyset)= \{ (\a_{\emptyset}, \tau_{\emptyset}, 0) \}$, where $\a_{\emptyset}, \tau_{\emptyset}$ are understood as a contact form and normal trivialization on the empty set.  
\edefi

\defi  \label{definition:adapted-contact-form}
Given a TN contact pair $(Y, \xi, V)$ and $\fk{r}=(\a_V, \tau, r) \in \fk{R}(Y, \xi, V)$, we say that a contact form $\op{ker} \l = \xi$ is \emph{adapted} to $\fk{r}$ if:
\begin{itemize}
\item $\l$ is non-degenerate; 
\item $\l|_V = \a_V$;
\item $V$ is a strong contact submanifold of $(Y, \l)$ (see \Cref{definition:strong-contact});
\item We have $\CZ_N^{\tau} (\g) = 1 + 2 \lfloor r P_\g \rfloor$  for all Reeb orbits $\g \sub V$, where $P_\g$ is the period of $\g$.
\end{itemize}

In case $V=\emptyset$, any contact form $\l$ is considered to be adapted to the unique element $(\a_{\emptyset}, \tau_{\emptyset}, 0) \in \fk{R}(Y, \xi, \emptyset)$. Given a contactomorphism $f: (Y, \xi, V) \to (Y',\xi', V')$, we write $f_*\fk{r}= (f_*\a_V, f_*\tau, r) \in \fk{R}(Y', \xi', V')$. If $\phi_t: V \to Y$ is an isotopy of contact embeddings where $\phi_0$ is the tautological embedding $V \xrightarrow{\op{id}} V \sub Y$ and $\phi_1(V)= V'$, then $\phi_t$ extends to a family of contactomorphisms $f_t$. We then write $(\phi_1)_*\fk{r}:= (f_1)_* \fk{r}$ (this is independent of the choice of extension). 

We say that $\l$ is \emph{positive elliptic} near $V \neq \emptyset$ if it is adapted to some $\fk{r}=(\a_V, \tau, r) \in \fk{R}(Y, \xi, V)$; we refer to $r > 0$ as the \emph{rotation parameter}.
\edefi

\rmk
Our insistence on allowing the case where $Y=\emptyset$ in the above definitions is explained by the need to treat Liouville manifolds as special cases of Liouville cobordisms in the arguments of \Cref{section:enriched-setups-twisted-counts}. 
\ermk

We will prove in \Cref{proposition:adapted-contact-forms-exist} that adapted contact forms always exist, i.e. for any TN contact pair $(Y, \xi, V)$ and $\fk{r}=(\a_V, \tau, r) \in \fk{R}(Y, \xi, V)$ there exists a contact form adapted to $\fk{r}$. The first step is to construct a suitable local model.

\begin{construction}\label{construction:contact-form}
For $0<\e \leq 1$, let $D^2 \sub \R^2$ be the standard disk of radius $\e$ (in the sequel, we will often denote this disk by $D^2_\e$). Let $(V,\a_V)$ be a contact manifold and let $\phi : D^2 \to \R_{> 0}$ be a smooth positive function which has a nondegenerate critical point at $0$ and satisfies $\phi(0) = 1$. We define
\eq \a_V^\phi = \frac{1}{\phi}(\a_V + \lambda_{D^2}) \eeq
where $\lambda_{D^2} = \frac{1}{2}(x\,dy - y\,dx)$ is the usual Liouville form on $D^2$. This is a contact form on $V \times D^2$ whose restriction to $V = V \times \{0\}$ coincides with $\a_V$. Its Reeb vector field is given by
\eq R_\phi = (\phi - Z_{D^2}\phi) R_V + X_\phi \eeq
where $Z_{D^2} = \frac{1}{2}(x \partial_x + y \partial_y)$ is the Liouville vector field of $\lambda_{D^2}$ and $X_\phi = -(\partial_y\phi)\partial_x + (\partial_x\phi)\partial_y$ is the Hamiltonian vector field of $\phi$ with respect to the symplectic form $\omega_{D^2} = d\lambda_{D^2}$. Our assumptions on $\phi$ imply that $R_\phi = R_V$ on $V \times \{0\}$, so that $(V,\a_V)$ is a strong contact submanifold of $(V \times D^2, \a_V^\phi)$. We will let
\eq S_\phi = \begin{pmatrix} \partial_{xx}\phi(0) & \partial_{yx}\phi(0) \\ \partial_{xy}\phi(0) & \partial_{yy}\phi(0) \end{pmatrix} \in \R^{2\times 2} \eeq 
denote the Hessian of $\phi$ at the origin. Since $S_\phi$ is symmetric and nondegenerate, its eigenvalues are real and nonzero, so its signature $\Sign(S_\phi)$ is one of $0,\pm 2$. We will say that $\phi$ is \emph{hyperbolic} if $\Sign(S_\phi) = 0$, \emph{positive elliptic} if $\Sign(S_\phi) = 2$ and \emph{negative elliptic} if $\Sign(S_\phi) = -2$. In the elliptic case, we define $c_\phi = \sqrt{\det(S_\phi)}/(2\pi)$; this is a positive real number since $\det(S_\phi) > 0$. Finally, we note that the splitting $(\xi_\phi)|_V = (\xi_\phi)|_V^{\top} \oplus (\xi_\phi)|_V^{\perp}$ (see \eqref{equation:splitting}) is given by $(\xi_\phi)|_V^{\top} = \xi_V$ and $(\xi_\phi)|_V^{\perp} = T_0 D^2$. We will let $\tau_\phi$ denote the trivialization of $(\xi_\phi)|_V^{\perp}$ by $\{\partial_x,\partial_y\}$.
\end{construction}

We say that $\a_V^\phi$ is \emph{non-degenerate on $V$} if every Reeb orbit of $\a_V$ is non-degenerate when viewed as a Reeb orbit of $\a_V^\phi$.
\prop \label{proposition:nondegeneracy-local-model}
Carrying over the notation of \Cref{construction:contact-form}, suppose that $\a_V$ is non-degenerate. If $\phi$ is elliptic, then $\a_V^\phi$ is non-degenerate on $V$ if and only if $(1/c_\phi)\Z \cap \mc{S}(\a_V) = \emptyset,$ where $\mc{S}(\a_V)$ denotes the action spectrum of $\a_V$.
\eprop

\pf
Let $\g$ be a Reeb orbit of period $P$ contained in $V$. Recall that $\g$ is non-degenerate if and only if its asymptotic operator is non-degenerate. Choose a trivialization $\tau$ and an almost complex structure $J$ on $(\xi_\phi)|_\g$ which preserve the splitting $(\xi_\phi)|_\g = (\xi_V)|_\g \oplus T_0 D^2$ and coincide with $\tau_\phi$ and $J_0$ respectively on $T_0 D^2$, where $J_0$ denotes the standard almost complex structure on $\R^2 = T_0 D^2$. The asymptotic operator $A_\g$ is compatible with this splitting and can therefore be written as $A_\g = A_{\g}^{\top} \oplus A_{\g}^{\perp}$. The tangential part $A_\g^\top$ is non-degenerate since it coincides with the asymptotic operator of $\g$ as a Reeb orbit in $V$. The normal part $A_\g^\perp$ is given explicitly by
\eq A_\g^\perp = -J_0 \partial_t - P \cdot S_\phi  \eeq
(this follows from a short computation using the formula for the Reeb vector field $R_\phi$ given in \Cref{construction:contact-form}). Define a path $\Psi$ of symplectic matrices by $\Psi(t) = \exp(t P \cdot J_0 S_\phi)$. Then $A_\g^\perp$ is non-degenerate if and only if $\Psi(1)$ doesn't have $1$ as an eigenvalue.
If $\phi$ is elliptic, then the eigenvalues of $\Psi(1)$ are $\exp(\pm i P\sqrt{\det(S_\phi)})$. Hence, $\g$ is non-degenerate if and only if $P\sqrt{\det(S_\phi)}$ is not an integer multiple of $2\pi$, i.e. $P \notin (1/c_\phi)\Z$. It follows that $\lambda_\phi$ is non-degenerate if and only if $(1/c_\phi)\Z \cap \mc{S}(\a_V) = \emptyset$, as claimed.
\epf
The important feature of \Cref{construction:contact-form} is that the normal Conley-Zehnder indices of the Reeb orbits in $V$ can be computed explicitly.

\prop \label{proposition:normal-CZ-indices}
Assume $\a_V^\phi$ is non-degenerate on $V$. If $\phi$ is elliptic, then
\eq \CZ_N^{\tau_\phi}(\g) = \pm(1 + 2\lfloor c_\phi P_\g \rfloor) \eeq
for every Reeb orbit $\g$ contained in $V$, where $P_\g > 0$ denotes the period of $\g$ and the sign is $+$ or $-$ depending on whether $\phi$ is positive elliptic or negative elliptic.
\eprop

\pf
We have $\CZ_N^{\tau_\phi}(\g) = \CZ(\Psi)$, where $\Psi(t) = \exp(t P \cdot J_0 S_\phi)$ is the path of symplectic matrices defined in the proof of \Cref{proposition:nondegeneracy-local-model} (see \cite[Sec.\ 3.4]{wendlsft}). Proposition~41 of \cite{guttcz} implies that
\eq  \CZ(\Psi) = \pm(1 + 2\lfloor c_\phi P \rfloor) \eeq
if $\Sign(S_\phi) = \pm 2$.
\epf

\prop \label{proposition:adapted-contact-forms-exist}
Fix a TN contact pair $(Y,\xi, V)$ and an element $\fk{r}=(\a_V, \tau, r) \in \fk{R}(Y, \xi, V)$. Then there exists a contact form $\l$ on $(Y, \xi)$ which is adapted to $\fk{r}$. 
\eprop

\pf
Let $\phi$ be as in \Cref{construction:contact-form}.  The standard neighborhood theorem for contact submanifolds (see \cite[Thm.\ 2.5.15]{geiges}) implies that the inclusion map $V \to Y$ extends to a contact embedding $\iota : (V \times D^2_\epsilon,\ker(\a_V^\phi)) \to (Y,\xi)$ such that $\iota^*\tau$ is homotopic to $\tau_\phi$, for some $\e>0$ sufficiently small. Hence there exists a contact form $\lambda$ for $\xi$ such that $\iota^*\lambda = \a^\phi_V$ near $V$. In addition to choosing $\phi$ so that $\lambda$ is adapted to $\fk{r}$, we also need to make sure that $\lambda$ can be modified away from $V$ so that it becomes non-degenerate. By \cite[Thm.\ 13]{abw}, this can be achieved by choosing a $\phi$ such that the following two conditions are satisfied:
\begin{itemize}
	\item $\a^{\phi}_V$ is non-degenerate on $V$;
	\item all the Reeb orbits of $\a^{\phi}_V$ in $V \times D^2_\epsilon$ are contained in $V$.
\end{itemize}

Let us set $\phi = 1 + \pi r(x^2 + y^2)$. \Cref{proposition:nondegeneracy-local-model} implies that $\a_V^\phi$ is non-degenerate on $V$ since $c_\phi = \sqrt{(2\pi r)^2}/(2\pi) = r$. Since $R_\phi = R_V + X_\phi$, every Reeb orbit $\g$ of $\a_V^\phi$ is of the form $\g = (\g_V,\g_\phi)$ where $\g_V$ is an orbit of $R_V$ and $\g_\phi$ is an orbit of $X_\phi$ with the same period $P > 0$. From the formula $X_\phi = -2\pi r y \partial_x + 2\pi r x \partial_y$, we see that if $\g_\phi$ were not constant, we would have $P \in (1/r)\Z$, contradicting our assumption on $r$ (see \Cref{definition:local-datum}). Thus $\g$ is contained in $V$.
\epf

\subsection{Open book decompositions} \label{subsection:open-book-decomp}

In this section, we consider normal Reeb dynamics for bindings of open book decompositions. We begin by recalling the definition of an open book decomposition (we refer to \cite[A.1]{ranicki} and \cite[Sec.\ 4.4.2]{geiges} for a historically informed survey of this theory). 

\defi  \label{definition:openbook}
An \emph{open book decomposition} $(Y,B, \pi)$ of a closed, oriented $n$-manifold $Y$ consists of the following data: 
\begin{itemize}
\item[(i)] An oriented, closed, codimension $2$ submanifold $B \sub Y$ with trivial normal bundle.  
\item[(ii)] A fibration $\pi: Y-B \to S^1$ which coincides with the angular coordinate in some neighborhood $B \tms \{0\} \sub B \tms D^2 = B \tms \{ (x,y) \mid x^2+y^2 <1 \}$. 
\end{itemize}
The submanifold $B \sub Y$ is called the \emph{binding} and the fibers of $\pi$ are called \emph{pages}. 
\edefi

Observe that the data of an open book decomposition induces a natural trivialization of the normal bundle to the binding. We also recall what it means for an open book decomposition to support a contact structure.

\defi[see \cite{girouxicm}]	\label{definition:girouxform}
Given an odd-dimensional manifold $Y^{2n-1}$, an open book decomposition $(Y,B, \pi)$ is said to \emph{support} a contact structure $\xi$ if there exists a contact form $\xi= \op{ker} \a$ such that the following properties hold:
\begin{itemize}
\item[(i)] The restriction of $\a$ to $B$ is a contact form.
\item[(ii)] The restriction of $d\a$ to any page $\pi^{-1}(\theta)$ is a symplectic form.
\item[(iii)] The orientation of $B$ induced by $\a$ coincides with the orientation of $B$ as the boundary of the symplectic manifold $(P_{\t}, d\a)$, where $P_{\theta}=\pi^{-1}(\theta)$ is any page. 
\end{itemize}
Such a contact form is called a \emph{Giroux form} (and is also said in the literature to be \emph{adapted} to the open book decomposition).
\edefi

\rmk \label{remark:transverse-reeb}
Condition (ii) in the above definition is equivalent to the Reeb vector field of $\a$ being transverse to the pages.
\ermk

For future convenience, we state the following definition. 
\defi \label{definition:obd-contact-pairs-G}
Let $\mc{G}$ be the set of TN contact pairs $(Y, \xi, V)$ having the property that $\xi$ is supported by an open book decomposition $\pi: Y-V \to S^1$ with binding $V$. 
\edefi

\lem	\label{lemma:g-form}
Let $(Y, B, \pi)$ be an open book decomposition supporting the contact structure $\xi$. Let $\alpha_B$ be a contact form for $(B,\xi_B)$ and let $f : [0,1) \to \R$ be a positive smooth function such that $f(0) = 1$ and $f'(r) < 0$ for $r > 0$. Then there exists a Giroux form $\a$ and an embedding $\phi: B \tms D^2_{\e} \to Y$ (for some small $\epsilon > 0$) with the following properties:
\begin{enumerate}
\item $\a|_B = \a_B$.
\item The projection $\pi \circ \phi$ is given by $(r, \t) \mapsto \t$ on $B \tms D^2_{\e} - B \tms \{0\}$.
\item $\phi^*\a = f(r)(\a|_B + \l_{D^2})$.
\end{enumerate}
\elem
\pf
The proof is in two steps. First, we will show that there exists a Giroux form $\tilde{\a}$ such that $\tilde{\a}|_B= \a_B$. Second, we will construct $\a$ by modifying $\tilde{\a}$ so that (1) -- (3) are satisfied. 

\emph{Step 1: } Let $\a'$ be an arbitrary Giroux form for $(Y, B, \pi)$. Since $\a'|_B$ and $\alpha_B$ define the same contact structure, we can write $\alpha_B = (1 + h)(\a'|_B)$ for some smooth function $h : B \to \R$. We may assume without loss of generality that $h \geq 0$ everywhere, since any positive constant multiple of $\a'$ is also a Giroux form. 

It is shown in the proof of \cite[Prop.\ 2]{dgz} that there exists a tubular neighborhood $B \tms D^2_{\e}$ of the binding on which $\pi = \theta$ and $\a' = g(\a'|_B + \l_{D^2})$, where $g : B \times D^2_\e \to \R$ is a positive smooth function satisfying $g \equiv 1$ on $B \tms \{0\}$,  $\l_{D^2} = \frac{1}{2}(x\,dy - y\,dx)$ and $\e>0$ is a suitably small constant. Note that $g(\a'|_B + \l_{D^2})$ is a Giroux form on $(B \times D^2_\e, B \times \{0\}, \pi = \theta)$ if and only if $\partial g/\partial r < 0$ for $r > 0$.  

Let $\sigma : [0,\epsilon] \to \R$ be a nonincreasing smooth function such that $\sigma(r) = 1$ for $r$ near $0$ and $\sigma(r) = 0$ for $r$ near $\epsilon$. Set $\tilde{g} := (1 + \sigma(r)h)g$. Then $\partial_r \tilde{g} = \partial_r \s   \cdot hg + (1+\sigma h) \cdot \partial_r g  <0$. Now we define $\tilde{\a}$ by replacing $g$ with $\tilde{g}$. 

\emph{Step 2: } By the previous step, we may fix a Giroux form $\tilde{\a}$ so that $\tilde{\a}|_B = \a_B$. Appealing again to the proof of \cite[Prop.\ 2]{dgz}, there exists a tubular neighborhood $B \tms D^2_{\e'}$ of the binding on which $\pi = \theta$ and $\tilde{\a} = \g(\tilde{\a}|_B + \l_{D^2})$, where $\g : B \times D^2_{\e'} \to \R$ is a positive smooth function satisfying $\g \equiv 1$ on $B \tms \{0\}$,  $\l_{D^2} = \frac{1}{2}(x\,dy - y\,dx)$ and $\e'>0$ is a suitably small constant. Again, we have that $\g(\tilde{\a}|_B + \l_{D^2})$ is a Giroux form on $(B \times D^2_{\e'}, B \times \{0\}, \pi = \theta)$ if and only if $\partial \g/\partial r < 0$ for $r > 0$.  

Let $\delta : B \times D^2_{\e'} \to \R$ be a positive smooth function such that $\delta = f$ near $B \times \{0\}$, $\delta = \g$ near $B \times \partial D^2_{\epsilon'}$, and $\partial_r \delta< 0$ for $r > 0$. Let $\a$ be the unique contact form on $Y$ which coincides with $\tilde{\a}$ outside the image of $\phi$ and satisfies $\phi^*\a = \delta(\tilde{\a}|_B + \l_{D^2})$. Then $\a$ is a Giroux form and satisfies conditions (1) -- (3).
\epf

\cor	\label{proposition:girouxstructure}
Consider an open book decomposition $(Y, B, \pi)$ which supports a contact structure $\xi$ and let $\tau$ denote the induced trivialization of the normal bundle of $B \sub Y$. Choose an element $\fk{r}=(\a_B, \tau, r)\in \fk{R}(Y, \xi, B)$. Then there exists a Giroux form $\a$ which is adapted to $\fk{r}$ (see \Cref{definition:adapted-contact-form}). 
\ecor

\pf
Let $\kappa= \pi r$ and define $f(s)= (1 + \kappa s^2)^{-1}$ for $s \in{} [0,1)$. Since $f(0) = 1$ and $f'(s) < 0$, it follows that there exists a Giroux form $\tilde{\a}$ satisfying the conditions stated in \Cref{lemma:g-form}. As we observed in the proof of \Cref{proposition:adapted-contact-forms-exist}, there exists a neighborhood $\mc{U}$ of $B$ with the following properties:
\begin{itemize}
\item $\tilde{\alpha}$ is nondegenerate on $\mc{U}$;
\item all the Reeb orbits in $\mc{U}$ are contained in $B$.
\end{itemize}
According to \cite[Thm.\ 13]{abw}, we can obtain a nondegenerate contact form by multiplying $\tilde{\a}$ by a smooth function $g: Y \to \R_+$ with $g \equiv 1$ near $B$. Moreover, we can assume that $g - 1$ is arbitrarily $C^1$-small and hence that $g\tilde{\a}$ is still a Giroux form.

Since $g\tilde{\a}=\a$ near $B$, it follows that $(g \tilde{\a})|_B= \a_B$ and that $B$ is a strong contact submanifold with respect to $g \tilde{\a}$. Finally, the last point in \Cref{definition:adapted-contact-form} can be verified just as in the proof of \Cref{proposition:normal-CZ-indices}.
\epf

\section{Standard setups and tree categories.}

Contact homology is defined in \cite{pardon} by counting pseudo-holomorphic curves (and more generally, pseudo-holomorphic buildings) in four setups. To keep track of the combinatorics of these curves, Pardon introduces certain categories of decorated trees. We briefly review this formalism here, referring the reader to \cite[Sec.\ 2.1]{pardon} for details. 

\subsection{Standard setups} \label{subsection:standard-setups} 

\begin{setup1}
A datum $\mc{D}$ for Setup I consists of a triple $(Y, \lambda, J)$, where $Y$ is a closed manifold, $\lambda$ is a non-degenerate contact form on $Y$ and $J$ is a $d\lambda$-compatible almost complex structure on $\xi = \ker \lambda$.
\end{setup1}

\begin{setup2}
A datum $\mc{D} = (\mc{D}^+,\mc{D}^-,\hat{X},\hat{\l},\hat{J})$ for Setup II consists of
\begin{itemize}
	\item data $\mc{D}^\pm = (Y^\pm,\l^\pm,J^\pm)$ as in setup I;
	\item an exact symplectic cobordism $(\hat{X},\hat{\l})$ with positive end $(Y^+,\l^+)$ and negative end $(Y^-,\l^-)$;
	\item a $d\hat{\l}$-tame almost complex structure $\hat{J}$ on $\hat{X}$ which agrees with $\hat{J}^\pm$ at infinity.
\end{itemize}
\end{setup2}

\begin{setup3}
A datum $\mc{D} = (\mc{D}^+,\mc{D}^-,(\hat{X},\hat{\l}^t,\hat{J}^t)_{t \in [0,1]})$ for this setting consists of
\begin{itemize}
	\item data $\mc{D}^\pm = (Y^\pm,\l^\pm,J^\pm)$ as in setup I;
	\item a family of exact symplectic cobordisms $(\hat{X},\hat{\l}^t)_{t \in [0,1]}$  with positive end $(Y^+,\l^+)$ and negative end $(Y^-,\l^-)$;
	\item a $d\hat{\l}^t$-tame almost complex structure $\hat{J}^t$ on $\hat{X}$ which agrees with $\hat{J}^\pm$ at infinity.
\end{itemize}
Note that for every $t_0 \in [0,1]$, there is a datum $\mc{D}^{t = t_0} = (\mc{D}^+,\mc{D}^-,\hat{X},\hat{\l}^{t_0},\hat{J}^{t_0})$ as in Setup II.
\end{setup3}

\begin{setup4}
A datum $\mc{D} = (\mc{D}^{01},\mc{D}^{12},(\hat{X}^{02,t},\hat{\l}^{02,t},\hat{J}^{02,t})_{t \in{} [0,\infty)})$ for this setting consists of
\begin{itemize}
	\item data
	\begin{align*}
		\mc{D}^{01} &= (\mc{D}^0,\mc{D}^1,\hat{X}^{01},\hat{\l}^{01},\hat{J}^{01}) \\
		\mc{D}^{12} &= (\mc{D}^1,\mc{D}^2,\hat{X}^{12}, \hat{\l}^{12},\hat{J}^{12})
	\end{align*} as in setup II, where $\mc{D}^i = (Y^i,\l^i,J^i)$, $i = 0,1,2$;
	\item a family of exact symplectic cobordisms $(\hat{X}^{02,t},\hat{\l}^{02,t})_{t \in{} [0,\infty)}$ with positive end $(Y^0,\l^0)$ and negative end $(Y^2,\l^2)$, which for $t$ large coincides with the $t$-gluing of $(\hat{X}^{01},\hat{\l}^{01})$ and $(\hat{X}^{12},\hat{\l}^{12})$;
	\item a $d\hat{\l}^{02,t}$-tame almost complex structure $\hat{J}^{02,t}$ on $\hat{X}^{02,t}$ which agrees with $\hat{J}^0$, $\hat{J}^2$ at infinity and is induced by $\hat{J}^{01}$, $\hat{J}^{02}$ for $t$ large.
\end{itemize}
\end{setup4}

\subsection{Trees}\label{subsection:trees}


For each setup, Pardon \cite[Sec.\ 2.1]{pardon} defines a category $\Sch_*$ ($* = \I, \II, \III, \IV$) which depends on some datum $\mc{D}$. Each object $T \in \Sch$ is a decorated tree (or forest) representing a certain class of pseudo-holomorphic curves (or more generally of buildings). Geometrically, the vertices correspond to curves, edges correspond to asymptotic orbits, and the decorations keep track of additional information (such as the homology classes of the components, and their ``level" in the SFT compactification). 

A morphism of trees in $\Sch_*$ consists of two pieces of data. First, a contraction of some edges, with the important caveat that only certain contractions are allowed which depend on the decorations on the tree. Second, one specifies some additional data on external edges which depends on the decorations of the external edges. Geometrically, a morphism of trees correspond to gluing holomorphic curves, and the data which one specifies on the external edges encores different ways of moving asymptotic markers. For $T \in \Sch_*$, we let $\op{Aut}(T)$ denote the group of automorphisms of $T$. Given a morphism $T' \to T$, we let $\op{Aut}(T'/T) \sub \op{Aut}(T')$ be the subgroup of automorphisms of $T'$ which are compatible with $T' \to T$. 

In each category $\Sch_*$, there is an operation called \emph{concatenation} whose input is a collection trees (satisfying certain conditions, and with additional matching data), and whose output is a single tree. Geometrically, concatenations of trees correspond to ``stacking" holomorphic buildings. The precise rules for concatenations are rather involved and depend on the individual setups. 

\rmk
A datum $\mc{D}$ for setups II, III, IV determines multiple categories of trees: this is because such a datum itself contains (by definition) data for multiple setups. \emph{We always follow the notation of \cite[Sec.\ 2.1]{pardon} to denote the resulting tree categories.} So, for example, if $\mc{D} = (\mc{D}^+,\mc{D}^-,\hat{X},\hat{\l},\hat{J})$ is a datum for Setup II, we write $\Sch_\II:= \Sch_\II(\mc{D}), \Sch_\I^+:= \Sch_\I(\mc{D}^+)$ and $\Sch_\I^-:= \Sch_\I(\mc{D}^-)$.  Similarly, a datum for Setup III determines categories $\Sch_\III, \Sch_\II^{t=0}, \Sch_\II^{t=1}, \Sch_\I^\pm$, and a datum for Setup IV determines categories $\Sch_\IV, \Sch_\II^{01}, \Sch_\II^{12}, \Sch_\I^0,\Sch_\I^1,\Sch_\I^2.$
\ermk

\subsection{Virtual moduli counts} 
To an object $T \in \Sch_*$, we can associate a moduli space $\M(T)$ \cite[Sec.\ 2.3]{pardon} which carries an action of $\op{Aut}(T)$ (this action corresponds geometrically to changing asymptotic markers). Note that $T \in \Sch_*$ has a well-defined notion of index and virtual dimension \cite[Def.\ 2.42]{pardon}. The compactified moduli space $\Mbar(T)$ is defined by (see \cite[Def.\ 2.13]{pardon})
\eq
	\Mbar(T):=\bigsqcup_{T'\to T}\M(T')/\!\Aut(T'/T).
\eeq

Theorem 1.1 in \cite{pardon} provides a perturbation datum $\theta \in \Theta_*(\mc{D})$ and associated virtual moduli counts $\#\Mbar(T)^{\vir}_{\theta} \in \Q$ (which are zero for $\vdim(T) \neq 0$) satisfying the \emph{master equations}
\begin{align} 
0 &= \sum_{\codim(T'/T)=1}\frac 1{\left|\Aut(T'/T)\right|}\#\Mbar(T')^{\vir}_{\theta}  \label{eq-boundary-zero}\\
\#\Mbar(\#_iT_i)^{\vir}_{\theta} &= \frac 1{\left|\Aut(\{T_i\}_i/\#T_i)\right|}\prod_i\#\Mbar(T_i)^{\vir}_{\theta} \label{eq-concatenation-additive}
\end{align}
By standard arguments, this can be used to define the various maps involved in the definition of contact homology (e.g.\ the differential $d$) and show that they satisfy the expected relations (e.g.\ $d^2 = 0$); see \cite[Sec.\ 1.7]{pardon}.

\section{Intersection theory for punctured holomorphic curves}

\subsection{Definition of the Siefring intersection number} \label{subsection:def-intersection-number}

We will make use in this paper of an intersection theory for asymptotically cylindrical maps and submanifolds. The four-dimensional theory was constructed by Siefring \cite{siefring} and assigns an integer to a pair of asymptotically cylindrical maps in a $4$-dimensional symplectic cobordism  (see also the book by Wendl \cite{wendlintersection}). The higher-dimensional theory, also due to Siefring, assigns an integer to the pairing of a codimension $2$ (asymptotically) cylindrical hypersurface with a (asymptotically) cylindrical map. A detailed overview can be found in \cite{moreno-siefring}. 

Let us consider a strong exact symplectic cobordism $(\hat X, \hat \l)$ from $(Y^+, \l^+)$ to $(Y^-, \l^-)$. Let $(V^{\pm},\l^{\pm}|_V) \sub (Y^{\pm}, \l^{\pm})$ be strong contact submanifolds and let $H \sub \hat X$ be a codimension $2$ submanifold with cylindrical ends $V^+ \sqcup V^-$. 

We let $\tau$ denote a choice of trivialization of $\xi^{\pm}|_{V^\pm}^{\perp}$ along every Reeb orbit in $V^\pm$. We require that the trivialization along a multiply covered orbit be pulled back from the chosen trivialization along the underlying simple orbit.
Let $u: \S - (\bf{p}_u^+ \sqcup \bf{p}_u^-) \to \hat X$ be a map which is positively/negatively 
asymptotic at $z \in \bf{p}_u^{\pm}$ to the Reeb orbit $\g_z$. Now set 
\eq\label{equation:perturb-depends} u \bullet_{\tau} H:= u^{\tau} \cdot H,\eeq
where $u^{\tau}$ is a perturbation of $u$ which is transverse to $H$ and constant with respect to $\tau$
at infinity, and $(- \cdot - )$ is the usual algebraic intersection number for transversely intersecting smooth maps. While \eqref{equation:perturb-depends} depends on the choice of trivialization $\tau$, Siefring showed that this count can be corrected so as to become independent of $\tau$. This leads to the following definition. 

\defi[see Sec.\ 2 in \cite{moreno-siefring}] \label{definition:generalized-intersection-number}
The \emph{generalized (or Siefring) intersection number} $u * H \in \Z$ of $u$ and $H$ is defined by
\eq 
u * H = u \bullet_{\tau} H + \sum_{ z \in \bf{p}_u^+} \a_N^{\tau; -}(\g_z) - 
		\sum_{ z \in \bf{p}_u^-} \a_N^{\tau; +}(\g_z)
\eeq
\edefi
\prop \label{proposition:invariance-homotopy}
The intersection number $u * H$ only depends on the equivalence classes of $u$ in $\pi_2(\hat{X}, \G^+ \sqcup \G^-)$ and $H$ in $\O_{2n-2}(\hat X, V^+ \sqcup V^-)$ (see \Cref{definition:homotopy-classes-submanifold}).
\eprop

\pf 
The intersection number $u^{\tau} \cdot H$ is clearly invariant under compactly supported isotopies of $H$.

Given a truncation $X_0 \subset \hat{X}$, we can proceed as in \Cref{subsection:homotopy-classes-maps} to associate to $u^\tau$ a map $\bar{u}^\tau : \overline{\S} \to X_0$. Let $H_0 = H \cap X_0$. If we choose $X_0$ sufficiently large (so that $H$ is cylindrical in its complement), then $H_0$ will be a submanifold with boundary $\partial H_0 = H_0 \cap \partial X_0 = V^+ \sqcup V^-$. Note that $\bar{u}^\tau \cdot H_0$ only depends on $[u] \in \pi_2(\hat{X}, \G^+ \sqcup \G^-)$. Moreover, we have $\bar{u}^\tau \cdot H_0 = u^\tau \cdot H$; indeed, if $X_0$ is sufficiently large, then the intersections of $\bar{u}^\tau$ with $H_0$ are exactly the same as those of $u^\tau$ with $H$.
\epf

\subsection{Positivity of intersection} \label{subsection:positivity-intersection}

We now discuss positivity of intersection for the Siefring intersection number. Given a contact manifold $(Y, \xi= \op{ker} \l)$ and an almost-complex structure $J$ on $\xi$, we adopt the usual convention of letting $\hat{J}$ denote the induced almost complex structure on the symplectization. An almost complex structure on a cobordism $(\hat{X},\hat{\l})$ between two contact manifolds $(Y^\pm, \l^\pm)$ is called \emph{cylindrical} if it agrees at infinity with $\hat{J}^\pm$ for some choice of $d\lambda^\pm$-compatible almost complex structures $J^\pm$ on $\ker(\lambda^\pm)$. 

\prop[see Cor.\ 2.3 and Thm.\ 2.5 in \cite{moreno-siefring}] \label{prop-positivity-intersections}
Let $(\hat{X},\hat{\l})$ be an exact symplectic cobordism from $(Y^+,\l^+)$ to $(Y^-,\l^-)$. Let $u$ and $H$ denote an asymptotically cylindrical map and a cylindrical submanifold of codimension $2$ in $\hat{X}$ respectively.
	
Suppose that $u$ and $H$ are $\hat J$-holomorphic for some cylindrical almost-complex structure $\hat J$ on $\hat X$ which is compatible with $d\hat{\l}$. If the image of $u$ is not contained in $H$, then $\op{Im}(u) \cap H$ is a finite set and 
\eq
u * H \ge u \cdot H.
\eeq
(Note that by ordinary positivity of intersection for two pseudo-holomorphic submanifolds, this implies that $u * H \ge 0$ and that $\op{Im}(u)$ and $H$ are disjoint if $u * H = 0$.)
\eprop

When the image of $u$ is contained in $H$, positivity of intersection does not hold. The following computation, which will be useful to us later, is one example of this. The notation $\widehat{\g}$ refers to the trivial cylinder $\R \times S^1 \to \hat{Y}$ over the Reeb orbit $\g$; similarly, $\hat{V} = \R \times V \subset \hat{Y}$ is the cylinder over the strong contact submanifold $V$.

\cor \label{corollary:intersection-cylinders}
Let $\g$ be a Reeb orbit in $Y$. If $\g$ is contained in $V$, then 
\eq \widehat{\g} * \hat{V} = -p_N(\g).\eeq
\ecor

\pf
By definition, 
\eq \widehat{\g} * \hat{V} = \widehat{\g}^\tau \cdot \hat{V} + \a_N^{\tau; -}(\g)  -\a_N^{\tau; +}(\g).\eeq
We can choose the perturbation $\widehat{\g}^\tau$ so that its image is disjoint from $\hat{V}$. The result follows since $\a_N^{\tau; +}(\g)  -\a_N^{\tau; -}(\g)= p_N(\g)$ by definition.
\epf

\rmk \label{remark:intersection-disjoint-cylinders}
If $\g$ is disjoint from $V$, then $\widehat{\g} * \hat{V} = 0$.
\ermk

\Cref{corollary:intersection-cylinders} shows that positivity of intersection fails for curves contained in $\hat{V}$. However, we still have a lower bound on the intersection number $u * \hat{V}$ when $u = \widehat{\g}$ is a trivial cylinder (namely, $\widehat{\g} * \hat{V} \ge -1$ since $p_N(\g) \in \{0,1\}$). In the remainder of this section, we show that if $(Y, \xi, V)$ is a TN contact pair and $\l$ is positive elliptic near $V$, then the intersection number $u * \hat{V}$ is bounded below for \emph{all} asymptotically cylindrical curves $u$ contained in $\hat{V}$. We also give an analoguous result for cylindrical submanifolds of symplectic cobordisms $H \subset (\hat{X},\omega)$ with trivial normal bundle.

\prop \label{proposition:positivity-symplectization}
Fix a TN contact pair $(Y, \xi, V)$ and a datum $\fk{r} = (\a_V, \tau, r) \in \fk{R}(Y, \xi, V)$. Consider a contact form $\l$ on $(Y, \xi)$ which is adapted to $\fk{r}$, and an almost-complex structure $J$ on $\xi$ which is compatible with $d\l$ and which preserves $\xi_V$. Suppose that $u$ is a $\hat{J}$-holomorphic curve whose image is entirely contained in $\hat{V}$. If $\lambda$ is positive elliptic near $V$, then $u * \hat{V} \ge 1 - p_u$, where $p_u$ denotes the number of punctures (positive and negative) of $u$.
\eprop

\pf
We have by definition that
\eq \a_N^{\tau; -}(\g_z) = \lfloor \CZ_N^\tau(\g_z)/2 \rfloor = \lfloor rP_z \rfloor \eeq
for $z \in \bf{p}_u^+$ and
\eq \a_N^{\tau; +}(\g_z) = \lceil \CZ_N^\tau(\g_z)/2 \rceil = 1 + \lfloor r P_z \rfloor \eeq
for $z \in \bf{p}_u^-$ (here $P_z$ denotes the period of the Reeb orbit $\g_z$). Using the trivial bounds $x - 1 < \lfloor x \rfloor \le x$ and the fact that $u^\tau \cdot \hat{V} = 0$, we obtain
\eq
	u * \hat{V} > \sum_{ z \in \bf{p}_u^+} (r P_z - 1) - \sum_{ z \in \bf{p}_u^-} (1 + r P_z)
	\ge -p_u + r \left(\sum_{ z \in \bf{p}_u^+} P_z - \sum_{ z \in \bf{p}_u^-} P_z \right).
\eeq
The fact that $u$ is $\hat{J}$-holomorphic implies that $\sum_{ z \in \bf{p}_u^+} P_z - \sum_{ z \in \bf{p}_u^-} P_z$ is nonnegative (see \cite[p.\ 60]{wendlsft}). Thus $u * \hat{V} \ge 1 - p_u$ as desired.
\epf

We will need an analogue of \Cref{proposition:positivity-symplectization} for cobordisms. Note that if $V \subset Y$ is a codimension $2$ contact submanifold, then the normal bundle of $\hat{V} = \R \times V \subset \R \times Y = \hat{Y}$ can be identified with the pullback of $\xi|_V^\perp$ under the projection $\hat{V} \to V$. Hence, any trivialization $\tau$ of $\xi|_V^\perp$ induces a trivialization of the normal bundle of $\hat{V}$, which we will denote by $\hat{\tau}$.

\prop \label{proposition:positivity-cobordism}
	Fix TN contact pairs $(Y^\pm, \xi^\pm, V^\pm)$ and elements $\fk{r}^\pm= (\a_V^{\pm}, \tau^\pm, r^\pm) \in \fk{R}(Y^\pm, \xi^\pm, V^\pm)$. Let $\l^{\pm}$ be contact forms on $(Y^\pm, \xi^\pm)$ which are adapted to $\fk{r}^\pm$, and let $(\hat{X},\hat{\l},H)_{\l^-}^{\l^+}$ be a strong relative symplectic cobordism from $(Y^+, \xi^+, V^+)$ to $(Y^-, \xi^-, V^-)$. We assume that there exists a global trivialization $\tau$ of the normal bundle of $H$ which coincides with $\hat{\tau}^\pm$ near $\pm \infty$.
	
Let $\hat{J}$ be an almost-complex structure on $\hat{X}$ which is cylindrical and compatible with $d\l^\pm$ outside a compact set, and such that $H$ is $\hat{J}$-holomorphic. Let $u$ be an asymptotically cylindrical map in $\hat{X}$ which is $\hat{J}$-holomorphic and whose image is entirely contained in $H$. Following the notation of \Cref{proposition:positivity-symplectization}: if $\lambda^+$ is positive elliptic near $V^+$, then $u*H >-p_u + r^+ \sum_{z \in \bf{p}_u^+} P_z - r^- \sum_{z \in \bf{p}_u^-} P_z$. In particular, $u* H \geq 0$ if $u$ has no negative puncture.
\eprop

\pf
We have
\eq
u * H = u^\tau \cdot H + \sum_{ z \in \bf{p}_u^+} \a_N^{\tau; -}(\g_z) - \sum_{ z \in \bf{p}_u^-} \a_N^{\tau; +}(\g_z).
\eeq
We can choose the perturbation $u^\tau$ so that it is disjoint from $H$, so $u^\tau \cdot H = 0$. We now argue as in \Cref{proposition:positivity-symplectization} to find that 
\eq u*H >  \sum_{ z \in \bf{p}_u^+} (r^+ P_z - 1) - \sum_{ z \in \bf{p}_u^-} (1 + r^- P_z) \geq -p_u +r^+ \sum_{ z \in \bf{p}_u^+} P_z - r^- \sum_{ z \in \bf{p}_u^-} P_z.\eeq
\epf

\subsection{The intersection number for buildings} \label{subsection:intersection-buildings}


In this section, we use Siefring's intersection theory to define an intersection number for buildings of asymptotically cylindrical maps and buildings of asymptotically cylindrical codimension $2$ submanifolds. Since the differences between $\Sch_\I$, $\Sch_\II$, $\Sch_\III$ and $\Sch_\IV$ don't matter for this purpose, we start by defining a category $\widehat{\Sch}$ of labeled trees which only keeps track of the information needed for intersection theory (in particular, there are obvious ``forgetful'' functors $\Sch_* \to \widehat{\Sch}$).

The category $\widehat{\Sch} = \widehat{\Sch}(\{\hat{X}^{ij}\}_{ij})$ depends on the following data:
\begin{itemize}
	\item[(i)] An integer $m \ge 0$ and a collection of $m + 1$ co-oriented contact manifolds $(Y^i,\xi^i)$, each equipped with a choice of contact form $\l^i$ ($0 \le i \le m$).
	\item[(ii)] For each pair of integers $0 \le i \le j \le m$, an exact symplectic cobordism $(\hat{X}^{ij},\hat{\l}^{ij})$ with positive end $(Y^i,\xi^i)$ and negative end $(Y^j,\xi^j)$. We require that $\hat{X}^{ii} = SY^i$ be the symplectization of $Y^i$ and that $\hat{X}^{ik} = \hat{X}^{ij} \# \hat{X}^{jk}$ for $i \le j \le k$ (this makes sense in light of \Cref{remark:associativity}).
\end{itemize}

An object $T \in \widehat{\Sch}$ is a finite directed forest (i.e. a finite collection of finite directed trees). We require that every vertex has a unique incoming edge. Edges which are adjacent to only one vertex are allowed; we will refer to them as input or output edges depending on whether they are missing a source or a sink. The other edges will be called interior edges. We also have the following decorations:
\begin{itemize}
	\item For each edge $e \in E(T)$, a symbol $*(e) \in \{0,\dots,m\}$ such that $*(e) = 0$ for input edges and $*(e) = m$ for output edges, together with a Reeb orbit $\g_e$ in $(Y^{*(e)},\l^{*(e)})$.
	\item For each vertex $v \in V(T)$, a pair $*(v) = (*^+(v),*^-(v)) \in \{0,\dots,m\}^2$ such that $*^+(v) \le *^-(v)$ and a homotopy class $\beta_v \in \pi_2(\hat{X}^{*(v)}, \g_{e^+(v)} \sqcup \{\g_{e^-}\}_{e^- \in E^-(v)})$, where $e^+(v)$ denotes the unique incoming edge of $v$ and $E^-(v)$ denotes the set of its outgoing edges. We require that $*(e^+(v)) = *^+(v)$ and $*(e^-) = *^-(v)$ for every $e^- \in E^-(v)$. 
\end{itemize}

\rmk
Geometrically, these decorations specify how different curves and orbits fit together to form a holomorphic building. For example, suppose $m=1$ and $T$ is a tree with one vertex $v$ and one input edge $e$. If $*(e)=0$ and $*(v)=00$, then $T$ describes a curve in $SY^0$ with one positive puncture and no negative punctures. If $*(e)=0$ and $*(v)=01$, then $T$ describes a curve in $\hat{X}^{01}$ with one positive puncture and no negative punctures. This labeling scheme of course follows \cite[Sec.\ 2.1]{pardon}.
\ermk

We will let $\G_T^+$ and $\G_T^-$ denote the collections of Reeb orbits associated to the input and output edges of an object $T \in \widehat{\Sch}$. In the case where $T$ is a tree, the unique element of $\G_T^+$ will be denoted by $\g_T^+$.

A morphism $\pi : T \to T'$ consists of a contraction of the underlying forests (meaning that $T'$ is identified with the forest obtained by contracting a certain subset of the interior edges of $T$) subject to the following conditions:
\begin{itemize}
	\item For every non-contracted edge $e \in E(T)$, we require that $*(\pi(e)) = *(e)$ and $\g_{\pi(e)} = \g_e$.
	\item For every vertex $v \in V(T)$, we have $*^+(\pi(v)) \le *^+(v)$ and $*^-(\pi(v)) \ge *^-(v)$.
	\item For every vertex $v' \in V(T')$, we require that $\beta_{v'} = \#_{\pi(v) = v'} \beta_v$.
\end{itemize}
Note that for any morphism $T \to T'$, we have $\G_T^+ = \G_{T'}^+$ and $\G_T^- = \G_{T'}^-$.
\rmk \label{rmk:maximal-contraction}
For every $T \in \widehat{\Sch}$, we get a morphism $T \to T_\mathrm{max}$ by contacting all of the interior edges of $T$. Each component of $T_\mathrm{max}$ is a tree with a unique vertex. In the case where $T$ is connected, we will write $\beta_T := \#_v \b_v \in \pi_2(\hat{X}^{0m},\g_T^+ \sqcup \G_T^-)$ for the homotopy class labeling the unique vertex of $T_\mathrm{max}$. Note that for every morphism $T \to T'$, we have $T_\mathrm{max} = T'_\mathrm{max}$. In particular, if $T$ and $T'$ are trees, then $\b_T = \b_{T'}$.
\ermk

\defi
Let $T \in \widehat{\Sch}$ and let $\{T_i\}_i$ denote its connected components. The \emph{intersection number} $T * H$ of $T$ with a codimension $2$ cylindrical submanifold $H \subset \hat{X}^{0m}$ is defined to be
\eq T * H = \sum_{i} \beta_{T_i} * H. \eeq
\edefi

By \Cref{proposition:invariance-homotopy}, this intersection number only depends on the class of $H$ in $\O_{2n-2}(\hat{X}^{0m}, V^0 \sqcup V^m)$.
By \Cref{rmk:maximal-contraction}, it is ``invariant under gluing'':
\begin{proposition} \label{proposition:invariance-gluing}
	Let $T, T' \in \widehat{\Sch}$. If there exists a morphism $T \to T'$, then $T * H = T' * H$.
\end{proposition}

Suppose now that $m = 0$, so that objects $T \in \widehat{\Sch}$ represent buildings of curves in the symplectization $\hat{Y}$ of a single contact manifold $(Y,\l) := (Y^0,\l^0)$, and that $H = \hat{V} := \R \times V$ is the trivial cylinder over some strong contact submanifold $V \subset Y$ of codimension $2$. In that case, the intersection number $T * H$ can be expressed more explicitly as follows.

\prop \label{proposition:additivity-symplectization}
	For any $T \in \widehat{\Sch}$, we have
	\eq \label{equation:additivity-symplectization}
	T * \hat{V} = \sum_{v \in V(T)} \beta_{v} * \hat{V} - \sum_{e \in E^\mathrm{int}(T)} \widehat{\g}_e * \hat{V}.
	\eeq
\eprop
\pf
	The proof will be by induction on the number of interior edges. If this number is zero, then \eqref{equation:additivity-symplectization} is true by definition. Otherwise, pick an edge $e \in E^\mathrm{int}(T)$ and contract it to obtain a morphism $\pi : T \to T'$ where $T'$ has one less interior edge than $T$. We can assume inductively that $T'$ satisfies \eqref{equation:additivity-symplectization}. Since $T * \hat{V} = T' * \hat{V}$, it suffices to show that
	\eq
		\beta_{v^+} * \hat{V} + \beta_{v^-} * \hat{V} - \widehat{\g}_e * \hat{V} = \beta_{v'} * \hat{V},
	\eeq
	where $v^+$ and $v^-$ are the source and sink of $e$ respectively and $v' = \pi(v^+) = \pi(v^-)$.

	To do this, start by picking curves $u_\pm : \Sigma^\pm \to \hat{Y}$ representing the classes $\beta_{v^\pm}$. Fix a choice of cylindrical coordinates near the positive puncture of $u^-$ and near the negative puncture of $u^+$ corresponding to $e$. We can assume that $u_\pm$ is cylindrical at infinity, so that there exists a constant $C > 0$ such that
	\eq
		u_\pm(s,t) = (P_e \cdot s, \g_e(t))
	\eeq
	for $\mp s \ge C$, where $P_e$ is the period of the orbit corresponding to $e$. Now let
	\begin{align}
		\overline{\Sigma}^+ &= \Sigma^+ \setminus ((-\infty,-3C) \times S^1), \\
		\overline{\Sigma}^- &= \Sigma^- \setminus ((3C,\infty) \times S^1),
	\end{align}
	and let $\Sigma = \Sigma^+ \# \Sigma^-$ be obtained by identifying $[-3C,-C] \times S^1 \subset \overline{\Sigma}^+$ with $[C,3C] \times S^1 \subset \Sigma^-$ via translation by $4C$. The curve $u_+ \# u_- : \Sigma \to \hat{Y}$ which is given by $\tau_{2CP_e} \circ u_+$ on $\overline{\Sigma}^+$ and $\tau_{-2CP_e} \circ u_-$ on $\overline{\Sigma}^-$ (where $\tau_s : \hat{Y} \to \hat{Y}$ denotes translation by $s$) then represents the homotopy class $\b_{v^+} \# \b_{v^-} = \b_{v'}$.

	Choose a trivialization $\tau$ of $\xi|_V^{\perp}$ along the relevant Reeb orbits and use it to produce perturbations $u_\pm^\tau$, $(u_+ \# u_-)^\tau$ as in section~\ref{subsection:def-intersection-number}. We can do this in such a way that $(u_+ \# u_-)^\tau$ is obtained by gluing $u_+^\tau$ and $u_-^\tau$. Then
	\eq (u_+ \# u_-)^\tau \cdot \hat{V} = u_+^\tau \cdot \hat{V} + u_-^\tau \cdot \hat{V}, \eeq
	so
	\begin{align*}
		(u_+ \# u_-) * \hat{V}
		&= u_+^\tau \cdot \hat{V} + u_-^\tau \cdot \hat{V} + \a_N^{\tau; -}(\g^+) - \sum_{ z \in \bf{p}_{u_+ \# u_-}^-} \a_N^{\tau; +}(\g_z) \\
		&= u_+^\tau \cdot \hat{V} + u_-^\tau \cdot \hat{V} + \a_N^{\tau; -}(\g^+) + \a_N^{\tau; +}(\g_e) - \sum_{ z \in \bf{p}_{u_+}^-} \a_N^{\tau; +}(\g_z) - \sum_{ z \in \bf{p}_{u_-}^-} \a_N^{\tau; +}(\g_z) \\
		&= u_+ * \hat{V} + u_- * \hat{V} + \a_N^{\tau; +}(\g_e) - \a_N^{\tau; -}(\g_e) \\
		&= u_+ * \hat{V} + u_- * \hat{V} + p_N(\g_e).
	\end{align*}
	We have $p(\g_e) = -\widehat{\g}_e * \hat{V}$ by \Cref{corollary:intersection-cylinders}, so this implies that
	\eq \beta_{v'} * \hat{V} = (u_+ \# u_-) * \hat{V} = u_+ * \hat{V} + u_- * \hat{V} - \widehat{\g}_e * \hat{V} = \beta_{v^+} * \hat{V} + \beta_{v^-} * \hat{V} - \widehat{\g}_e * \hat{V}, \eeq
	as desired.
\epf

\defi \label{defi:rep-buildings-symplectization}
	Given $T \in \widehat{\Sch}$, we say that $T$ is \emph{representable by a holomorphic building} if there exists a $d\l$-compatible almost complex structure $J$ on $\xi$ such that, for every vertex $v \in V(T)$, $\beta_v \in \pi_2(\hat{Y}, \g_{e^+(v)} \sqcup \{\g_{e^-}\}_{e^- \in E^-(v)})$ admits a $\hat{J}$-holomorphic representative. We say that $T$ is representable by a $\hat{J}$-holomorphic building if we wish to specify $\hat{J}$.
\edefi

\begin{corollary} \label{corollary:full-positivity}
Let $T \in \widehat{\Sch}$. Suppose that there exists a morphism $T' \to T$ and a $d\l$-compatible almost complex structure $J$ on $\xi$ such that $T'$ is representable by a $\hat{J}$-holomorphic building. Suppose also that $\hat{V}$ is $\hat{J}$-holomorphic. If $\l$ is positive elliptic near $V$, then $T * \hat{V} \ge -\Gamma^-(T,V)$, where $\Gamma^-(T,V)$ denotes the number of output edges $e$ of $T$ such that $\g_e$ is contained in $V$.
\end{corollary}

\begin{proof}
By \Cref{proposition:invariance-gluing}, $T * \hat{V} = T' * \hat{V}$. By \Cref{proposition:additivity-symplectization},
	\eq
		T' * \hat{V} = \sum_{v \in V(T')} \beta_{v} * \hat{V} - \sum_{e \in E^\mathrm{int}(T')} \widehat{\g}_e * \hat{V}.
	\eeq

According to \Cref{prop-positivity-intersections}, we have $\b_v * \hat{V} \ge 0$, unless the holomorphic representative of $\b_v$ is entirely contained in $\hat{V}$, in which case \Cref{proposition:positivity-symplectization} tells us that $\b_v * \hat{V} \ge -\# E^-(v)$. By \Cref{corollary:intersection-cylinders}, we have
	\eq -\# E^-(v) = \sum_{e \in E^-(v)} \widehat{\g}_e * \hat{V}. \eeq
	
	Given $v \in V(T')$, let us denote by $\G^-(v, V)$ the number of output edges $e \in E^-(v)$ such that $\g_e \sub V$.  Appealing again to \Cref{proposition:additivity-symplectization}, we have:
\begin{align}
T'^* \hat{V} &= \sum_{ v \in V(T')} \beta_v * \hat{V} - \sum_{e \in E^{\op{int}}(T')} \widehat{\g}_e * \hat{V} \\
&= \sum_{v \in V(T')} \lt( \beta_v * \hat{V} - \sum_{e \in E^-(v)} \hat{\g}_e * \hat{V} \rt) + \sum_{e \in E^-(T)} \widehat{\g}_e * \hat{V} \\
&=  \sum_{v \in V(T')} ( \beta_v * \hat{V} + \G^-(v, V) ) - \Gamma^-(T', V)\\
&\geq - \Gamma^-(T, V),
\end{align}
where we have used the fact that $T, T'$ have the same exterior edges in the last line. This completes the proof. 
\end{proof}

More generally, suppose we are given the following data, where $m$ is now allowed to be any nonnegative integer:
\begin{itemize}
	\item For each $0 \le i \le m$, a strong contact submanifold $V^i \subset Y^i$ of codimension $2$.
	\item For each $0 \le i \le j \le m$, a homotopy class $\eta_{ij} \in \O_{2n-2}(\hat{X}^{ij}, V^i \sqcup V^j)$. We require that $\eta_{ii} := [\hat{V}_i]$ be the homotopy class of $\hat{V}^i = \R \times V^i$ and that $\eta_{ik} = \eta_{ij} \# \eta_{jk}$ for any $i \le j \le k$.
\end{itemize}
Let $\eta := \eta_{0m} \in \O_{2n - 2}(\hat{X}^{0m}, V^0 \sqcup V^m)$.

\prop \label{proposition:gluing-additivity}
Let $T \in \widehat{\Sch}$. Then
\eq \label{equation:additivity-cobordism}
T * \eta = \sum_{v \in V(T)} \b_v * \eta_{*(v)} - \sum_{e \in E^\mathrm{int}(T)} \widehat{\g}_e * \hat{V}_{*(e)}.
\eeq
\eprop
\pf
We will say a vertex $v \in V(T)$ is a \emph{symplectization vertex} if $*(v) = ii$ for some $i$ and a \emph{cobordism vertex} otherwise. This induces a partition $E^\mathrm{int}(T) = E^{ss}(T) \sqcup E^{sc}(T) \sqcup E^{cc}(T)$ of the set of interior edges according to the types of the vertices they are adjacent to (here the superscripts $s, c$ stand for ``symplectization" and ``cobordism" respectively). Similarly, the set of exterior edges admits a partition $E^\mathrm{ext}(T) = E^s(T) \sqcup E^c(T)$.

We can (and will) assume without loss of generality that $E^{ss}(T)$ is empty. Indeed, let $T \to T'$ be the morphism obtained by contracting all the edges in $E^{ss}(T)$. Replacing $T$ with $T'$ doesn't change the left-hand side of \eqref{equation:additivity-cobordism} by \Cref{proposition:invariance-gluing} and doesn't change the right-hand side by \Cref{proposition:additivity-symplectization}.

Let $I = V(T) \sqcup E^c(T) \sqcup E^{cc}(T)$ and choose a familty of curves $\{u_i\}_{i \in I}$ with the following properties:
\begin{itemize}
	\item For each $v \in V(T)$, $u_v$ is a curve in the homotopy class $\b_v$ which is cylindrical at infinity.
	\item For each $e \in E^c(T) \sqcup E^{cc}(T)$, $u_e = \widehat{\g}_e$ is the trivial cylinder over the Reeb orbit $\g_e$.
\end{itemize}
 For $t > 0$ sufficiently large, we can glue the $u_i$'s to obtain a curve $u$ in $X^{00} \#_t X^{01} \#_t X^{11} \#_t \cdots \#_t X^{mm} \cong X^{0m}$ representing $T$. We can also choose representatives $H_{ij}$ of $\eta_{ij}$ so that $H_{ii} = \hat{V_i}$ and $H := H_{0m}$ coincides with $H_{00} \#_t H_{01} \#_t \cdots \#_t H_{mm}$.

 As in the proof of \Cref{proposition:additivity-symplectization}, we can choose perturbation $u^\tau$, $\{u_i^\tau\}$ so that
 \eq u^\tau \cdot H = \sum_{i \in I} u_i^\tau \cdot H_i, \eeq
 where $H_i := H_{*(v)}$ for $i = v \in V(T)$ and $H_i := \hat{V}_{*(e)}$ for $i = e \in E^c(T) \sqcup E^{cc}(T)$. The difference $\sum_i u_i * H_i - u * H$ is therefore equal to
\begin{align}
	&\sum_{e \in E^{sc}(T) \sqcup E^c(T)} \a_N^{\tau; -}(\g_e) - \a_N^{\tau; +}(\g_e) + 2 \sum_{e \in E^{cc}(T)} \a_N^{\tau; -}(\g_e) - \a_N^{\tau; +}(\g_e) \\
	&= \sum_{e \in E^{sc}(T) \sqcup E^c(T)} \widehat{\g}_e * \hat{V}_{*(e)} + 2 \sum_{e \in E^{cc}(T)} \widehat{\g}_e * \hat{V}_{*(e)} \\
	&= \sum_{e \in E^\mathrm{int}(T)} \widehat{\g}_e * \hat{V}_{*(e)} + \sum_{e \in E^{c}(T) \sqcup E^{cc}(T)} \widehat{\g}_e * \hat{V}_{*(e)}
\end{align}
Since $u_i * H_i = \widehat{\g}_e * \hat{V}_{*(e)}$ for $i = e \in E^c(T) \sqcup E^{cc}(T)$, we conclude that
\eq
	u * H = \sum_{v \in V(T)} u_v * H_{*(v)} - \sum_{e \in E^\mathrm{int}(T)} \widehat{\g}_e * \hat{V}_{*(e)},
\eeq
which implies \eqref{equation:additivity-cobordism}.
\epf

\defi \label{defi:rep-buildings-full}
Given $T \in \widehat{\Sch}$, we say that $T$ is \emph{representable by a holomorphic building} if for every vertex $v \in V(T)$, there exists an adapted almost complex structure $\hat{J}^v$ on $\hat{X}^{*(v)}$ such that $\beta_v \in \pi_2(\hat{X}^{*(v)}, \g_{e^+(v)} \sqcup \{\g_{e^-}\}_{e^- \in E^-(v)})$ and $\eta_{*(v)} \in \O_{2n-2}(\hat{X}^{*(v)}, V^{*^+(v)} \sqcup V^{*^-(v)})$ admit $\hat{J}^v$-holomorphic representatives. 
\edefi

\prop \label{proposition:positivity-intersection-buildings}
Let $T \in \widehat{\Sch}$. Suppose that there exists a morphism $T' \to T$ where $T'$ is representable by a holomorphic building. Suppose that $\l^i$ is positive elliptic near $V^i$ for all $0 \le i \le m$, and let $r_i > 0$ be the rotation parameter (see \Cref{definition:adapted-contact-form}). If $\beta_v * \eta_{*(v)} \geq -  \# \{ E^-(v) \}$, then $T * \eta \ge  -  \Gamma^-(T, V^m)$ (recall that $\eta:= \eta_{0m}$). 
\eprop

\pf
By \Cref{proposition:invariance-gluing}, we have $T * \eta = T' * \eta$. Given $v \in V(T')$, let us denote by $\Gamma^-(v, V^{*^-(v)})$ the number of output edges $e \in E^-(v)$ such that $\g_e \sub V^{*^-(v)}$. 
	
Arguing as in the proof of \Cref{corollary:full-positivity}, we obtain from \Cref{proposition:gluing-additivity} that
\begin{align}
T' * \hat{V} &= \sum_{ v \in V(T')} \beta_v * \eta_{*(v)} - \sum_{e \in E^{\op{int}}(T')} \widehat{\g}_e * \hat{V}_{*(e)} \\
&= \sum_{v \in V(T')} \lt( \beta_v * \eta_{*(v)} -  \sum_{e \in E^-(v)} \widehat{\g}_e * \hat{V}_{*(e)} \rt) +   \sum_{e \in E^-(T')} \widehat{\g}_e * \hat{V}_{*(e)} \\
&=  \sum_{v \in V(T')} ( \beta_v * \hat{V} +  \# \Gamma^-(v, V^{*^-(v)})) - \Gamma^-(T', V)\\
&\geq -   \Gamma^-(T, V),
\end{align}
where we have used the fact that $T, T'$ have the same exterior edges in the last line. This completes the proof. 
\epf

\subsection{The intersection number for cycles}

For future reference, we collect some basic facts about intersection numbers for cycles in oriented manifolds. This subsection takes places entirely in the smooth category and does not involve any contact topology. 

\defi \label{definition:intersection-number} 
Let $M$ be an oriented, compact manifold of dimension $n$, possibly with boundary. Let $S_1, S_2 \sub M$ be disjoint closed embedded submanifolds of $M$. (We allow the $S_i$ to intersect $\d M$, in which case the $S_i$ are required to be embedded submanifolds after enlarging $M$ by a collar). Then we can define a pairing
\begin{align}
- \cdot - : H_{k}(M, S_1; \Z) \tms H_{n-k}(M, S_2; \Z)  &\to \Z, \\
(A, B) &\mapsto A \cdot B, 
\end{align}  
where $A \cdot B$ is a signed count of intersections between cycles representing $A$ and $B$. More precisely, we represent cycles by $C^\infty$ chains; by general position, these chains may be assumed to intersect transversally after an arbitrarily small perturbation which does not affect their homology class. It is a folklore result which is beyond the scope of this paper that the resulting count is well-defined and graded-symmetric (see e.g.\ \cite[Sec.\ 2.3]{griffiths-morgan}). 
\edefi

We note that the intersection number in \Cref{definition:intersection-number} could be defined under much milder hypotheses, but this is not necessary for our purposes. If $A, B$ are (the pushforward of the fundamental class of) oriented manifolds, then $A\cdot B$ coincides with the usual intersection number for submanifolds. By abuse of notation, we will view the intersection pairing as being defined on both homology classes and oriented submanifolds. 

\defi \label{definition:linking-number}
Fix a closed manifold $Y$ of dimension $m \geq 3$ and a closed codimension $2$ submanifold $V\sub Y$. Suppose that $H_1(Y;\Z)= H_2(Y; \Z)=0$. 
	
Let $\g: S^1 \to Y-V$ be a loop with embedded image. The \emph{linking number} of $\g$ with respect to $V$ is denoted $\op{link}_V(\g)$ and defined by
\eq \op{link}_V(\g):= V \cdot C_{\g}, \eeq
where $C_{\g}$ is a cycle bounding $\g$ (which exists since $H_1(Y;\Z)=0$). This is well-defined due to our assumption that $H_2(Y; \Z)=0$.
	
Suppose now that $\L \sub Y-V$ is a submanifold with $\pi_0(\L)=\pi_1(\L)=0$. Let $c: [0,1] \to Y-V$ be a path with embedded image having the property that $c(0), c(1) \in \L$. Let $\ov{c}: S^1 \to Y-V$ be a loop with embedded image obtained by connecting $c(1)$ to $c(0)$ by a path in $\L$. The \emph{(path) linking number} of $c$ with respect to $V$ is denoted $\op{link}_V(c;\L)$ and is defined by setting 
\eq \op{link}_V(c; \L):= V \cdot C_{\ov{c}}, \eeq
where $C_{\ov{c}}$ is a cycle bounding $\ov{c}$. This is independent of $\ov{c}$ since $\pi_1(\L)=0$ and independent of $C_{\ov{c}}$ since $H_2(Y; \Z)=0$. 
\edefi

\rmk \label{rmk:obd-linking-number}
Fix an open book decomposition $(Y, B, \pi)$ and let $\g: S^1 \to Y-B$ be a loop. Then it is not hard to show that we have $\op{link}_B(\g) = \op{deg}(\pi \circ \g).$
	
Similarly, suppose that $\L \sub Y$ is a submanifold which is contained in a page of $(Y, B, \pi)$. Let $c: [0,1] \to Y-B$ be a path with the property that $c(0), c(1) \in \L$. Then the composition $\pi \circ c: [0,1] \to S^1$ induces a map $\ov{c}: [0,1]/\{0,1\} \to S^1$. We then have $\op{link}_B(c; \L) = \op{deg} \ov{c}. $
\ermk

\lem \label{lemma:cobordism-linking}
Let $Y^{\pm}$ be oriented manifolds with $Y^+ \neq \emptyset$ and let $B^{\pm} \sub Y^{\pm}$ be oriented submanifolds.  Let $W$ be an oriented, smooth cobordism from $Y^+$ to $Y^-$ and let $H \sub W$ be an oriented sub-cobordism from $B^+$ to $B^-$ (i.e. $H$ is an embedded submanifold which admits a collar neighborhood near the boundary of $W$.) Suppose that $H_1(Y^{\pm}; \Z)= H_2(Y^{\pm}; \Z)= H_2(W, Y^+; \Z)=0$.       

Let $\S$ be a Riemann surface with $k+1$ boundary components labelled $\g^+, \g^-_1,\dots,\g^-_k$. Suppose that $u: (\S, \d \S) \to (W, \d W)$ is a smooth map sending $\g^+$ into $Y^+ - B^+$ and $\g^-_i$ into $Y^- -B^-$.

Then \eq \label{equation:cobordism-linking} \op{link}_{B^+}(\g^+)- \sum_{i=1}^k \op{link}_{B^-}(\g^-_i)= H \cdot u(\S), \eeq
where we have identified the boundary components of $\S$ with the restriction of $u$ to these components. 
\elem

\pf
Choose a $2$-chain $B^{-} \in C_2(Y^-; \Z)$ with $\partial B^- =\g^-_1 \cup \dots \cup \g^-_k$. Glue $B^-$ to $u(\S)$ along the $\g^-_i$ and call the resulting chain $C \in C_2(W; \Z)$.  We now have $H \cdot C=  H \cdot u(\S)  +  \sum_{i=1}^k \op{link}_{B^-}(\g^-_i).$

By the long exact sequence of the triple $(W, Y^+, \g^+)$ and our assumption that $H_2(W, Y^+; \Z)=0$, the natural map $H_2(Y, \g^+; \Z) \to H_2(W, \g^+; \Z)$ is surjective. Let $\tilde{C} \in H_2(Y^+, \g^+; \Z)$ be a lift of $C \in H_2(W, \g^+; \Z)$. Then $H \cdot C= H \cdot \tilde{C}= B^+ \cdot \tilde{C} = \op{link}_{B^+}(\g^+). $
\epf

\lem \label{lemma:cobordism-linking-legendrian}
We carry over the setup and notation from \Cref{lemma:cobordism-linking}. In addition to the data considered there, let $\L^{\pm} \sub Y^{\pm}- B^{\pm}$ be an oriented, smooth submanifold and let $\L \sub W$ be an oriented sub-cobordism from $\L^+$ to $\L^-$ which is disjoint from $H$. We suppose in addition that $\pi_0(\L^{\pm})= \pi_1(\L^{\pm})=0$. 

Let $\S$ be a closed, oriented surface of genus zero with $s+1$ boundary components labeled $\g^*, \g_1,\dots,\g_n$. For $\s \in \N_+$, we place $2 \s$ disjoint marked points on $\g^*$, this partitioning $\g^*$ into $2\s$ sub-intervals. Let us label these subintervals by the symbols $c^+, b_{01}, c_{1}^-, b_{12}, c_2^-,\dots,b_{(\s-1)\s}, c_\s^-, b_{\s0}$, in the order induced by the orientation.

Suppose now that $u: (\S, \d \S) \to (W, \d W \cup \L)$ is a smooth map sending $(c^+, \d c^+)$ into $(Y^+-B^+, \L^+)$, sending $(c_i^-, \d c_i^-)$ into $(Y^- - B^-, \L^-)$, sending $b_{i(i+1)}$ into $\L$, and sending the $\g_i^-$ into $Y^- - B^-$. 

Then \eq \op{link}_{B^+}(c^+; \L^+)- \sum_{i=1}^{\s} \op{link}_{B^-}(c^-_i; \L^-)- \sum_{i=1}^s \op{link}_{B^-}(\g^-_i)= H \cdot u_*[\S], \eeq 
where we have again identified the boundary components of $\S$ with the restriction of $u$ to these components. 
\qed
\elem

\section{Energy and twisting maps} \label{section:energy-twisting-maps}

\subsection{Twisting maps} \label{subsection:twistinghomologies}

In order to define invariants of codimension $2$ contact submanifolds, we will proceed as follows. First, we will use Siefring's intersection theory to define maps $\psi : \Sch \to \mc{R}$. Here $\mc{R}$ could be any $\Q$-algebra, though we will only use $\mc{R} = \Q[U]$ and $\mc{R} = \Q$. We will then use these maps to define ``twisted'' moduli counts
\eq
	\#_\psi\Mbar(T)^{\vir} := \#\Mbar(T)^{\vir} \cdot \psi(T) \in \mc{R}.
\eeq
The maps $\psi$ will have the property that
\begin{align}
\psi(T') &= \psi(T) \quad \text{for every morphism $T' \to T$} \\
\psi(\#_i T_i) &= \prod_i \psi(T_i)
\end{align}
which implies that the \emph{master equations} \eqref{eq-boundary-zero} and \eqref{eq-concatenation-additive} still hold if $\#\Mbar^{\vir}$ is replaced by $\#_\psi\Mbar^{\vir}$.

The properties which must be satisfied by the maps $\psi$ in order to obtain twisted counts which are suitable for defining our invariants can be conveniently axiomatized in the notion of a \emph{twisting map}. We now define precisely this notion in each of the four setups. 

\begin{setup1}
Fix a datum $\mc{D}$ for Setup I.  Let $\Sch_\I^{\neq \emptyset}$ denote the full subcategory of $\Sch_\I$ spanned by objects $T$ for which the moduli space $\Mbar(T)$ is nonempty.
\defi \label{definition:twistingI}
	Let $\mc{R}$ be a $\Q$-algebra. The set $\Psi_\I(\mc{D};\mc{R})$ of $\mc{R}$-valued \emph{twisting maps} consists of all maps $\psi : \Sch_\I^{\neq \emptyset}(\mc{D}) \to \mc{R}$ satisfying the following two properties:
	\begin{itemize}
		\item for any morphism $T' \to T$, $\psi(T') = \psi(T)$;
		\item for any concatenation $\{T_i\}_i$, $\psi(\#_i T_i) = \prod_i \psi(T_i)$.
	\end{itemize}
\edefi

Fix a twisting map $\psi \in \Psi_\I(\mc{D};\mc{R})$.
Let
\eq
	CC_\bullet(Y,\xi,\psi)_\lambda := \bigoplus_{n\geq 0}\Sym^n_{\mc{R}}\Bigl(\bigoplus_{\gamma\in\cP_\good}\oo_\gamma\Bigr)
\eeq
be the free supercommutative $\Z/2$-graded unital $\mc{R}$-algebra generated by the good Reeb orbits. The grading of a Reeb orbit is given by its \emph{parity}, which is defined as \eq \label{equation:parity} 
	|\g|= \op{sign} \op{det}(I- A_{\g}) \in \{ \pm 1\} = \Z/2,
\eeq
where $A_\g$ is the linearized Poincar\'{e} return map of $\xi$ along $\g$ (see \cite[Sec.\ 2.13]{pardon}). Recall that a Reeb orbit is good iff it is not bad; a Reeb orbit $\g$ is bad if is an even multiple of some simple Reeb orbit $\g_s$ such that $\g$ and $\g_s$ have different parity; \cite[Def.\ 2.49]{pardon}.

Theorem 1.1 of \cite{pardon} provides a set of perturbation data $\Theta_\I(\mc{D})$ and associated virtual moduli counts $\# \Mbar_\I(T)^{\vir}_{\theta} \in \bb{Q}$ satisfying \eqref{eq-boundary-zero} and \eqref{eq-concatenation-additive}. We define the twisted moduli counts 
\eq \#_{\psi} \Mbar_\I(T)^{\vir}_{\theta}:= \#\Mbar_\I(T)^{\vir}_{\theta} \cdot \psi(T) \in \mc{R}.\eeq 

It follows easily from \Cref{definition:twistingI} that the twisted moduli counts also satisfy \eqref{eq-boundary-zero} and \eqref{eq-concatenation-additive}. We may therefore endow $CC_\bullet(Y,\xi,\psi)_{\lambda}$ with a differential $d_{\psi, J, \t}$, which is given by
\eq \label{equation:twist-differential-I}
	d_{\psi, J, \t}(\oo_{\gamma^+})
	= \sum_{\mu(T)=1}\frac 1{\left|\Aut(T)\right|}\cdot\#_{\psi} \Mbar_\I(T)^\vir_{\t} \oo_{\Gamma^-},
\eeq
where the sum is over all trees $T \in \Sch_I(\mc{D})$ representing curves with positive orbit $\gamma^+$ and negative orbits $\Gamma^-\to\cP_\good$. 
 
The homology of $(CC_\bullet(Y,\xi,\psi)_{\lambda}, d_{\psi, J, \theta})$ is a supercommutative $\Z/2$-graded unital $\mc{R}$-algebra which is denoted 
\eq \label{equation:inv-1} CH_{\bullet}(Y, \xi, \psi)_{\l, J, \theta}.\eeq
\end{setup1}

\begin{setup2}

Fix a datum $\mc{D}$ for Setup II. Suppose now we are given a map of $\Q$-algebras $m : \mc{R}^+ \to \mc{R}^-$ and twisting maps $\psi^\pm \in \Psi_\I(\mc{D}^\pm;\mc{R}^\pm)$.
\defi \label{definition:twistingII}
	The set $\Psi_\II(\mc{D};\psi^+,\psi^-)$ consists of all maps $\psi : \Sch_\II^{\neq \emptyset}(\mc{D}) \to \mc{R}^-$ satisfying the following two properties:
	\begin{itemize}
		\item for any morphism $T' \to T$, $\psi(T') = \psi(T)$;
		\item for any concatenation $\{T_i\}_i$,
		\eq
			\psi(\#_i T_i) = \left( \prod_{T_i \in \Sch_\I^+} m(\psi^+(T_i)) \right)\left( \prod_{T_i \in \Sch_\II} \psi(T_i) \right)\left( \prod_{T_i \in \Sch_\I^-} \psi^-(T_i) \right).
		\eeq
	\end{itemize}
\edefi

Fix a twisting map $\psi \in \Psi_\II(\mc{D};\psi^+,\psi^-)$. Theorem 1.1 of \cite{pardon} provides a set of perturbation data $\Theta_\II(\mc{D})$ together with a forgetful map 
\eq\label{equation:type-2-forget} \T_\II(\mc{D}) \to \T_\I(\mc{D}^+) \times \T_\I(\mc{D}^-)\eeq 
and associated virtual moduli counts $\# \Mbar_\II(T)^{\vir}_{\theta} \in \bb{Q}$. We define the twisted moduli counts 
\eq \#_{\psi} \Mbar_\II(T)^{\vir}_{\theta}:= \#\Mbar_\II(T)^{\vir}_{\theta} \cdot \psi(T) \in \mc{R}^-.\eeq

For any $\t \in \T_\II(\mc{D})$ mapping to $(\t^+,\t^-) \in \T_\I(\mc{D}^+) \times \T_\I(\mc{D}^-)$, we obtain a unital $\mc{R}^+$-algebra map
\eq\label{equation:phi-setup-2} \Phi(\hat{X},\hat{\l},\psi)_{\hat{J},\t} : CC_\bullet(Y^+,\xi^+,\psi^+)_{\lambda^+, J^+, \theta^+} \to CC_\bullet(Y^-,\xi^-,\psi^-)_{\lambda^-, J^-, \theta^-}\eeq
which maps $\oo_{\g^+}$ to
\eq
	\sum_{\mu(T)=0}\frac 1{\left|\Aut(T)\right|}\cdot\#_{\psi} \Mbar_\II(T)^\vir_{\t} \oo_{\Gamma^-},
\eeq
with the sum over all trees $T \in \Sch_\II(\mc{D})$ representing curves with positive orbit $\g^+$ and negative orbits $\Gamma^-\to\cP_\good(Y^-)$. This is a chain map since it follows from \Cref{definition:twistingII} that the twisted moduli counts satisfy \eqref{eq-boundary-zero} and \eqref{eq-concatenation-additive}.
\end{setup2}

\begin{setup3}

Fix a datum $\mc{D}$ for Setup III. There are three types of concatenations $\{T_i\}_i$ in $\Sch_\III = \Sch_\III(\mc{D})$:
\begin{enumerate}
	\item $\{T_i\} \subset \Sch_\I^+ \sqcup \Sch_\II^{t=0} \sqcup \Sch_\I^-$, in which case $\mathfrak{s}(\#_i T_i) = \{0\}$;
	\item $\{T_i\} \subset \Sch_\I^+ \sqcup \Sch_\II^{t=1} \sqcup \Sch_\I^-$, in which case $\mathfrak{s}(\#_i T_i) = \{1\}$;
	\item $\{T_i\} \subset \Sch_\I^+ \sqcup \Sch_\III \sqcup \Sch_\I^-$ and $T_i \in \Sch_\III$ for a unique $i = i_0$, in which case $\mathfrak{s}(\#_i T_i) = \mathfrak{s}(T_{i_0})$.
\end{enumerate}

(Here $\Sch_\I^{\pm}, \Sch_\II^{t\in \{0,1\}}, \Sch_\III$ are tree categories determined by $\mc{D}$, following the notation of \cite[Sec.\ 2.1]{pardon}). Suppose now we are given a map of $\Q$-algebras $m : \mc{R}^+ \to \mc{R}^-$ and twisting maps $\psi^\pm \in \Psi_\I(\mc{D}^\pm;\mc{R}^\pm)$, $\psi^0 \in \Psi_\II(\mc{D}^{t=0};\psi^+,\psi^-)$ and $\psi^1 \in \Psi_\II(\mc{D}^{t=1};\psi^+,\psi^-)$.
\defi
	The set $\Psi_\III(\mc{D};\psi^0,\psi^1)$ consists of all maps $\psi : \Sch_\III^{\neq \emptyset}(\mc{D}) \to \mc{R}^-$ satisfying the following properties:
	\begin{itemize}
		\item for any morphism $T' \to T$, $\psi(T') = \psi(T)$;
		\item for any concatenation $\{T_i\}_i$ of the first type,
		\eq
			\psi(\#_i T_i) = \left( \prod_{T_i \in \Sch_\I^+} m(\psi^+(T_i)) \right)\left( \prod_{T_i \in \Sch_\II^{t=0}} \psi^0(T_i) \right)\left( \prod_{T_i \in \Sch_\I^-} \psi^-(T_i) \right)
		\eeq
		\item for any concatenation $\{T_i\}_i$ of the second type,
		\eq
			\psi(\#_i T_i) = \left( \prod_{T_i \in \Sch_\I^+} m(\psi^+(T_i)) \right)\left( \prod_{T_i \in \Sch_\II^{t=1}} \psi^1(T_i) \right)\left( \prod_{T_i \in \Sch_\I^-} \psi^-(T_i) \right)
		\eeq
		\item for any concatenation $\{T_i\}_i$ of the third type,
		\eq
			\psi(\#_i T_i) = \left( \prod_{T_i \in \Sch_\I^+} m(\psi^+(T_i)) \right) \psi(T_{i_0}) \left( \prod_{T_i \in \Sch_\I^-} \psi^-(T_i) \right).
		\eeq
	\end{itemize}
\edefi

Fix a twisting map $\psi \in \Psi_\III(\mc{D};\psi^0,\psi^1)$.
Theorem 1.1 of \cite{pardon} provides a set of perturbation data $\Theta_\III(\mc{D})$ together with a forgetful map $\T_\III(\mc{D}) \to \T_\II(\mc{D}^0) \times_{\T_\I(\mc{D}^+) \times \T_\I(\mc{D}^-)} \T_\II(\mc{D}^1)$ and associated virtual moduli counts $\# \Mbar_\III(T)^{\vir}_{\theta} \in \bb{Q}$ (note that the fiber product is defined with respect to \eqref{equation:type-2-forget}). We define the twisted moduli counts 
\eq \#_{\psi} \Mbar_\III(T)^{\vir}_{\theta}:= \#\Mbar_\III(T)^{\vir}_{\theta} \cdot \psi(T) \in \mc{R}^-.\eeq

If $(\hat{X},\hat{\l}^t)$ is a family of exact cobordisms, then for any $\t \in \T_\III(\mc{D})$, we obtain an $\mc{R}^+$-linear map
\eq\label{equation:phi-homotopy-3} K(\hat{X},\{\l^t\}_t,\psi)_{\hat{J}^t,\t } : CC_\bullet(Y^+,\xi^+,\psi^+)_{\lambda^+, J^+, \theta^+} \to CC_{\bullet + 1}(Y^-,\xi^-,\psi^-)_{\lambda^-, J^-, \theta^-} \eeq
which sends the monomial $\prod_{i \in I} \oo_{\g_i^+}$ to
\eq
	\sum_{\vdim(\{T_i\}_{i\in I})=0}\frac 1{\left|\Aut(\{T_i\}_{i \in I})\right|}\cdot\#_{\psi} \Mbar_\III(\{T_i\}_{i\in I})^\vir_{\t} \prod_{i \in I}\oo_{\Gamma_i^-},
\eeq
with the sum over trees $T_i \in \Sch_\III(\mc{D})$ with positive orbit $\gamma_i^+$ and negative orbits $\Gamma_i^-\to\cP_\good(Y^-)$.

Equations \eqref{eq-boundary-zero} and \eqref{eq-concatenation-additive} applied to the twisted moduli counts imply that this is a chain homotopy between $\Phi(\hat{X},\hat{\l}^0,\psi^0)_{\hat{J}^0,\t^0}$ and $\Phi(\hat{X},\hat{\l}^1,\psi^1)_{\hat{J}^1,\t^1}$ and hence that the induced maps on homology
\eq\label{equation:phi-setup-3}
\begin{tikzcd}[column sep = huge]
CH_\bullet(Y^+,\xi^+,\psi^+)_{\lambda^+,J^+,\theta^+}\ar[shift left=0.6ex]{r}{\Phi(\hat X,\hat\lambda^0,\psi^0)_{\hat J^0,\theta^0}}\ar[shift left=-0.6ex]{r}[swap]{\Phi(\hat X,\hat\lambda^1,\psi^1)_{\hat J^1,\theta^1}}&CH_\bullet(Y^-,\xi^-,\psi^-)_{\lambda^-,J^-,\theta^-}
\end{tikzcd}
\eeq
are equal.

\end{setup3}

\begin{setup4}
Fix a datum $\mc{D}$ for Setup IV.  There are three types of concatenations $\{T_i\}_i$ in $\Sch_\IV = \Sch_\IV(\mc{D})$:
\begin{enumerate}
	\item $\{T_i\} \subset \Sch_\I^0 \sqcup \Sch_\II^{02} \sqcup \Sch_\I^2$, in which case $\mathfrak{s}(\#_i T_i) = \{0\}$;
	\item $\{T_i\} \subset \Sch_\I^0 \sqcup \Sch_\II^{01} \sqcup \Sch_\I^1 \sqcup \Sch_\II^{12} \sqcup \Sch_\I^2$, in which case $\mathfrak{s}(\#_i T_i) = \{\infty\}$; \label{item:concatenationIV-infinity}
	\item $\{T_i\} \subset \Sch_\I^0 \sqcup \Sch_\IV \sqcup \Sch_\I^2$ and $T_i \in \Sch_\IV$ for a unique $i = i_0$, in which case $\mathfrak{s}(\#_i T_i) = \mathfrak{s}(T_{i_0})$.
\end{enumerate}

(We again follow the notation of \cite[Sec.\ 2.1]{pardon} for tree categories determined by $\mc{D}$). Suppose now we are given maps of $\Q$-algebras $m^{01} : \mc{R}^0 \to \mc{R}^1$, $m^{12} : \mc{R}^1 \to \mc{R}^2$ and twisting maps
\begin{align}
	\psi^i &\in \Psi_\I(\mc{D}^i;\mc{R}^i) \quad (i = 0,1,2) \\
	\psi^{ij} &\in \Psi_\II(\mc{D}^{ij};\psi^i,\psi^j) \quad (ij = 01,12,02)
\end{align}
Set $m^{02} = m^{12} \circ m^{01} : \mc{R}^0 \to \mc{R}^2$.
\defi
	The set $\Psi_\IV(\mc{D};\{\psi^{ij}\})$ consists of all maps $\psi : \Sch_\IV^{\neq \emptyset}(\mc{D}) \to \mc{R}^2$ satisfying the following properties:
	\begin{itemize}
		\item for any  morphism $T' \to T$, $\psi(T') = \psi(T)$;
		\item for any concatenation $\{T_i\}_i$ of the first type,
		\eq
			\psi(\#_i T_i) = \left( \prod_{T_i \in \Sch_\I^0} m^{02}(\psi^0(T_i)) \right)\left( \prod_{T_i \in \Sch_\II^{02}} \psi^{02}(T_i) \right)\left( \prod_{T_i \in \Sch_\I^2} \psi^2(T_i) \right)
		\eeq
		\item for any concatenation $\{T_i\}_i$ of the second type,
		\begin{multline}
			\psi(\#_i T_i) = \left( \prod_{T_i \in \Sch_\I^0} m^{02}(\psi^0(T_i)) \right)\left( \prod_{T_i \in \Sch_\II^{01}} m^{12}(\psi^{01}(T_i)) \right)\left( \prod_{T_i \in \Sch_\I^1} m^{12}(\psi^1(T_i)) \right)\\
			\cdot \left( \prod_{T_i \in \Sch_\II^{12}} \psi^{12}(T_i) \right)\left( \prod_{T_i \in \Sch_\I^2} \psi^2(T_i) \right)
		\end{multline}
		\item for any concatenation $\{T_i\}_i$ of the third type,
		\eq
			\psi(\#_i T_i) = \left( \prod_{T_i \in \Sch_\I^0} m^{02}(\psi^0(T_i)) \right) \psi(T_{i_0}) \left( \prod_{T_i \in \Sch_\I^2} \psi^2(T_i) \right).
		\eeq
	\end{itemize}
\edefi

Fix a twisting map $\psi \in \Psi_\IV(\mc{D};\{\psi^{ij}\})$.
Theorem 1.1 of \cite{pardon} provides a set of perturbation data $\Theta_\IV(\mc{D})$ together with a forgetful map $\T_\IV(\mc{D}) \to \T_\II(\mc{D}^{02}) \times_{\T_\I(\mc{D}^0) \times \T_\I(\mc{D}^2)} (\T_\II(\mc{D}^{01}) \times_{\T_\I(\mc{D}^1)} \T_\II(\mc{D}^{12}))$ and associated virtual moduli counts $\# \Mbar_\IV(T)^{\vir}_{\theta} \in \bb{Q}$ (here again, the fiber product is defined using \eqref{equation:type-2-forget}). We define the twisted moduli counts 
\eq \#_{\psi} \Mbar_\IV(T)^{\vir}_{\theta}:= \#\Mbar_\IV(T)^{\vir}_{\theta} \cdot \psi(T) \in \mc{R}^2.\eeq

As in the previous section, we obtain an $\mc{R}^0$-linear map
\eq CC_\bullet(Y^0,\xi^0,\psi^0)_{\lambda^0, J^0, \theta^0} \to CC_{\bullet + 1}(Y^2,\xi^2,\psi^2)_{\lambda^2, J^2, \theta^2} \eeq
which is a chain homotopy between the maps $\Phi(\hat{X}^{02},\hat{\l}^{02},\psi^{02})_{\hat{J}^{02},\t^{02}}$ and $\Phi(\hat{X}^{12},\hat{\l}^{12},\psi^{12})_{\hat{J}^{12},\t^{12}} \circ \Phi(\hat{X}^{01},\hat{\l}^{01},\psi^{01})_{\hat{J}^{01},\t^{01}}$, so that the diagram
\eq\label{equation:phi-commute-4}
\begin{tikzcd}[row sep = small]
&CH_\bullet(Y^1,\xi^1,\psi^1)_{\lambda^1,J^1,\theta^1}\ar{rd}{\Phi(\hat X^{12},\hat\lambda^{12},\psi^{12})_{\hat J^{12},\theta^{12}}}\\
CH_\bullet(Y^0,\xi^0,\psi^0)_{\lambda^0,J^0,\theta^0}\ar{rr}[swap]{\Phi(\hat X^{02},\hat\lambda^{02},\psi^{02})_{\hat J^{02},\theta^{02}}}\ar{ru}{\Phi(\hat X^{01},\hat\lambda^{01},\psi^{01})_{\hat J^{01},\theta^{01}}}&&CH_\bullet(Y^2,\xi^2,\psi^2)_{\lambda^2,J^2,\theta^2}
\end{tikzcd}
\eeq
commutes.

\end{setup4}

\subsection{The energy of a (strict) symplectic cobordism}

In this section, we introduce a notion of energy for (families of strict) exact symplectic cobordisms, and for certain classes of almost-complex structures. 

\begin{notation}
Recall that a strict exact symplectic cobordism from $(Y^+, \l^+)$ to $(Y^-, \l^-)$ is the data of an exact symplectic cobordism $(\hat X, \hat \l)$ and embeddings
\eq e^{\pm}: (\R \times Y^{\pm}, \hat{\l}^{\pm}) \to (\hat X,\hat\l) \eeq
which preserve the Liouville forms and satisfy certain additional properties stated in \Cref{definition:strict-sympcobordism}. When we consider strict exact symplectic cobordisms in this section, we will routinely abuse notation by identifying subsets of $\R \tms Y^{\pm}$ with their image under $e^\pm$. We hope that this abuse will make this section easier to read without introducing any substantial ambiguities.
\end{notation}

We begin with the following definition. 

\defi \label{definition:cobordism-decompII}
Let $(\hat{X}, \hat{\l})$ be a strict exact symplectic cobordism from $(Y^+, \l^+)$ to $(Y^-, \l^-)$. A \emph{Type A cobordism decomposition} is the data of a pair of hypersurfaces
\begin{align}
\mc{H}_-= \{-C_-\} \tms Y^-  &\text{ and } \mc{H}_+ = \{ C_+ \} \tms Y^+
\end{align}
for $C_\pm \in \R$, such that 
\eq \label{equation:cob-non-intersect} ( (-\infty, -C_-) \tms Y^-)  \cap  (( C_+, \infty) \tms Y^+) = \emptyset. \eeq The intersection in \eqref{equation:cob-non-intersect} takes place inside $\hat{X}$; if $Y^-= \emptyset$, we set $\mc{H}_- = \emptyset$, $C_-=0$ and we consider that \eqref{equation:cob-non-intersect} is tautologically satisfied. 
We let $\S(\hat{X}, \hat{\l})= \S(\hat{X}, \hat{\l}; \l^+, \l^-)$ be the set of all such Type A cobordism decompositions. 
\edefi

\rmk\label{remark:hypersurface-constant} 
Since we are working with \emph{strict} cobordisms, the real numbers $C_\pm \in \R$ are uniquely determined by the hypersurfaces $\mathcal{H}_\pm$. The data of the pair $(\mathcal{H}_-,\mathcal{H}_+)$ is therefore equivalent to the data of the pair $(C_-,C_+)$.
\ermk

\defi\label{definition:type-A-energy}
Let $(\hat{X}, \hat{\l})$ be as in \Cref{definition:cobordism-decompII} and let $\s \in \S(\hat{X}, \hat{\l})$ be a Type A cobordism decomposition. We let $\mc{E}(\s):= C_-+C_+ $ be the \emph{energy} of the decomposition $\s$. We define 
\eq \mc{E}(\hat{X}, \hat{\l}) = \mc{E}(\hat{X}, \hat{\l}; \l^+, \l^-):= \inf_{\s \in \S(\hat{X}, \hat{\l})} \mc{E}(\s) \in \R \cup \{-\infty\} \eeq 
to be the energy of $(\hat{X}, \hat{\l})$ (this is well-defined since a cobordism decomposition clearly always exists). We note that the energy may in general be negative. 

Given $C \in \R$, let $\S(\hat{X}, \hat{\l})_{<C} \sub  \S(\hat{X}, \hat{\l})$ (resp. $\leq C$) denote the subset of cobordism decompositions of energy strictly less than $C$ (resp. at most $C$). 
\edefi

\lem[Energy of a symplectization] \label{lemma:energy-marked-symplectization}
Suppose that $(\hat{X}, \hat{\l}) = (SY,\l_Y)$ is a symplectization which is endowed with the canonical structure of a (strict) exact symplectic cobordism from $(Y, \l^+)$ to $(Y, \l^-)$; see \Cref{example:symplectization}. Fix $f: Y \to \R$ so that $\l^+= e^f \l^-$. Then $\mc{E}(SY, \l_Y) \leq -\op{min} f$. 
\elem

\pf
Let $e^{\pm}: (\R \times Y, \hat{\l}^\pm) \to (SY,\l_Y)$ be the canonical identifications induced by $\l^{\pm}$. Let $\mc{H}_+ = \{0\} \tms Y$ in the coordinates induced by $e^+$. This means that $\mc{H}_+= \{(f(y), y) \mid y \in Y \} \sub \R \tms Y$ in the coordinates induced by $e^-$. Now given any $C_- > - \op{min} f$, we can let $\mc{H}_-= \{- C_-\} \tms Y$ in the coordinates induced by $e^-$. It follows that $\mc{E}(SY, \hat \l_Y) \leq C_-$. Since $C_- > - \op{min} f$ was arbitrary, the claim follows. 
\epf

\rmk \label{remark:energy-symplectization-same-marking}
If we assume in addition that $\l^+= \l^-$, then it is easy to verify that in fact $\mc{E}(SY, \l_Y) = 0$. 
\ermk

\lem \label{lemma:minus-infinity-filling}
We have $\mc{E}(\hat{X}, \hat{\l}) = - \infty$ if and only if $Y^- = \emptyset$. 
\elem

\pf
Suppose that $Y^-$ is non-empty and choose a cobordism decomposition $\s$ for $(\hat X, \hat \l)$ given by a pair of hypersurfaces $\mc{H}_-, \mc{H}_+ \sub \hat{X}$. Let $(X, \l)$ be the truncated Liouville cobordism with negative boundary $\mc{H}_-$ and positive boundary $\mc{H}_+$. Observe that the image of the negative boundary under the Liouville flow must touch the positive boundary in some finite time $T< \infty$ -- indeed, this follows from from the fact that $(X, \l)$ has finite volume. Given any other cobordism decomposition $\s'$, we now have $\mc{E}(\s') \geq \mc{E}(\s)- T$. 
	
Suppose now that $Y^-$ is empty. Then \eqref{equation:cob-non-intersect} is a vacuous condition. Since the backwards Liouville flow of any slice $\{C_+\} \tms Y^+$ is defined for all time, it follows that we can find a cobordism decomposition of arbitrarily negative energy. 
\epf

\lem \label{lemma:cob-decomp-connected-II}
Fix a strict exact symplectic cobordism $(\hat{X}, \hat{\l})$ from $(Y^+, \l^+)$ to $(Y^-, \l^-)$. Then $\S(\hat{X}, \hat{\l})_{<C} \sub  \S(\hat{X}, \hat{\l})$ is 
\begin{itemize}
\item[(a)] nonempty for $C> \mc{E}(\hat{X}, \hat{\l})$,
\item[(b)] path-connected for all $C \in \R$ (note that the empty set is path-connected). 
\end{itemize}
If moreover $(\hat{X}, \hat{\l})= (SY,\l_Y)$ is a symplectization and $\l^+= \l^-$, then $\S(\hat{X}, \hat{\l})_{\leq 0}$ is non-empty and path-connected.
\elem

\pf 
Note first that (a) is tautologically true. Next, note that (b) is obvious when $Y^-= \emptyset$. It therefore remains to prove (b) under the assumption that $Y^- \neq \emptyset$.

Let us consider a pair of cobordism decompositions $\s, \s' \in \S(\hat{X}, \hat{\l})_{<C}$. By definition, $\s, \s'$ are entirely determined by the constants $-C_-, C_+ \in \R$ (resp.\ $-C_-', C_+'$), where we are following the notation of \Cref{definition:cobordism-decompII}. 

Suppose first that $-C_-=-C_-'$. Up to relabelling $\s$ and $\s'$, we can assume that $C_+ \leq C_+'$. Now just translate $C_+'$ in the negative direction until $C_+'=C_+$. This translation defines a $1$-parameter family of cobordism decompositions taking $\s'$ to $\s$, whose energy is clearly bounded by $\op{max}(\mc{E}(\s), \mc{E}(\s'))= \mc{E}(\s')<C$. An analogous argument works if we now suppose $C_+=C_+'$ and $-C_- \neq -C_-'$. 

Suppose finally that $-C_- \neq -C_-'$ and $C_+ \neq C_+'$. Up to relabelling $\s$ and $\s'$, we can assume that $C_+< C_+'$. If $-C_-' < -C_-$, then we translate $-C_-'$ in the positive direction until $-C_-'=-C_-$. If instead $-C_-< -C_-'$, then we simultaneously translate $-C_-$ and $C_+$ in the positive direction until either $-C_-=-C_-'$ or $C_+=C_+'$. This takes us back to the case treated in the previous paragraph.

Finally, if $(\hat{X}, \hat{\l})= (SY,\l_Y)$ is a symplectization with $\l^+=\l^-$, then any Type A cobordism decomposition $\s \in \S(\hat{X}, \hat{\l})_{\leq 0}$ has vanishing energy (\Cref{remark:energy-symplectization-same-marking}) and is equivalent to a choice of hypersurface $\mc{H}= \mc{H}_-= \mc{H}_+ = \{\tilde{C} \tms Y\}$. The space of such choices is in natural bijection with $\R$, so it is in particular non-empty and connected.  
\epf

\defi 
Let $(\hat{X}, \hat{\l}_t)_{t \in [0,1]}$ be a one-parameter family of (strict) exact symplectic cobordisms (cf.\ \Cref{definition:family-sympcobordisms}). A one-parameter family of Type A cobordism decompositions is just the data of a family of hypersurfaces
\begin{align}
\mc{H}_-(t)= \{-C_-(t)\} \tms Y^-  &\text{ and } \mc{H}_+(t) = \{ C_+(t) \} \tms Y^+
\end{align}
such that 
\eq \label{equation:cob-non-intersect-family} ( (-\infty, -C_-(t)) \tms Y^-)  \cap  (( C_+(t), \infty) \tms Y^+) = \emptyset.\eeq (If $Y^-= \emptyset$, we again set $\mc{H}_-(t) = \emptyset$, $C_-(t)=0$ and we consider that \eqref{equation:cob-non-intersect-family} is tautologically satisfied). We let $\S(\hat{X}, \hat{\l}_t)_{t \in [0,1]}$ be the set of all such families of cobordism decompositions. (Note that $\S(\hat{X}, \hat{\lambda}_{t_0})$ is a Type A cobordism decomposition for each fixed choice of $t_0$).  
\edefi

\defi
With the notation as above, with define the energy of a family of Type A cobordism decompositions $\s \in \S(\hat{X}, \hat{\l}_t)_{t \in [0,1]}$ to be $\mc{E}(\s):= \op{sup}_t(C_-(t) + C_+(t))$. 
\edefi

Let $(\hat{X}^{01}, \hat{\l}^{01})$ (resp. $(\hat{X}^{12}, \hat{\l}^{12})$) be a strict exact symplectic cobordism from $(Y^0, \l^0)$ to $(Y^1, \l^1)$ (resp. from $(Y^1, \l^1)$ to $(Y^2, \l^2)$).  Let $(\hat{X}, \hat{\l}_t)_{t \in{} [0,\infty)}$ be a one-parameter family of strict exact symplectic cobordisms which agrees for $t \geq a$ large enough with the $t$-gluing $(\hat{X}^{01} \#_t \hat{X}^{12}, \hat{\l}^{01} \#_t \hat{\l}^{02})_{t \in{} [a, \infty)}$; see \Cref{definition:gluing-cobordisms}. For $t \geq a$, note that there are canonical Liouville embeddings 

\begin{center}
\begin{tikzcd}[row sep = small, column sep = large]
\iota_{0,t}: (\hat{X}^{01}, \hat{\l}^{01}) \ar{r}{\mu_{t/2}} & (\hat{X}^{01}, e^{t/2} \hat{\l}^{01}) \ar{r} &(\hat{X}^{01} \#_t \hat{X}^{12}, \hat{\l}^{01} \#_t \hat{\l}^{02}) \\
\iota_{2,t}: (\hat{X}^{12}, \hat{\l}^{12}) \ar{r}{\mu_{-t/2}}  &(\hat{X}^{12}, e^{-t/2} \hat{\l}^{12}) \ar{r} &(\hat{X}^{01} \#_t \hat{X}^{12}, \hat{\l}^{01} \#_t \hat{\l}^{02}).
\end{tikzcd}
\end{center}

\defi \label{definition:typeb-decomp}
A \emph{Type B cobordism decomposition} of $(\hat{X}, \hat{\l}_t)_{t \in{} [0, \infty)}$ is the data of a family of hypersurfaces
\begin{align}
\mc{H}_2(t)= \{-C_2(t)\} \tms Y^2  &\hspace{1cm} \mc{H}_0(t) = \{ C_0(t) \} \tms Y^0
\end{align}
and a Liouville embedding $( [-C_1(t), \tilde{C}_1(t)] \tms Y^1, e^s\l_1) \hookrightarrow (\hat{X}, \hat{\l}_t)$ such that 
\eq  \label{equation:disjoint-4} (-\infty, -C_2(t) ) \tms Y^2, \, (-C_1(t), \tilde{C}_1(t)) \tms Y^1, \, (C_0(t), \infty) \tms Y^0 \eeq are pairwise disjoint. (In case $Y^-=\emptyset$, we set $\mc{H}_2(t)=\emptyset$, $C_2(t)=0$ and replace \eqref{equation:disjoint-4} by the condition that $(-C_1(t), \tilde{C}_1(t)) \tms Y^1$ and  $(C_0(t), \infty) \tms Y^0$ are pairwise disjoint.)

We let
\begin{align}
\mc{H}_1(t)= \{-C_1(t)\} \tms Y^1,  &\hspace{1cm} \tilde{\mc{H}}_1(t) = \{ \tilde{C}_1(t) \} \tms Y^1.
\end{align}

This data is required to satisfy the following hypotheses:
\begin{itemize}
	\item[(1)] $\tilde{C}_1(0)=-C_1(0)$, 
	\item[(2)] for $t$ large enough, $\mc{H}_0, \tilde{\mc{H}}_1(t)$ (resp.\ $\mc{H}_1(t), \mc{H}_2(t)$) are in the image of the canonical embedding $\iota_{0, t}$ (resp.\ $\iota_{2,t}$). Moreover, their preimages define a Type A decomposition on $(\hat{X}^{01}, \hat{\l}^{01})$ (resp.\ on $(\hat{X}^{12}, \hat{\l}^{12})$) \emph{which is independent of $t$.}  
\end{itemize}

We let $\S_B((\hat{X}, \hat{\l}_t)_{t \in{} [0,\infty)})$ denote the set of all such cobordism decompositions. We will write $\S(-)$ instead of $\S_B(-)$ when the subscript is understood from the context. As in \Cref{remark:hypersurface-constant}, note that the data of the hypersurfaces $\mc{H}_2(t), \mc{H}_1(t), \tilde{\mc{H}}_1'(t), \mc{H}_0(t)$ is equivalent to the data of the constants $C_2(t), C_1(t), \tilde{C}_1(t), C_0(t)$. 
\edefi

\defi
It follows from property (1) of \Cref{definition:typeb-decomp} that a Type B cobordism decomposition $\s^{02} \in \S_B((\hat{X}, \hat{\l}_t)_{t \in{} [0,\infty)})$ induces a Type A cobordism decomposition $\s \in \S_A(\hat{X}, \hat{\l}_0)$ by taking $\mc{H}_-=\mc{H}_2(0)$ and $\mc{H}_+= \mc{H}_0(0)$. We say that $\s$ is \emph{induced at zero} by $\s^{02}$. 

Similarly, property (2) of \Cref{definition:typeb-decomp} states that a Type B cobordism decomposition $\s^{02} \in \S_B((\hat{X}, \hat{\l}_t)_{t \in{} [0,\infty)})$ induces a pair of Type A decompositions $\s^{01} \in \S_A(\hat{X}^{01}, \hat{\l}^{01})$ and $\s^{12} \in \S_A(\hat{X}^{02}, \hat{\l}^{02})$. We say that the pair $(\s^{01}, \s^{12})$ is \emph{induced at infinity} by $\s^{02}$. 
\edefi

\defi
With the notation as above, we define the energy of a Type B cobordism decomposition $\s \in \S_B(\hat{X}, \hat{\l}_t)$ to be $\mc{E}(\s):= \op{sup}_t (C_{2}(t) + C_{0}(t) - C_1(t) - \tilde{C}_1(t))$. We let 
\eq \mc{E}(\hat{X}, \hat{\l}_t):= \op{inf}_{\s \in \S_B(\hat{X}, \hat{\l}_t)}\mc{E}(\s) \in \R \cup \{-\infty\}. \eeq
Given $C \in \R$, let $\S(\hat{X}, \hat{\l}_t)_{<C} \sub  \S(\hat{X}, \hat{\l}_t)$ (resp. $\leq C$) denote the subset of Type B cobordism decompositions of energy strictly less than $C$ (resp. at most $C$). 
\edefi

The following lemma asserts that our notions of energy for Type A and Type B decompositions are compatible with the map which associates to a Type B decomposition the Type A decomposition induced at zero or infinity.  It will be used implicitly in the sequel. 
\lem \label{lemma:energy-types-compatible}
Let $\s^{02}$ be a Type B cobordism decomposition. Suppose that $\s$ is induced at zero by $\s^{02}$ and that $(\s^{01}, \s^{12})$ is induced at infinity. Then $\mc{E}(\s) \leq \mc{E}(\s^{02})$ and $\mc{E}(\s^{01})+ \mc{E}(\s^{12}) \leq \mc{E}(\s^{02})$.  
\elem
\pf
The first claim follows from (1) in \Cref{definition:typeb-decomp} and the definition of energy for Type A and Type B cobordism decomposition. The second claim follows similarly from (2) in \Cref{definition:typeb-decomp}. 
\epf

\cor \label{lemma:minus-infinity-filling-typeb}
We have $\mc{E}(\hat{X}, \hat{\l}_t)= -\infty$ if and only if $Y^2=\emptyset$.
\ecor

\pf
One direction follows from \Cref{lemma:minus-infinity-filling} and \Cref{lemma:energy-types-compatible}. The other one can be checked by inspection, using the backwards Liouville flow as in the proof of the corresponding statement in \Cref{lemma:minus-infinity-filling}. 
\epf

\defi
	Let $(\hat{X}, \hat{\l})$ and $(\hat{X}', \hat{\l}')$ be exact symplectic cobordisms. For $C \in \R$, let $\S_A((\hat{X}, \hat{\l}),(\hat{X}', \hat{\l}'))_{< C} \sub  \S_A(\hat{X}, \hat{\l}) \tms \S_A(\hat{X}, \hat{\l})$ (resp. $(-)_{\leq C}$) be the subspace of pairs $(\s, \s')$ of Type A cobordism decompositions such that $\mc{E}(\s)+ \mc{E}(\s') <C$ (resp. $\leq C$). 
\edefi 

\lem \label{lemma:gluing-surjective-energy-decomp}
Given $C \in \R$ such that $\S(X^{02,t}, \hat{\l}^{02,t})_{<C}$ is nonempty, the map which associates to a decomposition $\s^{02} \in \S(X^{02,t}, \hat{\l}^{02,t})_{<C}$ the pair $(\s^{01}, \s^{12}) \in \S_A((\hat{X}^{01}, \hat{\l}^{01}),(\hat{X}^{12}, \hat{\l}^{12}))_{<C}$ induced by $\s^{02}$ at infinity is surjective.  If $\l^0=\l^1=\l^2$ and $(\hat{X}^{01}, \hat{\l}^{01}),(\hat{X}^{12}, \hat{\l}^{12})$ are symplectizations, the same statement holds with $``\leq"$ in place of $``<"$.
\elem

\pf
Choose a Type B decomposition $\tilde{\s}^{02}$. Let $(\tilde{\s}^{01}, \tilde{\s}^{12})$ be the Type A decompositions induced by $\tilde{\s}^{02}$ at infinity. According to \Cref{definition:typeb-decomp}, this means that there exists a $T>0$ so that for $t \geq T$, we have that $\mc{H}_0(t), \tilde{\mc{H}}_1(t)$ are independent of $t$ after pulling back via the canonical embedding $\iota_{01}$ (and similarly $\mc{H}_1(t), \mc{H}_2(t)$ are independent of $t$ after pulling back by $\iota_{12}$).  By a routine modification of the arguments of \Cref{lemma:cob-decomp-connected-II}(b), one can now construct a Type B decomposition $\s^{02}$ so that $\s^{02}_t= \tilde{\s}^{02}_t$ for $t \in [0,T]$, $\mc{E}(\s^{02}) \leq \mc{E}(\tilde{\s}^{02})$ and $\s^{02}$ induces the pair $(\s^{01}, \s^{12})$.
\epf

\lem \label{lemma:energy-trivial-component}
Suppose that $(\hat{X}^{02,t}, \hat{\l}^{02,t})$ is the $(t+T)$-gluing of two exact symplectic cobordisms $(\hat{X}^{01}, \hat{\l}^{01})$ and $(\hat{X}^{12}, \hat{\l}^{12})$, for $T\geq 0$ an arbitrary fixed constant and $t \in{} [0, \infty)$ a parameter (see \Cref{definition:gluing-cobordisms} for the definition of this gluing and \Cref{lemma:t-gluing} for the parametric version). Suppose that either $(\hat{X}^{01}, \hat{\l}^{01})$ or $(\hat{X}^{12}, \hat{\l}^{12})$ is a symplectization (see \Cref{example:symplectization}). Then $\mc{E}(\hat{X}^{02,t}, \hat{\l}^{02,t})= \mc{E}(\hat{X}^{01}, \hat{\l}^{01}) + \mc{E}(\hat{X}^{12}, \hat{\l}^{12})$. 
\elem

\pf
By \Cref{lemma:minus-infinity-filling} and \Cref{lemma:minus-infinity-filling-typeb}, we may assume that $\hat X^{12}$ has a non-empty negative end.

We only treat the case where $(\hat{X}^{01}, \hat{\l}^{01})$ is a symplectization and $T=0$ since the other cases are analogous. 

Choose $\s^{01}$ so that $\mc{E}(\s^{01}) \leq \mc{E}(\hat{X}^{01}, \hat{\l}^{01})+ \e$ and choose $\s^{12}$ so that  $\mc{E}(\s^{12}) \leq \mc{E}(\hat{X}^{12}, \hat{\l}^{12})+ \e$.  Let $\tilde{X}^{01}\sub \hat{X}^{01}$ and $\tilde{X}^{12} \sub \hat{X}^{12}$ be the Liouville subdomains which determine the Type A decompositions $\s^{01}$ and $\s^{12}$ respectively.

Note that $\hat{X}^{02,t}$ comes equipped with tautological embeddings $\iota_{0,t}: \hat{X}^{01} \to \hat{X}^{02,t}$ and $\iota_{2,t}: \hat{X}^{12} \to \hat{X}^{02,t}$ (see \Cref{definition:gluing-cobordisms}). For $T'$ large enough and $t \geq T'$, note that $\iota_{0,t}(\mc{H}^0_-)$ is in the image of $\iota_{2,t}(\mc{H}^2_+)$ under the Liouville flow. These hypersurfaces therefore bound Liouville domains $( [-C_1(t), \tilde{C}_1(t)] \tms Y^1, e^s \l_1)$. 

Let $f: [0, \infty) \to \R$ be a function which equals $-(C_1(T')+\tilde{C}_1(T'))$ on $[0, T']$, is non-decreasing on $[T', T'+1]$ and is zero on $[T'+1, \infty)$. Let $\tau_f: \hat{X}^{01} \tms{} [0, \infty) \to  \hat{X}^{01}$ be defined by $\tau_f(x, t) = \phi^{01}_{f(t)+T'+1-t}(x)$, where $\phi^{01}_{t}$ is the time-$t$ Liouville flow on $\hat{X}^{01}$. Now define a map $\ov{\iota}_{0,t}: \hat{X}^{01} \to \hat{X}^{02,t}$ by letting $\ov{\iota}_{0,t} (x)= \iota_{0,t} \circ \tau_f(x,t)$. 

We now define the data of a Type B cobordism decomposition by letting 
\eq 
\mc{H}_2(t)= \iota_{2,t}(\mc{H}^2_-), \hspace{0.3cm} \mc{H}_1(t)= \iota_{2,t}(\mc{H}^2_+),  \hspace{0.3cm}  \tilde{\mc{H}}_1= \ov{\iota}_{0,t}( \mc{H}^0_-(t)), \hspace{0.3cm}  \mc{H}_0 = \ov{\iota}_{0,t}(\mc{H}^0_+(t)).
\eeq
One can check that this data indeed defines a Type B cobordism decomposition, which has energy precisely equal to $\mc{E}(\s^{01})+ \mc{E}(\s^{12}) \leq \mc{E}(\hat{X}^{01}, \hat{\l}^{01}) + \mc{E}(\hat{X}^{12}, \hat{\l}^{12})+ 2\e$. Since $\e$ was arbitrary, we conclude that $\mc{E}(\hat{X}^{02,t}, \hat{\l}^{02,t}) \leq \mc{E}(\hat{X}^{01}, \hat{\l}^{01}) + \mc{E}(\hat{X}^{12}, \hat{\l}^{12})$. 
\epf

We now discuss almost-complex structures for Setups II-IV.

\begin{setup2}

Fix a datum $\mc{D} = (\mc{D}^+,\mc{D}^-,\hat{X},\hat{\l},\hat{J})$ for Setup II, $\mc{D}^\pm = (Y^\pm,\l^\pm,J^\pm)$. Let $(V^{\pm}, \tau^{\pm}) \sub (Y^{\pm},\xi^{\pm})$ be framed codimension 2 contact submanifolds; let $\a^{\pm}:= \l^{\pm}|_{V^\pm}$ and assume that $V^\pm$ is a strong contact submanifold with respect to $\l^\pm$.  Let $J^{\pm}$ be $d\l^\pm$-compatible almost-complex structures on $\xi^\pm \sub TY^{\pm}$ which preserve $\xi^{\pm} \cap TV^{\pm}$. 

Let $H \sub \hat{X}$ be a codimension $2$ symplectic submanifold such that $(\hat{X},\hat{\l}, H)$ is an exact relative symplectic cobordism from $(Y^+,\xi^+,V^+)$ to $(Y^-,\xi^-,V^-)$. We will also consider (see \Cref{notation:choice-forms}) the strict symplectic cobordisms $(\hat{X},\hat{\l})^{\l^+}_{\l^-}$ and $(H, \hat{\l}|_H)^{\a^+}_{\a^-}$. 

\defi \label{definition:ac-adapted-II}
	Fix a Type A cobordism decomposition $\s \in \S(H, \hat{\l}|_H)$, which is specified by a pair of hypersurfaces $\mc{H}_-= \{ -C_- \} \tms V^-$ and $\mc{H}_+ = \{ C_+ \} \tms V^+$. We say that an almost-complex structure $\hat{J}$ on $\hat{X}$ is \emph{adapted} to $\s$ if the following properties hold:
	\begin{itemize}
		\item $\hat{J}$ is compatible with $d \hat{\l}$,
		\item $\hat{J}$ coincides with $\hat{J}^{\pm}$ near the ends (where $\hat{J}^{\pm}$ is the canonical cylindrical almost-complex structure induced on $(\hat{Y},\hat{\l}^\pm)$ by $J^{\pm}$),
		\item $H \sub \hat{X}$ is a $\hat{J}$-complex hypersurface,
		\item $\hat{J}$ preserves $\op{ker} \a^+ \sub TV^+$ on $[C_2, \infty) \tms V^+$ (resp. preserves $\op{ker}\a^- \sub TV^-$ on $(-\infty, -C_1] \tms V^-$) and the induced almost-complex structure is $d \a^+$-compatible (resp. $d\a^-$-compatible). 
	\end{itemize} 
In case $V^-=\emptyset$, the conditions involving $V^-$ are considered to be vacuously satisfied.
\edefi
	
\defi \label{definition:energy-ac-II}	
Given an almost-complex structure $\hat{J}$ on $(\hat{X}, \hat{\l})$ we define its energy 
\eq \mc{E}(\hat{J}):= \op{inf}\{\mc{E}(\s) \mid \s \in \S(H, \hat{\l}|_{H}), \; \hat{J} \; \text{is adapted to } \s \} \in \R \cup \{\pm \infty \}. \eeq We define $\mc{E}(\hat{J})= \infty$ if $\hat{J}$ is not adapted to any cobordism decomposition.  
\edefi
		
Let $\mc{J}(\hat{X}, \hat{\l}, H)_{< C}$ (resp. $\leq C$) be the set of almost-complex structures of energy less than $C$ (resp. at most $C$). Let $\mc{J}(\hat{X}, \hat{\l}, H):= \mc{J}(\hat{X}, \hat{\l}, H)_{<\infty}$ be the set of almost-complex structures adapted to some decomposition $\s \in \S(\hat{H}, \hat{\l}|_H)$. 
	
\lem \label{lemma:ac-connected-II}
	The set $\mc{J}(\hat{X}, \hat{\l}, H)_{<C}$ is 
	\begin{itemize}
		\item[(a)] nonempty for $C> \mc{E}(H,\hat{\l}|_H)$, 
		\item[(b)] path-connected for all $C \in \R$ (note that the empty set is path-connected).
	\end{itemize}
	If moreover $(H, \hat{\l}|_H) = (SV,\lambda_V)$ is a symplectization and $\a^+=\a^-$, then $\mc{J}(\hat{X}, \hat{\l}, H)_{ \leq 0}$ is non-empty and path-connected.
\elem

\pf
To prove (a), it is enough to show that given any cobordism decomposition $\s$, there exists an almost-complex structure adapted to it, i.e. meeting the conditions of \Cref{definition:ac-adapted-II}. To prove (b), it follows from \Cref{lemma:cob-decomp-connected-II} that it is enough to prove a similar statement in families: namely, if $\{\s_t\}_{t \in [a,b]}$ is a family of cobordism decompositions and $J_a, J_b$ are almost-complex structures adapted to $\s_a, \s_b$ respectively, then there is a family $\{J_t\}_{t \in [a,b]}$ adapted to $\s_t$.  All of these statements can be proved by standard arguments, using the fact that the space of almost-complex structures compatible with a given symplectic structure can be viewed as the space of sections of a bundle with contractible fibers (see e.g.\ \cite[Prop.\ 2.6.4]{mcduff-sal-intro}). 

If $\a^+= \a^-$ and $(H, \hat{\l}|_H)$ is a symplectization, then \Cref{lemma:cob-decomp-connected-II} implies that $\S(H, \hat{\l}|_H)_{\leq 0}$ is non-empty and path-connected. Hence the same arguments involving extensions of almost-complex structures imply that $\mc{J}(\hat{X}, \hat{\l}, H)_{\le 0}$ is non-empty and path-connected.
\epf

\end{setup2}

\begin{setup4}

Fix a datum $\mc{D}$ for Setup IV.  We write $\mc{D}= (\mc{D}^{01},\mc{D}^{12},(\hat{X}^{02,t},\hat{\l}^{02,t},\hat{J}^{02,t})_{t \in{} [0,\infty)})$, where
\begin{align*}
	\mc{D}^{01} &= (\mc{D}^0,\mc{D}^1,\hat{X}^{01},\hat{\l}^{01},\hat{J}^{01}) \\
	\mc{D}^{12} &= (\mc{D}^1,\mc{D}^2,\hat{X}^{12},\hat{\l}^{12},\hat{J}^{12}) \\
	\mc{D}^i &= (Y^i,\l^i,J^i) \quad (i = 0,1,2)
\end{align*}

Let $(V^i, \tau^i) \sub (Y^i, \xi)$ be framed codimension 2 contact submanifolds; set $\a_{V^i}= \l^i|_{V^i}$ and assume that $V^i$ are strong contact submanifolds with respect to $\l^i$. Let $H^{01} \subset \hat{X}^{01}$, $H^{12} \subset \hat{X}^{12}$, and $(H^{02,t} \subset \hat{X}^{02,t})_{t \in{} [0,\infty)}$ be cylindrical symplectic submanifolds such that $(\hat{X}^{02,t},\hat{\l}^{02,t},H^{02,t})_{t \in{} [0,\infty)}$ is a family of relative symplectic cobordisms that agrees for $t$ large with the $t$-gluing of the relative symplectic cobordisms $(\hat{X}^{01},\hat{\l}^{01},H^{01})$ and $(\hat{X}^{12},\hat{\l}^{12},H^{12})$. Note that $\{H^{02,t} \}$ forms a family of Liouville manifolds with respect to (the restriction of) $\hat{\l}^{02,t}$. 

\defi
Fix a Type B cobordism decomposition $\s^{02} \in \S_B(\hat{H}^{02,t}, \hat{\l}^{02,t}_{H^{02,t}})$. Recall that $\s^{02}$ consists in the data of hypersurfaces $\mc{H}_2(t)= \{ - C_2(t) \} \tms V^2, \mc{H}_1(t) = \{-C_1(t) \} \tms V^1, \tilde{\mc{H}}_1(t) = \{ \tilde{C}_1(t) \} \tms V^1, \mc{H}_0(t) = \{ C_0(t) \} \tms V^0$. We say that an almost-complex structure $\hat{J}^{02,t}$ is \emph{adapted} to $\s^{02}$ if the following properties hold:
\begin{itemize}
\item $\hat{J}^{02,t}$ is compatible with $d\hat{\l}^{02,t}$,
\item $\hat{J}^{02,t}$ coincides with $\hat{J}^0$ (resp.\ $\hat{J}^2$) near the positive (resp.\ negative) end,
\item $H^{02,t}$ is a $\hat{J}^{02,t}$-complex hypersurface, and $\hat{J}^{02,t}$ is compatible with the restriction of $d\hat{\l}^{02,t}$ to $H^{02,t}$,
\item $\hat{J}^{02,t}$ preserves $\op{ker} \a_0$ on $[C_0(t), \infty) \tms V^0$ (resp. $\op{ker} \a_1$ on $[-C_1(t), \tilde{C}_1(t)] \tms V_1$, resp. $\op{ker} \a_2$ on $(-\infty, -C_2(t)] \tms V_2)$. Moreover, the induced almost-complex structure is $d\a_0$-compatible (resp. $d\a_1$-compatible, resp. $d\a_2$-compatible). 
\end{itemize}
In case $V^2=\emptyset$, all conditions involving $V^2$ are considered to be vacuously satisfied.
\edefi

\defi
Given a family of almost-complex structures $\hat{J}_t$, we define its energy 
\eq \mc{E}(\hat{J}_t):= \op{inf} \{ \mc{E}(\s) \mid \s \in \S(\hat{X}^{02,t}, \hat{\l}^{02,t}), \hat{J}_t\; \text{is adapted to } \s \} \in \R \cup \{\pm \infty \}.\eeq 
If $\hat{J}_t$ is not adapted to any cobordism decomposition, we set $\mc{E}(\hat{J}_t)= \infty$. 
\edefi

Let $\mc{J}(\hat{X}^{02,t}, \hat{\l}^{02,t}, H^{02,t})$ be the set of almost-complex structures adapted to some Type B decomposition $\s \in \S_B(H^{02,t}, \hat{\l}^{02,t}|_{H^{02,t}})$. For $C \in \R$, let $\mc{J}(\hat{X}^{02,t}, \hat{\l}^{02,t}, H^{02,t})_{<C}$ (resp. $\leq c$) be the set of all such decompositions having energy less than $C$ (resp. at most $C$). 

Let $\mc{J}((\hat{X}^{01}, \hat{\l}^{01}), (\hat{X}^{12}, \hat{\l}^{12}))_{<C} \sub  \mc{J}(\hat{X}^{01}, \hat{\l}^{01}) \tms \mc{J}(\hat{X}^{12}, \hat{\l}^{12})$ (resp. $\leq C$) be the subspace of pairs $(J, J')$ with the property that $\mc{E}(J)+ \mc{E}(J') <C$ (resp. $\leq C$). 

The following lemma is an analog of \Cref{lemma:ac-connected-II} and can be proved by similar arguments. 

\lem \label{lemma:ac-connected-IV}
	The set $\mc{J}(\hat{X}^{02,t}, \hat{\l}^{02,t}, H^{02,t})_{<C}$ is	nonempty for $C> \mc{E}(\hat{X}^{02,t}, \hat{\l}^{02,t}, H^{02,t} )$. If moreover $\a_0=\a_1=\a_2$ and $(H^{01}, \hat{\l}^{01}|_{H^{01}}), (H^{12}, \hat{\l}^{12}|_{H^{12}})$ and $(H^{02,t}, \hat{\l}^{02,t}|_{H^{02,t}})$ are symplectizations, then $\mc{J}(\hat{X}^{02,t}, \hat{\l}^{02,t}, H^{02,t})_{\leq 0}$ is non-empty.  \qed
\elem

We will also need the following lemma, which follows from \Cref{lemma:gluing-surjective-energy-decomp} and standard arguments for extending compatible almost-complex structures. 

\lem \label{lemma:gluing-surjective-energy-ac}
Suppose that $\mc{J}(\hat{X}^{02,t}, \hat{\l}^{02,t}, H^{02,t})_{<C}$ is nonempty. Then the map which associates to an almost-complex structure $\hat{J}_t \in \mc{J}(\hat{X}^{02,t}, \hat{\l}^{02,t}, H^{02,t})_{<C}$ the pair $(\hat{J}^{01}, \hat{J}^{12}) \in \mc{J}((\hat{X}^{01}, \hat{\l}^{01}), (\hat{X}^{12}, \hat{\l}^{12}))_{<C}$ is surjective for all $C>0$. 

If moreover $\a_0=\a_1=\a_2$ and $(H^{01}, \hat{\l}^{01}|_{H^{01}}), (H^{12}, \hat{\l}^{12}|_{H^{12}})$ and $(H^{02,t}, \hat{\l}^{02,t}|_{H^{02,t}})$ are symplectizations, then the same statement holds for $C=0$ with $``\leq"$ in place of $``<"$.  \qed
\elem

\end{setup4}

\section{Enriched setups and twisted moduli counts}  \label{section:enriched-setups-twisted-counts}

\subsection{Enriched setups} \label{subsection:enriched-setups}

The construction of invariants of codimension $2$ contact submanifolds in this paper follows the same general scheme as Pardon's construction of contact homology. However, we work with a class of ``enriched" setups I*-IV*, which contain more information than the standard setups I-IV considered by Pardon and reviewed in \Cref{subsection:standard-setups}.

We will show in \Cref{subsection:twisting-maps-contact-subman} that the data associated to our enriched setups give rise to twisting maps. These twisting maps are constructed using Siefring's intersection theory, and will be used to define ``twisted" moduli counts, following the construction of \Cref{subsection:twistinghomologies}. 

Given a datum $\mc{D}$ for any of Setups I*-IV*, there is a ``forgetful functor" which allows one to view $\mc{D}$ as a datum of Setup I-IV.  However, it is not the case that every datum of Setup I-IV admits an enrichment. Nevertheless, we will show in \Cref{subsection:mainconstruction} that the class of enriched data is large enough for the purpose of defining invariants in the spirit of contact homology. 

\begin{setup1star}
A datum $\mc{D}=((Y, \xi, V), \fk{r}, \l, J)$ for Setup I* consists of:
\begin{itemize}
	\item A TN contact pair $(Y, \xi, V)$,
	\item an element $\fk{r}=(\a_V, \tau, r) \in \fk{R}(Y, \xi, V)$,
	\item a contact form $\op{ker} \l = \xi$ which is adapted to $\fk{r}$,
	\item an almost-complex structure $J$ which is compatible with $d\l$ and preserves $\xi_V$.
\end{itemize}
Observe that there is a ``forgetful functor" from Setup I* to Setup I which remembers $(Y, \l, J)$ but forgets $V$ and $\fk{r}$. One has analogous forgetful functors for the other setups.
\end{setup1star}

\begin{setup2star}
A datum $\mc{D}=(\mc{D}^+, \mc{D}^-, \hat X, \hat \l, H, \hat J)$ for Setup II* consists of:
\begin{itemize}
	\item data $\mc{D}^\pm = ((Y^\pm,\xi^{\pm}, V^{\pm}), \fk{r}^{\pm}, \l^\pm, J^\pm)$ for Setup I*, where we write $\fk{r}^\pm=(\a^\pm_V, \tau^\pm, r^\pm)$;
	\item an exact relative symplectic cobordism $(\hat{X},\hat{\l}, H)$ with positive end $(Y^+,\l^+, V^+)$ and negative end $(Y^-,\l^-, V^{\pm})$;
	\item an $d\hat \l $-tame almost complex structure $\hat{J}$ on $\hat{X}$ which agrees with $\hat{J}^\pm$ at infinity.
\end{itemize}
This datum is moreover subject to the following conditions: 
\begin{itemize}
	\item there exists a trivialization of the normal bundle of $H$ which restricts to $\tau^+$ (resp.\ $\tau^-$) on the positive (resp.\ negative) end; 
	\item $\mc{E}(\hat{J})<\infty$ and $r^+ \geq e^{\mc{E}(\hat{J})} r^-$.
\end{itemize}
\end{setup2star}
 
\begin{setup3star}
 A datum $\mc{D} = (\mc{D}^+,\mc{D}^-, \hat{X}, \hat{\l}^t,H^t, \hat{J}^t)_{t \in [0,1]}$ for Setup III* consists of:
\begin{itemize}
	\item data $\mc{D}^\pm = ((Y^\pm,\xi^{\pm}, V^{\pm}), \fk{r}^{\pm}, \l^\pm, J^\pm)$ for Setup I*;
	\item a family of exact relative symplectic cobordisms $(\hat{X}, \hat{\l}^t, \hat{H}^t)$ for $t \in [0,1]$, with positive end $(Y^+,\l^+, V^+)$ and negative end $(Y^-,\l^-, V^-)$; 
	\item a family $d\hat \l^t $-tame almost complex structures $\hat{J}^t$ on $\hat{X}$, which agree with $\hat{J}^\pm$ at infinity.
\end{itemize}
This datum is moreover subject to the following conditions: 
\begin{itemize}
	\item there exists a trivialization of the normal bundle of $H$ which restricts to $\tau^+$ (resp.\ $\tau^-$) on the positive (resp.\ negative) end; 
	\item $\mc{E}(\hat{J}^t)<\infty$ and $r^+ \geq e^{\mc{E}(\hat{J}^t)} r^-$. 
	\end{itemize}
\end{setup3star}

\begin{setup4star}
A datum $\mc{D}= (\mc{D}^{01},\mc{D}^{12},(\hat{X}^{02,t},\hat{\l}^{02,t},H^{02,t},\hat{J}^{02,t})_{t \in{} [0,\infty)})$ for Setup IV* consists of: 
\begin{itemize}
	\item data $\mc{D}^i = ((Y^i, \xi^i, V^i); \fk{r}^i,\l^i,J^i)$ for Setup I*, for $i = 0,1,2$;
	\item a datum $\mc{D}^{01} = (\mc{D}^0,\mc{D}^1,\hat{X}^{01},\hat{\l}^{01},H^{01}, \hat{J}^{01})$ for Setup II*;
	\item a datum $\mc{D}^{12} = (\mc{D}^1,\mc{D}^2,\hat{X}^{12},\hat{\l}^{12}, H^{12},\hat{J}^{12})$ for Setup II*;
	\item a family of cylindrical symplectic submanifolds $H^{02,t} \subset \hat{X}^{02,t}$, for $t \in{} [0,\infty)$, such that $(\hat{X}^{02,t},\hat{\l}^{02,t},H^{02,t})_{t \in{} [0,\infty)}$ is a family of exact relative symplectic cobordisms that agrees for $t$ large with the $t$-gluing of the relative symplectic cobordisms $(\hat{X}^{01},\hat{\l}^{01},H^{01})$ and $(\hat{X}^{12},\hat{\l}^{12},H^{12})$.
\end{itemize}
This datum is moreover subject to the following conditions: 
\begin{itemize}
	\item there exists a trivialization of the normal bundle of $H^{02,t}$ which restricts to $\tau^0$ (resp.\ $\tau^2$) on the positive (resp.\ negative) end; 
	\item there exists a trivialization of the normal bundle of $H^{01}$ which restricts to $\tau^0$ (resp.\ $\tau^1$) on the positive (resp.\ negative) end; 
	\item there exists a trivialization of the normal bundle of $H^{12}$ which restricts to $\tau^1$ (resp.\ $\tau^2$) on the positive (resp.\ negative) end; 
	\item $\mc{E}(\hat{J}^{02,t})<\infty, \mc{E}(\hat{J}^{01})<\infty$ and $\mc{E}(\hat{J}^{12}) <\infty$;
	\item $r^0 \geq e^{\mc{E}(\hat{J}^{02,t})} r^2$; $r^0 \geq e^{\mc{E}(\hat{J}^{01})} r^1$; and $r^1 \geq e^{\mc{E}(\hat{J}^{12})} r^2$.
\end{itemize}
\end{setup4star}

We note that the requirement in Setups II*--IV* that the almost-complex structures have finite energy is of course vacuous if $r^-, r_1, r_2$ are nonzero (or equivalently $V^-, V_1, V_2$ are non-empty). 

\subsection{Twisting maps associated to enriched setups} \label{subsection:twisting-maps-contact-subman} 

In this section, we construct twisting maps on the contact homology algebra. These maps depend on geometric data involving codimension $2$ contact submanifolds and relative symplectic cobordisms. 

\begin{setup1star}

Let $\mc{D}=((Y, \xi, V), \fk{r}, \l, J)$ be a datum for Setup I*, where $\fk{r}=(\a_V, \tau, r)$. There is an obvious functor from $\Sch_\I(\mc{D})$ to the category $\widehat{\Sch}(\hat{Y})$ defined in section~\ref{subsection:intersection-buildings}. We therefore have a well-defined intersection number $T * \hat{V}$ for $T \in \Sch_\I(\mc{D})$. We now introduce twisting maps associated to the above setup. 

\defi \label{definition:twist-basic-I}
We define a map $\psi_V(T) : \Sch_\I^{\neq \emptyset}(\mc{D}) \to \Q[U]$ by
\eq \psi_V(T) = U^{T * \hat{V} + \G^-(T,V)}, \eeq
where $\G^-(T,V)$ denotes the number of output edges $e$ of $T$ such that the corresponding Reeb orbit $\g_e$ is contained in $V$. \Cref{corollary:full-positivity} ensures that the exponents appearing in these definitions are nonnegative. 
\edefi

\rmk \label{remark:non-empty-tree-category}
\Cref{corollary:full-positivity} only applies to trees $T$ such that $\Mbar(T) \neq \emptyset$. This is why the definition of twisting maps only requires them to be defined on $\Sch^{\neq \emptyset}$ and not on the whole category $\Sch$. 
\ermk

\defi \label{definition:twist-ell-red-I} 
We define a map $\widetilde{\psi}_{V} : \Sch_\I^{\neq \emptyset}(\mc{D}) \to \Q$ by
\eq \widetilde{\psi}_V(T) = \begin{cases} 1 & \text{if $ T * \hat{V} = 0$ and $\left|\gamma_e\right| \cap V = \emptyset$ for every $e \in E(T)$} \\ 0 & \text{otherwise} \end{cases} \eeq
\edefi

We must now check that the maps in Definitions \ref{definition:twist-basic-I} and \ref{definition:twist-ell-red-I} satisfy the axioms of \Cref{definition:twistingI}.

\prop \label{proposition:basic-twisting-mapI}
The map $\psi_V$ introduced in \Cref{definition:twist-basic-I} is a twisting map.
\eprop
\pf
It follows from \Cref{proposition:invariance-gluing} that $\psi_V(T) = \psi_V(T')$ for any morphism $T \to T'$.

Let $\{T_i\}_{i}$ be a concatenation in $\Sch_\I^{\neq \emptyset}$. We need to show that $\psi_V(\#_i T_i) = \prod_i \psi_V(T_i)$, i.e.
\eq \label{equation:elliptic-additivityI}
	(\#_i T_i) * \hat{V} +  \G^-(\#_i T_i,V) = \sum_i T_i * \hat{V} +  \G^-(T_i,V).
\eeq

We assume $V \neq \emptyset$ (otherwise there is nothing to say). Since the contact form $\l$ is positive-elliptic near $V$, we have $p_N(\g) = 1$ for every Reeb orbit $\g$ contained in $V$ by \Cref{proposition:normal-CZ-indices}. \Cref{remark:intersection-disjoint-cylinders} and \Cref{corollary:intersection-cylinders} therefore imply that $\hat{\g} * \hat{V}$ is equal to $-1$ if $\g$ is contained in $V$ and $0$ otherwise. By \Cref{proposition:additivity-symplectization}, this means that
\eq \label{equation:elliptic-intersectionI}
T * \hat{V} = \sum_{v \in V(T)} \b_v * \hat{V} + \G^\mathrm{int}(T,V)
\eeq
for all $T \in \Sch_\I$, where $\G^\mathrm{int}(T,V)$ denotes the number of edges $e \in E^\mathrm{int}(T)$ such that $\g_e$ is contained in $V$. Equation~\eqref{equation:elliptic-additivityI} is therefore equivalent to
\eq
\G^\mathrm{int}(\#_i T_i,V) + \G^-(\#_i T_i,V) = \sum_i (\G^\mathrm{int}(T_i,V) + \G^-(T_i,V)).
\eeq
The result now follows from the observation that there is a (canonical) label preserving bijection between $E^\mathrm{int}(\#_i T_i) \cup E^-(\#_i T_i)$ and $\cup_i (E^\mathrm{int}(T_i) \cup E^-(T_i))$ (this is an immediate consequence of the definition: every interior edge of $T_j$ corresponds to an interior edge of $\#_i T_i$, and every output edge of $T_j$ corresponds either to an interior or an output edge of $\#_i T_i$ depending on whether it is identified with another edge in the concatenation or not).
\epf

It will be convenient to introduce the following definition. 

\defi
Given a tree $T\in \Sch_\I$, a vertex $v \in V(T)$ is \emph{mean} if it is an interior vertex and $\lvert \g_e \rvert \sub V$ for all $e \in e^+(v) \sqcup E^-(v)$.  All other vertices are said to be \emph{nice}. These sets are denoted $V_m(T) \sub V(T)$ and $V_n(T) \sub V(T)$ respectively. 
\edefi

\rmk
This notion of \emph{nice/mean} vertices is purely auxiliary (and has nothing to do with good/bad Reeb orbits!). Geometrically, mean vertices correspond to holomorphic buildings which have intermediate orbits intersecting $\hat{V}$. Nice orbits do not affect the intersection number of the building, but mean orbits do affect it and must therefore be treated carefully (hence the adjective). 
\ermk

\prop \label{proposition:reduced-elliptic-twisting-mapI}
The map $\widetilde{\psi}_{V}$ introduced in \Cref{definition:twist-ell-red-I} is a twisting map.
\eprop

\pf
Fix a tree $T \in \Sch^{\neq \emptyset}_\I$. We first show that $\widetilde{\psi}_{V}(T')= \widetilde{\psi}_{V}(T)$ for any tree $T' \in \Sch^{\neq \emptyset}_\I$ admitting a morphism $T' \to T$. Observe that we may assume without loss of generality that $T'$ is representable by a $\hat{J}$-holomorphic building (see \Cref{defi:rep-buildings-symplectization}). Indeed, since $T, T' \in \Sch^{\neq \emptyset}_\I$, there exists $T'' \to T' \to T$ such that $T''$ is representable by a $\hat{J}$-holomorphic building. So we may as well prove that $\widetilde{\psi}_{V}(T'') = \widetilde{\psi}_{V}(T')$ and $\widetilde{\psi}_{V}(T'') =\widetilde{\psi}_{V}(T)$.  

Let us therefore fix $T' \in \Sch^{\neq \emptyset}_\I$ such that $T'$ is representable by a $\hat{J}$-holomorphic building, and a morphism $T' \to T$. It follows from \Cref{proposition:invariance-gluing} that $T'* \hat{V}= T* \hat{V}$. Note that $T', T$ have the same exterior edges. If one of these edges is contained in $V$, then $\widetilde{\psi}_{V}(T')= \widetilde{\psi}_{V}(T)=0$. So we can assume that the exterior edges of $T', T$ are not contained in $V$. 
	
Suppose now that $T'$ has an interior edge contained in $V$.  For $i = 0,1,2$, let $X_i  \geq 0$ be the number of edges $e \in E(T')$ such that $|\g_e| \sub V$ and $e$ is adjacent to exactly $i$ mean vertices. By assumption, we have $X_2 + X_1 + X_0 \geq 1$. According to \Cref{proposition:additivity-symplectization}, we have 
\eq T' * \hat{V}= \sum_{v \in V_n(T')  } \beta_v * \hat{V} + \sum_{v \in V_m(T') } \beta_v * \hat{V}  +  X_2 + X_1 + X_0.\eeq 
According to \Cref{prop-positivity-intersections}, we also have that $\sum_{v \in V_n(T')} \beta_v * \hat{V}  \geq 0$ (here we use the fact that $T'$ is representable by a $\hat{J}$-holomorphic building). If there are no mean vertices, then we have that $\sum_{v \in V_m(T') } \beta_v * \hat{V}=0$, $X_1= X_2 = 0$ and $X_0 \geq 1$. So $T' * \hat{V} >0$. If there exists at least one mean vertex, observe that we have $ X_2 \leq \# V_m(T') -1$. Moreover, given $v \in V_m(T')$, \Cref{proposition:positivity-symplectization} together with the fact that $T'$ is representable by a $\hat{J}$-holomorphic building imply that $\beta_v * \hat{V} \geq 1- p_v$, where $p_v$ is the number of edges adjacent to $v$. It follows that $ \sum_{v \in V_m(T')} \beta_v * \hat{V} + X_2 + X_1 \geq (\#V_m(T')- X_1 - 2 X_2)  + X_2 + X_1 = \#V_m(T') - X_2 \geq 1$.  It thus follows again that $T' * \hat{V} >0$. We conclude that $\widetilde{\psi}_{H}(T')= \widetilde{\psi}_{H}(T)=0$ if $T'$ has an interior edge contained in $V$. 

We are left with the case where $T'$ and hence $T$ have no edges contained in $V$. It is then immediate that $\widetilde{\psi}_{V}(T')= \widetilde{\psi}_{V}(T)$. 
	
We now show that any concatenation $\{T_i\}_i$ satisfies $\widetilde{\psi}_{V}(\#_iT_i)= \prod_i \widetilde{\psi}_V(T_i)$. If one of the $T_i$ has an edge contained in $V$, then $\#_iT_i$ also has an edge contained in $V$ and we have $\widetilde{\psi}_{V}(\#_iT_i)= \prod_i \widetilde{\psi}_{V}(T_i)=0$. If none of the $T_i$ have an edge contained in $V$, then the same is true for $\#_iT_i$. Hence \Cref{proposition:additivity-symplectization} implies that $\#_iT_i * \hat{V}= \sum_i T_i*\hat{V}$.  By positivity of intersection (\Cref{prop-positivity-intersections}), $\sum_i T_i*\hat{V}=0$ if and only $T_i*\hat{V}=0$ for all $i$. It then follows that $\widetilde{\psi}_{V}(\#_iT_i)= \prod_i \widetilde{\psi}_V(T_i)$.
\epf

\end{setup1star}

\begin{setup2star}

Fix a datum $\mc{D}=(\mc{D}^+, \mc{D}^-, \hat X, H, \hat \l, \hat J)$ for Setup II*, where we write $\mc{D}^\pm = ((Y^\pm,\xi^{\pm}, V^{\pm}), \fk{r}^{\pm}, \l^\pm, J^\pm)$ and $\fk{r}^{\pm}= (\a^{\pm}, \tau^{\pm}, r^{\pm})$. 

We now introduce the following twisting maps.
 
\defi \label{definition:basic-twist-II}
We define a map $\psi_H : \Sch_\II^{\neq \emptyset}(\mc{D}) \to \Q[U]$ by
\eq \psi_H(T) = U^{T * H + \G^-(T,V^-)}. \eeq
\edefi

\defi \label{definition:ell-reduced-twist-II}
We define a map $\widetilde{\psi}_{H}:  \Sch_\II^{\neq \emptyset}(\mc{D}) \to \Q$ by
\eq \widetilde{\psi}_{H}(T) = \begin{cases} 1 & \text{if $ T * H= 0$ and $\left|\gamma_e\right| \cap V^\pm = \emptyset$ for every $e \in E(T)$} \\ 0 & \text{otherwise} \end{cases} \eeq
\edefi

We need to verify that the above definitions satisfy the axioms of twisting maps. The first step is to prove that the $\psi_H(T)$ are non-negative powers of $U$. This is the content of \Cref{corollary:positive-II}, whose proof requires some preparatory lemmas. (In the next two lemmas, $\dot{\S}$ always denotes an arbitrary punctured Riemann surface.)

\lem \label{lemma:stokes-II}
For $n \geq 1$, suppose that $\b \in \pi_2(\hat{X}, \g^+ \sqcup (\cup_{i=1}^n \g^-_i))$ is represented by a $\hat{J}$-holomorphic curve $u: \dot{\S} \to \hat{X}$ which is contained in $H$. Then  $P^+ - e^{-\mc{E}(\hat{J})} (\sum_{i=1}^n P^-_i) \geq 0$, where $P^+$ (resp. $P^-_i$) is the period of $\g^+ \sub (Y^+, \l^+)$ (resp. the period of $\g^-_i \sub (Y^-, \l^-)$ ). 
\elem
   
\pf	
The claim is trivial if $\mc{E}(\hat{J})=\infty$, so let us assume that $\hat{J} \in \mc{J}(\hat{X}, \hat{\l}, H)$. We may therefore fix a Type A decomposition $\s$ of $(H, \l_H)$, which is specified by a pair of hypersurfaces $\mc{H}_-= \{ -C_1 \} \tms V^-$ and $\mc{H}_+ = \{C_+ \} \tms V^+$. 

It will be convenient to define the regions $R^- := (-\infty, -C_1] \tms V^-$, $R^+:= [C_2 \tms \infty) \tms V^+$ and $\tilde{H} = H - (\op{int}(R^-) \cup \op{int}(R^+))$. Let us first assume that $u$ is transverse to the boundary of $\tilde{H}$. 
Consider now the sum 
\eq \int_{u^{-1}(R^-)} e^{-C_1} u^* d\a^- + \int_{u^{-1}(\tilde{H})} u^* d\hat{ \l} + e^{C_2} \int_{u^{-1}(R^+)} u^* d\a^+. \eeq 
Each summand is non-negative due to the fact that $u$ is $\hat{J}$-holomorphic and that $\hat{J}$ is adapted to $\s$. By Stokes' theorem, the sum of the integrals is $e^{C_2}P^+ - e^{-C_1} (\sum_{i=1}^n P^-_i) \geq 0$. This implies that $P^+ \geq e^{-\mc{E}(\s)} (\sum_{i=1}^n P^-_i)$. 

If $u$ is not transverse to the boundary of $\tilde{H}$, observe by Sard's theorem that transversality can be achieved for a sequence of domains $\tilde{H}^n := \tilde{H} \cup [-C_1^n, -C_1] \cup [C_2, C_2^n]$, where $\{C_i^n\}_{n=0}^{\infty}$ is monotonically decreasing and $C_i^n \to C_i$. It is easy to verify that $\hat{J}$ is still adapted to the Type A decompositions induced by the boundary of $\tilde{H}^n$, so the above argument goes through and passing to the limit gives $P^+ \geq e^{-\mc{E}(\s)} (\sum_{i=1}^n P^-_i)$.

The lemma now follows from the definition of $\mc{E}(\hat{J})$. 
\epf

\lem \label{lemma:additive-II}
For $n \geq 0$, suppose that $\b \in \pi_2(\hat{X}, \g^+ \sqcup (\cup_{i=1}^n \g^-_i))$ is represented by a $\hat{J}$-holomorphic curve $u: \dot{\S} \to \hat{X}$. (Note that unlike in \Cref{lemma:stokes-II}, we allow $n=0$ in which case the union is interpreted as being empty.) Then $\b * H \geq -n_u$ where $n_u$ is the total number of negative punctures of $u$ contained in $V^- \sub Y^-$.	
\elem

\pf 
According to \Cref{prop-positivity-intersections}, we only need to consider the case where the image of $u$ is contained in $H$. By definition of a datum for Setup II*, the trivializations $\tau^\pm$ extend to a global trivialization $\tau$ of the normal bundle of $H$, which implies that $u^\tau \cdot H = 0$. Using the fact that $u^\tau \cdot H = 0$, we have (see \Cref{definition:generalized-intersection-number} and the proof of \Cref{proposition:positivity-symplectization})
\begin{align}
u*H&= \a_N^{{\tau}; -}(\g^+) - \sum_{i=1}^n \a_N^{{\tau}; +}(\g_i^-)\\
&=  \lfloor \CZ_N^{\tau}(\g^+)/2 \rfloor - \sum_{i=1}^n   \lceil \CZ_N^{\tau}(\g_i^-)/2 \rceil \\
&=   \lfloor r^+ P^+ \rfloor -  \sum_{i=1}^n  \lfloor r^- P_i \rfloor,
\end{align}
where the sum is interpreted as zero if $u$ has no negative punctures. 

We may assume that $n \geq 1$ and $r^->0$ (otherwise the lemma is automatic). Let $p_u=n_u+1$ be the total number of punctures (positive and negative) of $u$ contained in $V^\pm \sub Y^\pm$. Using the trivial bounds $x - 1 < \lfloor x \rfloor \le x$, we obtain
\eq u * H >  (r^+ P^+ - 1) - \sum_{i=1}^n (1 + r^- P^-_i) = -p_u+ r^+ P^+  - r^- \sum_{i=1}^n  P^-_i. \eeq
Using now \Cref{lemma:stokes-II} and the fact that $r^+ \geq e^{\mc{E}(\hat{J})} r^-$, we have
\eq-p_u+ r^+  P^+  - r^- \sum_{i=1}^n  P^-_i \geq  -p_u + r^+  \e^{-\mc{E}(\hat{J})} \sum_{i=1}^n P^-_i - r^- \sum_{i=1}^n P^-_i \geq -p_u. \eeq
The claim follows.
\epf

\cor \label{corollary:positive-II}
We have $T*H \geq - \Gamma^-(T, V^-)$ for any $T \in \mc{S}^{\neq \emptyset}_\II(\mc{D})$. Hence $\psi_H(T) \in \Q[U]$. 
\ecor

\pf 
Since $T \in \mc{S}^{\neq \emptyset}_\II(\mc{D})$, there exists $T' \to T$ such that $T'$ is representable by a holomorphic building. Since the Siefring number is invariant under gluing (\Cref{proposition:invariance-gluing}), we may assume that $T$ is representable by a holomorphic building. We now apply \Cref{proposition:positivity-intersection-buildings}: it therefore suffices to check that for each $v \in T$, the intersection number $\beta_v * \eta_{*(v)}$ is bounded below by $-\#\{E^-(v)\}$. In case $*(v) =01$, this follows from \Cref{lemma:additive-II}. In case $*(v)=00$ or $*(v)=11$, this follows either from \Cref{lemma:additive-II} or (more directly) from \Cref{proposition:positivity-symplectization}.
\epf

\prop \label{proposition:basic-additivityII}
	Let $\{T_i\}_i$ be a concatenation in $\Sch_\II$. Then we have:
	\begin{multline*}
	(\#_i T_i) * H +  \G^-(\#_i T_i,V^-) = \sum_{T_i \in \Sch_\I^+} (T_i * \hat{V}^+ + \G^-(T_i,V^+)) \\
	+ \sum_{T_i \in \Sch_\II} (T_i * H + \G^-(T_i,V^-)) + \sum_{T_i \in \Sch_\I^-} (T_i * \hat{V}^- + \G^-(T_i,V^-)).
	\end{multline*}
\eprop


\pf
As in the proof of \Cref{proposition:basic-twisting-mapI}, our assumptions imply that $\hat{\g} * V^\pm = -1$ if $\g$ is contained in $V^\pm$ and $0$ otherwise. By \Cref{proposition:gluing-additivity}, we have
\eq \label{equation:pos-ellipticII}
	T * H
	= \sum_{\begin{smallmatrix} v \in V(T) \cr *(v) = 00 \end{smallmatrix}} \b_v * \hat{V}^+ + \sum_{\begin{smallmatrix} v \in V(T) \cr *(v) = 01 \end{smallmatrix}} \b_v * H + \sum_{\begin{smallmatrix} v \in V(T) \cr *(v) = 11 \end{smallmatrix}} \b_v * \hat{V}^- + \G^\mathrm{int}(T,V^+) +\G^\mathrm{int}(T,V^-)
\eeq
for all $T \in \Sch_\II$. By applying this formula to $T = \#_i T_i$ (and also using \eqref{equation:elliptic-intersectionI}), we see that it suffices to prove that
\begin{align}\label{equation:edge-bijection}
	& \G^\mathrm{int}(\#_i T_i,V^+) + \G^\mathrm{int}(\#_i T_i,V^-) +\G^-(\#_i T_i,V^-) \\
	&= \sum_{T_i \in \Sch_\I^+}  \G^\mathrm{int}(T_i,V^+) +  \G^-(T_i,V^+) \nonumber \\
	&\hphantom{= } + \sum_{T_i \in \Sch_\II}  \G^\mathrm{int}(T_i,V^+) + \G^\mathrm{int}(T_i,V^-) + \G^-(T_i,V^-) \nonumber \\
	&\hphantom{= } + \sum_{T_i \in \Sch_\I^-}  \G^\mathrm{int}(T_i,V^-) + \G^-(T_i,V^-). \nonumber
\end{align}
As in the proof of \Cref{proposition:basic-twisting-mapI}, this is just a matter of understanding how the edges of $\#_i T_i$ are obtained from the edges of the $T_i$'s, following the discussion in \cite[Sec.\ 2.2]{pardon}. More precisely, let us check that every edge counted on the right hand side of \eqref{equation:edge-bijection} is also counted on the left hand side. Note that under concatenation, interior edges remain interior edges. Output edges either remain output edges, or they become interior edges. The output edges corresponding to $\G^-(T_i, V^+)$ for $T_i \in  \Sch_\I^+$ must all become interior edges of $\#_iT_i$: indeed, any such output edge has label $*(e)=0$, but the output edges of $\#_iT_i$ have label $*(e)=1$. These output edges are thus counted in $\G^\mathrm{int}(\#_i T_i,V^+)$.

The output edges corresponding to $\G^-(T_i, V^-)$ for $T_i \in  \Sch_\II$ may either become interior edges of $\#_iT_i$ (in which case they are counted in $\G^\mathrm{int}(\#_i T_i,V^-)$), or remain output edges (in which case they are counted in $\G^-(\#_i T_i,V^-)$). Similarly, the output edges corresponding to $\G^-(T_i, V^-)$ for $T_i \in  \Sch_\I^-$ may either become interior edges of $\#_iT_i$ (counted in $\G^\mathrm{int}(\#_i T_i,V^-)$) or remain output edges (counted in $\G^-(\#_i T_i,V^-)$).
\epf

\cor
Under the assumptions of \Cref{proposition:basic-additivityII}, $\psi_H \in \Psi_\II(\mc{D};\psi_{V^+},\psi_{V^-})$. 
\ecor

\pf
\Cref{proposition:invariance-gluing} implies that $\psi_{H}(T) = \psi_{H}(T')$ for any morphism $T \to T'$. \Cref{proposition:basic-additivityII} implies that $\psi_{H}$ acts correctly on concatenations.
\epf

We now want to show that $\widetilde{\psi}_{H}$ is a twisting map.  We will need the following definition. 

\defi
Given a tree $T\in \Sch_\II$, a vertex $v \in V(T)$ is \emph{mean} if it is an interior vertex and $\lvert \g_e \rvert \sub V^{\pm}$ for all $e \in e^+(v) \sqcup E^-(v)$.  All other vertices are said to be \emph{nice}. These sets are denoted $V_m(T) \sub V(T)$ and $V_n(T) \sub V(T)$ respectively. 
\edefi

\prop \label{proposition:reduced-elliptic-twisting-mapII}
Under the assumptions of \Cref{definition:ell-reduced-twist-II}, $\widetilde{\psi}_{ H} \in \Psi_\II(\mc{D}; \widetilde{\psi}_{V^+},\widetilde{\psi}_{V^-})$.
\eprop

\pf
Consider a tree $T' \in \Sch^{\neq \emptyset}_\II$ with a morphism $T' \to T$. We wish to show that $\widetilde{\psi}_{H}(T')=\widetilde{\psi}_{H}(T)$. As in the proof of \Cref{proposition:reduced-elliptic-twisting-mapI}, we may assume that $T'$ is representable by a building (see \Cref{defi:rep-buildings-full}). 
	
It follows from \Cref{proposition:invariance-gluing} that $T'* H= T* H$. Note that $T', T$ have the same exterior edges. If one of these edges is contained in $V^{\pm}$, then $\widetilde{\psi}_{H}(T')= \widetilde{\psi}_{H}(T)=0$. So we can assume that the exterior edges of $T', T$ are not contained in $V^{\pm}$. 
	
Suppose now that $T'$ has an interior edge contained in $V^{\pm}$. Arguing as in the proof of \Cref{proposition:reduced-elliptic-twisting-mapI}, let $X_i \ge 0$ ($i = 0,1,2$) denote the number of edges $e \in E(T')$ such that $|\g_e| \sub V^{\pm}$ and $e$ is adjacent to exactly $i$ mean vertices. By assumption $X_0 + X_1 + X_2 \geq 1$. By \Cref{proposition:gluing-additivity}, we have 
\eq T' * H= \sum_{v \in V_n(T')  } \beta_v * H_v + \sum_{v \in V_m(T') } \beta_v * H_v  +  X_2 + X_1 + X_0,\eeq
where we write $H_v=\hat{V}^+$ if $*(v)=00$, $H_v=H$ if $*(v)=01$ and $H_v=\hat{V}^-$ if $*(v)=11$.  
According to \Cref{proposition:positivity-symplectization} and the fact that $T'$ is representable by a building, we have that $\sum_{v \in V_n(T')} \beta_v * H_v  \geq 0$. If there are no mean vertices, then $ \sum_{v \in V_m(T') } \beta_v * H_v = X_1 = X_2 =0$ and $X_0 \geq 1$. Hence $T' * H  \geq 1$. If there exists at least one mean vertex, observe that $ X_2 \leq \# V_m(T') -1$.  According to \Cref{lemma:additive-II} and the fact that $T'$ is representable by a building, we have that $ \sum_{v \in V_m(T)} \beta_v * H_v + X_2 + X_1 \geq \#V_m(T')- X_1 - 2 X_2  + X_2 + X_1 = \#V_m(T') - X_2 \geq 1$.  It thus follows again that $T' * \hat{V} \geq 1$. We conclude that $\widetilde{\psi}_{V}(T')= \widetilde{\psi}_{V}(T)=0$ if $T'$ has an interior edge contained in $V^{\pm}$.
	
We are left with the case where $T'$ and hence $T$ have no edges contained in $V^{\pm}$. It's then immediate that $\widetilde{\psi}_{H}(T')= \widetilde{\psi}_{H}(T)$. 
	
If $\{T_i\}_i$ is a concatenation, then the argument is the same as in the proof of \Cref{proposition:reduced-elliptic-twisting-mapI} (using \Cref{proposition:gluing-additivity} instead of \Cref{proposition:additivity-symplectization}).
\epf

\end{setup2star}

\begin{setup3star}

Fix a datum $\mc{D} = (\mc{D}^+,\mc{D}^-, \hat{X}, \hat{\l}^t,H^t, \hat{J}^t)_{t \in [0,1]}$ for Setup III*, where $\mc{D}^\pm = ((Y^\pm,\xi^{\pm}, V^{\pm}), \fk{r}^{\pm}, \l^\pm, J^\pm)$.

We now introduce the following twisting maps.

\defi \label{definition:ell-twist-III}
We define a map $\psi_{H^t} : \Sch_\III^{\neq\emptyset}(\mc{D}) \to \Q[U]$ by
\eq \psi_{H^t}(T) = U^{T * H^t + \Gamma^-(T, V^-)}. \eeq
\edefi

\defi \label{definition:ell-reduced-twist-III}
We define a map $\widetilde{\psi}_{H^t}:  \Sch_\III^{\neq \emptyset}(\mc{D}) \to \Q$ by
\eq\widetilde{\psi}_{H^t}(T) = \begin{cases} 1 & \text{if $ T * H^t= 0$ and $\left|\gamma_e\right| \cap V^\pm = \emptyset$ for every $e \in E(T)$} \\ 0 & \text{otherwise} \end{cases} \eeq
\edefi

There is no difference between $\Sch_\III$ and $\Sch_\II$ from the point of view of the intersection theory defined in \cref{subsection:intersection-buildings}. It can therefore be shown by essentially the same arguments as in the previous section that the above definitions do indeed satisfy the axioms for twisting maps. 

\cor \label{corollary:III-twisting-verification}
	We have $\psi_{H^t} \in \Psi_\III(\mc{D};\psi_{H^0},\psi_{H^1})$ and $\widetilde{\psi}_{H^t} \in \Psi_\III(\mc{D};\widetilde{\psi}_{H^0},\widetilde{\psi}_{H^1})$.
\ecor

\end{setup3star}

\begin{setup4star}

Fix datum $\mc{D}= (\mc{D}^{01},\mc{D}^{12},(\hat{X}^{02,t},\hat{\l}^{02,t},H^{02,t},\hat{J}^{02,t})_{t \in{} [0,\infty)})$ for Setup IV*.  Here, we have that:

\begin{itemize}
	\item $\mc{D}^{01} = (\mc{D}^0,\mc{D}^1,\hat{X}^{01},\hat{\l}^{01},H^{01}, \hat{J}^{01})$ is a datum for Setup II*;
	\item $\mc{D}^{12} = (\mc{D}^1,\mc{D}^2,\hat{X}^{12},\hat{\l}^{12}, H^{12},\hat{J}^{12})$ is a datum for Setup II*;
	\item $\mc{D}^i = ((Y^i, \xi^i, V^i); \fk{r}^i,\l^i,J^i)$ is a datum for Setup I*, for $i = 0,1,2$;
\end{itemize}

We introduce the following twisting maps.

\defi \label{definition:basic-twist-IV}
We define $\psi_{H^{02,t}} : \Sch_\IV^{\neq\emptyset}(\mc{D}) \to \Q[U]$ by
\eq \psi_{H^{02,t}}(T) = U^{T * \eta + \G^-(T,V^2)}. \eeq
\edefi

\defi \label{definition:ell-red-twist-IV}
We define $\widetilde{\psi}_{H^{02,t}} : \Sch_\IV^{\neq\emptyset}(\mc{D}) \to \Q$ by
\eq
\widetilde{\psi}_{H^{02,t}}(T) =
\begin{cases}
1 & \text{if $ T * \eta = 0$ and $\lvert\gamma_e\rvert \cap V^i = \emptyset$ for all $e \in E(T), i \in \{0,1,2\}$} \\
0 & \text{otherwise}
\end{cases}
\eeq
\edefi

We need to show that the powers of $U$ appearing in \Cref{definition:basic-twist-IV} are non-negative. This will be the content of \Cref{corollary:positive-IV}, which requires some preparatory lemmas.

\lem \label{lemma:stokes-IV}
For $n \geq 1$, suppose that $u: \dot{\S} \to \hat{X}^{02,t}$ is $\hat{J}^{02,t}$-holomorphic with positive orbit $\g^+$ and negative orbits $\cup_{i=1}^n \g^-_i$. Then we have $P^+ - e^{-\mc{E}(\hat{J}^{02,t})} \sum_i P^-_i \geq 0$, where $P^+$ (resp. $P^-_i$) is the period of $\g^+ \sub V^0$ (resp. $\g^-_i \sub V^2$).
\elem
\pf 
The proof is analogous to that of \Cref{lemma:stokes-II}. If $\mc{E}(\hat{J}^{02,t})= \infty$, the result is trivial. Hence we may assume that $\hat{J}^{02,t} \in \mc{J}(\hat{X}^{02,t}, \hat{\l}^{02,t}, H^{02,t})$ and fix a type B cobordism decomposition $\s^{02,t}$ of $(H^{02,t}, \hat{\l}^{02,t}|_{H^{02,t}})$ to which $\hat{J}^{02,t}$ is adapted. The decomposition $\s^{02,t}$ is specified by a family of hypersurfaces $\mc{H}_2(t)= \{ - C_2(t) \} \tms V^2, \mc{H}_1(t) = \{-C_1(t) \} \tms V^1, \tilde{\mc{H}}_1(t) = \{ \tilde{C}_1(t) \} \tms V^1, \mc{H}_0(t) = \{ C_0(t) \} \tms V^0$. 

It will be convenient to define the regions $R_2(t)= (-\infty, -C_{2}(t)] \tms V^2, R_0(t)= [C_0(t), \infty) \tms V^0$ and $R_1(t) = [-C_1(t), \tilde{C}_1(t)] \tms V^1$. 

Suppose first that $R_1(t)$ is empty. Then $\tilde{C}_1(t)+ C_1(t)=0$, and hence $\mc{E}(\s^{02,t})= C_0(t)+ C_2(t)$. Hence $\mc{E}(\s^{02,t})$ coincides with the energy of $\s^{02,t}$ if it is viewed as a Type A cobordism decomposition (\Cref{definition:type-A-energy}) by forgetting $C_1, \tilde{C}_1$. Hence, when $R_1(t)$ is empty, the claim reduces to \Cref{lemma:stokes-II}.

We now assume $R_1(t)$ is non-empty. We let $\tilde{H}^{02,t}_{12} \sqcup \tilde{H}^{02,t}_{01}$ be the connected components of $X^{02,t} - \op{int} (R_2(t) \cup R_1(t) \cup R_0(t))$. Let us first assume that the image of $u$ intersects the boundaries of $\tilde{H}^{02,t}_{01}$ and $\tilde{H}^{02,t}_{12}$ transversally. We then have the following computations:
\begin{itemize}
\item $\int_{u^{-1}(R_2(t))} u^* \a_2 = \int_{u^{-1}(\mc{H}_2(t))} u^* \a_2 - \sum_{i=1}^n P^-_i \geq 0$,
\item $\int_{u^{-1}(H^{02,t}_{21})} u^* d(e^s \a_2) = e^{-C_1(t)} \int_{u^{-1}(\mc{H}_1(t))} u^* \a_1 - e^{-C_{2}(t)} \int_{u^{-1}(\mc{H}_2(t))} u^* \a_2  \geq 0$,
\item $\int_{u^{-1}(R_1(t))} u^* \a_1 = \int_{u^{-1}(\tilde{\mc{H}}_1(t))} u^* \a_1 - \int_{u^{-1}(\mc{H}_1(t))} u^* \a_1 \geq 0$,
\item $\int_{u^{-1}(\tilde{H}^{02,t}_{01})} u^* d(e^s \a_1) = e^{C_0(t)} \int_{u^{-1}(\mc{H}_0(t))} u^* \a_0 - e^{\tilde{C}_1(t)} \int_{u^{-1}(\tilde{\mc{H}}_1(t))} u^* \a_1  \geq 0$,
\item $\int_{u^{-1}(R_0(t))} u^* \a_0 = P^+_i - \int_{u^{-1}(\mc{H}_0(t))} u^* \a_0 \geq 0$.
\end{itemize}

After appropriate rescalings, these terms form a telescoping sum. We find: $P^+ - e^{C_2(t) + C_0(t)- C_1(t)- \tilde{C}_1(t)} \sum_i P^-_i = P^+-  e^{-\mc{E}(\s^{02,t})} \sum_i P^-_i \geq 0$. 

Suppose now that the image of $u$ does not intersect the boundaries of $\tilde{H}^{02,t}_{01}$ and $\tilde{H}^{02,t}_{12}$ transversally. For $\e_n \downarrow 0$, set $R_2^{(n)}:= (-\infty, -C_2(t) - \e_n] \tms V^2, R_0^{(n)}= [C_0(t)+ \e_n, \infty) \tms V^0$ and $R_1^{(n)} = [-C_1(t)+\e_n, \tilde{C}_1(t)-\e_n] \tms V^1$. By Sard's theorem, we may assume by choosing $\e_n$ appropriately that $u$ intersects the boundary of the $R_i^{(n)}$ transversally. Now repeat the above argument with $R_i^{(n)}$ in place of $R_i(t)$. This yields the inequality $P^+ - e^{C_2(t) + C_0(t)- C_1(t)- \tilde{C}_1(t)+ 4 \e_n} \sum_i P^-_i \geq 0$. The claim now follows by passing to the limit.
\epf

\lem \label{lemma:additive-IV} 
For $n \geq 0$, suppose that $u: \dot{\S} \to \hat{X}^{02,t}$ is a $\hat{J}^{02,t}$-holomorphic curve in the class $\beta \in \pi_2(\hat{X}^{02,t}, \g^+ \sqcup (\cup_{i=1}^n \g^-_i))$ for $t<\infty$.  Then $\beta * [H^{02,t}] \geq - n_u$, where $n_u$ is the total number of negative punctures. (Note that unlike in \Cref{lemma:stokes-IV}, we allow $n=0$ here in which case the union is interpreted as being empty.)
\elem

\pf We argue as in the proof of \Cref{lemma:additive-II}. It is enough to consider the case where the image of $u$ is contained in $H^{02,t}$. The trivialization $\tau$ extends to a global trivialization along $H^{02,t}$, implying that $\b \cdot_{\tau} H^{02,t}=0$. 

We thus have
\begin{align*}
u*H^{02,t}&= \a_N^{{\tau}; -}(\g^+) - \sum_{i=1}^n \a_N^{{\tau}; +}(\g_i^-)\\
&=  \lfloor \CZ_N^{\tau}(\g^+)/2 \rfloor - \sum_{i=1}^n   \lceil \CZ_N^{\tau}(\g_i^-)/2 \rceil \\
&=   \lfloor r_0 P^+ \rfloor -  \sum_{i=1}^n +  \lfloor r P_i \rfloor,
\end{align*}
where the sum is interpreted as zero if $u$ has no negative punctures. Thus the lemma is verified if $n=0$ or $r_2=0$. It remains only to consider the case where $n\geq 1$ and $r_2>0$. 

Using the trivial bounds $x - 1 < \lfloor x \rfloor \le x$, we obtain
$$
	u * \hat{V} > \sum_{ z \in \bf{p}_u^+} (r_0 P_z - 1) - \sum_{ z \in \bf{p}_u^-} (1 + r_2 P_z)
	\ge -p_u + e^{\mc{E}(\s^{02})} r_2 (\sum_{ z \in \bf{p}_u^+} P_z) - r_2 \sum_{ z \in \bf{p}_u^-} P_z .
$$
It follows from \Cref{lemma:stokes-IV} that $e^{\hat{J}^{02,t}} r_2 (\sum_{ z \in \bf{p}_u^+} P_z) - r_2 \sum_{ z \in \bf{p}_u^-} P_z  \geq 0$. The claim follows.
\epf

\cor \label{corollary:positive-IV}
We have that $T* \eta + \G^-(T,V^2) \geq 0$. 
\ecor

\pf 
We need to consider two cases. If $\fk{s}(T) \in{} [0, \infty)$, then the claim follows by combining \Cref{proposition:positivity-intersection-buildings} and \Cref{lemma:additive-IV}. If $\fk{s}(T) = \{\infty\}$, then the argument is the same as in the proof of \Cref{corollary:positive-II}.
\epf

\prop \label{proposition:basic-additivityIV}
	Let $\{T_i\}_i$ be a concatenation in $\Sch_\IV$ of type~\ref{item:concatenationIV-infinity} (see page~\pageref{item:concatenationIV-infinity}). Then we have
	\begin{align*}
	&(\#_i T_i) * \eta + \G^-(T,V^2) \\
	&= \sum_{T_i \in \Sch_\I^0} (T_i * \hat{V}^0 + \G^-(T_i,V^0)) + \sum_{T_i \in \Sch_\II^{01}} (T_i * H^{01} +  \G^-(T_i,V^1)) + \sum_{T_i \in \Sch_\I^1} (T_i * \hat{V}^1 +  \G^-(T_i,V^1)) \\
	&\hphantom{= } + \sum_{T_i \in \Sch_\II^{12}} (T_i * H^{12} +  \G^-(T_i,V^2)) + \sum_{T_i \in \Sch_\I^2} (T_i * \hat{V}^2 + \G^-(T_i,V^2)).
	\end{align*}
\eprop
\pf
As in the proof of \Cref{proposition:basic-twisting-mapI}, our assumptions imply that $\hat{\g} * V^j = -1$ if $\g$ is contained in $V^j$ and $0$ otherwise. \Cref{proposition:gluing-additivity} implies that
\begin{align*}
	&(\#_i T_i) * \eta \\
	&= \sum_{\begin{smallmatrix} v \in V(\#_i T_i) \cr *(v) = 00 \end{smallmatrix}} \b_v * \hat{V}^0 +  \sum_{\begin{smallmatrix} v \in V(\#_i T_i) \cr *(v) = 01 \end{smallmatrix}} \b_v * H^{01} + \sum_{\begin{smallmatrix} v \in V(\#_i T_i) \cr *(v) = 11 \end{smallmatrix}} \b_v * \hat{V}^1 + \sum_{\begin{smallmatrix} v \in V(\#_i T_i) \cr *(v) = 12 \end{smallmatrix}} \b_v * H^{12} + \sum_{\begin{smallmatrix} v \in V(\#_i T_i) \cr *(v) = 22 \end{smallmatrix}} \b_v * \hat{V}^2 \\
	&\hphantom{= } + \G^\mathrm{int}(\#_i T_i, V^0) +  \G^\mathrm{int}(\#_i T_i, V^1) + \G^\mathrm{int}(\#_i T_i, V^2).
\end{align*}
As in the proof of \Cref{proposition:basic-additivityII}, it follows that the result is equivalent to
\begin{align*}
	& \G^\mathrm{int}(\#_i T_i, V^0) + \G^\mathrm{int}(\#_i T_i, V^1) + \G^\mathrm{int}(\#_i T_i, V^2) +  \G^-(T,V^2) \\
	&=  \sum_{T_i \in \Sch_\I^0}  \G^\mathrm{int}(T_i,V^0) +   \G^-(T_i,V^0) \\
	&\hphantom{= } +  \sum_{T_i \in \Sch_\II^{01}}  \G^\mathrm{int}(T_i,V^0) +  \G^\mathrm{int}(T_i,V^1) +  \G^-(T_i,V^1) \\
	&\hphantom{= } + \sum_{T_i \in \Sch_\I^1} \G^\mathrm{int}(T_i,V^1) + \G^-(T_i,V^1) \\
	&\hphantom{= } + \sum_{T_i \in \Sch_\II^{12}}   \G^\mathrm{int}(T_i,V^1) +  \G^\mathrm{int}(T_i,V^2) +  \G^-(T_i,V^2) \\
	&\hphantom{= } + \sum_{T_i \in \Sch_\I^2} \G^\mathrm{int}(T_i,V^2) +  \G^-(T_i,V^2),
\end{align*}
which is a consequence of the way the edges of $\#_i T_i$ are obtained from the edges of the $T_i$'s.
\epf

\cor
We have $\psi_{H^{02,t}} \in \Psi_\IV(\mc{D};\psi_{H^{01}},\psi_{H^{12}},\psi_{H^{02,0}})$ and we have that $\widetilde{\psi}_{H^{02,t}} \in \Psi_\IV(\mc{D};\widetilde{\psi}_{H^{01}},\widetilde{\psi}_{H^{12}},\widetilde{\psi}_{H^{02,0}})$.
\ecor
\pf
\Cref{proposition:basic-additivityIV} shows that $\psi_{H^{02,t}}$ acts correctly on concatenations of type~\ref{item:concatenationIV-infinity}; the proof that $\psi_{H^{02,t}}$ behaves well with respect to the other two types of concatenation is virtually identical. \Cref{proposition:invariance-gluing} implies that $\psi_{H^{02,t}}(T) = \psi_{H^{02,t}}(T')$ for any morphism $T \to T'$. The argument that $\widetilde{\psi}_{H^{02,t}}(T) = \widetilde{\psi}_{H^{02,t}}(T')$ is essentially the same as the proof of \Cref{proposition:reduced-elliptic-twisting-mapII}, except that we appeal to \Cref{lemma:additive-IV} instead of \Cref{lemma:additive-II}.
\epf

\end{setup4star}

The results from the previous sections can be conveniently packaged into the following theorem. 

\thm[cf.\ Thm.\ 1.1 in \cite{pardon}] \label{theorem:twisted-counts}
Let $\mc{D}$ be a datum for any one of Setups I*-IV*. Then there exists a set of perturbation data $\Theta(\mc{D})$ and twisted moduli counts $\#_{\psi}\ov{\mc{M}}(T)^\vir_\theta \in \Q[U], \#_{\tilde{\psi}} \ov{\mc{M}}(T)^\vir_\theta \in \Q$ (for $\t \in \Theta(\mc{D})$ and $T \in \Sch_*(\mc{D})$, $* = \I, \II, \III, \IV$) satisfying the obvious analogs of (i)--(v) in Thm.\ 1.1 in \cite{pardon}. 
\ethm
\pf
There is a forgetful functor taking a datum for the enriched setups I*-IV* (\Cref{subsection:enriched-setups}) to a datum for the standard setups I-IV (\Cref{subsection:standard-setups}). So the set of perturbation data is furnished by Thm.\ 1.1 in \cite{pardon}. We showed in \Cref{subsection:twisting-maps-contact-subman} a datum for setups I*-IV* gives rise to twisting maps, from which we may define our twisted moduli counts as in \Cref{subsection:twistinghomologies}. The properties (i)-(iv) are tautological and (v) is a consequence of the axioms of twisting maps, as explained in \Cref{subsection:twistinghomologies}.
\epf

\section{Construction of the main invariants} \label{section:invariant}

In this section, we construct the invariants which are the central objects of this paper. To the data of a TN contact pair $(Y, \xi, V)$ and an element $\fk{r} \in \fk{R}(Y, \xi, V)$, we associate a unital, $\Z/2$-graded $\Q[U]$-algebra 
\eq \label{equation:main-invariant-repeat} CH_\bullet(Y, \xi, V; \fk{r}).\eeq 
There is a natural map to ordinary contact homology $CH_\bullet(Y, \xi, V; \fk{r}) \to CH_\bullet(Y, \xi)$ given by setting $U=1$. 

A contactomorphism $f:(Y, \xi, V) \to (Y', \xi', V')$ induces an identification $CH_\bullet(Y, \xi, V; \fk{r})= CH_\bullet(Y', \xi', V'; f_*\fk{r})$. An exact relative symplectic cobordism $(\hat X, \hat \l, H)$ from $(Y^+, \xi^+, V^+)$ to $(Y^-, \xi^-, V^-)$ satisfying an energy condition induces a map $CH_\bullet(Y^+, \xi^+, V^+; \fk{r}^+) \to CH_\bullet(Y^-, \xi^-, V^-; \fk{r}^-)$. Unfortunately, our notions of energy are not well behaved under compositions of arbitrary relative symplectic cobordisms, so the composition of maps is not always defined. 

We also define a reduced version of \eqref{equation:main-invariant-repeat} which only counts Reeb orbits in the complement of a codimension $2$ submanifold, and certain ``asymptotic invariants" which have good functoriality properties. 

\subsection{Construction and basic properties of the invariants} \label{subsection:mainconstruction}

The following susection is entirely parallel to \cite[Sec.\ 1.7]{pardon}. More precisely, Pardon constructs (ordinary) contact homology by applying \cite[Thm.\ 1.1.]{pardon} to data from Setups I-IV. We construct our new invariants by applying \Cref{theorem:twisted-counts} to data from Setups I*-IV*. 

\begin{setup1star}

Fix a TN contact pair $(Y, \xi, V)$ and an element $\fk{r} \in \fk{R}(Y, \xi, V)$. According to \Cref{proposition:adapted-contact-forms-exist}, we may choose a contact form $\xi= \op{ker} \l$ which is adapted to $\fk{r}$. Let $J: \xi \to \xi$ be a $d\lambda$-compatible almost complex structure which preserves $\xi_V$. We therefore obtain a datum $\mc{D}$ for Setup I*.  \Cref{theorem:twisted-counts} applied to $\mc{D}$ furnishes a $\Z/2$-graded, unital $\Q[U]$-algebra 
\eq CH_\bullet(Y, \xi, V; \fk{r})_{\l, J, \theta} \eeq 
for any choice of perturbation datum $\theta \in \T_\I(\mc{D})$; cf.\ \eqref{equation:twist-differential-I} and \eqref{equation:inv-1}. 
\end{setup1star}

\begin{setup2star}
Fix pairs $\mc{D}^\pm = ((Y^\pm,\xi^{\pm}, V^{\pm}), \fk{r}^{\pm}, \l^\pm, J^\pm)$ of data for Setup I*, where we write $\fk{r}^{\pm}=(\a^{\pm}, \tau^{\pm}, r^{\pm})$.  Let $(\hat X, \hat \l, H)$ be an exact relative symplectic cobordism with positive end $(Y^+, \l^+, V^+)$ and negative end $(Y^-, \l^-, V^-)$, and suppose that there exists a trivialization of the normal bundle of $H$ which restricts to $\tau^{\pm}$ on the positive/negative end. 

\prop  \label{proposition:map-II}
	Suppose that $ r^+> e^{\mc{E}(H, \hat{\l}|_H)} r^-$. Then there is an induced map on homology
\eq \label{equation:cob-map-depends} \Phi(\hat{X},\hat{\l},H)_{\hat{J}, \theta} : {CH}_\bullet(Y^+, \xi^+, V^+; \fk{r}^+)_{\lambda^+, J^+, \theta^+} \to {CH}_\bullet(Y^-,\xi^-, V^-; \fk{r}^-)_{\lambda^-, J^-, \theta^-}. \eeq

If $\a^+ = \a^-$ and $(H, \hat{\l}|_H)$ is a symplectization, then the same conclusion holds provided that $r^+ \geq r^-$. 
\eprop

\pf
According to \Cref{lemma:ac-connected-II}, we can choose an almost-complex structure $\hat{J}$ on $\hat{X}$ which is $d\hat{\l}$-compatible and agrees with $\hat{J}^{\pm}$ at infinity, and such that $r^+ \geq e^{\mc{E}(\hat{J})} r^-$.  We thus obtain a datum $\mc{D}=(\mc{D}^+, \mc{D}^-, \hat X, H, \hat \l, \hat J)$ for Setup II*. 

Given $(\theta^+, \theta^-) \in \T_\I(\mc{D}^+)  \tms \T_\I(\mc{D}^-)$, \Cref{theorem:twisted-counts} thus provides a perturbation datum $\theta \in \T_\II(\mc{D})$ with $\t \mapsto (\t^+, \t^-)$, and twisted moduli counts which give rise to the map \eqref{equation:cob-map-depends}; cf.\ \eqref{equation:phi-setup-2}.
\epf

\end{setup2star}

\begin{setup3star} We have the following proposition.

\prop  \label{proposition:map-III}
Under the assumptions of \Cref{proposition:map-II}, the map \eqref{equation:cob-map-depends} is independent of the pair $(\hat{J}, \theta)$. 
\eprop

\pf
Let $(\hat{J}_0, \theta_0)$ and $(\hat{J}_1, \theta_1)$ be two possible choices of such pairs. Let us first treat the caes where $(H, \hat{\l}|_H)$ is not a symplectization. For any $\e>0$, \Cref{lemma:ac-connected-II} provides an interpolating family of almost-complex structure $\{\hat{J}_t\}_{t\in [0,1]}$ such that $\mc{E}(\hat{J}_t) \leq \op{max}(\mc{E}(\hat{J}_0), \mc{E}(\hat{J}_1)) + \e$. Choosing $\e$ small enough so that $r^+> e^{\mc{E}(\hat{J}_t )} r^-$, we thus get a datum $\mc{D}$ for Setup III*.

\Cref{theorem:twisted-counts} now provides perturbation data $\theta \in \T_\III(\mc{D})$ mapping to $(\theta_0, \theta_1)$, and a chain homotopy between the maps $\Phi(\hat{X},\hat{\l},H)_{\hat{J}_0, \theta_0}$ and $\Phi(\hat{X},\hat{\l},H)_{\hat{J}_1, \theta_1}$; cf.\ \eqref{equation:phi-homotopy-3} and \eqref{equation:phi-setup-3}.

If $\a^+ = \a^-$ and $(H, \hat{\l}|_H)$ is a symplectization, then \Cref{lemma:ac-connected-II} implies that we may repeat the above argument for a family of almost-complex structures $\hat{J}_t$ which have vanishing energy. Tracing through the proof, it is straightforward to check that desired conclusion goes through provided that $r^+ \geq r^-$. 
\epf

\end{setup3star}

\begin{setup4star}
Let us consider data $\tilde{\mc{D}} = (\tilde{\mc{D}}^{01},\tilde{\mc{D}}^{12},(\hat{X}^{02,t},\hat{\l}^{02,t})_{t \in{} [0,\infty)})$, where
\begin{align*}
	\tilde{\mc{D}}^{01} &= (\mc{D}^0,\mc{D}^1,\hat{X}^{01},\hat{\l}^{01}, H^{01}) \\
	\tilde{\mc{D}}^{12} &= (\mc{D}^1,\mc{D}^2,\hat{X}^{12},\hat{\l}^{12}, H^{12}) \\
	\mc{D}^i &= ((Y^i,\xi^i,V^i), \fk{r}^i, \l^i, J^i) \quad (i = 0,1,2)
\end{align*}

Here $\tilde{\mc{D}}^{01}, \tilde{\mc{D}}^{12}, \tilde{\mc{D}}$ are ``partial data" for Setups II* and IV*, since they do not contain any information about almost-complex structures. These ``partial data" are assumed to obey all the axioms stated in \Cref{subsection:enriched-setups} which do not involve complex structures. 

The $\mc{D}^i$ are (ordinary) data for Setup I*. 

\prop  \label{proposition:map-IV}
	Suppose that the following conditions hold:
	\begin{itemize}
		\item $r_0> e^{\mc{E}(H^{02,t}, \hat{\l}^{02,t}|_{H^{02,t}})} r_2$,
		\item $r_0> e^{\mc{E}(H^{01}, \hat \l^{01}|_{H^{01}})} r_1$,
		\item $r_1> e^{\mc{E}(H^{12}, \hat \l^{12}|_{H^{12}})} r_2$. 
	\end{itemize}
Then the following diagram commutes:
\begin{center}
\begin{tikzcd}[row sep = small]
&CH_\bullet(Y^1,\xi^1, V^1; \fk{r}^1)_{\lambda^1,J^1,\theta^1}\ar{rd}{\Phi(\hat X^{12},\hat\lambda^{12},H^{12})}\\
CH_\bullet(Y^0,\xi^0,V^0; \fk{r}^0)_{\lambda^0,J^0,\theta^0}\ar{rr}[swap]{\Phi(\hat X^{02},\hat\lambda^{02},H^{02})}\ar{ru}{\Phi(\hat X^{01},\hat\lambda^{01},H^{01})}&&CH_\bullet(Y^2,\xi^2,V^2; \fk{r}^2)_{\lambda^2,J^2,\theta^2}
\end{tikzcd}
\end{center}

If $\a_0= \a_1= \a_2$ and if $(H^{01}, \hat{\l}^{01}|_{H^{01}})$, $(H^{12}, \hat{\l}^{12}|_{H^{12}})$, $(H^{02,t}, \hat{\l}^{02,t}|_{H^{02,t}})$ are symplectizations, then the conclusion still holds if we only assume that $r_i \geq r_j$ for $i \geq j$. 
\eprop

\pf
According to \Cref{lemma:ac-connected-IV}, one can choose a family of almost complex structures $\hat{J}^{02,t}$ so that $r_0 > e^{\mc{E}(\hat{J}^{02,t})} r_2$. Moreover, by \Cref{lemma:gluing-surjective-energy-ac}, one may also assume that $r_0 > e^{\mc{E}(\hat{J}^{01})} r_1$ and $r_1 > e^{\mc{E}(\hat{J}^{12})} r_2$, where $J^{01}$ and $J^{12}$ are the almost-complex structures induced at infinity by $J^{02,t}$.  We therefore obtain a datum for Setup IV* by considering $\mc{D}^{01}=(\tilde{\mc{D}}^{01}, \hat{J}^{01}), \mc{D}^{12} = (\tilde{\mc{D}}^{12}, \hat{J}^{12})$ and $\mc{D}=(\tilde{\mc{D}}, \hat{J}^{02,t})$. \Cref{theorem:twisted-counts} applied to $\mc{D}$ now implies the commutativity of the above diagram; cf.\ \eqref{equation:phi-commute-4}. 

Under the additional hypotheses that $\a_0=\a_1=\a_2$ and that the relevant cobordisms are symplectizations, \Cref{lemma:ac-connected-IV} and \Cref{lemma:gluing-surjective-energy-ac} allow us to work with (families of) almost-complex structures with vanishing energy. Retracing through the above argument, we find that the desired conclusion follows if $r_i \geq r_j$ for $i \geq j$. 
\epf

\prop \label{proposition:commutativity}
	Let $(\hat{X}, \hat{\l}, H)$ be a relative symplectic cobordism from $(Y^{+}, \xi^{+}, V^+)$ to $(Y^-, \xi^-, V^-)$.  Let $(\hat{V}^{\pm}, \hat{\l}_{\hat{V}}^{\pm})$ be the Liouville structure induced on $\hat{V}^{\pm}$ from the canonical Liouville structure of the symplectization $(\hat{Y}^{\pm}, \lambda_{Y^{\pm}})$.  For $i\in \{1,2\},$ consider elements $\fk{r}_i^{\pm} \in \fk{R}(Y^{\pm}, \xi^{\pm}, V^{\pm})$ and let $\l_i^{\pm}$ be a contact form on $Y^{\pm}$ which is adapted to $\fk{r}_i^{\pm}$. Suppose finally that we have:
\begin{itemize}
	\item[(1)] $r_i^+ > e^{\mc{E}(H, \hat{\l}|_H)} r_i^-$;
	\item[(2)] $r_1^+ > e^{\mc{E}(\hat{V}^+, \hat{\l}_{\hat{V}}^+)} r_2^+$ and $r_1^- > e^{\mc{E}(\hat{V}^-, \hat{\l}_{\hat{V}}^-)} r_2^-$. 
\end{itemize} 

Then the following diagram commutes: 
	\begin{center} \label{diagram:commutativity}
	\begin{tikzcd} [row sep = large, column sep = huge]
		{CH}_\bullet(Y^+,\xi^+,V^+; \fk{r}_1^+)_{\l_1^+, J^+, \theta^+} \ar{d}{\Phi(\hat{X},\hat{\l},H)} \ar{r}{\Phi(\hat{Y}^+,\hat{V}^+)} & {CH}_\bullet(Y^+,\xi^+,V^+; \fk{r}_2^+)_{\l_2^+, J^+, \theta^+} \ar{d}{\Phi(\hat{X},\hat{\l},H)} \\
		{CH}_\bullet(Y^-,\xi^-,V^-; \fk{r}_1^-)_{\l_1^-, J^-, \theta^-} \ar{r}{\Phi(\hat{Y}^-,\hat{V}^-)} & {CH}_\bullet(Y^-,\xi^-,V^-; \fk{r}_2^-)_{\l_2^-, J^-, \theta^-}
	\end{tikzcd}
	\end{center}

As usual, if $\a^+_1=\a^-_1 = \a^+_2 = \a^-_2$ and $(H, \hat{\l}|_H)$ is a symplectization, then it is enough to assume that $r_i^+ \geq r_i^-, r^+_1 \geq r^+_2$ and $r^-_1 \geq r^-_2$. 
\eprop

\pf
Observe first that the conditions (1-2) along with \Cref{proposition:map-II} ensure that the maps appearing in the commutative diagram are well-defined. Let us now consider the strict exact symplectic cobordisms $(\hat{X}, \hat{\l}, H)^{\l_1^+}_{\l_1^-}$ and $(\hat{Y}^-, \l_2^-, \hat{V}^-)^{\l_1^-}_{\l_2^-}$.  For $t \in{} [0, \infty)$ and $T_0>0$ large enough, we can consider their $(t+T_0)$-gluing $(\hat{X}^t, \hat{\l}^t, H^t)$; cf.\ \Cref{definition:gluing-relative-cobordisms}. According to \Cref{lemma:energy-trivial-component}, we have that (cf.\ \Cref{notation:emphasis-relative})
\eq \mc{E}((H^t, {\hat{\l}^t}|_{H^t})^{\a_1^+}_{\a_2^-}) = \mc{E}((H, \hat{\l}|_H)^{\a_1^+}_{\a_1^-}) + \mc{E}((\hat{V}^-, \hat{\l}^-_{\hat{V}})^{\a_1^-}_{\a_2^-}).\eeq
It then follows from (1-2) that $r_1^+ > e^{\mc{E}((H^t, \hat{\l}^t|_{H^t})^{\a_1^+}_{\a_2^-})} r_2^-$.

We can now appeal to \Cref{proposition:map-IV}, which implies that the composition $\Phi(\hat{Y}, \hat{V}^-) \circ \Phi(\hat{X}, \hat{\l}, H)$ agrees with the map induced by $(\hat{X}^0, \hat{\l}^0, H^0)= (\hat{X}, \hat{\l}, H)^{\l_1^+}_{\l_2^-}$; see \Cref{example:gluing-trivial-end}. The same argument shows that composition along the upper right hand side of the diagram agrees with the map induced by  $(\hat{X}, \hat{\l}, H)^{\l_1^+}_{\l_2^-}$. This proves the claim.  
\epf

\end{setup4star}

We obtain the following corollary by putting together the results of the previous section. 

\cor \label{corollary:change-contact-form}
	Consider a TN contact pair $(Y, \xi, V)$ and fix an element $\fk{r} \in \fk{R}(Y, \xi, V)$. Let $\mc{D}^{\pm}=(Y, \xi, V), \fk{r}, \l^{\pm}, J^{\pm})$ be a pair of data for Setup I* and fix $\theta^{\pm} \in  \T_\I(Y^{\pm},\lambda^{\pm},J^{\pm})$.  

	The map
	\eq \label{equation:iso-change-contact-forms}
	\Phi(\hat{Y},\hat{\l},\hat{V}) : {CH}_\bullet(Y,\xi,V; \fk{r})_{\lambda^+, J^+, \theta^+} \to {CH}_\bullet(Y,\xi,V; \fk{r})_{\lambda^-, J^-, \theta^-}
	\eeq
	defined in \Cref{proposition:map-II} is an isomorphism.
\ecor

\pf 
In light of \Cref{proposition:commutativity} and \Cref{lemma:energy-marked-symplectization}, it's enough to consider the case $\l^+ = \l^- = \l$ and $J^+ = J^- = J$. Let $\theta \in \Theta_\II(\hat{Y},\hat{\l},\hat{J})$ be a lift of $(\theta^+, \theta^-)$ under the forgetful map $\Theta_\II(\hat{Y},\hat{\l},\hat{J}) \to \Theta_\I(Y,\l,J) \tms \Theta_\I(Y,\l,J)$. The proof of \cite[Lem.\ 1.2]{pardon} can be adapted to show that the map
\eq \Phi(\hat{Y},\hat{\l},\hat{V})_{\hat{J},\t} : {CC}_\bullet(Y,\xi, V; \fk{r})_{\lambda, J, \theta^+} \to {CC}_\bullet(Y,\xi, V; \fk{r})_{\lambda, J, \theta^-} \eeq
is an isomorphism of chain complexes: one simply needs to observe that the twisted counts of trivial cylinders coincide with the usual counts.
\epf

We now arrive at the definition of our main invariants.

\defi[Full invariant] \label{definition:invariant-choice-depend} 
Consider a TN contact pair $(Y, \xi, V)$ and choose an element $\fk{r} \in \fk{R}(Y, \xi, V)$.  Let 
\eq \label{equation:main-invariant} {CH}_{\bullet}(Y, \xi, V; \fk{r})\eeq
 be the limit (or equivalently the colimit) of $\{ {CH}_{\bullet}(Y, \xi, V; \fk{r})_{\l, J, \t} \}_{\l, J, \t}$ along the maps \eqref{equation:iso-change-contact-forms}. \Cref{proposition:commutativity} and \Cref{corollary:change-contact-form} imply that ${CH}_{\bullet}(Y, \xi, V; \fk{r})$ is canonically isomorphic to ${CH}_{\bullet}(Y, \xi, V; \fk{r})_{\l, J, \t}$ for any admissible choice of $(\l, J, \t)$. 
\edefi

Given $s \in \Q$, define 
\eq CH_{\bullet}^{U=s}(Y, \xi, V; \fk{r}):= CH_{\bullet}(Y, \xi, V; \fk{r}) \otimes_{\Q[U]} \Q,\eeq 
where the map $\Q[U] \to \Q$ sends $U\mapsto s$. There is a natural evaluation morphism of $\Q[U]$-algebras 
\eq \label{equation:evaluation} \op{ev}_{U=s}: CH_{\bullet}(Y, \xi, V; \fk{r})  \to CH_{\bullet}^{U=s}(Y, \xi, V;\fk{r}). \eeq

It follows tautologically from the construction that $CH_{\bullet}^{U=1}(Y, \xi, V; \fk{r})= CH_\bullet(Y, \xi)$. The invariant $CH_{\bullet}(Y, \xi, V; \fk{r})$ therefore admits a $\Q[U]$ algebra morphism to ordinary contact homology (which is viewed as a $\Q[U]$-algebra by letting $U$ act by the identity).

We can also define a ``reduced" variant of the invariants \eqref{equation:main-invariant} which are based on the twisting map $\tilde{\psi}$. These invariants are naturally $\Q$-algebras (as opposed to $\Q[U]$-algebras) and only take into account Reeb orbits in the complement of the codimension $2$ submanifold. 

More precisely, given a datum $((Y, \xi, V), \fk{r}, \l, J)$ for Setup I*, we may proceed as in \Cref{subsection:mainconstruction}(I*) and let 
\eq (\widetilde{CC}_{\bullet}(Y, \xi, V; \fk{r})_{\l}, d_{\widetilde{\psi}, J, \theta}) \eeq 
be the complex generated by the (good) Reeb orbits \emph{not contained in $V \sub Y$}, for some perturbation datum $\theta \in \T_\I(\mc{D})$. By repeating the above arguments with the twisting maps $\widetilde{\psi}_{-}$ in place of the twisting maps $\psi_{-}$, one can establish the obvious analog of \Cref{proposition:commutativity} and \Cref{corollary:change-contact-form}. In particular, given choices of data $(\l^+, J^+, \theta^+), (\l^-, J^-, \theta^-)$ as in \Cref{corollary:change-contact-form}, there is an isomorphism
\eq \label{equation:reduced-isoms} \Phi(\hat{Y}, \hat{\l}, \hat{V}): \widetilde{CH}_{\bullet}(Y, \xi, V; \fk{r})_{\l^+, J^+, \theta^+} \to \widetilde{CH}_{\bullet}(Y, \xi, V; \fk{r})_{\l^+, J^+, \theta^+}. \eeq 

\defi[Reduced invariant] \label{definition:reduced-invariant}
Consider a TN contact pair $(Y, \xi, V)$ and fix an element $\fk{r} \in \fk{R}(Y, \xi, V)$.  Let \eq \widetilde{CH}_{\bullet}(Y, \xi, V; \fk{r}) \eeq be the limit (or equivalently the colimit) of the algebras $\{ \widetilde{CH}_{\bullet}(Y, \xi, V; \fk{r})_{\l, J, \t} \}_{\l, J, \t}$ along the maps \eqref{equation:reduced-isoms}. 
\edefi

For future reference, we record the following corollary of the above discussion.

\cor \label{corollary:cobordism-map}
Let $(Y^{\pm}, \xi^{\pm}, V^{\pm})$ be TN contact pairs and choose elements $\fk{r}^{\pm}= (\a^{\pm}, \tau^{\pm}, r^{\pm}) \in \fk{R}(Y^{\pm}, \xi^{\pm}, V^{\pm})$. Consider an exact relative symplectic cobordism $(\hat{X}, \hat \l, H)$ with positive end $(Y^{+}, \xi^{+}, V^{+})$ and negative end $(Y^{-}, \xi^{-}, V^{-})$, and suppose that $\tau^+, \tau^-$ extend to a global trivialization of the normal bundle of $H$.  If $r^+ \geq e^{\mc{E}((H, \hat{\l}|_H)^{\a^+}_{\a^-})} r^-$, then there is an induced map 
\eq \Phi(\hat{X}, \hat \l, H): CH_{\bullet}(Y^+, \xi^+, V^+; \fk{r}^+) \to CH_{\bullet}(Y^-, \xi^-, V^-; \fk{r}^-).\eeq

Similarly, suppose that $(\hat X, \hat \l^t, H^t)_{t \in [0,1]}$ is a family of exact relative symplectic cobordisms with ends $(V^{\pm}, \xi^{\pm}, V^{\pm})$ and such that $\tau^{\pm}$ extends to a global trivialization of the normal bundle of $H^t$. If $r^+ \geq e^{\mc{E}(H^t, (\hat \l^t)|_{H^t})} r^-$, then 
\eq \Phi(\hat{X}, \hat \l^0, H^0)= \Phi(\hat X, \hat \l^1, H^1).\eeq

The analogous statement holds for the reduced invariants $\widetilde{CH}_{\bullet}(-)$. 
\ecor

\subsection{(Bi)gradings}  \label{subsection:bigradings-obd}

The $\Z/2$-grading by parity on the deformed invariants $CH_{\bu}(-), \widetilde{CH}_\bu(-)$ shall be referred to as the \emph{homological grading}. As in the case of (ordinary) contact homology, the homological grading can be lifted to a $\Z$-grading under certain topological assumptions. We will also to refer to this $\Z$-grading as the homological grading when it exists. 

\defi[see Sec.\ 1.8 in \cite{pardon}] \label{definition:homological-grading}
Let $(Y^{2n-1}, \xi, V)$ be a TN contact pair and choose $\fk{r} \in \fk{R}(Y, \xi, V)$.  Suppose that $H_1(Y; \Z)=0$ and $c_1(\xi)=0$. Then the homological $\Z/2$-grading lifts to a canonical $\Z$-grading defined on generators by 
\eq \label{equation:orbits-grading} |\g| = \op{CZ}^{\tau}(\g)+ n-3, \eeq
where $\tau$ is any trivialization of the contact distribution along $\g$ (this is independent of $\tau$ due to our assumption that $c_1(\xi)=0$). 
\edefi

\rmk
In \Cref{definition:homological-grading}, our assumption that $c_1(\xi)=0$ is equivalent to the statement that the canonical bundle $\L^{n-1}_{\C} \xi$ is trivial. The grading in general depends on a trivialization of the canonical bundle; however, our assumption that $H_1(Y; \Z)=0$ along with the universal coefficients theorem implies that $H^1(Y; \Z)=0$. Hence the canonical bundle admits a unique trivialization.
\ermk

\lem[see (2.50) in \cite{pardon}] \label{lemma:homological-grading-cobordism}
With the notation of \Cref{corollary:cobordism-map}, let us suppose that $H_1(Y^{\pm}; \Z)= 0$ and that $c_1(\xi^{\pm})=c_1(TX)=0$. Then the cobordism maps described in \Cref{corollary:cobordism-map} preserve the homological $\Z$-grading.
\qed 
\elem

Under certain topological assumptions, the reduced invariant $\widetilde{CH}_{\bu}(-;-)$ admits an additional $\Z$-grading which we will refer to as the \emph{linking number grading}. 

\defi \label{definition:bigrading}
Let $(Y, \xi, V)$ be a TN contact pair and choose $\fk{r} \in \fk{R}(Y, \xi, V)$. Suppose that $H_1(Y; \Z)= H_2(Y; \Z)=0$. Then the \emph{linking number grading} $| \cdot |_{\op{link}}$ on $\widetilde{CH}_\bu(Y, \xi, V; \fk{r})$ is given on generators by (see \Cref{definition:linking-number})
\eq \label{equation:linking-orbits} |\g|_{\op{link}} = \op{link}_V(\g). \eeq
The linking number grading of a word of generators is then defined to be the sum of the linking number grading of each letter. One can verify using \Cref{lemma:cobordism-linking} that this grading is well-defined.
		
We let 
\eq \label{equation:bigrading} \widetilde{CH}_{\bu, \bu}(Y, \xi, V;\fk{r}) \eeq  be the (super)-commutative bigraded $\Q$-algebra, where 
\begin{itemize}
\item the first bullet refers to the homological $\Z$-grading (which exists in view of our topological assumption and the universal coefficients theorem, see \Cref{definition:homological-grading});
\item the second bullet refers to the linking number $\Z$-grading. 
\end{itemize}
	
We sometimes drop the second grading in our notation, so the reader should keep in mind that the notation $\widetilde{CH}_{\bu}(-;-)$ always refers to the homological grading. 
\edefi	

We have the following lemma as a consequence of \Cref{lemma:cobordism-linking} and \Cref{lemma:homological-grading-cobordism}. 

\lem \label{lemma:linking-grading-cobordism}
With the notation of \Cref{corollary:cobordism-map}, suppose that $H_1(Y^{\pm}; \Z)=H_2(X, Y^+; \Z)=0$. Then the cobordism maps described in \Cref{corollary:cobordism-map} preserve the linking number $\Z$-grading. In case we also have that $c_1(\xi^{\pm})=c_1(TX)=0$, then the cobordism maps preserve the $(\Z \tms \Z)$-bigrading \eqref{equation:bigrading}. 
\qed
\elem

\subsection{Asymptotic invariants}  \label{subsection:asymptotic-invariants}

Given a TN contact pair $(Y, \xi, V)$ and a trivialization $\tau$ of the normal bundle $N_{Y/V}$, let 
\eq \fk{R}^{\tau}(Y, \xi, V)= \{\fk{r}= (\a, \tau', r') \in \fk{R}(Y, \xi, V) \mid \tau'=\tau \} \sub \fk{R}(Y, \xi, V). \eeq
We equip $\fk{R}^{\tau}(Y, \xi, V)$ with a preorder $\preceq$ defined by setting $(\a^-_V, \tau, r^-) \preceq (\a^+_V, \tau, r^+) $ if $r^+ \geq e^{-\op{min} f} r^-$, where $\a^+_V= e^f \a^-_V$.\footnote{A preorder on a set is a binary relation which is reflexive and transitive. Equivalently, a preordered set is a category with at most one morphism from any object $x$ to any other object $y$.} We let $\preceq^{\op{op}}$ denote the opposite preorder.

We now define a functor $\mc{F}(Y, \xi, V)$ from the preordered set $(\fk{R}^{\tau}(Y, \xi, V), \preceq^{\op{op}})$ to the category of $\Q[U]$-algebras. On objects, the functor takes $\fk{r}$ to ${CH}_{\bullet}(Y, \xi, V; \fk{r})$. It remains to define the functor on morphisms.  

Given elements $\fk{r}^{\pm}=(\a^{\pm}_V, \tau, r^{\pm}) \in \fk{R}^{\tau}(Y, \xi, V)$, let $\l^{\pm}$ be a contact form on $Y$ which is adapted to $\fk{r}^{\pm}$. Consider the symplectization $(\hat{Y}, \hat{\l}, \hat{V})^{\l^+}_{\l^-}$.  If $\fk{r}^- \preceq \fk{r}^+$, then \Cref{lemma:energy-marked-symplectization}, \Cref{proposition:map-II} and \Cref{proposition:commutativity} imply that there is a map 
\eq \label{equation:cobordism-asymptotic-invariants} \Phi(\hat{Y}, \hat{\l}, \hat{V}): CH_{\bullet}(Y, \xi, V; \fk{r}^+) \to CH_{\bullet}(Y, \xi, V; \fk{r}^-). \eeq 
This defines $\mc{F}(Y, \xi, V)$ on morphisms. One can check using \Cref{proposition:commutativity} that $\mc{F}(Y, \xi, V)$ is indeed a functor. 

We can similarly define a functor $\mc{F}_+(Y, \xi, V)$ from $(\fk{R}^{\tau}(Y, \xi, V), \preceq^{\op{op}})$ to the category of $\Q$-algebras using $\widetilde{CH}_\bu(-)$.

\defi[Asymptotic invariants] \label{definition:asymptotic-invariants}
Noting that the category of $\Q[U]$-algebras is complete and co-complete, we denote by 
\eq \underleftarrow{CH}_{\bullet}(Y, \xi, V; \tau) \text{ and } \underrightarrow{CH}_{\bullet}(Y, \xi,V; \tau) \eeq 
the limit (resp.\ the colimit) of the $\Q[U]$-algebras $\{ CH_{\bullet}(Y, \xi, V; \fk{r}) \}$ over the preordered set $(\fk{R}^{\tau}(Y, \xi, V), \preceq^{\op{op}})$. We let $\underleftarrow{\widetilde{CH}}_{\bullet}(Y, \xi, V; \tau)$ and $\underrightarrow{\widetilde{CH}}_{\bullet}(Y, \xi, V; \tau)$ be defined similarly over the category of $\Q$-algebras. 
\edefi

It's easy to check that $(\fk{R}^{\tau}(Y, \xi, V), \preceq^{\op{op}})$ is a filtered preordered set. In particular, (co)limits can be computed by restricting to (co)final subsets. In contrast to the invariants defined in \Cref{subsection:mainconstruction}, the asymptotic invariants are fully functorial under compositions of arbitrary relative symplectic cobordisms which respect normal trivializations. (A verification of this is tedious and essentially consists of repeating the arguments of \Cref{subsection:mainconstruction} -- the key is that the energy conditions can always be satisfied by sending the rotation parameter to zero or infinity.)

\subsection{Mixed morphisms} 
Consider a TN contact pair $(Y, \xi, V)$ and elements $\fk{r}^{\pm}= (\a^{\pm}, \tau^{\pm}, r^{\pm}) \in \fk{R}(Y, \xi, V)$. In this section, we exhibit a $\Q$-algebra map
\eq CH_\bu^{U=0}(Y,\xi, V; \fk{r}^+) \to \widetilde{CH}_\bu(Y, \xi, V; \fk{r}^-) \eeq 
under certain assumptions on $r^+, r^-$. Precomposing with \eqref{equation:evaluation} gives a $\Q$-algebra map
\eq CH_\bu(Y, \xi, V; \fk{r}^+) \to \widetilde{CH}_\bu(Y, \xi, V; \fk{r}^-). \eeq

Let us begin by considering a datum $\mc{D}=(\mc{D}^+, \mc{D}^-, \hat X, H, \hat \l, \hat J)$ for Setup II*, where we let $\mc{D}^\pm = ((Y,\xi, V), \fk{r}^{\pm}, \l^\pm, J^\pm)$.

\defi \label{definition:twisting-mixed}
We define a map $\psi_{\op{mix}}: \Sch_\II^{\neq \emptyset}(\mc{D}) \to \Q$ by 
\eq  \psi_{\op{mix}}(T) = \begin{cases} 1 & \text{if $ T * \hat{V} = 0$ and $\left|\gamma_e\right| \cap V^- = \emptyset$ for every $e \in E(T)$}; \\ 0 & \text{otherwise.} \end{cases}
\eeq
\edefi

\prop \label{proposition:twist-full-to-red}
The map $\psi_{\op{mix}}(-)$ is a twisting map provided that the following properties hold:
\begin{itemize}
\item[(i)] $r^+ \geq 2 e^{\mc{E}(\hat{J})} r^-$, 
\item[(ii)] $r^+ > 2 / R^{\op{min}}_{\a}$, where $R^{\op{min}}_{\a}$ denotes the smallest action of all Reeb orbits of $\a$. 
\end{itemize}
\eprop

We need the following lemma for proving \Cref{proposition:twist-full-to-red}. 

\lem \label{lemma:pos-full-to-red}
Suppose that $\b \in \pi_2(\hat{Y}, \g^+ \sqcup (\cup_{i=1}^n \g^-_i))$ is represented by a $\hat{J}$-holomorphic curve $u: \dot{\S} \to \hat{X}$ contained in $\hat{V}$. Suppose also that $r^+, r^-$ satisfy the assumptions (i) and (ii) in \Cref{definition:twisting-mixed}. Then $\beta*\hat{V} \geq 2-p_u$, where $p_u$ is the total number of punctures (positive and negative) of $u$ contained in $V^{\pm}$. 
\elem

\pf 
By \Cref{lemma:stokes-II}, we have that $r^+ P^+ - r^- \sum_{i=1}^n P^-_i \geq r^+ P^+ - r^- P^+ e^{\mc{E}(\hat{J})} \geq P^+(r^+ - r^- e^{\mc{E}(\hat{J})}) \geq R^{\op{min}}_{\a}  (( r^+/2 - r^- e^{\mc{E}(\hat{J})})+ r^+/2) \geq R^{\op{min}}_{\a} r^+/2 \geq 1$. The lemma now follows from \Cref{proposition:positivity-cobordism}. 
\epf

\defi
Given a tree $T \in \Sch^{\neq \emptyset}_\II$, we say that a vertex $v \in V(T)$ is \emph{mean} if all adjacent edges are contained in $V^{\pm}$. Otherwise, we say that $v \in V(T)$ is \emph{nice}. We denote by $V_b(T)$ (resp.\ $V_n(T)$) the set of mean (resp. nice) vertices of $T$.
\edefi

\pf[Proof of \Cref{proposition:twist-full-to-red}]
Choose a tree $T' \in \Sch^{\neq \emptyset}_\II$.  Let $T' \to T$ be a morphism. It follows from \Cref{proposition:invariance-gluing} that $T'* \hat{V} = T* \hat{V}$. If $T'$ has no edges contained in $V^-$, then neither does $T'$ and we see that ${\psi}_{\op{mix}}(T')= \psi_{\op{mix}}(T)$. 

Let us now suppose that $T'$ has an edge contained in $V^-$. Note that $T', T$ have the same exterior edges. If one of these edges is contained in $V^{-}$, then ${\psi}_{\op{mix}}(T')= \psi_{\op{mix}}(T)=0$. Let us therefore assume that the exterior edges of $T', T$ are not contained in $V^{-}$. 
	
We are left with the case where $T'$ has at least one interior edge contained in $V^{-}$. If $T'$ had no mean vertices, then it would follow from \Cref{proposition:gluing-additivity} that there are no interior edges contained in $V^{\pm}$, which is a contradiction. It follows that $T'$ has at least one mean vertex. Let $E^{\op{int}}_b(T') \sub E^{\op{int}}(T')$ be the set of interior edges which occur as an outgoing edge of some mean vertex.  According to \Cref{proposition:gluing-additivity}, we have $T' * \hat{V} = \sum_{v \in V_n(T')} \beta_v * \hat{V} + \sum_{v \in V_m(T')} \beta_v * \hat{V} + | E^{\op{int}} (T') - E^{\op{int}}_b(T')  | + | E^{\op{int}}_b(T') | \geq \sum_{v \in V_m(T')} \beta_v * \hat{V} + | E^{\op{int}}_b(T') |  = \sum_{v \in V_m(T')} (\beta_v* \hat{V} + p^-_v)$, where $p^-_v$ denotes the number of outgoing edges of $v$.  (Here, we have used the fact that the outgoing edges of a mean vertex are all interior edges, which follows from our assumption that the exterior edges of $T', T$ are not contained in $V^{-}$.) It now follows from \Cref{lemma:pos-full-to-red} and the fact that $T'$ has at least one mean vertex that $\sum_{v \in V_m(T')} (\beta_v* \hat{V} + p_v) \geq \sum_{v \in V_m(T')} (2 -p_v + p^-_v) \geq \sum_{v \in V_m(T')} 1 \geq 1$. Hence ${\psi}_{\op{mix}}(T')= \psi_{\op{mix}}(T)=0$.  This completes the proof that $\psi_{\op{mix}}(T')= \psi_{\op{mix}}(T)$. 

If $\{T_i\}_i$ is a concatenation, then the argument is the same as in the proof of \Cref{proposition:basic-additivityII} since every edge in $\#_iT _i$ appears in at least one of the $T_i$. 
\epf

\prop \label{proposition:mixed-map}
Consider a TN contact pair $(Y, \xi, V)$ and elements $\fk{r}^{\pm}= (\a^{\pm}, \tau^{\pm}, r^{\pm}) \in \fk{R}(Y, \xi, V)$. If $r' > 2 e^{\mc{E}(\hat{V}, \hat \l |_{\hat V} )} r^-$ and $r^+ > 2 / R^{\op{min}}_{\a}$, then there is a map of $\Q$-algebras 
\eq \label{equation:mixed} CH_{\bullet}^{U=0}(Y, \xi, V; \fk{r}^+) \to \widetilde{CH}_{\bullet}(Y, \xi, V; \fk{r}^-).\eeq 
\eprop

\pf
The argument is essentially the same as the proof of \Cref{proposition:map-II}. Choose data of Type I* $\mc{D}^\pm = ((Y,\xi, V), \fk{r}^{\pm}, \l^\pm, J^\pm)$. Now consider the symplectization $(\hat Y, \hat \l, \hat V)$. \Cref{lemma:ac-connected-II} furnishes an almost-complex structure $\hat{J}$ on $\hat{Y}$ which is $d\hat{\l}$-compatible and agrees with $\hat{J}^{\pm}$ at infinity, and such that $r^+ \geq e^{\mc{E}(\hat{J})} r^-$.  
It now follows as in \Cref{subsection:twistinghomologies}(II) that we have a $\Q$-algebra chain map
\eq \Phi(\hat{Y}, \hat{\l}, \psi_{\op{mix}})_{\hat{J}, \T}:  CC_{\bullet}^{U=0}(Y, \xi, V; \fk{r}')_{J^+, \t^+} \to \widetilde{CC}_{\bullet}(Y, \xi, V; \fk{r})_{J^-, \t^-}, \eeq	
for perturbation data $\T \mapsto (\t^+, \t^-)$. 
\epf

\rmk
The proof of \Cref{proposition:mixed-map} does not show that \eqref{equation:mixed} is independent of auxiliary choices (i.e. $\hat{J}, J^{\pm}, \T, \t^{\pm}$). To show this, one needs to extend the definition of the twisting map $\psi_{\op{mix}}$ to Setups III* and IV*. One can then prove analogs of \Cref{theorem:twisted-counts} and \Cref{corollary:cobordism-map} for $\psi_{\op{mix}}$. All of the ingredients for this are already in place, but we omit the details since \Cref{proposition:mixed-map} is sufficient for our applications. 
\ermk

\cor \label{corollary:non-zero-non-reduced}
Suppose that $\widetilde{CH}_{\bullet}(Y, \xi, V; \fk{r}) \neq 0$ for some $\fk{r}=(\a_V, \tau, r) \in \fk{R}(Y, \xi, V)$. Letting $\fk{r}'= (\a_V', \tau', r')$, we have $CH_{\bullet}(Y, \xi, V; \fk{r}') \neq 0$ if $r'$ is large enough. In particular, $\underleftarrow{CH}_{\bullet}(Y, \xi, V; \tau) \neq 0$. 
\qed
\ecor

\section{Augmentations and linearized invariants}

\subsection{Differential graded algebras}\label{subsection:dgas}
Let $R$ be a commutative ring containing $\Q$. The following two categories occur naturally in contact topology: the category \bf{cdga} of unital $\Z$-graded (super-)commutative dg $R$-algebras; the category \bf{dga} of unital $\Z$-graded associative dg $R$-algebras. When we speak of a \emph{dg algebra}, we mean an object of either one of these categories. 



Let us say that a dg algebra is \emph{action-filtered} if the underlying graded algebra is the free algebra on a free graded module $U$ having the following property: $U$ admits a basis $\{ x_{\a} \mid \a \in \mc{A}\}$ for some well-ordered set $\mc{A}$ such that $dx_{\a}$ is a sum of words in the letters $x_\b$ for $\b < \a$.  \emph{Note that the dg algebras which arise in Symplectic Field Theory are naturally action filtered by Reeb length.}

The categories \bf{cdga} and \bf{dga} each carry a model structure described by Hinich (see \cite{hinich} or \cite[Sec.\ 1.11]{vallette}) with the following properties: (i) the weak equivalences are the quasi-isomorphisms; (ii) all objects are fibrant and the set of cofibrant objects includes all objects which are \emph{action-filtered}.\footnote{Both \cite{hinich} and \cite{loday-vallette} work in the setting of dg algebras over operads. \bf{cdga} and \bf{dga} are respectively the category of dg algebras over the operad uComm and uAss ($u$ stands for \emph{unital}), so they can be treated on equal footing.  A good reference for the material in this section which mostly avoids operadic language (but only treats the case of $\bf{cdga}$) is \cite{lazarev-markl}.} Note that in any model category, there is a notion of two maps being left (resp.\ right) homotopic, defined in terms of cylinder (resp.\ path) objects \cite[Sec.\ 7.3.1]{hirschhorn}. These notions coincide on objects which are fibrant and cofibrant. 

We let $\bf{hcdga}$ and $\bf{hdga}$ be the associated homotopy categories: the objects are those objects of $\bf{cdga}$ (resp.\ $\bf{dga}$) which are fibrant and cofibrant, and the morphisms are the homotopy classes of morphisms in $\bf{cdga}$ (resp.\ $\bf{dga}$). \emph{In particular, $\bf{hcdga}$ and $\bf{hdga}$ contain all action-filtered dg algebras.}

An \emph{augmentation} of a dg $R$-algebra $A$ is a morphism of dg algebras $\e: A \to R$, where $R$ is viewed as a dg algebra concentrated in degree $0$.  A dg algebra equipped with an augmentation is said to be augmented. We let $\bf{cdga}_{/R}$ and $\bf{dga}_{/R}$ be the overcategory of augmented objects, which naturally inherit model structures. We let $\bf{hcdga}_{/R}$ (resp.\ $\bf{hdga}_{/ R}$) be the associated homotopy categories. 

\defi \label{definition:linearization}
Given an augmented dg algebra $\e: (A, d) \to R$, we consider the graded $R$-module $A_{\e}:= \op{ker} \e/ (\op{ker} \e)^2$. The differential $d$ descends to a differential $d_{\e}$ on $A_{\e}$. The resulting differential graded module $(A_{\e}, d_{\e})$ is called the \emph{linearization} of $(A,d)$ at the augmentation $\e$.\footnote{The linearization is sometimes also called the ``indecomposable quotient" in the rational homotopy theory literature. In the contact topology literature, one also often encounters an equivalent construction of the linearization in which one twists the differential by the augmentation; see \cite[Rmk.\ 2.8]{ng-intro}.} 
\edefi
It follows from the definition that linearization defines a functor from $\bf{cdga}_{/R}$ (resp.\ $\bf{dga}_{/R}$) to the category of chain complexes of $R$-modules. It is an important fact that this functor is left Quillen \cite[12.1.7]{loday-vallette}, and therefore induces a functor of homotopy categories. We state this as a corollary:

\cor\label{corollary:linearization-functor}
Linearization defines a functor from $\bf{hcdga}_{/R}$ (resp.\ $\bf{hdga}_{/ R}$) to the homotopy category of chain complexes of $R$-modules. 
\qed
\ecor

In particular, \Cref{corollary:linearization-functor} implies that homotopic augmentations of action-filtered dg algebras induce isomorphic linearizations. Special cases of this statement have already appeared in the Legendrian contact homology literature; see e.g.\ \cite[Sec.\ 5.3.2]{aug-sheaves}). 

\rmk \label{remark:zero-augmentation}
Let $(A,d)$ be a dg algebra, where $A$ is the free $R$-algebra generated by the set $\{x_{\a} \mid \a \in \mc{A}\}$. Let $\e: A \to R$ be the unique $R$-algebra map sending the generators $x_{\a}$ to zero. Then $\e$ is an augmentation if and only if $d x_{\a}$ is contained in the ideal $(x_\b \mid \b \in \mc{A} )$ for all $\a \in \mc{A}$ (or equivalently, iff the differential has no constant term). If $\e$ is an augmentation, it is called the \emph{zero augmentation}. 

Suppose now that $(A, d)$ is the (possibly deformed) contact algebra of some contact manifold, i.e. $(A,d)$ is the commutative $R$-algebra generated by good Reeb orbits (for $R=\Q, \Q[U]$) and the differential is defined as in \Cref{subsection:twistinghomologies}. Suppose that $\e: A \to R$ is the zero augmentation. Then $\op{ker} \e$ is the free $R$-module generated by (good) Reeb orbits and $d_{\e}$ counts curve with one input and one output (i.e. $d_\e$ is defined as in \eqref{equation:twist-differential-I}, where the sum is restricted to curves with $|\G^-|=1$). It follows that the homology of the complex $(A_{\e}, d_{\e})$ can be interpreted as the (possibly deformed) \emph{cylindrical} contact homology. This latter invariant admits a rigorous definition for contact structures under certain assumptions, such as the existence of a contact form with no contractible Reeb orbits; see \cites{bao-honda, hutchings-nelson} in dimension $3$ and \cite[Sec.\ 1.8]{pardon} in general. 
\ermk

\subsection{Cyclic homology} 

\defi[cf.\ \cite{connes}] \label{definition:reduced-cyclic}
Let $S$ be a countable, well-ordered set equipped with a map $| \cdot |: S \to \Z$. Let $A= R \langle S \rangle$ be the free $\Z$-graded $R$-algebra generated by $S$, where the $\Z$-grading is induced by extending $| \cdot |$ multiplicatively. 
	
Let $d: A \to A$ be a differential of degree $-1$. Let $\ov{A} := A/R$, and consider the cyclic permutation map $\tau: \ov{A} \to \ov{A}$ which is defined on monomials by $\tau(\g_1\dots \g_l)= (-1)^{|\g_1|(|\g_2|+\dots+|\g_l |) } \g_2\dots \g_l \g_1$ and extended $R$-linearly. Let $\ov{A}^{\tau}:= \ov{A}/(1-\tau)$ be the $\Z$-graded $R$-module of coinvariants. Observe that $d$ passes to the quotient. We denote the induced differential by $d^{\tau}$.
	
We now define 
\eq \ov{HC}_{\bullet}(A):= H_{\bullet}(\ov{A}^{\tau}, d^{\tau}) \eeq 
and refer to this invariant as the \emph{reduced cyclic homology} of the dg algebra $(A, d)$.
\edefi

\rmk	
\Cref{definition:reduced-cyclic} agrees with other definitions of reduced cyclic homology of dg algebras (such as \cite[Sec.\ 5.3]{loday}) which may be more familiar to the reader, when both are defined. We adopt the present definition for consistency with \cite{bee}. 

In the special case where $A$ is the Chekanov-Eliashberg dg algebra of a Legendrian knot in a contact manifold satisfying the assumptions of \cite[Sec.\ 4.1]{bee}, the algebraic invariants considered in \cite[Sec.\ 4]{bee} can be translated as follows: $L\bb{H}^{\op{cyc}}(A) = \overline{HC}_\bullet(A)$, $L\bb{H}^{\op{Ho}+}(A)=\overline{HH}_\bullet(A)$ and $L\bb{H}^{\op{Ho}}(A)=HH_\bullet(A)$. Here $HH_\bullet(-)$ and $\overline{HH}_\bullet(-)$ denote respectively Hochschild homology and reduced Hochschild homology. 
\ermk

We record the following computation which will be useful to us later on.

\lem \label{lemma:reduced-cyclic-acyclic}
Under the assumptions of \Cref{definition:reduced-cyclic}, if $(A, d)$ is acyclic, then
\eq
	\ov{HC}_k(A) = \begin{cases} R & \text{if }  k\; \text{is odd and positive,}\\
	0 & \text{otherwise.} \end{cases}
\eeq
\elem

\pf
Let us first prove the lemma under the assumption that $S$ is a finite set. Note first that the Hochschild homology of an acyclic finitely-generated dg algebra vanishes identically. Moreover, we have an exact triangle 
\eq \label{equation:reduced-hochschild} R[0] \to HH_\bullet(A) \to \ov{HH}_\bullet(A) \xrightarrow{[-1]}, \eeq which implies that $\ov{HH}_\bullet(A)$ is just a copy of $R$ concentrated in degree $1$. 
	
We now consider the following Gysin-type exact triangle (see \cite[Prop.\ 4.9]{bee}):  
\eq \label{equation:hochschild-cyclic} \ov{HC}_\bullet(A) \xrightarrow{[-2]} \ov{HC}_\bullet(A) \xrightarrow{[+1]} \ov{HH}_\bullet(A) \xrightarrow{[0]}\eeq

The desired result now follows immediately by induction, using \eqref{equation:hochschild-cyclic} and the fact that $\ov{HC}_\bullet(A)$ vanishes in sufficiently large positive and negative degrees due to our finiteness hypotheses. We remark that \eqref{equation:hochschild-cyclic} is constructed from a spectral sequence, whose convergence can only be verified under finiteness assumptions.  

Let us now drop our assumption that $S$ is a finite set.  We instead consider an exhaustion of $S$ by finite subsets $S^{(1)} \sub S^{(2)} \sub \dots \sub S$. Let $(A^{(k)}, d) \sub (A, d)$ be the dg sub-algebra generated by $S^{(k)}$. One can readily verify that 
\eq \varinjlim \ov{HC}_\bullet(A^{(k)}) = \ov{HC}_\bullet( \varinjlim A^{(k)}) = \ov{HC}_\bullet(A). \eeq
	
Observe that $(A^{(k)}, d)$ is acyclic for $k$ large enough and satisfies the assumption of \Cref{definition:reduced-cyclic}. Since we have already proved the lemma under the assumption that $S$ is finite, it is enough to prove that the natural maps $\ov{HC}_\bullet(A^{(k)}) \to \ov{HC}_\bullet(A^{(k+1)})$ are isomorphisms for $k$ large enough. 
	
To this end, note that the exact triangles \eqref{equation:reduced-hochschild} and \eqref{equation:hochschild-cyclic} can be shown to be functorial under morphisms of bounded dg algebras. Since quasi-isomorphisms induce isomorphisms on Hochschild homology, it follows from \eqref{equation:reduced-hochschild} that the natural map $\ov{HH}_\bullet(A^{(k)}) \to \ov{HH}_\bullet(A^{(k+1)})$ is an isomorphism. Since $\ov{HH}_\bullet(A^{(k)})=\ov{HH}_\bullet(A^{(k+1)})$ is concentrated in degree $1$, and since $\ov{HC}_i(A^{(k)})$ and $\ov{HC}_i(A^{(k+1)})$ vanish for $|i|$ sufficiently large, the desired claim can be checked by inductively applying the five-lemma (cf.\ \cite[Sec.\ 2.2.3]{loday}).
\epf

\subsection{Augmentations from relative fillings} \label{subsection:augmentations-fillings}
According to the philosophy of \cite{sft-egh}, contact homology is supposed to be well-defined as a differential graded $\Q$-algebra up to strict isomorphism. However, \cite{pardon} only proves that contact homology is well-defined as a graded $\Q$-algebra, i.e.\ after passing to homology. Similarly, the invariants introduced in \Cref{section:invariant} are merely graded $R$-algebras for $R=\Q, \Q[U]$. For the purpose of linearizing the invariants of \Cref{section:invariant}, we saw in \Cref{subsection:dgas} that the following intermediate assumption (weaker than what is conjectured in \cite{sft-egh} but stronger than what is proved in \cite{pardon}) is sufficient:

\begin{assumption} \label{assumption:dga-homotopy}
The constructions of \Cref{section:invariant} can be lifted to the category \bf{hcdga}. (i.e.\ cobordisms induce a unique morphism of commutative dg algebras up to homotopy, and composition of cobordism is functorial up to homotopy.)
\end{assumption}


The rest of this section depends on the unproved \Cref{assumption:dga-homotopy};  all statements which depend on this assumption are therefore starred, in accordance with the convention stated in the introduction.

\begin{definition-star} \label{definition:contact-algebra}
Fix a TN contact pair $(Y, \xi, V)$ and an element $\fk{r} \in \fk{R}(Y, \xi, V)$. Let 
\eq \mc{A}(Y, \xi, V; \fk{r}) \in \Q[U]-\bf{hcdga} \eeq 
be the limit (or equivalently the colimit) of the dg algebras $\{ (CC_{\bullet}(Y, \xi, V; \fk{r})_{\l}, d_{\psi, J, \theta}) \}_{\l, J, \t} $ under the lifts of the maps \eqref{equation:iso-change-contact-forms} which are furnished by \Cref{assumption:dga-homotopy}. We also analogously define 
\eq \widetilde{\mc{A}}(Y, \xi, V; \fk{r}) \in \Q-\bf{hcdga}. \eeq 
\end{definition-star}

\rmk[Bigradings] \label{remark:bigradings-contact-algebra}
With the notation of \Cref{definition:contact-algebra}, suppose that $H_1(Y; \Z)=H_2(Y;\Z)=0$. Combining \Cref{assumption:dga-homotopy} with the discussion of \Cref{subsection:bigradings-obd}, it then follows that $\widetilde{\mc{A}}(Y, \xi, V; \fk{r})$ is a $(\Z \tms \Z)$-bigraded differential algebra, where the differential has bidegree $(-1,0)$. 
\ermk

\begin{definition-star} \label{definition:bigrading-lift}
Given an augmentation $\e: \mc{A}(Y, \xi, V; \fk{r}) \to \Q[U]$, we let $\mc{A}^{\e}(Y, \xi, V; \fk{r}) $ be the linearized chain complex (in the sense of \Cref{definition:linearization}) with respect to $\e$ and let $CH_\bullet^{\e}(Y, \xi, V; \fk{r})$ be the resulting homology. 
	
We have analogous invariants in the reduced case, which are denoted by $\widetilde{\mc{A}}^{\tilde{\e}}(Y, \xi, V; \fk{r})$ and $\widetilde{CH}_{\bullet}^{\tilde{\e}}(Y, \xi, V; \fk{r})$ for an augmentation $\tilde{\e}: \widetilde{\mc{A}}(Y, \xi, V; \fk{r}) \to \Q]$.
\end{definition-star}

\begin{definition} \label{definition:relative-filling}
Given a contact manifold $(Y, \xi)$ and a codimension 2 contact submanifold $V$, a \emph{relative filling} $(\hat X, \hat \l, H)$ is a relative symplectic cobordism from $(Y, \xi, V)$ to the empty set. 
\end{definition}

Let $(\hat X, \hat \l, H)$ be a relative filling of $(Y, \xi, V)$ and fix $\fk{r} \in \fk{R}(Y, \xi, V)$. Suppose that $\tau$ extends to a normal trivialization of $H$. Then  \Cref{lemma:homological-grading-cobordism} and
\Cref{assumption:dga-homotopy} furnish an augmentation $\e(\hat X, \hat{\l}, H): \mc{A}(Y, \xi, V; \fk{r}) \to \Q[U]$. Similarly, we have an augmentation $\tilde{\e}(\hat X, \hat \l, H): \widetilde{\mc{A}}(Y, \xi, V; \fk{r}) \to \Q$.

If we suppose that $H_1(Y;\Z)=H_2(Y; \Z)= H_2(X, Y; \Z)=0$ and $c_1(TX)=0$, then \Cref{lemma:linking-grading-cobordism} and \Cref{assumption:dga-homotopy} imply that $\tilde{\e}(\hat X, \hat \l, H)$ preserves the $(\Z \tms \Z)$-bigrading defined in \Cref{remark:bigradings-contact-algebra}. It follows that the linearized complex 
\eq \widetilde{\mc{A}}^{\tilde{\e}}(Y, \xi, V; \fk{r}) \eeq
 inherits a $(\Z \tms \Z)$-bigrading with differential of bidegree $(-1,0)$. Hence 
\eq \widetilde{CH}_{\bullet}^{\tilde{\e}}(Y, \xi, V; \fk{r}) \eeq 
is a $(\Z \tms \Z)$-bigraded $\Q$-vector space. 

We end this section by collecting some lemmas which will be useful later. The reader is referred to \Cref{subsection:open-book-decomp} for a review of open book decompositions.

\begin{lemma-star} \label{lemma:indep-relative-filling}
Suppose that $(\hat W, \hat \l, H)$ is a relative filling of $(Y,\xi, V)$. Suppose that $V$ is the binding of an open book decomposition $(Y, V, \pi)$ which supports $\xi$. Fix an element $\fk{r}=(\a_V, \tau, r) \in \fk{R}(Y, \xi, V)$ where $\tau$ is the canonical trivialization induced by the open book. 
	
Suppose that $H$ admits a normal trivialization which restricts to $\tau$. Suppose also that $H_1(Y; \Z)= H_2(Y; \Z)=H_2(W, Y; \Z)=0$ and that $c_1(TW)=0$. Then the augmentation 
\eq \tilde{\e}: \widetilde{\mc{A}}(Y, \xi, V; \fk{r}) \to \Q \eeq 
is the zero augmentation. In particular, it depends only on $(\hat W, \hat \l)$ and not on $H$.
\end{lemma-star}

\pf It is shown in \Cref{proposition:girouxstructure} that there exists a non-degenerate contact form $\a$ for $(Y, \xi)$ which is adapted to $\fk{r}$ and has the property that all Reeb orbits are transverse to the pages of the open book decomposition. Given auxiliary choices of almost-complex structures and perturbation data, the augmentation $\tilde{\e}$ counts (possibly broken) holomorphic planes $u$ which are asymptotic to a Reeb orbit $\g$ disjoint from $H$, and such that $[u] * H=0$. However, our topological assumptions and \Cref{lemma:cobordism-linking} implies that $[u]*H$ is precisely the linking number of $\g$ with the binding $V$, which is strictly positive by assumption (see \Cref{rmk:obd-linking-number}). 
\epf

\begin{lemma-star} \label{lemma:fillings-commute}
Let $(\hat X, \hat \l, H)$ and $(\hat X', \hat \l', H')$ be relative symplectic fillings of $(Y, \xi, V)$ and $(Y', \xi', V')$. Let $f: (Y,\xi ,V) \to (Y', \xi', V')$ be a contactomorphism. Suppose that there exist a symplectomorphism $\phi: (\hat X, \hat \l) \to (\hat X', \hat \l')$ which coincides near infinity with the induced map $\tilde{f}: SY \to SY'$. 

Given any $\fk{r} \in \fk{R}(Y, \xi, V)$, the following diagram commutes:

\eq \label{equation:fillings}
	\begin{tikzcd}
		CH_\bullet(Y, \xi, V; \fk{r}) \ar{rrr}{\Phi(\hat X, \hat \l, H)} \ar{d}{=} &&& \Q[U] \ar{d} \\
		CH_\bullet(Y', \xi', V'; f_*\fk{r}) \ar{rrr}{\Phi(\hat X', \phi_*\hat \l, \phi(H))}  \ar{d}{=}  &&& \Q[U] \ar{d} \\
		CH_\bullet(Y', \xi', V'; f_*\fk{r}) \ar{rrr}{\Phi(\hat X', \hat \l', \phi(H))} &&& \Q[U]
	\end{tikzcd}
\eeq

The analogous statement holds for $\widetilde{CH}_{\bullet}(-)$ (with $\Q$ in place of $\Q[U]$).  In addition, if $H_1(Y; \Z)=H_2(Y; \Z)= 0$ and $c_1(TX)=0$, then all arrows can be assumed to preserve the bigrading in \Cref{definition:bigrading}. 

\end{lemma-star}

\pf
The commutativity of the top square is essentially tautological; more precisely, it follows from the functoriality of the moduli counts in \Cref{theorem:twisted-counts}. The commutativity of the bottom square follows from the observation that $\phi$ preserves the Liouville form outside a compact set. Hence $(\hat X', \phi_* \hat \l)$ and $(\hat X', \hat \l')$ are deformation equivalent. It follows by \Cref{corollary:cobordism-map} that they induce the same morphism on homology.  
The fact that the maps preserve the bigradings on $\widetilde{CH}(-;-)$ (under the above topological assumptions) is a consequence on \Cref{lemma:linking-grading-cobordism}. 
\epf

\begin{corollary-star} \label{corollary:commuting-fillings}
Let $(\hat X, \hat \l, H)$ (resp.\ $(\hat X, \hat \l, H')$) be relative fillings for $(Y, \xi, V)$ (resp.\ $(Y, \xi, V')$). Suppose that $V$ is the binding of an open book decomposition of $(Y, \xi)$ and fix $\fk{r}=(\a, \tau, r) \in \fk{R}(Y, \xi, V)$, where $\tau$ is induced by the open book. Let $f: (Y, \xi, V) \to (Y, \xi, V')$ be a contactomorphism.
	
Suppose that $H_1(Y;\Z)=H_2(Y;\Z)=H_2(X, Y; \Z)=0$ and that $c_1(TX)=0$. Then 
\eq \widetilde{CH}_{\bu,\bu}^{\tilde{\e}}(Y, \xi, V; \fk{r}) = \widetilde{CH}_{\bu,\bu}^{\tilde{\e'}}(Y, \xi, V'; f_*\fk{r}), \eeq 
where both augmentations are induced by the relative fillings and the $(\Z \tms \Z)$-bigrading is defined in \Cref{definition:bigrading}. 
\end{corollary-star}

\pf  
	Indeed, since the lift of a contactomorphism to the symplectization is a Hamiltonian symplectomorphism, it is easy to construct a symplectic automorphism of $(X, d \hat \l)$ satisfying the conditions of \Cref{lemma:fillings-commute} (see e.g.\ \cite[Sec.\ 3.2]{chantraine}). The claim now follows from \Cref{lemma:indep-relative-filling}. 
\epf

\section{Invariants of Legendrian submanifolds} \label{section:legendrian-invariants}

\subsection{Invariants of contact pushoffs} \label{subsection:contact-pushoff}

\defi[see Def.\ 3.1 in \cite{casals-etnyre}] \label{definition:contact-pushoff} Let $(Y, \xi)$ be a contact manifold and let $\Lambda \hookrightarrow Y$ be a Legendrian embedding.  By the Weinstein neighborhood theorem, the map extends to an embedding $\op{Op}(\L) \sub (J^1\L, \xi_{\mathrm{std}}) \to (Y, \xi)$, where $\op{Op}(\L) \sub (J^1\L, \xi_{\mathrm{std}})$ denotes an open neighborhood of the zero section. 

Let $\tau(\L)$ be the induced codimension $2$ contact embedding 
\eq \tau(\L):  \d (D_{\e, g}^* \L)= \d (D_{\e, g}^* \L) \tms 0 \sub T^*\L \tms \R= J^1 \L \hookrightarrow (Y, \xi).\eeq 
Here $D_{\e, g}^* \L$ is the sphere bundle of covectors of length $\e$ with respect to some metric $g$, which is a contact manifold with respect to the restriction of the canonical $1$-form on $T^*\L$. We refer to $\tau(\L)$ as the \emph{contact pushoff} of $\L \hookrightarrow Y$.
\edefi

Standard arguments establish that the contact pushoff is canonical up to isotopy through codimension $2$ contact embeddings. By abuse of notation, we will routinely identify $\tau(\L)$ with its image. Observe that it follows that $CH_\bullet(Y, \xi, \tau(\L); \fk{r})$ and $\widetilde{CH}_\bullet(Y, \xi, \tau(\L); \fk{r}) $ can be viewed as invariants of $\L$. 

\subsection{Deformations of the Chekanov-Eliashberg dg algebra} \label{subsection:deform-ce-dga}

In the spirit of the previous sections, we now consider deformations of the Chekanov-Eliashberg dg algebra of a Legendrian induced by a codimension $2$ contact submanifold.  We begin with some preliminary definitions. 

\defi \label{definition:non-deg-chord}
Let $\L \sub (Y, \xi)$ be a Legendrian submanifold. Given a contact form $\op{ker} \a= \xi$, consider a Reeb chord $c: [0,R] \to Y$. The linearized Reeb flow defines a path of symplectomorphisms $\mc{P}_r: \xi|_{c(0)} \to \xi|_{c(r)}$. We say that the Reeb chord $c$ is non-degenerate if $\mc{P}_R(T_{c(0)}\L) \cap T_{c(R)} \L= \{0\}$. 
\edefi

\defi[cf.\ Sec.\ 2.1 in \cite{bee}] \label{definition:reeb-chord-index}
With the notation of \Cref{definition:non-deg-chord}, let $\bigwedge^{n-1}_{\C}(\xi, d\a)$ be the canonical bundle of $\xi$ and suppose that it admits a trivialization $\s$. Let $\L_1,\dots, \L_k$ be an enumeration of the components of $\L$. Suppose that each $\L_i$ has vanishing Maslov class.
	
Suppose first of all that $k=1$ (i.e. $\L$ is connected). Given a non-degenerate Reeb chord $c$, pick a path $c_-$ in $\L$ connecting $c(R)$ to $c(0)$. Observe that $\bigwedge^{n-1} T_{c_-} \L_i \sub \bigwedge^{n-1}_{\C}(\xi, d\a)$ is a path of Lagrangian subspaces along $c_-$. We call this path $L_{c_-}$. The map $\mc{P}_r$ also defines a path of Lagrangian subspaces $\bigwedge^{n-1} \mc{P}_r(T_{c(0)}\L_i) \sub \bigwedge^{n-1}_{\C}(\xi, d\a)$ along $c$. We call this path $L_c$. 
			
Let $\tilde{c}=c_- * c$ be obtained by concatenating $c_-$ and $c$ (the concatenation is from left to right). Now consider the path of Lagrangian subspaces $L_{\tilde{c}}= L_{c^-} * L_c * \bf{P}^+$, where $\bf{P}^+$ is a positive rotation from $\mc{P}_R(T_{c(0)}\L)$ to $T_{c(R)} \L$ (this is well-defined by our assumption that $c$ is non-degenerate). 
	
The \emph{Conley-Zehnder index for chords} of $c$ with respect to $\s$ is denoted $\CZ^{+, \s}(c)$ and defined by 
\eq \CZ^{+, \s}(c) = \mu^{\s}(L_{\tilde{c}}), \eeq 
where $\mu^{\s}(-)$ is the Maslov index with respect to $\s$ \cite[Thm.\ 2.3.7]{mcduff-sal-intro}. This definition is independent of the choice of $c_-$ due to our assumption that $\L$ has vanishing Maslov class. Note also that the resulting index depends on $\s$, but its parity does not. 
	
In case $k>1$, the definition of the Conley-Zehnder index for chords is more complicated, and depends on additional choices. We refer the reader to \cite[Sec.\ 2.1]{bee} (we warn the reader that there is a typo in the formula stated there: the correct formula for the Conley-Zehnder index for chords should read $\CZ^{+, \s}(c)= |c|-1= (\phi_- - \phi_{\L}(x_1))/ \pi + (n-1)/2$.)
\edefi

\rmk
It may happen that a Reeb orbit can also be viewed as a Reeb chord with same starting and end point. In this case, we have in general that $\op{CZ}^+(c) \neq \op{CZ}(c)$. 
\ermk

Let us now consider a TN contact pair $(Y, \xi, V)$ and a Legendrian submanifold $\L \sub (Y-V, \xi)$. We let $\L_1,\dots,\L_k$ be an enumeration of the connected components of $\L$. As in \Cref{definition:reeb-chord-index}, we assume that the $\L_i$ have vanishing Maslov class.

\begin{definition-assumption-star}[cf.\ Prop. 2.8.2 in \cite{sft-egh}]  \label{definition:legendrian-invariance}

Fix $\fk{r} \in \fk{R}(Y, \xi, V)$. Let us also choose the following additional data:
\begin{itemize}
\item a contact form $\l \in \O^1(Y)$ which is adapted to $\fk{r}$ and has the property that all Reeb orbits and $\L$-Reeb chords are non-degenerate
\item a $d\l$-compatible almost-complex structure $J$ on $\xi$ which preserves $TV$.
\end{itemize}
	
Given a class $\beta \in \pi_2(\hat Y; c^+, \G^-_\L, \G^-)$, we let 
\eq \mc{M}(c^+, \G^-_{\L}, \G^-; \beta)_J\eeq 
be the moduli space of connected $\hat{J}$-holomorphic curves, modulo $\R$-translation representing the class $\beta$. (Here we follow the notation of \Cref{subsection:homotopy-classes-maps}, where $c^+$ is a Reeb chord of $\L \sub (Y, \l)$, $\G^-_\L= \{c^-_1,\dots, c^-_\s\}$ is an ordered collections of (not necessarily distinct) Reeb chords, and $\G^-$ is a collection of Reeb orbits.) Since $V \sub (Y-\L, \l)$ is a strong contact submanifold, a straightforward extension of Siefring's intersection theory defines an intersection number $\hat{V}* \beta \in \Z$. 
	
Let us now consider the semi-simple ring 
\eq \mc{R}= \oplus_{i=1}^k \Q[U],\eeq 
and let $e_1,\dots,e_k$ be the idempotents corresponding to the unit in each summand. 
	
Let $CL_\bullet(Y, \xi, V, \L; \fk{r})_\l$ be the free $\mc{R}$ algebra generated by (good) Reeb orbits of $(Y, \a)$ and Reeb chords of $\L$, subject to the following relations:
\begin{itemize}
\item $\g_1\g_2= (-1)^{|\g_1||\g_2|} \g_2\g_1$ for Reeb orbits $a,b$,
\item If $c_{ij}$ is a Reeb chord from $\L_i$ to $\L_j$, then $e_k c_{ij} e_l= \delta_{jk} c_{ij} \delta_{il}$. 
\end{itemize}

We assume that there exists a suitable virtual perturbation framework compatible with \cite{pardon}, so that we can define a differential $d_J$ (squaring to zero) on generators as follows: 
\begin{itemize}
\item for a Reeb chord $c^+$, we let 
\eq d_J(c^+)=  \sum_{} \frac{1}{|\op{Aut}|} \# \ov{\mc{M}}(c^+, \G^-_{\L}, \G^-; \beta)_J U^{\hat{V} * \beta} c^-_1\dots c^-_{\s} \g_1\dots \g_s, \eeq where the sum is over choices of $\b \in \pi_2(\hat Y; c^+, \G^-_\L, \G^-)$, for all possible choices of $\G^-_{\L}, \G^-$;	
\item for a Reeb orbit $\g$, we let $d_J(\g)$ be the usual deformed contact homology differential, as described in \Cref{subsection:mainconstruction}.
\end{itemize}

We assume that $(CL_\bullet(Y, \xi, V, \L; \fk{r})_{\l}, d_J)$ is independent of $\l, J$ up to canonical isomorphism in $\Q[U]$-\bf{hdga}. We denote the resulting object by 
\eq \mc{L}(Y, \xi, V, \L; \fk{r}) \eeq 
and we let $CH_\bullet(Y, \xi, V, \L; \fk{r})$ be its homology. We assume that $\mc{L}(Y, \xi, V, \L; \fk{r})$ satisfies the limited functoriality described in \Cref{corollary:cobordism-map-legendrian}. 
\end{definition-assumption-star}

\begin{definition-assumption-star}[cf.\ Prop. 2.8.2 in \cite{sft-egh}] \label{definition:legendrian-invariance-reduced}

Carrying over the hypotheses and notation from \Cref{definition:legendrian-invariance}, let us consider the semi-simple ring 
\eq \tilde{\mc{R}}= \oplus_{i=1}^k \Q,\eeq 
where we again let $e_1,\dots,e_k$ be the idempotents corresponding to the unit in each summand. 

We let $\widetilde{CL}_\bullet(Y, \xi, V, \L; \fk{r})_{\l}$ be the free $\tilde{\mc{R}}$ algebra generated by (good) Reeb orbits of $(Y, \a)$ \emph{which are not contained in $V$} and $\L$ Reeb chords, subject to the following relations:
\begin{itemize}
	\item $\g_1\g_2= (-1)^{|\g_1||\g_2|} \g_2\g_1$ for Reeb orbits $a,b$,
	\item If $c_{ij}$ is a Reeb chord from $\L_i$ to $\L_j$, then $e_k c_{ij} e_l= \delta_{jk} c_{ij} \delta_{il}$. 
\end{itemize}
	
This algebra is again $\Z/2$-graded in general, and $\Z$-graded when the canonical bundle is trivialized. 	We assume again that there exists a suitable virtual perturbation framework so that one can define a differential $\tilde{d}_L$ (squaring to zero) on generators as follows:

\begin{itemize}
\item for a Reeb chord $c^+$, we let 
\eq d_J(c^+)=  \sum_{} \frac{1}{|\op{Aut}|} \# \ov{\mc{M}}(c^+, \G^-_{\L}, \G^-; \beta)_J \delta(\hat{V} * \beta) c^-_1\dots c^-_{\s} \g_1\dots \g_s, \eeq 
where $\delta: \R \to \{0, 1\}$ satisfies $\delta(0)=1$ and $\delta(s)=0$ for $s \neq 0$ and the sum is over all possible choices of homotopy classes as in \Cref{definition:legendrian-invariance}.
\item for a Reeb orbit $\g$, we let $\tilde{d}_J(\g)$ be the reduced contact homology differential associated to the twisted moduli counts $\#_{\tilde{\psi}}\mc{M}$, which only counts curves disjoint from $\hat{V}$
\end{itemize}

We assume that $(\widetilde{CL}_\bullet(Y, \xi, V, \L; \fk{r})_{\l}, \tilde{d}_J)$ is independent of $\l, J$ up to canonical isomorphism in $\Q$-\bf{hdga}. We denote the resulting object by 
\eq \widetilde{\mc{L}}(Y, \xi, V, \L; \fk{r}) \eeq 
and we let $CH_\bullet(Y, \xi, V, \L; \fk{r})$ be its homology.  
We also assume that $\widetilde{\mc{L}}(Y, \xi, V, \L; \fk{r})$  satisfies the limited functoriality described in \Cref{corollary:cobordism-map-legendrian}. 
\end{definition-assumption-star}

\begin{proposition-star}[cf.\ \Cref{corollary:cobordism-map}] \label{corollary:cobordism-map-legendrian}
Let $(Y^{\pm}, \xi^{\pm}, V^{\pm})$ be TN contact pairs and choose elements $\fk{r}^{\pm}= (\a^{\pm}, \tau^{\pm}, r^{\pm}) \in \fk{R}(Y^{\pm}, \xi^{\pm}, V^{\pm})$. Consider an exact relative symplectic cobordism $(\hat{X}, \hat \l, H)$ with positive end $(Y^{+}, \xi^{+}, V^{+})$ and negative end $(Y^{-}, \xi^{-}, V^{-})$, and suppose that $\tau^+, \tau^-$ extend to a global trivialization of the normal bundle of $H$. 

Suppose that $L \sub (\hat{X}, \hat \l, H)$ is a cylindrical Lagrangian submanifold which is disjoint from $H$, with ends $\L^{\pm} \sub (Y^{\pm}- V^\pm, \xi^{\pm})$. 

If $r^+ \geq e^{\mc{E}((H, \hat{\l}_H)^{\a^+}_{\a^-})} r^-$, then there is an induced map 
\eq \Phi(\hat{X}, \hat \l, H, L): \mc{L}(Y^+, \xi^+, V^+; \fk{r}^+) \to \mc{L}(Y^-, \xi^-, V^-; \fk{r}^-). \eeq

The analogous statement holds for the reduced invariants $\widetilde{\mc{L}}(-)$. 
\end{proposition-star}

\begin{definition-star} \label{definition:augmentated-cedga}
Let $\e: \mc{A}(Y, \xi, V; \fk{r}) \to \Q[U]$ be an augmentation. Then we let 
\eq \label{equation:tensor-augmented-cedga} \mc{L}^{\e}(Y, \xi, V, \L; \fk{r}):= \mc{L}(Y, \xi, V, \L; \fk{r}) \otimes_{\mc{A}(Y, \xi, V; \fk{r})} \Q[U], \eeq 
with differential $d_L \otimes 1$. The structure maps implicit in forming the tensor product \eqref{equation:tensor-augmented-cedga} are, respectively, furnished by the inclusion $\mc{A}(Y, \xi, V; \fk{r}) \sub \mc{L}(Y, \xi, V, \L; \fk{r})$ and the augmentation $\e: \mc{A}(Y, \xi, V; \fk{r}) \to \Q[U]$. The resulting tensor product is naturally also a differential graded $\Q[U]$ algebra.
	
We similarly define 
\eq \widetilde{\mc{L}}^{\tilde{\e}}(Y, \xi, V, \L; \fk{r}):= \widetilde{\mc{L}}(Y, \xi, V, \L; \fk{r}) \otimes_{\widetilde{\mc{A}}(Y, \xi, V; \fk{r})} \Q,\eeq 
(using the maps $\widetilde{\mc{A}}(Y, \xi, V; \fk{r}) \sub \widetilde{\mc{L}}(Y, \xi, V, \L; \fk{r})$ and $\tilde{\e}: \widetilde{\mc{A}}(Y, \xi, V; \fk{r}) \to \Q$) which is naturally a differential graded $\Q$-algebra.
\end{definition-star}

\rmk
The algebra $\mc{L}^{\e}(Y, \xi, V, \L; \fk{r})$ is the twisted analog of the Legendrian homology dg algebra (or Chekanov-Eliashberg dg algebra) described in \cite[Sec.\ 4.1]{bee}. 
\ermk
	
We now discuss gradings on the above Legendrian invariants. 

\begin{definition-star} \label{definition:legendrian-homological-grading}

Let $(Y, \xi, V)$ be a TN contact pair and choose $\fk{r} \in \fk{R}(Y, \xi, V)$. Let $\L \sub (Y-V, \xi)$ be a Legendrian submanifold. Suppose that $H_1(Y; \Z)=0$ and that $c_1(\xi)=0$. Then the Legendrian homological $\Z/2$-grading of $\mc{L}(Y, \xi, V, \L; \fk{r})$ (resp. $\widetilde{\mc{L}}(Y, \xi, V, \L; \fk{r})$) lifts to a canonical $\Z$-grading given on orbits by \eqref{equation:orbits-grading} and given on chords by
\eq \label{equation:chords-grading} |c| = \op{CZ}^{+, \tau} (c)-1, \eeq 
which is well-defined due to our topological assumptions. 
	
The invariants $\mc{L}^{\e}(Y, \xi, V, \L; \fk{r}), \ov{HC}(\mc{L}^{\e}(Y, \xi, V, \L; \fk{r})), \widetilde{\mc{L}}^{\e}(Y, \xi, V, \L; \fk{r})$ and $\ov{HC}(\widetilde{\mc{L}}^{\e}(Y, \xi, V, \L; \fk{r}))$ inherit a $\Z$-grading which we also refer to as the homological grading. 
\end{definition-star}

\begin{lemma-star} \label{lemma:homological-grading-cobordism-legendrian}
With the notation of \Cref{corollary:cobordism-map-legendrian}, suppose that $H_1(Y^{\pm}; \Z)= 0$ and that $w_2(L)=c_1(\xi^{\pm})=c_1(TX)=0$. Then the cobordism maps described in \Cref{corollary:cobordism-map-legendrian} preserve the Legendrian homological $\Z$-grading.
\qed
\end{lemma-star} 

As in \Cref{subsection:bigradings-obd}, there is also a linking number $\Z$-grading on the reduced Legendrian invariants under certain topological assumptions.

\begin{definition-star} \label{definition:bigradings-legendrian}
Let $(Y, \xi, V)$ be a TN contact pair and choose $\fk{r} \in \fk{R}(Y, \xi, V)$ Let $\L \sub (Y, \xi)$ be a Legendrian submanifold. Suppose that $H_1(Y;\Z)=H_2(Y; \Z)= \pi_0(\L)=\pi_1(\L) = 0$.  Then the \emph{linking number grading} $| \cdot |_{\op{link}}$ on $\widetilde{\mc{L}}(Y, \xi, V, \L; \fk{r})$ is given on Reeb chords by
\eq |c|_{\op{link}} = \op{link}_V(c; \L).\eeq
It is given on Reeb orbits by \eqref{equation:linking-orbits}. The grading is extended to arbitrary words by defining the grading of word to be the sum of the gradings of its letters. One can verify using \Cref{lemma:cobordism-linking-legendrian} that this grading is well-defined.
	
We let 
\eq \label{equation:bigrading-legendrian} \widetilde{\mc{L}}_{\bu,\bu}(Y, \xi, V, \L; \fk{r}) \eeq be the bigraded differential $\Q$-algebra of bidegree $(-1,0)$, where
\begin{itemize}
	\item the first bullet refers to the (Legendrian) homological $\Z$-grading (which is well-defined in view of our topological assumptions and the universal coefficients theorem, see \Cref{definition:legendrian-homological-grading});
	\item the second bullet refers to the (Legendrian) linking number grading. 
\end{itemize}
\end{definition-star}

We also have the following lemma which follows from \Cref{lemma:homological-grading-cobordism-legendrian} and \Cref{lemma:cobordism-linking-legendrian}.

\begin{lemma-star} \label{lemma:linkin-grading-cobordisms-leg}
With the notation of \Cref{corollary:cobordism-map-legendrian}, suppose that $H_1(Y^{\pm}; \Z)=H_2(Y^{\pm}; \Z)=H_2(X, Y^+; \Z)=0$ and that $\pi_0(\L^\pm)=\pi_1(\L^\pm)=0$. Then the cobordism maps described in \Cref{corollary:cobordism-map-legendrian} preserve the linking number $\Z$-grading. In case we also have that $w_2(\L^{\pm})=c_1(\xi^{\pm})=c_1(TX)=0$, then the cobordism maps preserve the $(\Z \tms \Z)$-bigrading \eqref{equation:bigrading-legendrian}. 
\qed
\end{lemma-star}

\begin{corollary-star} \label{corollary:bigradings-legendrian-cyclic}
Consider a TN contact pair $(Y, \xi, V)$ and an element $\fk{r} \in \fk{R}(Y, \xi, V)$. Let $\L \sub (Y, \xi)$ be a Legendrian submanifold. Let $(W, \l, H)$ be a relative filling for $(Y, \xi, V)$ and let $\tilde{\e}: \tilde{\mc{A}}(Y, \xi, V; \fk{r}) \to \Q$ be the induced augmentation. Suppose that $H_1(Y; \Z)=H_2(Y; \Z)= H_2(W,Y; \Z)=0$, that $\pi_0(\L)=\pi_1(\L)=0$ and that $w_2(\L)=c_1(\xi)=c_2(TW)=0$. 
	
Then 
\eq \widetilde{\mc{L}}_{\bu,\bu}^{\tilde{\e}}(Y, \xi, V; \fk{r}) \eeq
inherits the structure of a $(\Z \tms \Z)$-bigraded $\Q$-algebra with differential of bidegree $(-1,0)$.  Moreover,  
\eq \ov{HC}_{\bu,\bu}(\widetilde{\mc{L}}^{\tilde{\e}}(Y, \xi, V; \fk{r})) \eeq 
inherits the structures of a $(\Z \tms \Z)$-bigraded $\Q$-vector space.
\end{corollary-star}

\pf
According to \Cref{lemma:cobordism-linking} and our topological hypotheses, the augmentation $\tilde{\e}$ preserves the linking number. The first claim follows. For the second claim, note that both the homological grading and linking number grading are preserved by the cyclic permutation operator $\tau$, and hence pass to reduced cyclic homology (see \Cref{definition:reduced-cyclic}). 
\epf

\subsection{The effect of Legendrian surgery}

The familiar procedure of attaching a handle in differential topology can be performed in the symplectic category. There are various essentially equivalent approaches to doing this in the literature. For concreteness, we exclusively follow in this paper the construction described in \cite[Sec.\ 3.1]{van-koert} which we now summarize. 

\begin{construction}[Attaching a handle] \label{construction:handle-attaching}
Let $(X_0^{2n}, \l_0)$ be a Liouville cobordism with positive boundary $(Y_0^{2n-1}, \xi_0= \ker(\l_0|_{Y_0}))$. Let $\Lambda \sub (Y_0-V, \xi_0)$ be an isotropic sphere with trivialized conformal symplectic normal bundle (the latter condition is vacuous if $\L$ is a Legendrian). Choose an arbitrary open neighborhood $\mc{U}$ of $\L$ which we refer to as the \emph{attaching region}.

We may now glue a model handle $H$ along $Y_0$ inside $\mc{U}$, following the detailed construction given in \cite[Sec.\ 3.1]{van-koert}. The gluing is carried out by identifying the Liouville flow near $\L$ with the flow on $H$. We note that this gluing procedure involves some auxiliary choices which we do not state here. 

The outcome of the procedure (for any of the above auxiliary choices) is a Liouville cobordism $(X, \l)$ with positive boundary $(Y, \xi= \ker(\l|_Y))$. We say that this domain is obtained from $(X_0, \l_0)$ by \emph{attaching a handle along $\L$}, or \emph{Legendrian surgery on $\L$}.  As it well-known from differential topology, $Y$ differs from $Y_0$ by surgery along $\L$. 
\end{construction}

In \cite{bee}, Bourgeois, Ekholm and Eliashberg study the effect of handle attachment on various flavors of symplectic and contact homology. In particular, they describe conjectural exact sequences which govern the change in these invariants and describe the moduli spaces of holomorphic curves which should underly the existence of these exact sequences. In Theorem* \ref{theorem-star:reduced-attaching-triangles} below (see also \Cref{remark:other-extend-bee}), we state an analog of the surgery exact sequence for linearized contact homology in \cite[Thm.\ 5.1]{bee}. The proofs sketched in \cite[Sec.\ 6]{bee} also apply \emph{mutatis mutandis} in the present setting, so will not be repeated. As discussed in \Cref{subsection:intro-surgery}, we expect that if \cite[Thm.\ 5.1]{bee} can be made rigorous in the setting of Pardon's VFC package, it should pose no substantial additional difficulties to also establish Theorem* \ref{theorem-star:reduced-attaching-triangles}. 

To set the stage for Theorem* \ref{theorem-star:reduced-attaching-triangles}, let $(Y_0^{2n-1}, \xi_0)$ be a contact manifold and let $V \sub (Y_0, \xi_0)$ be a codimension 2 contact submanifold with trivial normal bundle. Let $(X_0, \l_0)$ be a Liouville domain with positive boundary $(Y_0, \xi= \op{ker} \l_0)$ and let $H_0 \sub (X, \l_0)$ be a symplectic submanifold which is preserved set-wise by the Liouville flow near $\d X_0=Y_0$ and such that $\d H=V$. 

Let $\Lambda \sub (Y_0-V, \xi_0)$ be an isotropic sphere with trivialized conformal symplectic normal bundle (the latter condition is vacuous if $\L$ is a Legendrian). Let $(X, \l)$ be the Liouville domain obtained by attaching a Weinstein handle along $\L$ according to \Cref{construction:handle-attaching} and let $(Y, \xi= \op{ker} \l)$ be the positive boundary. We may assume that the attaching region is disjoint from $V \sub Y_0$. By abuse of notation, we therefore view $V$ as a codimension $2$ contact submanifold of $(Y_0, \xi_0)$ and $(Y, \xi)$ and view $H_0$ as a submanifold of $X_0$ and $X$. We also identify $\fk{R}(Y_0, \xi_0, V)= \fk{R}(Y, \xi, V)$. 

We let $(\hat{X}_0, \hat \l_0, H)$ be the completion of $(X_0, \l_0, H_0)$ and let $(\hat{X}, \hat{\l}, H)$ be the completion of $(X, \l, H_0)$ There are a natural (strict) markings $ e_0: \R \tms Y_0 \to \hat{X}_0$ taking $(t, y_0) \mapsto \psi^0_t(y_0)$ and $e: \R \tms Y \to \hat{X}$ taking $(t, y) \mapsto \psi_t(y_0)$, where $\psi^0$ (resp.\ $\psi$) is the Liouville flow in $\hat X_0$ (resp.\ in $\hat X$).  

Finally, in order to have well-defined homological $\Z$-gradings, we assume that $H_1(Y_0;\Z)=H_1(Y;\Z)=0$ and that $c_1(TX_0)=c_1(TX)=0$.



\begin{theorem-assumption-star}[cf.\ Thm.\ 5.1 in \cite{bee}]  \label{theorem-star:reduced-attaching-triangles}
With the above setup and $\fk{r} \in \fk{R}(Y, \xi, V)$, consider the augmentations $\tilde{\e}((\hat X_0, \hat \l_0, H), e_0): \tilde{\mc{A}}(Y, \xi, V) \to \Q$ and $\tilde{\e}((\hat X, \hat \l, H), e): \tilde{\mc{A}}(Y, \xi, V) \to \Q$. If $\L$ is a Legendrian sphere, we have the following exact triangle, where the top horizontal arrow is the natural map induced by an exact relative symplectic cobordism. 
\eq
 \begin{tikzcd}
 \widetilde{CH}^{\tilde{\e}}_{\bullet -(n-3)}(Y,\xi, V; \fk{r}) \arrow{rr} &&  \widetilde{CH}^{\tilde{\e}_0}_{\bullet -(n-3)}(Y_0, \xi_0, V; \fk{r}) \arrow{dl}{[-1]} \\
& \ov{HC}_\bullet(\widetilde{\mc{L}}^{\tilde{\e}_0}(Y_0, \xi_0, V, \L; \fk{r})) \arrow{ul} 
\end{tikzcd}
\eeq

If $\L$ is an isotropic sphere of dimension $k \leq n-2$, then we have that	
\eq
H_* (\op{Cone} (\widetilde{CH}^{\tilde{\e}}_{\bullet-(n-3)}(Y,\xi, V; \fk{r}) \to \widetilde{CH}^{\tilde{\e}_0}_{\bullet-(n-3)}(Y_0, \xi_0, V; \fk{r}) )  = 
\begin{cases} \Q  & \text{if } *= n-k + 2 \N \\
0 &\text{otherwise.}
\end{cases}
\eeq
\end{theorem-assumption-star}

\rmk\label{remark:other-extend-bee}
There is a natural analog of Theorem* \ref{theorem-star:reduced-attaching-triangles} involving the invariants $CH^{\e_0}_{\bullet-(n-3)}(Y_0, \xi_0, V; \fk{r})$, $CH^{\e}_{\bullet-(n-3)}(Y,\xi, V; \fk{r}) $ and $\ov{HC}_\bullet(\mc{L}^{\e_0}(Y_0, \xi_0, V, \L; \fk{r}))$ which we also expect to hold. We do not state it here since we do not use it in this paper.
\ermk

\rmk \label{remark:exact-sequence-bigrading}
With the setup of Theorem* \ref{theorem-star:reduced-attaching-triangles}, let us in addition assume that $H_2(Y_0;\Z)=H_2(X_0, Y_0)=H_2(Y;\Z)=H_2(X, Y; \Z)=0$. Then \Cref{lemma:linking-grading-cobordism} and \Cref{corollary:bigradings-legendrian-cyclic} provide an additional linking number $\Z$-grading on the invariants appearing in the surgery exact sequences. 

The resulting $(\Z \tms \Z)$-bigrading is preserved by the maps in the surgery exact sequence.  Indeed, \Cref{lemma:linking-grading-cobordism} ensures the top horizontal map preserves the linking number $\Z$-grading.  The bottom right map counts holomorphic disks with one positive interior puncture, $k$-negative boundary punctures, and with boundary mapping to $S \L$ (the relevant moduli space is described in \cite[Sec.\ 2.6]{bee}). Hence one can readily verify (cf.\ \Cref{lemma:cobordism-linking-legendrian}) that this map also preserves the linking number grading. Finally, the bottom left map is defined algebraically as the connecting map in the long exact sequence. Since the internal differentials of the relevant chain complexes preserve the linking number grading, this connecting map does too. 
\ermk

\section{Some computations}

\subsection{Vanishing results}

Recall that a contact manifold $(Y^{2n-1}, \xi)$ is said to be \emph{overtwisted} if it contains an overtwisted disk; see \cite[Sec.\ 1]{bem}.  In general, if $(Y, \xi)$ is overtwisted and $C \sub Y$ is a closed subset, then $(Y-C, \xi)$ may not be overtwisted. 

\thm \label{theorem:loosevanishing} 
Suppose that $(Y, \xi, V)$ is a TN contact pair such that $(Y-V, \xi)$ is overtwisted. Given any element $\fk{r} \in \fk{R}(Y, \xi, V)$, we have 
\eq CH_{\bullet}(Y, \xi, V; \fk{r})= \widetilde{CH}_\bullet(Y, \xi, V; \fk{r})= 0.\eeq
\ethm

We collect some definitions which will be useful in proving \Cref{theorem:loosevanishing}.  Let $\alm_{\mathrm{U}(n-1)}(S^{2n-1})$ be the set of almost-contact structures on $S^{2n-1}$ (see \Cref{definition:almost-contact-structure}). It follows by the main theorem of \cite{bem} that $\alm_{\mathrm{U}(n-1)}(S^{2n-1})$ is in canonical correspondence with the set of overtwisted contact structures on the sphere, a fact which will be used implicitly in the proof of \Cref{theorem:loosevanishing}. 

A folklore result in contact topology (see e.g.\ \cite[Sec.\ 6]{c-m-p}) states that for any fixed element $\beta \in \alm_{\mathrm{U}(n-1)}(S^{2n-1})$, the operation of connected sum endows $\alm_{\mathrm{U}(n-1)}(S^{2n-1})$ with a group structure with identity element $\beta$. The isomorphism class of the resulting group is moreover independent of $\beta$. Since these facts are not to our knowledge available in the literature, we have provided careful proofs in the appendix. 

For the remainder of this section, we fix $\beta \in \alm_{\mathrm{U}(n-1)}(S^{2n-1})$ to be the almost-contact structure induced by the standard contact structure on the sphere. Given a pair of contact manifolds $(M_1, \a_1), (M_2, \a_2)$, one can also consider their connected sum $(M_1 \#M_2, \a_1 \# \a_2)$, which is obtained by gluing-in a neck along Darboux balls in $M_1, M_2$. This operation is discussed in \Cref{remark:connected-sum-genuine}. As noted there, the two a priori different notions of a connected sum of (almost-)contact manifolds commute with the forgetful map from contact manifolds to almost-contact manifolds. 

\pf[Proof of \Cref{theorem:loosevanishing}] 
It is enough to prove that the invariants vanish for a particular choice of non-degenerate contact form $\tilde{\a}$ on $Y$ which is adapted to $\fk{r}$. To construct such a form, we follow arguments of Bourgeois and Van Koert in \cite[Sec.\ 6.2]{bvk}. 
	
Using \Cref{construction:contact-form}, we define an auxiliary contact form $\a$ in a small neighborhood $\mc{N}$ of $V$ with the property that $V$ is a strong contact submanifold and that $\a$ is adapted to $\fk{r}$. After possibly shrinking $\mc{N}$, we can assume that $(Y-\mc{N}, \xi)$ is overtwisted. We now extend $\a$ arbitrarily to a globally-defined, non-degenerate contact form on $(Y, \xi)$. (Since $\a_V$ is non-degenerate, \Cref{construction:contact-form} produces a non-degenerate contact form on $\mc{N}$, so it extends unproblematically to a global non-degenerate contact form). 

Choose a Darboux ball $B \sub Y$ whose closure is disjoint from $\mc{N}$. Let $B' \sub B$ be a smaller Darboux ball and let $A=B- \ov{B'}$. Let $\beta_0$ denote the almost-contact structure on $B$ obtained by restricting $\xi$. Let $\alm_{\mathrm{U}(n-1)}(B, A; \beta_0)$ be the set of almost-contact structures on $B$ agreeing with $\beta_0$ near $A$. The group $\alm_{\mathrm{U}(n-1)}(S^{2n-1})$ acts on $\alm_{\mathrm{U}(n-1)}(B, A; \beta_0)$ by connect-summing with an almost-contact sphere along a disk whose closure is disjoint from $A$.

We now appeal to work of Bourgeois and van Koert: in \cite[Sec.\ 2.2]{bvk}, they construct a special contact form $\a_L$ on the sphere $S^{2n-1}$ (this form turns out to be overtwisted by \cite{c-m-p}, although we won't need this). They prove \cite[Sec.\ 2--5]{bvk} that $\a_L$ admits a Reeb orbit $\g$ which bounds a single, transversally cut-out $J$-holomorphic plane, for some suitable $J$ on the symplectization.  

We now form the connected sum of $(S^{2n-1}, \a_L)$ with $(Y, \a)$, where we assume that the gluing happens entirely inside of $B'$. Bourgeois and van Koert (following earlier work of Ustilovsky \cite{ustilovsky}) explain how to perform this connected sum so that the orbit $\g$ survives, and still has the property that it bounds a single transversally cut-out plane (basically, one can suitably adjust the neck to ensure that the plane cannot cross the neck; see \cite[Sec.\ 6.2]{bvk}). 

Finally, we further connect sum with another overtwisted contact sphere $(S^{2n-1}, \a'_L)$ so that $(B, \a \# \a_L \# \a'_L|_B)$ is formally contact isotopic to $(B, \a|_B)$, through a contact isotopy fixed near $A$. Note that we may freely assume that the two connected sums happen in disjoint regions of $B'$, so they do not interfere with each other. Unwinding the definitions, this means that there exists a diffeomorphism $\psi: B \to  B \# S^{2n-1} \# S^{2n-1}$ fixed near the boundary and a formal contact isotopy from $\op{ker} \psi^*( \a \# \a_L \# \a_L')|_B$ to $\xi_B= \op{ker}\a|_B$, which is fixed near $A$.

%
If we extend $\psi$ to a diffeomorphism $Y \to Y \#S^{2n-1} \# S^{2n-1}$ by letting it be the identity outside of $B$, we observe that $\op{ker} \psi^*( \a \# \a_L \# \a_L')$ is formally isotopic to $\xi= \op{ker}\a$. Moreover, these contact structures agree on $Y-B \supset \mc{N}$. Since $(Y- \mc{N}, \xi)$ is overtwisted, it follows from the relative h-principle for overtwisted contact structures (see \cite[Thm.\ 1.2]{bem}) that there is a smooth isotopy $\phi_t$ fixed on $\mc{N}$ so that $\tilde{\a}:= \phi_1^* \psi^* (\a \#\a_L \# \a_L') $ is a contact form for $(Y, \xi)$. By construction, $\tilde{\a}= \a$ on $\mc{N}$, so $\tilde{\a}$ is adapted to $\fk{r}$. Finally, it follows from the above discussion that $CH_{\bullet}(Y, \xi, V; \fk{r})$ vanishes when we compute it using the form $\tilde{\a}$ (since $\g$ bounds a rigid plane). An analogous argument shows that $\widetilde{CH}_{\bullet}(Y, \xi, V; \fk{r})$ vanishes as well.  
\epf

We also state a vanishing result for the deformed Chekanov-Eliashberg dg algebra of certain loose Legendrians. To set the notation, let us now assume that $(Y, \xi,V)$ is an arbitrary TN contact pair and fix $\fk{r} \in \fk{R}(Y, \xi, V)$.

\begin{proposition-star} \label{proposition:loose-acyclic}
Suppose that $\L \sub (Y-V, \xi)$ is a loose Legendrian submanifold. Then $\mc{L}(Y, \xi, V, \L; \fk{r})$ and $\widetilde{\mc{L}}(Y, \xi, V, \L; \fk{r})$ are acyclic. Given augmentations $\e: \mc{A}(Y, \xi, V; \fk{r}) \to \Q[U]$ and $\tilde{\e}: \tilde{\mc{A}}(Y, \xi, V; \fk{r}) \to \Q$, the invariants $\mc{L}^{\e}(Y, \xi, V, \L; \fk{r})$ and $\widetilde{\mc{L}}^{\e}(Y, \xi, V, \L; \fk{r})$ are also acyclic. 
\end{proposition-star}

\pf
The argument is the same as that which shows that the (undeformed) Chekanov-Eliashberg dg algebra of a loose Legendrian is acyclic (see e.g.\ \cite[Sec.\ 5]{murphy}): up to Legendrian isotopy in $Y-V$, we can find a chord $c$ of arbitrarily small action which bounds a single half-disk. This disk can be assumed to stay in a small ball disjoint from $V$ for action reasons. Hence we have $d(c)=1$.  
\epf

\subsection{Nonvanishing results: bindings of open books} \label{subsection:nonvanishing-results}

The following theorem is the main result of this section.

\thm \label{theorem:openbook} 
Consider a TN contact pair $(Y, \xi, V)$. Suppose that $Y$ admits an open book decomposition $(Y, B, \pi)$ which supports the contact structure $\xi$ and realizes $V=B$ as its binding. Viewing $(B, \tau)$ as a framed contact submanifold, where $\tau$ denote the trivialization of $B \sub Y$ induced by the open book decomposition, we have 
\eq \widetilde{CH}_{\bullet}(Y, \xi, B; \fk{r}) \neq 0 \eeq 
for any $\fk{r}=(\a_B, \tau, r) \in \fk{R}(Y, \xi, B)$. 
\ethm 

By combining \Cref{theorem:openbook} with \Cref{corollary:non-zero-non-reduced}, we obtain the following result. 
\cor \label{corollary:openbookvanish}
Under the hypotheses of \Cref{theorem:openbook}, if $r'$ is large enough and we write $\fk{r}'=(\a_B, \tau, r')$, then 
\eq CH_{\bullet}(Y, \xi, B; \fk{r}') \neq 0. \eeq
\ecor

\pf[Proof of \Cref{theorem:openbook}] 
According to \Cref{proposition:girouxstructure}, the open book decomposition $(Y,B, \pi)$ supports a non-degenerate Giroux form $\a$ which is adapted to $\fk{r}$, for any $\fk{r}=(\a_B, \tau, r) \in \fk{R}(Y, \xi, B).$
	
Consider the algebra $\widetilde{CC}_{\bullet}(Y,\xi, B; \fk{r})$ generated by (good) Reeb orbits of $\a$ not contained in $B$. After fixing an almost complex structure $J: \xi \to \xi$ which is compatible with $d \a$ and preserves $\xi|_B$, and a choice of perturbation data $\theta \in \Theta_\I((Y, \xi, B), \a, J)$, we get a differential $d_J= d(\tilde{\psi}_B,J, \theta)$ and the homology of the resulting chain complex is (canonically isomorphic to) $\widetilde{CH}_{\bullet}(Y, \xi, B; \fk{r})$. 

Let us suppose for contradiction that $\widetilde{CH}_{\bullet}(Y, \xi, B; \fk{r})=0$. This means that $1$ is in the image of the differential. By the Leibnitz rule, this implies that there exists some good Reeb orbit $\g : S^1 \to Y$ and a relative homotopy class $\beta \in \pi_2(Y,\g)$ such that the twisted moduli count of planes positively asymptotic to $\g$ in the homotopy class $\beta$ is non-zero. To state this more formally in the language of \Cref{subsection:intersection-buildings}, let $T \in \mc{S}_\I((Y, \xi, B), \a, J)$ be the tree with a single input edge $e$ and a single vertex $v$, where $e$ is decorated with the Reeb orbit $\g$ and $v$ is decorated with the $\beta \in \pi_2(Y,\g)$. Then we have that $\widetilde{\psi}_{ B}(T) \neq 0$. 

In particular, this implies that $\ov{\mc{M}}(T) \neq \emptyset$. Hence there exists $T' \to T$ such that $T'$ is representable by a $J$-holomorphic building. The proof of \Cref{proposition:reduced-elliptic-twisting-mapI} shows that we may assume that $T'$ does not have any edges contained in $B$ (since otherwise we would have $\widetilde{\psi}_{B}(T') = \widetilde{\psi}_{B}(T) =0$). 

It follows by \Cref{proposition:additivity-symplectization} that $T' * \hat{B}= \sum_{v \in V(T')} \beta_v * \hat{B}$, and \Cref{corollary:full-positivity} implies that all the terms on the right-hand side are non-negative. Since $\widetilde{\psi}_{B}(T') = \widetilde{\psi}_{ B}(T) \neq 0$, it follows by definition of the reduced twisting maps that $T' * \hat{V}=0$. Hence $\beta_v * \hat{B}=0$ for all $v \in V(T')$. 

For topological reasons, there exists $\tilde{v} \in V(T')$ such that $\tilde{v}$ has a single incoming edge and no outgoing edges. Hence $\tilde{v}$ is represented by a $J$-holomorphic plane $u$ which is asymptotic to some Reeb orbit $\tilde{\g}$. By positivity of intersection (see \Cref{prop-positivity-intersections}) and the fact that $\beta_{\tilde{v}} * \hat{B}=0$, the image of $u$ is contained in $\R \times (Y \setminus B)$. Thus $\tilde{\g}$ is contractible in $Y \setminus B$, which implies that the composition $\pi \circ \tilde{\g} : S^1 \to Y \setminus B \to S^1$ has degree $0$. This is a contradiction: since $\a$ is a Giroux form, $\pi \circ \tilde{\g}$ must be an immersion (by \Cref{remark:transverse-reeb}) and hence have nonzero degree. 
\epf

We also state a vanishing result for the (reduced) Chekanov-Eliashberg dg algebra introduced in \Cref{subsection:deform-ce-dga}. 

\begin{theorem-star} \label{theorem-star:legendrian-nonvanishing}
Let $(Y, \xi, V)$ be a TN contact pair and let $\L \sub (Y-V, \xi)$ be a Legendrian submanifold. Suppose that $\xi$ supports an open book decomposition $\pi$ with binding $B=V$, such that $\L$ is contained in a single page. Let $\tau$ be the trivialization of $N_{Y/V}$ induced by the open book. Then we have 
\eq \widetilde{\mc{L}}(Y, \xi, V, \L; \fk{r}) \neq 0.\eeq
\end{theorem-star}
\pf
The proof is identical to that of \Cref{theorem:openbook}; namely, one argues that the image of any Reeb orbit or chord under the differential cannot contain a term of degree zero, which immediately implies the claim.
\epf

We note that is was proved by Honda and Huang \cite[Cor.\ 1.3.3]{honda-huang} that any Legendrian $\L$ in a contact manifold $(Y, \xi)$ is contained in the page of some compatible open book decomposition. Hence it follows from \Cref{theorem-star:legendrian-nonvanishing} that every Legendrian is tight in the complement of some codimension $2$ contact submanifold.

\subsection{Explicit computations in open books} \label{subsection:linearized-open-books}

We now perform certain explicit computations in open book decompositions which will be used in applications in the next sections.  We assume throughout this section that $n \geq 4$. This assumption is needed for the purpose of obtaining a $(\Z \tms \Z)$-bigrading on $\widetilde{CH}(-)$ (see \Cref{lemma:first-top-facts} and \Cref{corollary:topological-facts}).

Let us endow $S^{n-1}$ with a Riemannian metric $h$ having the property that all geodesics are non-degenerate. Such metrics, which are typically referred to as ``bumpy" in the literature, are generic in the space of Riemannian metrics (see \cite{abraham} or \cite[3.3.9]{klingenberg}). It can be shown \cite{abraham} that any manifold endowed with a bumpy metric admits a closed geodesic of minimal length. We let $\rho>0$\label{length-shortest-geodesic} be the length of the shortest geodesic of $(S^{n-1}, h)$. 

To set the stage for this section, it will be useful to recall some general facts about coordinate systems. Given a system of local coordinates $(q_1,\dots,q_m)$ on some manifold $M$, the \emph{dual coordinates} $(p_1,\dots,p_m)$ in the fibers of the $T^*M$ are characterized by the property that 
\eq T^*M \ni (\bf{q}, \bf{p})= (q_1,\dots,q_m, p_1,\dots,p_m)= \sum_{i=1}^m p_i dq_i. \eeq 
Unless otherwise indicated, a pair $(\bf{q}, \bf{p})$ refers to a system of local coordinates in the cotangent bundle of a manifold, where $\bf{p}$ is dual to $\bf{q}$.

It will sometimes also be useful to work with Riemannian normal coordinates. Recall that on a Riemannian manifold $(M, g)$, a system of normal coordinates $(x_1,\dots, x_m) $ has the property that for any vector $\bf{a} \in T_xM$, the path $\t \mapsto \bf{a} t$ is a geodesic. If $(q_1,\dots,q_n)$ is a system of Riemannian normal coordinates, then the path $t \mapsto (\g(t), \dot{\g}^{\flat})$ can be written in coordinates $(\bf{q}, \bf{p})$ as 
\eq \label{equation:normal-geodesic} t \mapsto (\bf{a} t, \bf{a}) \in T^*M.\eeq

We now introduce a Liouville manifold which will be studied throughout the remainder of this section. For $a>0$, define
\eq \label{equation:a-filling} (\hat{W}_0, \hat{\l}^a)= (D^2 \tms T^*S^{n-1}, \hat{\l}^a:= \frac{1}{a} s^2 d\t + \l_{\op{std}}), \eeq where we have chosen local coordinates $(s, \t, \bf{q},\bf{p})$. We emphasize that the Liouville structure depends on the parameter $a>0$.

Let $\phi: \hat{W}_0 \to \R$ be the function 
\eq \phi(s, \t, \bf{q}, \bf{p})= s^2+ \|\bf{p}\|^2. \eeq
We consider the Liouville domain 
\eq \label{equation:liouville-domain-y0} (W_0, \l^a)= ( \{ \phi  \leq 1\}, \l^a:= \hat{\l}^a|_{W_0}), \eeq
and its contact-type boundary 
\eq \label{equation:y0} (Y_0, \xi_0) = (\{\phi=1\}, \xi= \op{ker} \l_0),\eeq
 where $\l_0={(\hat{\l}^a)}|_{Y_0}$ is the induced contact form. Consider also the codimension $2$ contact submanifold 
\eq V= \{ \phi=1, s=0 \} \sub (Y_0, \xi_0) \eeq 
and the Legendrian
\eq \label{equation:Lambda} \L:= \{ \phi=1, \t=\op{constant}, s=1, \|p\|=0\}. \eeq
We define $\a:={(\l_0)}|_V$ and let $\tau$ be the trivialization of $N_{Y_0/V}$ which is unique by \Cref{lemma:first-top-facts}. We set 
\eq\label{equation:r-a-dependent} \fk{r}= (\a, \tau, a) \in \fk{R}(Y_0, \xi_0, V).\eeq 
Finally, we let $H= \{0\} \tms T^*S^{n-1} \sub \hat{W}_0$. 

\emph{Observe that $\fk{r}$ depends on our choice of $a>0$.} More generally, the contact form $\l_0={(\hat{\l}^a)}|_{Y_0}$ on $(Y_0, \xi_0)$ obviously depends on $a>0$. The plan is now to study the Reeb dynamics on $(Y_0, \xi_0)$ with respect to this contact form. By taking $a \gg 0$ large enough, we will be able to obtain a sufficiently good understanding of the Reeb dynamics to compute the invariant $\widetilde{CH}(Y_0, \xi_0, V ; \fk{r})$ in low degrees; see \Cref{proposition:linearized-obd}. 

\lem \label{lemma:first-top-facts}
The manifolds $W_0, Y_0$ have vanishing first and second homology and cohomology with $\Z$ coefficients. In addition, we have $H^1(V;\Z)=0$ and $w_2(\L)=0$. 
\elem

\pf
The first claim is proved in \Cref{corollary:topological-facts}. To compute $H^1(V;\Z)$, note that $V$ is the sphere bundle associated to $T^*S^{n-1}$. Hence, we have a fibration $S^{n-2} \hookrightarrow V \to S^{n-1}$ giving rise to a Gysin sequence
\eq \dots  \to H^k(S^{n-1}; \Z) \to H^k(V;\Z) \to H^{k-(n-2)}(S^{n-1};\Z) \to \dots \eeq
Taking $k=1$ immediately gives the desired result since $n \geq 4$. Finally, note that $\L= S^{n-1}$, which has vanishing homology (with any coefficients) in degrees $1 \leq i \leq n-2$. Hence $w_2(\L)=0$ for $n \geq 4$. 
\epf

Observe that there is a natural marking 
\begin{align}
e_0: \R \tms Y_0 &\to (\hat W_0, \hat \l^a, H)\\
(t, y) &\mapsto \psi_t(y),
\end{align}
where $\psi_{(-)}$ is the Liouville flow associated to $\hat \l^a$. 

This endows $(\hat W_0, \hat \l^a, H)$ with the structure of a (strict) relative exact symplectic cobordism. We thus obtain an augmentation
\eq \tilde{\e}_0: \tilde{\mc{A}}(Y_0, \xi_0, V; \fk{r}) \to \Q. \eeq

It follows from \Cref{lemma:first-top-facts} and the discussion following \Cref{definition:bigrading-lift} that $\tilde{\mc{A}}(Y_0, \xi_0, V; \fk{r})$ and $\tilde{\mc{A}}^{\tilde{\e}_0}(Y_0, \xi_0, V; \fk{r})$ admit a $(\Z \tms \Z)$-bigrading.

We now analyze the structure of $(Y_0,\lambda_0)$ in more detail. First, observe that $(Y_0 - V, \lambda_0|_{Y_0 - V})$ is strictly contactomorphic to
\eq \label{morse-bott-neighborhood}
		(S^1 \times D^*S^{n-1},\alpha_V := \frac{1}{a}(1 - \lVert \bf{p} \rVert^2) d\theta + \lambda_{\mathrm{std}})
\eeq
via the map
\begin{align}
	S^1 \times D^*S^{n-1} &\to Y_0 - V \\
	(\theta, \bf{q}, \bf{p}) &\mapsto (\sqrt{1 - \lVert \bf{p} \rVert^2}, \theta, \bf{q}, \bf{p})
\end{align}
where $D^*S^{n-1} = \{ (\bf{q},\bf{p}) \in T^*S^{n-1} \mid \lVert \bf{p} \rVert < 1 \}$. We let $\mc{N} \subset Y_0$ denote the image of $S^1 \times S^{n-1}$ under this map; equivalently, $\mc{N} = \{ \lVert \bf{p} \rVert = 0 \}$. The complement $(Y_0 - \mc{N}, \lambda_0|_{Y_0 - \mc{N}})$ is strictly contactomorphic to $(B \times U, \alpha_{\mc{N}} := \frac{1}{a}(x\,dy - y\,dx) + \sqrt{1 - x^2 - y^2}\alpha_U)$, where $B \subset \R^2$ denotes the open unit disk and $(U,\alpha_U)$ denotes the unit cotangent bundle of $(S^{n-1},h)$, equipped with the contact form $\alpha_U := \lambda_{\mathrm{std}}$ induced by the canonical Liouville form on $T^*S^{n-1}$. A contactomorphism is given by
\begin{align}
	B \times U &\to Y_0 - \mc{N} \\
	(x,y,\bf{q},\bf{p}) &\mapsto (x,y,\bf{q},\sqrt{1 - x^2 - y^2} \bf{p}).
\end{align}
Observe that the size of the tubular neighborhood $B \tms U$ depends on our choice of $a>0$. 

Our first task is to study the Reeb orbits of $\lambda_0$ which are in the complement of $\mc{N}$. In particular, we wish to show that they are nondegenerate for a generic choice of $a$, and moreover that their Conley-Zehnder indices depend linearly on $a$. This is the content of \Cref{proposition:interior-orbits-cz} and \Cref{corollary:orbits-away-mb} below.

\begin{proposition} \label{proposition:interior-orbits-cz}
	Let $\gamma_U : \R/\Z \to U$ be a Reeb orbit of $\alpha_U$ of period $P_U$. Then
	\begin{enumerate}
		\item the map
			\begin{align}
				\gamma_1 : \R/\Z &\to B \times U \\
				t &\mapsto (0,0,\gamma_U(t))
			\end{align}
			is a Reeb orbit of $\alpha_{\mc{N}}$ of period $P_1 := P_U$;
		\item given any $r_0 \in (0,1)$ and integers $m,n > 0$ such that 
\eq\label{equation:r0-m-n} \frac{a P_U}{4\pi\sqrt{1 - r_0^2}} = \frac{m}{n}, \eeq the map
			\begin{align}
				\gamma_2 : \R/\Z &\to B \times U \\
				t &\mapsto (r_0\cos(2\pi m t),r_0\sin(2\pi m t),\gamma_U(nt))
			\end{align}
			is a Reeb orbit of $\alpha_{\mc{N}}$ of period \[ P_2 := (2 - r_0^2)\frac{2\pi m}{a} = (2 - r_0^2) \frac{n P_U}{2\sqrt{1 - r_0^2}}. \]
	\end{enumerate}
	Every Reeb orbit of $\alpha_{\mc{N}}$ is of the form (1) or (2) for some choice of $\gamma_U$, $r_0$, $m$, $n$.

	If $\alpha_U$ is nondegenerate and $a$ satisfies
	\begin{equation} \label{eq:generic-contact-form}
		a^{-1} \notin \bigcup_{q \in \Q_{> 0}} \frac{1}{4\pi\sqrt{q}} \mathcal{S}(\alpha_U),
	\end{equation}
	where $\mathcal{S}(\alpha_U) \subset \R$ is the action spectrum of $\alpha_U$, then $\alpha_{\mc{N}}$ is nondegenerate. Moreover, given any trivialization $\tau_0$ of $\gamma_U^*\ker(\alpha_U)$, there exist trivializations $\tau_i$ of $\gamma_i^*\ker(\alpha_{\mc{N}})$ ($i = 1,2$) such that
	\eq \CZ^{\tau_1}(\g_1) = 1 + 2\left\lfloor \frac{P_1a}{4\pi} \right\rfloor + \CZ^{\tau_0}(\gamma_U) \eeq
	and
	\eq \CZ^{\tau_2}(\g_2) = 1 + 2\left\lfloor \frac{P_2a}{\pi(2 - r_0^2)^2} \right\rfloor + \CZ^{\tau_0}(\gamma_U^n). \eeq
	If $\tau_0$ extends to a disk spanning $\gamma_U$, then $\tau_i$ extends to a disk spanning $\gamma_i$.
\end{proposition}
\begin{proof}
	The Reeb vector field of $\alpha_{\mc{N}}$ is given by
	\eq
		R_{\alpha_{\mc{N}}} = \frac{1}{2 - x^2 - y^2}\left( a(x \partial_y - y \partial_x) + 2\sqrt{1 - x^2 - y^2} R_U \right)
	\eeq
	where $R_U$ denotes the Reeb vector field of $\alpha_U$ (recall that $a>0$ is a constant fixed above). A simple computation shows that $\g_1$ and $\g_2$ are Reeb orbits with periods as claimed, and that there are no other orbits.

	Note that the contact structure $\xi = \ker(\alpha_{\mc{N}})$ splits as
	\eq
		\xi = \langle e_1, e_2 \rangle \oplus \ker(\alpha_U)
	\eeq
	where $e_1$ and $e_2$ are the vector fields on $B \times U$ defined by
	\begin{align}
		e_1 &= \partial_x + \frac{y}{a\sqrt{1 - x^2 - y^2}} R_U \\
		e_2 &= \partial_y - \frac{x}{a\sqrt{1 - x^2 - y^2}} R_U
	\end{align}
	In particular, given a trivialization $\tau_0$ of $\gamma_U^*\ker(\alpha_U)$, we get trivializations $\tau_1 = \langle \gamma_1^* e_1, \gamma_1^* e_2 \rangle \oplus \tau_0$ and $\tau_2 = \langle \gamma_2^* e_1, \gamma_2^* e_2 \rangle \oplus \tau_0^n$ of $\gamma_1^*\xi$ and $\gamma_2^*\xi$, where $\tau_0^n$ denotes the trivialization of $(\gamma_U^n)^*\ker(\alpha_U)$ induced by $\tau_0$.

	We have
	\begin{align}
		\mathcal{L}_{e_1}R_{\alpha_{\mc{N}}} &= -\partial_x\left( \frac{ay}{2 - x^2 - y^2} \right) e_1 + \partial_x\left( \frac{ax}{2 - x^2 - y^2} \right) e_2 \\
		&= \frac{a}{(2 - x^2 - y^2)^2} \left( -2xy e_1 + (2 + x^2 - y^2) e_2 \right) \\
		\mathcal{L}_{e_2}R_{\alpha_{\mc{N}}} &= \partial_y\left( \frac{ay}{2 - x^2 - y^2} \right) e_1 - \partial_y\left( \frac{ax}{2 - x^2 - y^2} \right) e_2 \\
		&= \frac{a}{(2 - x^2 - y^2)^2} \left( -(2 - x^2 + y^2) e_1 + 2xy e_2 \right)
	\end{align}
	Moreover, for any vector field $X$ on $U$ such that $X \in \ker(\alpha_U)$, we have
	\eq
		\mathcal{L}_{X} R_{\alpha_{\mc{N}}} = \frac{2\sqrt{1 - x^2 - y^2}}{2 - x^2 - y^2} \mathcal{L}_X R_U
	\eeq
	Hence, if $\Psi_i(t) : \xi_{\gamma_i(0)} \to \xi_{\gamma_i(t)}$ denotes the linearized Reeb flow along $\gamma_i$ (viewed as a matrix via the trivialization $\tau_i$), $i = 1,2$, then $\Psi_i'(t) = S_i(t)\Psi_i(t)$ with
	\begin{align}
		S_1(t) &= \frac{aP_U}{2} \begin{pmatrix} 0 & -1 \\ 1 & 0 \end{pmatrix} \oplus P_U S_U(t) \\
		S_2(t) &= \frac{2\pi m}{2 - r_0^2} \begin{pmatrix} -r_0^2 \sin(4\pi m t) & -2 + r_0^2 \cos(4\pi m t) \\ 2 + r_0^2 \cos(4\pi m t) & r_0^2 \sin(4\pi m t) \end{pmatrix} \oplus n P_U S_U(nt)
	\end{align}
	where $S_U(t)$ is the matrix such that the linearized Reeb flow $\Psi_U : (\xi_U)_{\gamma_U(0)} \to (\xi_U)_{\gamma_U(t)}$ of $R_U$ along $\gamma_U$ satisfies $\Psi_U'(t) = P_U S_U(t) \Psi_U(t)$.

	It follows that $\CZ^{\tau_1}(\g_1) = \CZ(\psi_1) + \CZ^{\tau_0}(\gamma_U)$ and $\CZ^{\tau_2}(\g_2) = \CZ(\psi_2) + \CZ^{\tau_0}(\gamma_U^n)$, where $\psi_1$ and $\psi_2$ are paths of $2\times 2$ matrices given by $\psi_i(t) = \exp(P_i(t))$ with
	\begin{align}
		P_1(t) &= t\frac{aP_U}{2} \begin{pmatrix} 0 & -1 \\ 1 & 0 \end{pmatrix} \\
		P_2(t) &= \int_0^t \frac{2\pi m}{2 - r_0^2} \begin{pmatrix} -r_0^2 \sin(4\pi m s) & -2 + r_0^2 \cos(4\pi m s) \\ 2 + r_0^2 \cos(4\pi m s) & r_0^2 \sin(4\pi m s) \end{pmatrix} ds \\
		&= \frac{2\pi m}{2 - r_0^2} \begin{pmatrix} \frac{r_0^2}{4\pi m} (\cos(4\pi m t) - 1) & -2t + \frac{r_0^2}{4\pi m} \sin(4\pi m t) \\ 2t + \frac{r_0^2}{4\pi m} \sin(4\pi m t) & -\frac{r_0^2}{4\pi m} (\cos(4\pi m t) - 1) \end{pmatrix}
	\end{align}
	Note that $P_1(t)$ and $P_2(t)$ are diagonalizable with eigenvalues $\pm 2\pi i \lambda_1(t)$ and $\pm 2\pi i \lambda_2(t)$ respectively, where
	\begin{align}
		\lambda_1(t) &= t\frac{a P_U}{4\pi} \\
		\lambda_2(t) &= \frac{1}{2 - r_0^2}\sqrt{4m^2 t^2 - \frac{r_0^4}{8\pi^2} (1 - \cos(4\pi m t))}
	\end{align}
	It follows that $\ker(\psi_i(t) - \mathrm{Id})$ is either $\R^2$ or $0$ depending on whether $\lambda_i(t)$ is an integer or not. Assumption \eqref{eq:generic-contact-form} implies that $\lambda_i(1)$ is not an integer and hence that $\psi_i(1)$ doesn't have $1$ as an eigenvalue, i.e. $\psi_i$ is nondegenerate. (This is clear for $\l_1(1)$; to check this for $\l_2(1)$, note that it follows from \eqref{equation:r0-m-n} that
	\eq \l_2(1)= \frac{2m}{2-r_0^2} =\frac{2m}{1+  \frac{n^2 a^2 P_U^2}{(4\pi)^2 m^2}}  = \frac{2m^3 (4\pi)^2 }{ (4\pi)^2 m^2 + n^2 a^2 P_U^2}.\eeq
If this expression were integral, then the reciprocal would be rational; hence $\frac{ n^2 a^2 P_U^2}{2m^3 (4\pi)^2}$ would be rational, contradicting Assumption \eqref{eq:generic-contact-form}.)
	
	Since $-J_0 P_i'(t)$ is positive-definite for all $t$, it follows from \cite[Prop.\ 52]{guttcz} that
	\eq
		\CZ(\psi_i) = 1 + 2\#\{ t \in (0,1) \mid \lambda_i(t) \in \Z \}.
	\eeq
	Since $\lambda_i$ is strictly increasing with $\lambda_i(0) = 0$ and $\lambda_i(1) \notin \Z$, the right-hand side is equal to $1 + 2\lfloor \lambda_i(1) \rfloor$. Thus
	\begin{align}
		\CZ(\psi_1) &= 1 + 2\left\lfloor \frac{a P_U}{4\pi} \right\rfloor = 1 + 2\left\lfloor \frac{a P_1}{4\pi} \right\rfloor \\
		\CZ(\psi_2) &= 1 + 2\left\lfloor \frac{2m}{2 - r_0^2} \right\rfloor = 1 + 2 \left\lfloor \frac{a P_2}{\pi(2 - r_0^2)^2} \right\rfloor
	\end{align}
	as desired.
\end{proof}

\cor \label{corollary:orbits-away-mb}
Suppose that $\g$ is a closed Reeb orbits of $(Y_0, \xi= \op{ker} \l_0)$ which is contained in the complement of $\mc{N} \sub Y_0$. Then 
\eq \op{CZ}^{\tau}(\g) > \left\lfloor \frac{a \rho}{\pi} \right\rfloor, \eeq 
where $\tau$ is a trivialization which extends to a spanning disk and $\rho > 0$ is as on page~\pageref{length-shortest-geodesic}.
\ecor

\pf
It is well-known that the Reeb orbits on $U$ correspond bijectively to geodesics on $(S^{n-1},h)$; with our normalization, the action of a closed Reeb orbit equals twice the length of the corresponding unit speed geodesic (see e.g.\ \cite[Sec.\ 1.5]{geiges}). Moreover, according to \cite[Prop.\ 1.7.3]{sft-egh}, given a Reeb orbit $\tilde{\g}$ which corresponds to a geodesic $\g$, we have 
\eq \mu_{M}(\g)= \op{CZ}^{\tau}(\tilde{\g}), \eeq 
where $\mu_{M}$ is the Morse index of the geodesic and $\tau$ extends to a spanning disk (see \Cref{remark:trivialization}). Since the Morse index of a geodesic is non-negative by definition, the corollary follows from \Cref{proposition:interior-orbits-cz}. 
\epf

\rmk \label{remark:trivialization} 
The trivialization considered in \cite[Prop.\ 1.7.3]{sft-egh} is in fact constructed as follows. Choose a spanning disk $\tilde{v}: D^2 \to U \sub T^*S^{n-1}$ for $\tilde{\g}$ and let $v:= \pi \circ \tilde{v}$, where $\pi: T^*S^{n-1} \to S^{n-1}$ is the projection. Let $\{\s^1,\dots, \s^{n-1} \}$ be a trivialization of $v^*TS^{n-1}$. For points $\pi: \tilde{x} \mapsto x$, let $Q_{x; \tilde{x}}: \pi^{-1}(x) \to T_{\tilde{x}}(\pi^{-1}(x))$ be the canonical identification. Now define $\tilde{\s}^i_p= Q_{v(p); \tilde{v}(p)}\s_p$ for $p \in D^2$. Then $\{\tilde{\s}^1,\dots,\tilde{\s}^{n-1} \}$ defines a Lagrangian subbundle of the symplectic vector bundle $(\tilde{v}^*(\xi), d\l_0)$. Hence it induces a unique trivialization of $\tilde{v}^*\xi$, which restricts on the boundary to a trivialization of $\tilde{\g}^*\xi$. 
\ermk

We now turn our attention to the Reeb dynamics near $\mc{N}$. Recall from page~\pageref{morse-bott-neighborhood} that $\mc{N}$ is contained in $(Y_0 - V,\lambda_0)$, which is strictly contactomorphic to $(S^1 \times D^*S^{n-1},\alpha_V)$, where $\alpha_V = \frac{1}{a}(1 - \lVert \bf{p} \rVert^2) d\theta + \lambda_{\mathrm{std}}$.

\lem \label{lemma:y0-reeb}
Let $\bf{q}=(q_1,\dots,q_{n-1})$ be Riemannian normal coordinates in some open set $\mc{U} \sub (S^{n-1},h)$ and let $\bf{p}=(p_1,\dots,p_n)$ be the dual coordinates. The Reeb vector field of $\alpha_V$ is given by
\eq
R_{\alpha_V} = \frac{1}{1 + \lVert p \rVert^2} \left( a\partial_\theta + 2\sum_{i,j} h^{ij} p_i \partial_{q_j} - \sum_{i,j,k} p_i p_j \partial_k h^{ij} \partial_{p_k} \right)
\eeq
on $S^1 \times D^*\mc{U}$. (Here we follow the convention of using superscripts $(h^{ij}) = (h_{ij})^{-1}$ to denote the coefficients of the metric induced by $h$ on $T^*S^{n-1}$.) 
\elem

\pf
A direct computation using the formulas
\begin{align}
\alpha_V &= \frac{1}{a}\left(1 - \sum_{i,j} p_i p_j h^{ij} \right) d\theta + \sum_i p_i\,dq_i \\
d\alpha_V &= -\frac{2}{a} \sum_{i,j} h^{ij} p_i\,dp_j \wedge d\theta - \frac{1}{a} \sum_{i,j,k} p_i p_j \partial_k h^{ij} \,dq_k \wedge d\theta + \sum_i dp_i \wedge dq_i
\end{align}
shows that $\alpha_V(R_{\alpha_V}) = 1$ and $d\alpha_V(R_{\alpha_V},-) = 0$.
\epf

\lem \label{lemma:y0-obd}
Consider the map $\pi: Y_0- V \to S^1$ given by $\pi(s, \t, q, p) = \t$. Then the pair $(V, \pi)$ defines an open book decomposition of $Y$. Moreover, $\l_0$ is a Giroux form for the contact structure $\xi_0= \op{ker} \l_0$. 
\elem
\pf
It is clear that $(V, \pi)$ defines an open book decomposition of $Y$. To verify that $\l_0$ is a Giroux form, observe by \Cref{lemma:y0-reeb} that the Reeb vector field is transverse to the pages of $\pi$. The claim then follows by combining \Cref{definition:girouxform} and \Cref{remark:transverse-reeb}. 
\epf
	 
By \Cref{lemma:y0-reeb}, the map $\g_0: \R/\Z \to S^1 \times D^*\mc{U}$ given by the formula $\g_0(t)= (2\pi t, 0, 0)$ defines a simple Reeb orbit in $Y_0$. Let $\g_0^k$ denote its $k$-fold cover. There is an obvious trivialization $\tau_0$ of $\xi|_{\g_0^k}$ given by 
\eq \label{trivialization-tau0}
	\tau_0= \{\d_{p_1},\dots,\d_{p_{n-1}},\d_{q_1},\dots, \d_{q_{n-1}} \}.
\eeq
Let $\tau$ be the trivialization of $\xi|_{\g_0^k}$ defined as follows: 
\eq \label{trivialization-tau}
\tau = \{\sin(2\pi k t) \d_{q_1} + \cos(2\pi k t) \d_{p_1},\d_{p_2}, \dots \d_{p_n}, \cos(2\pi k t) \d_{q_1} - \sin(2\pi k t) \d_{p_1}, \d_{q_2},\dots, \d_{q_n} \}.
\eeq 
Observe that $\tau$ extends to a disk spanning $\g_0$ in $Y_0$.  

\lem \label{lemma:linearized-flow}
With respect to the trivialization $\tau_0$, the linearized Reeb flow along $\g_0^k$ is given by the matrix
\eq \label{equation:matrix-reeb}
	\begin{pmatrix}
	1 & 0 \\
	2t & 1 
	\end{pmatrix},
\eeq
	where each entry of this matrix should be viewed as an $(n-1) \tms (n-1)$ diagonal matrix.
\elem
\pf
	Note that $\tau_0$ can be extended to a trivialization $\tilde{\tau}_0$ of $\ker(\alpha_V)$ over $S^1 \times D^*\mc{U}$, where
	\eq
		\tilde{\tau}_0 = \left\{ \partial_{p_1}, \dots, \partial_{p_{n-1}}, \partial_{q_1} - \frac{ap_1}{1 - \lVert \bf{p} \rVert^2} \partial_\theta, \dots, \partial_{q_{n-1}} - \frac{ap_{n-1}}{1 - \lVert \bf{p} \rVert^2} \partial_\theta \right\}.
	\eeq
	Using the formula for $R_{\alpha_V}$ given in \Cref{lemma:y0-reeb}, one can easily compute
	\begin{align}
		\mc{L}_{\partial_{p_i}} R_{\alpha_V}\big|_{\bf{p}=0,\bf{q}=0} &= 2\partial_{q_i} \\
		\mc{L}_{\partial_{q_i} - \frac{ap_i}{1 - \lVert \bf{p} \rVert^2} \partial_\theta} R_{\alpha_V}\Big|_{\bf{p}=0,\bf{q}=0} &= 0
	\end{align}
	Hence, the matrix $A(t)$ representing the linearized Reeb flow $\xi_{\g_0^k(0)} \to \xi_{\g_0^k(t)}$ with respect to the trivialization $\tau_0$ is given by
	\eq
		A(t) = \exp\left( t\begin{pmatrix} 0 & 0 \\ 2 & 0 \end{pmatrix}\right) = \begin{pmatrix}
			1 & 0 \\
			2t & 1
		\end{pmatrix}
	\eeq
	where each entry should be interpreted as a multiple of the $(n-1)\times(n-1)$ identity matrix.
\epf

\cor
The Robbin-Salamon index satisfies:  
\eq \mu^{\tau_0}_{RS}(\g_0^k)=(n-1)/2.\eeq  
Hence, 
\eq \label{equation:robbin-salamon-k} \mu_{RS}^{\tau}(\g_0^k)= (n-1)/2 + 2k.\eeq
\ecor

\pf
	The first computation follows from \cite[Prop.\ 54]{guttcz} (there is a sign change due to the fact that the matrix we are considering is the transpose of that considered in \cite[Prop.\ 54]{guttcz}, but the proof is entirely analogous). The second computation follows from the fact (see the proof of Lemma 57 in \cite{guttcz}) that the Robbin-Salamon index satisfies the so-called ``loop property", i.e. given a path of symplectic matrices $\phi: [0,1] \to \op{Sp}(2n, \R)$ with $\phi(0)=\phi(1)=\op{id}$, and given a path $\psi:[0,1] \to  \op{Sp}(2n, \R)$, we have 
	\eq \mu_{RS}(\phi \psi)= \mu_{RS}(\psi) + 2 \mu(\phi), \eeq where $\phi$ is the Maslov index of the path.  
\epf
	
By \Cref{lemma:y0-reeb}, $\mc{N} = \{ \| p \| =0 \}$ is preserved by the Reeb flow and is foliated by Reeb orbits in a Morse-Bott family.

Given $\e>0$ which will be fixed later, let $\mc{U}_\e= \{\|p \| < \e \} \cap Y_0$. This is a neighborhood of $\mc{N}$, which we identify with $S^1 \times D_\epsilon^*S^{n-1}$ via the contactomorphism defined on page~\pageref{morse-bott-neighborhood}. Let $f: \mc{U}_\e \to \R$ be the function corresponding to \label{morse-bott-perturbation}
\begin{align}
	S^1 \times D_\epsilon^*S^{n-1} &\to \R \\
	(\theta,\bf{q},\bf{p}) &\mapsto \rho(\lVert \bf{p} \rVert) g(\bf{q})
\end{align}
under this identification, where $g$ is a perfect Morse function on $S^{n-1}$ and $\rho : \R \to [0,1]$ is a smooth bump function with $\rho(x) = 1$ for $x$ near $0$ and $\rho(x) = 0$ for $x > \epsilon/2$ (cf.\ \cite[Sec.\ 2.2]{bourgeois-thesis}).

\lem \label{lemma:big-action} Fix $P>0$. If $\e$ is small enough, all closed Reeb orbits of $(Y_0, \l_0)$ which are contained in $\mc{U}_\e - \mc{N}$ have action at least $P$. 
\qed
\elem

We now consider a perturbed contact form $\l_{\delta}:= (1+\delta f) \l_0$. Since $f$ is compactly-supported in $\mc{U}_\e$, the form $\l_{\delta}$ can be viewed as a contact form both on $\mc{U}_\e$ and on $Y_0$. 

\lem \label{lemma:two-orbits}
Fix $P>0$. If $\e, \delta$ are small enough, then there are exactly two simple Reeb orbits in $\mc{U}_\e$ with action $<P$. We label them $\g_a$ and $\g_b$, and they correspond respectively to the minimum and maximum of $f$. 
\elem

\pf
 Combine \Cref{lemma:big-action} with the argument of \cite[Lem.\ 2.3]{bourgeois-thesis}. 
\epf

\lem \label{lemma:two-orbits-index}
	Let $P>0$ be as in \Cref{lemma:two-orbits}. After possibly further shrinking $\e, \delta$, we may assume that any Reeb orbit of $(Y_0, \l_{\delta})$ contained in $\mc{U}_\e$ and having Conley-Zehnder index (measured with respect to a trivialization which extends to a spanning disk) less than $P/a$ is a multiple of $\g_a$ or $\g_b$. In addition, we have
	
\eq \label{equation:index-gamma-a}
	\CZ^{\tau}(\g_a^k)= \mu_{RS}^{\tau}(\g_0^k)- (n-1)/2 + \op{ind}_a(\delta f) =  2k,
\eeq
and 
\eq \label{equation:index-gamma-b}
	\CZ^{\tau}(\g_b^k)= \mu_{RS}^{\tau}(\g_0^k)- (n-1)/2 + \op{ind}_b(\delta f) = (n-1) + 2k. 
\eeq
\elem

\pf
First of all, observe by \Cref{lemma:y0-reeb} that the boundary of $\mc{U}_\e$ is preserved by the Reeb flow of $\l_{0}$. It follows that the Reeb flow of $\l_0$ has ``bounded return time", in the terminology of \cite[Def.\ 2.5]{bourgeois-thesis}. 

Next, it follows from \eqref{equation:robbin-salamon-k} that the Robbin-Salamon index of any Reeb orbit $\g$ contained in the Morse-Bott submanifold $\mc{N}=\{\|p\|=0\} \sub Y_0$ satisfies 
\eq \mu_{\op{RS}}(\g)= (n-1)/2 + 2\op{wind}(\g) = (n-1)/2 + 2T_{\g} a,\eeq 
where $P_\g$ is the length of $\g$. It follows that these orbits satisfy ``index positivity" (with constant $2/a$), in the terminology of \cite[Def.\ 2.6]{bourgeois-thesis}. 

The first claim now follows from \cite[Lem.\ 2.7]{bourgeois-thesis}. The index computations follow by combining \eqref{equation:robbin-salamon-k} with \cite[Lem.\ 2.4]{bourgeois-thesis}.
\epf

We now put together the above results. For any integer $N>0$, let us define 
\eq \S^1_N= \{k \in \Z \mid 0< k < N, k\, \op{even} \} \eeq and let 
\eq \S^2_N= \{k \in \Z \mid k< N, k= n-1+2j, j \geq 1 \}. \eeq

\begin{proposition-star} \label{proposition:linearized-obd} Given any $N>0$, there exists $A>0$ so that 
\eq CH_{k-(n-3)}^{U=0}(Y_0, \xi_0, V; \fk{r}) = \widetilde{CH}_{k-(n-3)}(Y_0, \xi_0, V; \fk{r})= \begin{cases} \Q \oplus \Q &\text{if }  k \in \S^1_N \cap \S^2_N, \\
\Q &\text{if }  k \in \S^1_N \cup \S^2_N - (\S^1_N \cap \S^2_N),  \\
0 &\text{if } k \notin \S^1_N \cup \S^2_N, k<N
\end{cases}
\eeq
whenever $a>A$. (Recall from \eqref{equation:r-a-dependent} that $\fk{r}$ depends on $a>0$ and hence on $N>0$.) 
\end{proposition-star}

\pf
According to \Cref{corollary:orbits-away-mb}, we may fix $A>0$ large enough so that all Reeb orbits for $(Y_0, \l_0)$ in the complement of $\mc{N} \sub Y_0$ have index at least $N$. We now choose $\e, \delta$ small enough so that the conclusions of \Cref{lemma:two-orbits-index} hold with $P=N$. Since $f_{\delta}$ is compactly-supported in $\mc{U}_\e$, we find that the only Reeb orbits of $(Y_0, \l_{\delta})$ having index less than $N$ are multiples of $\g_a, \g_b$. 
	
According to \eqref{equation:index-gamma-a} and \eqref{equation:index-gamma-b}, it is now enough to check that the differential vanishes on the set $$\O_N= \{ \g_a^{k_a}, \g_b^{k_b}  \mid \CZ^{\tau}(\g_a^{k_a})<N, \CZ^{\tau}(\g_b^{k_b})<N \}.$$ To see this, observe that for $\g \in \O_N$ we have 
\eq \label{equation:index-winding}\op{CZ}^{\tau}(\g)= 2 \op{wind}(\g), \; \op{mod} (n-1) \eeq
	
Suppose that there exists a homotopy class $\beta$ of curves of index $1$ with $\hat{V} * \beta=0$. Then the linking number of the positive puncture equals the sum of the linking numbers of the negative punctures. Hence, by \eqref{equation:index-winding}, the index of the positive puncture equals the sum of the indices of the negative punctures, $\op{mod} (n-1)$. Since $\beta$ has index $1$, this means that $1=0, \op{mod} (n-1)$. This is a contradiction since $n>2$. 
\epf

\begin{corollary-star}  \label{corollary:linearized-obd} 
Let $N>0$ be as in \Cref{proposition:linearized-obd}. Then for all integers $k<N$ we have 
\eq \widetilde{CH}^{\tilde{\e}_0}_{k-(n-3)}(Y_0, \xi_0, V; \fk{r}) = \widetilde{CH}_{k-(n-3)}(Y_0, \xi_0, V; \fk{r}), \eeq 
where the right-hand side was computed in \Cref{proposition:linearized-obd}.
\qed
\end{corollary-star}

We now turn out attention to computing certain Legendrian invariants. Let $\L \subset (Y_0, \xi_0)$ be defined as above (see \eqref{equation:Lambda}). Recall that the relative symplectic filling $(\hat{W}_0, \hat{\l}_0, H)$ gives an augmentation $\tilde{\e}_0: \tilde{\mc{A}}(Y_0, \xi_0, V; \fk{r}) \to \Q$.

It follows from \Cref{corollary:bigradings-legendrian-cyclic} and \Cref{lemma:first-top-facts} that $\widetilde{\mc{L}}^{\tilde{\e}_0}(Y_0, \xi_0, V, \L; \fk{r})$  is a $(\Z\tms \Z)$-bigraded algebra with differential of bidegree $(-1,0)$, and that $\ov{HC}_{\bu,\bu}(\widetilde{\mc{L}}^{\tilde{\e}}(Y, \xi, V; \fk{r})) $ is a $(\Z \tms \Z)$-bigraded $\Q$-vector space.

We now have the following computation. 

\begin{proposition-star} \label{proposition:legendrian-dga}
Given a positive integer $N  \gg 1$, let 
\eq \oplus_{j \leq N} \widetilde{\mc{L}}^{\tilde{\e}_0}_{\bu, j}(Y_0, \xi_0, V, \L; \fk{r}) \sub \widetilde{\mc{L}}^{\tilde{\e}_0}(Y_0, \xi_0, V, \L; \fk{r}) \eeq 
be the bigraded sub-module of elements of winding number at most $N$. Then this sub-module can be generated by products of total winding number $ \leq N$ of Reeb chords $\{ a_k\}_{k \in \N_+}$ and $\{b_k\}_{k \in \N_+}$, where $|a_k|= 2k-1$ and $|b_k|=n-2+ 2k$. (Note that we do not say anything about the differential).
\end{proposition-star}

\pf	
Since $(Y_0 - V, \lambda_0)$ is strictly contactomorphic to $(S^1 \times D^*S^{n-1},\alpha_V)$, we have that $(Y_0 - V, \lambda_\delta)$ is strictly contactomorphic to $(S^1 \times D^*S^{n-1},\alpha_\delta := (1 + \delta f)\alpha_V)$.
	
Recall that $f$ depends on a positive real parameter $\e >0$ which can be taken to be arbitrarily small. Moreover, the restriction of $f$ to the Legendrian $\Lambda = \{0\} \times S^{n-1}$ is equal to $g$, a Morse function with exactly two critical points: one minimum $a$ and one maximum $b$.
	
Let $c$ denote either $a$ or $b$. As in \cref{lemma:two-orbits-index}, we let $\g_c$ denote the simple Reeb chord (with is also a Reeb orbit) passing through $c$ and let $\g^k_c$ denote its $k$-fold cover. Observe first of all that all Reeb chords for $\L$ are contained in $\mc{U}_{\e}$ -- this follows from the fact that $f$ is compactly supported in $\mc{U}_{\e}$ (see page~\pageref{morse-bott-perturbation}). Hence, as in \Cref{lemma:big-action}, if we assume that $\e>0$ is small enough, then there exists $P>0$ large enough so that all Reeb chords of action greater than $P$ have winding number greater than $N$.  By a routine adaptation of \Cref{lemma:two-orbits} (or rather the proof of \cite[Lem.\ 2.3]{bourgeois-thesis}), one concludes that the only Reeb chords of winding number less than or equal to $N$ are the $\g_a^k$ and $\g_b^k$.
	
	We can assume without loss of generality that there are normal coordinates $\bf{q} = (q_1,\dots,q_{n-1})$ defined in a neighborhood $U_c \subset \Lambda$ of $c$ in which $g$ is given by
	\eq
		g = g(c) + \epsilon \sum_{i = 1}^{n-1} q_i^2,
	\eeq
	where $\epsilon = 1$ if $c = a$ and $\epsilon = -1$ if $c = b$. The Reeb vector field of $\alpha_\delta$ is given by
	\eq
		R_{\alpha_\delta} =
		\frac{1}{1 + \delta f} R_\alpha + \frac{2\epsilon\delta}{(1 + \delta f)^2} \sum_i \left( q_i - \frac{2p_i}{1 + \lVert \bf{p} \rVert^2} \sum_{j,k} h^{jk} p_k q_j \right) \partial_{p_i}
	\eeq
	on $S^1 \times D^*U_c$ for $\lVert \bf{p} \rVert$ sufficiently small (i.e. satisfying $\rho(\lVert \bf{p} \rVert) = 1$).  We will now show that for every $k \ge 1$, the indices of $\g_a^k$ and $\g_b^k$ as Reeb chords are given by
	\begin{align}
		\CZ^+(\g_a^k) &= 2k \\
		\CZ^+(\g_b^k) &= 2k + n - 1
	\end{align}
	Hence, setting $a_k = \g_a^k$ and $b_k = \g_b^k$, we have $\lvert a_k \rvert = \CZ^+(a_k) - 1 = 2k - 1$ and $\lvert b_k \rvert = \CZ^+(b_k) - 1 = 2k + n - 2$, as desired. 

	To compute $\CZ^+(\g_c^k)$, we start by computing the linearized Reeb flow along $\g_c^k$ with respect to the trivialization $\tau_0$ (see \eqref{trivialization-tau0}). We proceed as in \Cref{lemma:linearized-flow}: we have
	\begin{align}
		\mc{L}_{\partial_{p_i}} R_{\alpha_\delta}\big|_{\bf{p}=0,\bf{q}=0} &= \frac{2}{1 + \delta g(c)} \partial_{q_i} \\
		\mc{L}_{\partial_{q_i} - \frac{ap_i}{1 - \lVert \bf{p} \rVert^2} \partial_\theta} R_{\alpha_\delta}\Big|_{\bf{p}=0,\bf{q}=0} &= \epsilon \frac{2\delta}{(1+\delta g(c))^2} \partial_{p_i}
	\end{align}
	Hence, the matrix $A(t)$ representing the linearized Reeb flow $\xi_{\g_c^k(0)} \to \xi_{\g_c^k(t)}$ with respect to the trivialization $\tau_0$ satisfies $A'(t) = S A(t)$ with
	\eq
		S = \begin{pmatrix} 0 & \epsilon \frac{2\delta}{(1 + \delta g(c))^2} \\ \frac{2}{1 + \delta g(c)} & 0 \end{pmatrix},
	\eeq
	where each entry should be interpreted as a multiple of the $(n-1)\times(n-1)$ identity matrix. Setting $\mu = \frac{2\delta}{(1 + \delta g(c))^2}$ and $\nu = \frac{2}{1 + \delta g(c)}$ for ease of notation, it follows that
	\eq
		A(t) = \exp(tS) =
		\begin{cases}
			\begin{pmatrix} \cosh(t\sqrt{\mu\nu}) & \sqrt{\mu/\nu} \sinh(t\sqrt{\mu\nu}) \\ \sqrt{\nu/\mu} \sinh(t\sqrt{\mu\nu}) & \cosh(t\sqrt{\mu\nu}) \end{pmatrix} & \text{if $\epsilon = 1$} \\
			\begin{pmatrix} \cos(t\sqrt{\mu\nu}) & -\mu/\nu \sin(t\sqrt{\mu\nu}) \\ \nu/\mu \sin(t\sqrt{\mu\nu}) & \cos(t\sqrt{\mu\nu}) \end{pmatrix} & \text{if $\epsilon = -1$}
		\end{cases}
	\eeq
	(note that $\mu, \nu > 0$ if $\delta$ is sufficiently small).

	Let $L(t) \subset \xi_{\g_c^k(t)}$ be the path of Lagrangian subspaces obtained by applying the linearized Reeb flow to the tangent space $T_c\Lambda \subset \xi_c$ and let $\tilde{L}(t)$ be the loop obtained by closing up $L(t)$ by a positive rotation. Since $T_c\Lambda$ is represented by $\begin{pmatrix} 0 \\ I_{n-1} \end{pmatrix}$ in the trivialization $\tau_0$, $L(t)$ is represented by $A(t)\begin{pmatrix} 0 \\ I_{n-1} \end{pmatrix}$. In the two-dimensional case (i.e. $n - 1 = 1$), one can easily deduce (e.g.\ using the standard properties of the Maslov index stated in \cite[Thm.\ 2.3.7]{mcduff-sal-intro}) that
	\eq
		\mu^{\tau_0}(\tilde{L}(t)) =
		\begin{cases}
			0 & \text{if $\epsilon = 1$} \\
			1 & \text{if $\epsilon = -1$}
		\end{cases}
	\eeq
	In general, $L(t)$ splits as a direct sum of $n - 1$ copies of the two-dimensional case, so the additivity property of the Maslov index \cite[Thm.\ 2.3.7]{mcduff-sal-intro} implies that
	\eq
		\mu^{\tau_0}(\tilde{L}(t)) =
		\begin{cases}
			0 & \text{if $\epsilon = 1$} \\
			n - 1 & \text{if $\epsilon = -1$}
		\end{cases}
	\eeq
	
		According to \Cref{definition:reeb-chord-index} and the definition of the Maslov index \cite[Thm.\ 2.3.7]{mcduff-sal-intro}, we have $\CZ^+(\g_c^k) =\mu^{\tau}(\L^{n-1}_{\C} \tilde{\L})=\mu^{\tau}(\tilde{L}(t))$, where $\tau$ is a trivialization of the contact structure along $\g_c^k$ which extends to a spanning disk. For example, we can take $\tau$ to be the trivialization defined in equation \eqref{trivialization-tau}. The difference $\mu^{\tau}(\tilde{L}(t)) - \mu^{\tau_0}(\tilde{L}(t))$ is equal to twice the Maslov index of the loop of symplectic matrices relating $\tau$ and $\tau_0$, i.e.
	\eq
		\mu^{\tau}(\tilde{L}(t)) - \mu^{\tau_0}(\tilde{L}(t)) = 2\mu\begin{pmatrix} \cos(2\pi k t) & -\sin(2\pi k t) \\ \sin(2\pi k t) & \cos(2\pi k t) \end{pmatrix} = 2k.
	\eeq
	It follows that
	\begin{align}
		\CZ^+(\g_a^k) &= 2k \\
		\CZ^+(\g_b^k) &= 2k + n - 1
	\end{align}
	as desired.
\end{proof}
	
It will be useful to record the following consequence of the above computation.

\begin{corollary-star} \label{corollary:leg-wind-computation}
Suppose that $n \geq 4$ is \emph{even}. Then we have
\eq \op{rk} \ov{HC}_{2n, 2}(\widetilde{\mc{L}}^{\tilde{\e}_0}(Y_0, \xi_0, V, \L; \fk{r}))=1. \eeq
\end{corollary-star}
\pf
Indeed, note that the generators described in \Cref{proposition:legendrian-dga} satisfy $\op{link}(a_k)= \op{link}(b_k)=k$.  It thus follows that 
\eq \ov{CC}_{2n-1,2}(\widetilde{\mc{L}}^{\e}(Y_0, \xi_0, V; \fk{r}))=\ov{CC}_{2n+1,2}(\widetilde{\mc{L}}^{\e}(Y_0, \xi_0, V; \fk{r}))=0. \eeq 
On the other hand, $\ov{CC}_{2n,2}(\widetilde{\mc{L}}^{\e}(Y_0, \xi_0, V; \fk{r}))$ is generated by the word $b_1b_1$. 
\epf

\section{Applications to contact topology}

\subsection{Contact and Legendrian embeddings}

We begin by introducing some standard definitions in the theory of contact and Legendrian embeddings. 

\defi \label{definition:almost-contact-structure}
Given a smooth manifold $Y^{2n-1}$, a \emph{formal contact structure} (or \emph{almost-contact structure}) is the data of a pair $(\eta, \o)$, where $\eta \sub TY$ is a codimension $1$ distribution and $\o \in \O^2(Y)$ is a $2$-form whose restriction to $\eta$ is non-degenerate. A formal contact structure is said to be \emph{genuine} if it is induced by a contact structure. 
	
If $Y^{2n-1}$ is orientable, then a formal contact structure is the same thing as a lift of the classifying map $Y \to BSO(2n+1)$ to a map $Y \to B(U(n) \tms \op{id}) = BU(n)$.
\edefi

\defi[see Def.\ 2.2 in \cite{casals-etnyre}]
Let $(Y^{2n-1}, \xi= \op{ker} \a)$ be a contact manifold. Given a formal contact manifold $(V^{2m-1}, \eta, \o)$ where $1 \leq m \leq n-1$, a formal (iso)contact embedding is a pair $(f, F_s)$ where

\begin{itemize}
\item $F_s$ is a fiberwise injective bundle map $TV \to TY$ defined for $s \in [0,1]$,
\item $f: V \to Y$ is a smooth map and $df= F_0$,
\item $F_1$ defines a fiberwise conformally symplectic map $(\eta, \o) \to (\xi, d\a)$. 
\end{itemize}
	
Observe that the above properties are independent of the choice of contact form $\a$. 
\edefi

Two formal contact embeddings $i_0, i_1: (V, \zeta, \o) \to (Y, \xi)$ are said to be formally isotopic if they can be connected by a family 
$\{i_t\}_{t \in [0,1]}$ of formal contact embeddings. 

A (genuine) contact embedding $(V, \zeta) \to (Y, \xi)$ is simply a smooth embedding $\phi: V \to Y$ such that $\phi_*(\zeta)= \xi_{\phi(V)}$. In particular, 
every contact embedding induces a formal contact embedding by taking $F_s=F_0=df$.

\defi[see Def.\ 2.1 in \cite{casals-etnyre}]
Let $(Y^{2n-1}, \xi)$ be a contact manifold. Given a smooth $n$-dimensional manifold $\L$, a formal Legendrian embedding is a pair $(f, F_s)$ where

\begin{itemize}
\item $F_s$ is a fiberwise injective bundle map $TV \to TY$ defined for $s \in [0,1]$,
\item $df= F_0$
\item $\op{im} (F_1) \sub \xi$. 
\end{itemize}
\edefi

Two formal Legendrian embeddings are said to be formally isotopic if they can be connected by a family of Legendrian embeddings. A (genuine) Legendrian embedding $\L \to (Y, \xi)$ is a smooth embedding $\phi: \L \to Y$ such that $d\phi(T\L) \sub \xi \sub TY$. In particular, a Legendrian embedding canonically induces a formal Legendrian embedding. 

We now review some foundational facts about loose Legendrians.  Recall that a Legendrian $\L$ in a (possibly non-compact) contact manifold $(Y, \xi)$ of dimension at last five is defined to be loose if it admits a loose chart. For concreteness, we adopt as our definition of a loose chart the one given in \cite[Sec.\ 7.7]{cieliebak-eliashberg}. 

Loose Legendrians satisfy the following h-principle due to Murphy \cite[Thm.\ 1.2]{murphy}: given a pair of loose Legendrian embeddings $f_0, f_1: \L \to (Y, \xi)$ which are formally isotopic, then $f_0, f_1$ are genuinely isotopic, i.e. isotopic through Legendrian embeddings.

Given an arbitrary Legendrian submanifold $\L_0$ in a contact manifold $(Y_0, \xi_0)$ of dimension at least five, one can perform a local modification called \emph{stabilization} which makes $\L_0$ loose without changing the formal isotopy class of the tautological embedding $\L_0 \xrightarrow{\op{id}} \L_0$. This modification can be realized in multiple essentially equivalent ways. In this paper, we will take as our definition of stabilization any construction which satisfies the properties stated in the following lemma. 

\lem \label{lemma:legendrian-stabilization}
Given a Legendrian submanifold $\L \sub (Y_0, \xi_0)$ and an open set $U \sub Y_0$ such that $U \cap \L_0$ is nonempty, there exists a Lagrangian embedding $f_1: \L_0 \to Y$ which is formally isotopic to the tautological embedding $\L_0 \xrightarrow{\op{id}} \L_0$ via a family of formal Legendrian embeddings $\{(f_t, F^t_s)\}_{t \in [0,1]}$ which are independent of $t$ on $\L_0 \cap (Y- U)$. We put $\L:= f_1(\L_0)$ and say that $\L$ is the \emph{stabilization} of $\L_0$ inside $U$. 
\elem	

\pf
To construct $\L$, we follow the procedure described in \cite[Sec.\ 7.4]{cieliebak-eliashberg}. As the reader may verify, this construction can be assumed to happen entirely inside a suitably chosen Darboux chart $\mc{U} \sub U$.  In addition, the construction depends on the choice of a function $f$; using that $Y$ has dimension at least five, we may (and do) assume that $\chi(\{f \geq 1 \} ) =0$. To construct the formal isotopy, we simply follow the proof of \cite[Prop.\ 7.23]{cieliebak-eliashberg} (using the assumption that $\chi(\{f \geq 1 \} ) =0$). The argument there is entirely local, so that the isotopy can be assumed to be fixed outside of $\mc{U}$ (and in particular outside of $U$).  
\epf

\subsection{Embeddings into overtwisted contact manifolds}

We begin with the following proposition.

\prop \label{proposition:hploosecontact}
Suppose that $(Y, \xi)$ is an overtwisted contact manifold and let 
\eq i: (V, \z) \to (Y, \xi) \eeq 
be a formal contact embedding. Then there exists an open subset $\O \sub Y$ such that $Y- \ov{\O}$ is overtwisted and a \emph{genuine} contact embedding 
\eq j: (V,\z) \to \O \sub (Y, \xi) \eeq such that $i$ and $j$ are formally contact isotopic in $(Y, \xi)$.
\eprop

\pf 
We will assume for simplicity that $V$ is connected but the proof can easily be generalized. Let $D_{\op{ot}} \sub (Y, \xi)$ be an overtwisted disk. Let $f_t$ be a family of formal contact embeddings such that $f_0$ is the underlying smooth map induced by $i$, and $\op{Im}(f_1) \cap D_{\op{ot}} = \emptyset.$  Let $\O \sub Y$ be a connected open subset such that $\op{Im} (f_1) \sub \O \sub \ov{\O} \sub Y-D_{\op{ot}}$. According to \cite[Prop.\ 3.8]{bem}, we can assume by choosing $\O$ large enough that $(\O, \xi)$ is overtwised. 

For purely algebro-topological reasons, there exists a family $\xi_t$ of formal contact structures on $Y$ with the following properties:
	
\begin{itemize}
\item $\xi_0=\xi$,
\item $\xi_t$ is constant in the complement of $\ov{\O}$,
\item $\xi_1$ is a genuine contact structure in a neighborhood $\mc{V} \sub \O$ of $\op{Im}(f_1)$ and $f_1$ is a genuine contact embedding with respect to $\xi_1$. 
\end{itemize}
	
Since $\xi_1$ is genuine on $\mc{V} \cup (Y- \ov{\O})$, it follows from the relative h-principle for overtwisted contact structures \cite[Thm.\ 1.2]{bem} that $\xi_1$ is homotopic to a genuine overtwisted contact structure through a homotopy fixed on $\mc{V} \cup (Y-\ov{\O})$. Thus we may as well assume in the third bullet above that $\xi_1$ is genuine everywhere. 

Since $(\O, \xi)$ is overtwisted, it follows from the \cite[Thm.\ 1.2]{bem} that there exists a homotopy $\tilde{\xi}_t$ of genuine contact structures such that $\tilde{\xi}_0=\xi_0=\xi$, $\tilde{\xi}_1=\xi_1$, and $\tilde{\xi}_t$ is independent of $t$ on $Y - \ov{\O}$. 
	
By Gray's theorem, there is an ambient isotopy $\psi_t: Y \to Y$ which is fixed on $Y- \ov{\O}$ and has the property that $\psi_t^*\tilde{\xi}_t= \tilde{\xi}_0= \xi_0$. The composition $\psi_1^{-1} \circ f_1$ is in the same class of formal contact embeddings as $f_1$ and gives the desired genuine embedding. 
\epf

We now describe a procedure for constructing pairs of codimension $2$ contact embeddings in overtwisted manifolds which are formally isotopic but fail to be isotopic as contact embeddings.  

\begin{construction} \label{construction:ot-infinite-topology} 
Let $(Y^{2n-1}_0, \xi_0)$ be a closed, overtwisted contact manifold and let $(Y_0, B, \pi)$ be an open book decomposition which supports $\xi_0$. Let $i_0: (B, ({\xi_0})_B) \to (Y_0, \xi_0)$ be the tautological embedding of the binding and let $j_0: (B, ({\xi_0})_B) \to (Y_0, \xi_0)$ be a contact embedding with overtwisted complement which is formally isotopic to $i_0$ (the existence of such an embedding follows from \Cref{proposition:hploosecontact}).  Let $D_{\op{ot}} \sub Y_0$ be an overtwisted disk which is disjoint from $j_0(B)$. 

Choose an open subset $\mc{U} \sub Y_0$ whose closure is disjoint from $i_0(B) \cup j_0(B) \cup D_{\op{ot}}$, and such that $i_0, j_0$ are formally isotopic in the complement of $\ov{\mc{U}}$. Now let $(Y, \xi)$ be obtained by attaching handles of arbitrary index along isotropic submanifolds contained inside $\mc{U}$ (see \Cref{construction:handle-attaching}). We let $(\hat{X}, \hat{\l})$ denote the resulting Weinstein cobordism with positive end $(Y, \xi)$ and negative end $(Y_0, \xi)$.

Observe that $(Y, \xi)$ is still overtwisted and that $i_0, j_0$ can also be viewed as codimension $2$ contact embeddings into $(Y,\xi)$. We denote these latter embeddings by $i, j: (B, ({\xi_0})_B) \to (Y, \xi)$. By construction, the embeddings $i, j$ are formally isotopic. 
\end{construction}

\thm \label{theorem:distinguish-infinite-ot}
The embeddings $i, j$ which arise from \Cref{construction:ot-infinite-topology} are not genuinely isotopic. In fact, $i$ is not genuinely isotopic to any reparametrization of $j$ in the source, meaning that the codimension $2$ submanifolds $(i(B), \xi_{i(B)}), (j(B), \xi_{j(B)})$ are not contact isotopic.
\ethm

\pf
According to \Cref{corollary:cobordism-map}, the cobordism $(\hat{X}, \hat{\l})$ induces a map of unital $\Q$-algebras $\widetilde{CH}_{\bullet}(Y, \xi, B; \fk{r}) \to \widetilde{CH}_{\bullet}(Y_0, \xi_0, B; \fk{r})$ for any element $\fk{r} \in \fk{R}(Y_0, \xi_0, B) \equiv \fk{R}(Y, \xi, B)$. Moreover, \Cref{theorem:openbook} implies that $\widetilde{CH}_\bullet(Y_0, \xi_0, B; \fk{r}) \neq 0$ for appropriate $\fk{r} \in \fk{R}(Y, \xi, B)$. It follows that $\widetilde{CH}_{\bullet}(Y, \xi, B; \fk{r}) \neq 0$.

If we assume that $(i(B), \xi_{i(B)}), (j(B), \xi_{j(B)})$ are isotopic as codimension $2$ contact submanifolds, then $\widetilde{CH}_{\bullet}(Y, \xi, B; \fk{r}) = \widetilde{CH}_{\bullet}(Y, \xi, j(B); \fk{r}')$ for some datum $\fk{r}' \in \fk{R}(Y, \xi, j(B))$. On the other hand, observe that $(Y-j(B), \xi)$ is overtwisted by construction. Hence \Cref{theorem:loosevanishing} implies that $\widetilde{CH}_{\bullet}(Y, \xi, j(B); \fk{r}') = 0$.
\epf

\ex
By a well-known theorem of Giroux and Mohsen \cite[Thm.\ 7.3.5]{geiges}, any contact manifold $(Y, \xi)$ admits an open book decomposition $(Y, B, \pi)$ which supports $\xi$. Hence \Cref{construction:ot-infinite-topology} and \Cref{theorem:distinguish-infinite-ot} can be applied to any overtwisted contact manifold.
\eex

We also consider the following modification of \Cref{construction:ot-infinite-topology}.

\begin{construction} \label{construction:ot-legendrian}
Let $(Y^{2n-1}_0, \xi_0)$ be a closed, overtwisted contact manifold and let $(Y_0, B, \pi)$ be an open book decomposition which supports $\xi_0$. Suppose that there exists a Legendrian submanifold $\L \sub Y_0$ such that $B = \tau(\L)$ is a contact pushoff of $\L$. Let $D_{\op{ot}} \sub Y_0$ be an overtwisted disk.

Let $\mc{U}_1 \sub Y_0-B - D_{\op{ot}}$ be an open ball which intersects $\L$. Let $\L' \sub (Y_0, \xi_0)$ be obtained by stabilizing $\L$ inside $\mc{U}_1$ (see \Cref{lemma:legendrian-stabilization}). Let $\mc{U}_2$ be the union of $\mc{U}_1$ with a tubular neighborhood of $\L$. Let $V'=\tau(\L') \sub \mc{U}_2$ be a choice of contact pushoff for $\L'$. 

Let $i_0: (B, ({\xi_0})_B) \to (Y_0, \xi_0)$ be the tautological embedding. \cite[Lem.\ 3.4]{casals-etnyre} implies that $i_0$ is formally isotopic to some codimension $2$ contact embedding $j_0: (B, (\xi_0)_B) \to (Y_0, \xi_0)$ where $j_0(B)= B'$.  Choose such a formal isotopy and let $\mc{T} \sub Y_0$ be its trace.

Let $\mc{U} \sub Y_0$ be an open set whose closure is disjoint from $\mc{T} \cup \mc{U}_2 \cup D_{\op{ot}}$. As in \Cref{construction:ot-infinite-topology}, let $(Y, \xi)$ be obtained by attaching handles of arbitrary index along some collection of isotropics inside $\mc{U}$. Let $(\hat{X}, \hat{\l})$ denote the resulting Weinstein cobordism with positive end $(Y, \xi)$ and negative end $(Y_0, \xi)$.

It follows from our choice of $\mc{U}$ that $(Y, \xi)$ is overtwisted and that $\L, \L', V, V'$ can be viewed as submanifolds of $(Y, \xi)$. It also follows that $\L'$ is the stabilization of $\L$ as submanifolds of $(Y, \xi)$, and that $V$ (resp. $V'$) is the contact pushoff of $\L$ (resp.\ $\L'$).  
\end{construction}

\cor \label{corollary:leg-distinguish-infinite-ot}
The submanifolds $(V, \xi_V)$ and $(V', \xi_{V'})$ are not isotopic through codimension $2$ contact submanifolds. Hence the Legendrian submanifolds $\L, \L' \sub (Y, \xi)$ are not isotopic through Legendrian submanifolds. 
\ecor

\pf
The proof of the first statement is identical to that of \Cref{theorem:distinguish-infinite-ot}.  The second statement follows from the fact that $V, V'$ are respectively the contact pushoff of $\L, \L'$. 
\epf

\ex
Let $(Y_0, \xi_0)= \op{obd}(T^*S^{n-1}, \tau^{-1})$, where $\tau^{-1}$ is a left-handed Dehn twist. Note by \cite[Thm.\ 1.1]{c-m-p} that $(Y_0, \xi_0)$ is overtwisted (in fact, $(Y, \xi)$ is contactomorphic to $(S^{2n-1}, \xi_{\op{ot}})$). Let $P=T^*S^{n-1} \sub Y_0$ be a page of the open book and let $\L \sub (Y_0, \xi_0)$ be the Legendrian which corresponds to the zero section of $P$. Then the binding of the open book decomposition is also a contact pushoff of $\L$. We may therefore apply \Cref{construction:ot-legendrian} to this data. 
\eex

\rmk
Consider the special case of \Cref{construction:ot-infinite-topology} and \Cref{construction:ot-legendrian} where $\mc{U}$ is empty, i.e. one does not attach any handles. In this case, \Cref{theorem:distinguish-infinite-ot} and \Cref{corollary:leg-distinguish-infinite-ot} are essentially equivalent to the statement that the binding of an open book decomposition is tight (i.e. must intersect any overtwisted disk). This statement was proved in dimension $3$ by Etnyre and Vela-Vick \cite[Thm.\ 1.2]{etnyre-velavick}; the higher dimensional case follows from work of Klukas \cite[Cor.\ 3]{klukas}, who proved (following an outline of Wendl \cite[Rmk.\ 4.1]{wendllocalfilling}) the stronger fact that a local filling obstruction (such as an overtwisted disk) in a closed contact manifold must intersect the binding of any supporting open book. 
\ermk

\subsection{Contact embeddings into the standard contact sphere} \label{subsection:embeddings-tight-manifold}

In this section, we exhibit examples of pairs of codimension $2$ contact embeddings into tight contact manifolds which are formally isotopic but are not isotopic through genuine contact embeddings. We begin with the following construction.

\begin{construction} \label{construction:tight-surgery-example}
Let $(Y^{2n-1}_0, \xi_0)$ be a contact manifold for $n \geq 3$. Let $V \subset (Y_0, \xi_0)$ be a codimension $2$ contact submanifold and let $\L \sub (Y_0, \xi_0)$ be a loose Legendrian such that $\L \cap V = \emptyset$.
	
Choose an open ball $\mc{U} \sub Y_0$ such that $(\mc{U}, \mc{U} \cap \L)$ is a loose chart for $\L$.  Next, choose an open ball $\mc{O} \sub Y_0-V- \mc{U}$ (By definition of a loose chart, $\mc{U} \cap \L$ is a proper subset of $\L$, so it is clear that such choices exist).
	
Let $\L'$ be obtained by stabilizing $\L$ inside $\mc{O}$. It follows from \Cref{lemma:legendrian-stabilization} that $\L$ and $\L'$ are formally isotopic via a formal isotopy fixed outside of $\mc{O}$. 

According to \Cref{lemma:formal-contact} below, we can (and do) fix a contactomorphism $f: (Y_0, \xi_0) \to (Y_0, \xi_0)$ with the following properties:
\begin{enumerate}
\item \label{enumerate:isotopic-identity}  $f$ is isotopic to the identity,
\item \label{enumerate:loose-to-loose}  $f(\L)=\L'$,
\item \label{enumerate:formal-isotopy} the tautological contact embedding $i'_0: (V, (\xi_0)_V) \to (Y_0-\L', \xi_0)$ is formally isotopic to the embedding $i'_1:= f \circ i'_0: (V, (\xi_0)_V) \to (Y_0-\L', \xi_0)$ (we emphasize here that the formal isotopy is contained in the open contact manifold $(Y_0-\L', \xi_0)$). 
\end{enumerate}

Finally, let $(Y, \xi)$ be obtained by attaching a Weinstein $n$-handle along $\L' \sub (Y_0, \xi_0)$ as described in \Cref{construction:handle-attaching}. We assume without loss of generality that the attaching region $\L' \sub \mc{V}$ disjoint from $V$ and $f(V)$, and that $i'_0$ and $i'_1$ are formally isotopic in $Y_0-\mc{V}$. We let $\iota: Y_0-\mc{V} \hookrightarrow Y$ be the canonical inclusion. 
	
Let 
\eq i_0= \iota \circ i'_0: (V, \xi_V) \to (Y, \xi) \eeq 
be the tautological contact embedding and define 
\eq i_1:= \iota \circ i'_1: (V, \xi_V) \to (Y, \xi). \eeq 
It is an immediate consequence of (\ref{enumerate:formal-isotopy}) and our choice of $\mc{V}$ that $i_0$ and $i_1$ are formally isotopic. 
\end{construction}

\lem \label{lemma:formal-contact}
With the notation of \Cref{construction:tight-surgery-example}, there exists a contactomorphism $f: (Y_0, \xi_0) \to (Y_0, \xi_0)$ satisfying the properties (\ref{enumerate:isotopic-identity}-\ref{enumerate:formal-isotopy}) stated in \Cref{construction:tight-surgery-example}.
\elem

\pf 
Recall that $\mc{U}$ is disjoint from $\mc{O}$. Recall also that $(\mc{U}, \mc{U} \cap \L)$ is a loose chart for $\L$, which means in particular that $\mc{U}$ deformation retracts onto $\mc{U} \cap \L$. Using these two facts, it is not hard to verify that there exists a family of \emph{formal} contact embeddings $j_t: (V, (\xi_0)_V) \to (Y_0, \xi_0)$, for $t \in [0,1]$, with the following properties:
\begin{itemize}
\item $j_0= i_0'$, 
\item $j_t(V)$ is disjoint from $\mc{O} \cup \L$ for all $t \in [0,1]$,
\item $j_1(V)$ is disjoint from $\mc{U} \cup \mc{O} \cup \L$.
\end{itemize}

By the $h$-principle for loose Legendrian embeddings \cite[Thm.\ 1.2]{murphy}, there exists a global contact isotopy $\phi_t$, for $t \in [0,1]$, such that $\phi_0=\op{Id}$ and $\phi_1(\L)=\L'$. By the Legendrian isotopy extension theorem \cite[Thm.\ 2.6.2]{geiges}, this isotopy can be assumed to be compactly-supported and constant in a neighborhood $\mc{W}$ of $j_1(V)$, where $\mc{W}$ is disjoint from $\mc{U} \cup \mc{O} \cup \L$. 

Let $f:= \phi_1$ and observe that $f$ satisfies (\ref{enumerate:isotopic-identity}-\ref{enumerate:loose-to-loose}). Observe that $f\circ j_t$ defines a formal contact isotopy from $f\circ i_0'=i_1'$ to $j_1$ in the complement of $\L'=f(\L)$. Since $i_0'$ is formally isotopic to $j_1$ in the complement of $\L' \sub \L \cup \mc{O}$, we find that $f$ satisfies (\ref{enumerate:formal-isotopy}).
\epf

It will be useful to record the following basic observation, which is a consequence of the fact stated in \Cref{definition:symplectic-lift} that an isotopy of contactomorphisms induces a Hamiltonian isotopy of symplectizations. 

\lem[cf.\ \Cref{definition:symplectic-lift}] \label{lemma:lift-end-liouville}
Let $(\hat X, \hat \l)$ be a relative cobordism from $(Y^+, \l^+)$ to $(Y^-, \l^-)$. Given contactomorphisms $f^{\pm}: (Y^{\pm}, \l^{\pm}) \to (Y^{\pm}, \l^{\pm})$ contact isotopic to the identity, there is a symplectomorphism $F: (\hat X, \hat \l) \to (\hat X, \hat \l)$ which agrees near infinity with the lifts $\tilde{f}^{\pm}: (SY^{\pm}, \l_{Y^{\pm}}) \to (SY^{\pm}, \l_{Y^{\pm}})$. 
\qed
\elem

Let us now return to the geometric setup considered in \Cref{subsection:linearized-open-books}.  In particular, we let 
\eq (\hat W_0, \hat{\l}^a):= (D^2 \tms T^*S^{n-1}, \frac{1}{a} r^2 d\t + \l_{\op{std}}), \eeq 
where $a>0$ is a constant which will be fixed later (see \eqref{equation:a-filling}). 

We let $(W_0, \l^a), (Y_0, \xi_0= \op{ker} \l_0)$, $V \sub Y_0$, $\L \sub Y_0$, and $H= \{0\} \tms T^*S^{n-1}$ be defined as in \Cref{subsection:linearized-open-books}.  Note that $\L$ is a loose Legendrian according to \cite[Prop.\ 2.9]{casals-murphy}. 

\Cref{construction:tight-surgery-example} applied to the above data produces a contactomorphism $f: (Y_0, \xi_0) \to (Y_0, \xi_0)$, a Liouville domain $(X, \l)$ with positive contact boundary $(Y, \xi= \op{ker} \l)$, and a pair of formally isotopic contact embeddings $i_0, i_1: (V, \xi_V) \to (Y, \xi)$. 

We let $\fk{r} =(\a, \tau, a) \in \fk{R}(Y_0, \xi_0, V)$, where $\a:= (\l_0)|_V$ and the trivialization $\tau$ is unique since $H^1(V; \Z)=0$ (see \Cref{corollary:topological-facts}).   We let $\fk{r}' = ((i_1')_*\a, \tau, a) \in \fk{R}(Y_0, \xi_0, V)$, where $i_1'$ is defined as in \Cref{construction:tight-surgery-example} and $\tau$ is again unique.  Since the surgery resulting from \Cref{construction:tight-surgery-example} away from $V$ and $i_1'(V)$, we may identify $\fk{R}(Y, \xi, V)=\fk{R}(Y_0, \xi_0, V)$ and $\fk{R}(Y, \xi, i_1(V))= \fk{R}(Y_0, \xi_0, i_1'(V))$. 

As in \Cref{subsection:linearized-open-books}, let $e_0: \R \tms Y_0 \to (\hat W_0, \hat \l_0, H)$ be the canonical marking furnished by the Liouville flow and let $\tilde{\e}_0: \widetilde{\mc{A}}(Y_0, \xi_0, V; \fk{r}) \to \Q$ be the associated augmentation. 

By \Cref{lemma:lift-end-liouville}, there is a symplectomorphism $\psi: (\hat W_0, \hat \l_a) \to (\hat W_0, \hat \l_a)$ which coincides near infinity with the lift $\tilde{f}: SY_0 \to SY_0$. Let $H' \sub \hat W_0$ be a symplectic submanifold which is cylindrical at infinity and coincides with the symplectization of $f(V)=i_1'(V)$ on $[0,\infty) \tms Y_0$.  Such a surface can be constructed by taking the backwards Liouville flow of $\psi(H)$. 
		
Let $\tilde{\e}'_0:  \widetilde{\mc{A}}(Y_0, \xi_0, i'_1(V); \fk{r}') \to \Q$ be the augmentation induced by the relative symplectic cobordism $((\hat W_0, \hat{\l}_a, H'), \tilde{\e}_0)$. 

Observe that $(Y_0, \xi_0, V) \in \mc{G}$ and hence also $(Y_0, f_*\xi_0, f(V))=(Y_0, \xi_0, i_1'(V)) \in \mc{G}$ (see \Cref{definition:obd-contact-pairs-G}). The following lemma shows that we also have $(Y, \xi, i_1(V)) \in \mc{G}$. 

\lem \label{lemma:i1-obd}
	Up to contactomorphism, $(Y, \xi)= \op{ob}(T^*S^{n-1}, \tau_S)=(S^{2n-1}, \xi_{\op{std}}) $, where $\tau_S$ denotes a right-handed Dehn twist. Moreover, the first contactomorphism can be assumed to take $i_1(V)$ to the binding of the open book decomposition $\op{ob}(T^*S^{n-1}, \tau_S)$. 
\elem
	
\pf
By construction, there is an open book decomposition of $(Y_0, \xi_0)$ agreeing (up to contactomorphism) with $\op{ob}(T^*S^{n-1}, \op{id})$, such that $i_1(V)$ is the binding and $\L'$ is the zero section of a page. Note now that attaching a handle to the zero section of a page of $(Y_0, \xi_0)=\op{ob}(T^*S^{n-1}, \op{id})$ simply changes the monodromy of the open book by a positive Dehn twist \cite[Thm.\ 4.6]{van-koert}. Hence, $i_1(V)$ is the binding of $\op{ob}(T^*S^{n-1}, \tau_S)=(S^{2n-1}, \xi_{\op{std}})$.
\epf

\cor \label{corollary:topological-facts}
The manifolds $Y_0, W_0, Y, W$ have vanishing first and second homology and cohomology with $\Z$-coefficients. Hence the same is also true for the pairs $(W_0, Y_0)$ and $(W,Y)$. Finally, we have $H^1(V; \Z)=0$. 
\ecor

\pf
By construction, $W$ is obtained by attaching a handle of index $n$ to $W_0$. The union of the core and co-core of this handle has codimension $n$. Hence, for $i \leq n-2$, we have $H_i(W_0; \Z)=H_i(W;\Z)$ and $H_i(Y_0;\Z)=H_i(Y;\Z)$. Now, $W_0$ is homotopy equivalent to $S^{n-1}$ by definition, while $Y$ is homeomorphic to $S^{2n-1}$ by \Cref{lemma:i1-obd}. Since $n \geq 4$, it follows that $Y_0, W_0, Y, W$ have vanishing first and second homology. The vanishing of cohomology in the same degrees follows by the universal coefficients theorem for cohomology.  The vanishing of $H^1(V;\Z)$ was proved in \Cref{lemma:first-top-facts}.
\epf

As a result of \Cref{corollary:topological-facts}, \Cref{definition:bigrading}, \Cref{lemma:linking-grading-cobordism} and \Cref{definition:bigradings-legendrian}, the invariants considered in the proof of \Cref{theorem:tight-examples} below, as well as the maps between these invariants, are all canonically $(\Z \tms \Z)$-bigraded.

\begin{theorem-star} \label{theorem:tight-examples}
For $n \geq 4$ \emph{even} and $a \gg 0$ large enough, the contact embeddings $i_0, i_1: (V, \xi_V) \to (Y, \xi)= (S^{2n-1}, \xi_{\op{std}})$ are not isotopic through contact embeddings.
\end{theorem-star}

\pf
We suppose for contradiction that $i_0$ and $i_1$ are isotopic through contact embeddings. This means that there exists a contactomorphism $g: (Y,\xi, V) \to (Y, \xi, i_1(V))$. It follows by \Cref{lemma:i1-obd} that $(Y, \xi, V) \in \mc{G}$.
	
According to \Cref{corollary:linearized-obd} (and the description of the generators in \Cref{proposition:linearized-obd}), we may (and do) fix $a \gg 0$ large enough so that
\eq \widetilde{CH}^{\tilde{\e}_0}_{k-(n-3), 2}(Y_0, \xi_0, V; \fk{r})= 
\begin{cases} 
\Z & \text{ if } k=4, 4+(n-1),\\
0 & \text{ otherwise.}
\end{cases}
\eeq
In particular, since $n \geq 4$, we have
\eq \label{equation:vanish-part} \widetilde{CH}^{\tilde{\e}_0}_{2n-(n-3),2}(Y_0, \xi_0, V; \fk{r})=\widetilde{CH}^{\tilde{\e}_0}_{2n+1-(n-3),2}(Y_0, \xi_0, V; \fk{r})=0.\eeq
Since $(Y_0, \xi_0, V) \in \mc{G}$, it follows by \Cref{corollary:commuting-fillings} that 
\eq \label{equation:inner-equal} \widetilde{CH}^{\tilde{\e}_0}_{\bullet, \bullet}(Y_0, \xi_0, V; \fk{r})=\widetilde{CH}^{\tilde{\e}_0'}_{\bullet, \bullet}(Y_0, \xi_0, i_1(V); \fk{r}'). \eeq  
Similarly, it follows by \Cref{definition:legendrian-invariance-reduced} that 
\eq \label{equation:inner-equal-legendrian} \widetilde{\mc{L}}^{\tilde{\e}_0}_{\bu, \bu} (Y_0, \xi_0, V, \L; \fk{r})=\widetilde{\mc{L}}^{\tilde{\e}_0'}_{\bu,\bu}(Y_0, \xi_0, i_1(V), \L'; \fk{r}'). \eeq  

Let $e: \R \tms Y \to \hat W$ be the canonical marking and consider the resulting relative filling $((\hat{W}, \hat \l, H),e)$. Let $\tilde{\e}: \widetilde{\mc{A}}(Y, \xi, V; \fk{r}) \to \Q$ be the induced augmentation. Let $\phi: (\hat W, \hat \l, H) \to (\hat W, \hat \l, H)$ be a symplectomorphism which agrees with the lift of $g$ near infinity. Let $\tilde{\e}': \widetilde{\mc{A}}(Y, \xi, i_1(V); \fk{r}') \to \Q$ be the augmentation induced by $((\hat W, \hat \l, H'), e)$. Then according to \Cref{lemma:fillings-commute} and \Cref{corollary:commuting-fillings}, we have
\eq \widetilde{CH}^{\tilde{\e}}_{\bu,\bu}(Y, \xi, V; \fk{r}) = \widetilde{CH}^{\tilde{\e}'}_{\bu,\bu} (Y, \xi, i_1(V); \fk{r}'). \eeq
It then follows by \Cref{lemma:non-zero} that $\widetilde{CH}^{\tilde{\e}}_{2n-(n-3),2}(Y, \xi, V; \fk{r}) \neq 0$. Hence \Cref{lemma:injective} implies that 
\eq \widetilde{CH}^{\tilde{\e}_0}_{2n-(n-3),2}(Y_0, \xi_0, V; \fk{r}) \neq 0. \eeq 
This contradicts \eqref{equation:vanish-part}.
\epf

\begin{lemma-star} \label{lemma:non-zero}
We have 
\eq \widetilde{CH}^{\tilde{\e}'}_{2n-(n-3),2} (Y, \xi, i_1(V); \fk{r}') \neq 0. \eeq
\end{lemma-star}
\pf
On the one hand, \Cref{corollary:leg-wind-computation} and \eqref{equation:inner-equal-legendrian} imply that 
\eq \op{rk} \ov{HC}_{2n,2}(\widetilde{\mc{L}}^{\tilde{\e}_0'}(Y_0, \xi_0, i_1(V), \L'; \fk{r}') ) = 1. \eeq
On the other hand, by \eqref{equation:vanish-part} and \eqref{equation:inner-equal}, we have that
\eq \label{equation:vanish-part-prime} \widetilde{CH}^{\tilde{\e}'_0}_{2n-(n-3),2}(Y_0, \xi_0, i_1(V); \fk{r}')=\widetilde{CH}^{\tilde{\e}'_0}_{2n+1-(n-3),2}(Y_0, \xi_0, i_1(V); \fk{r}')=0.\eeq
It then follows by Theorem* \ref{theorem-star:reduced-attaching-triangles} and \Cref{remark:exact-sequence-bigrading} that 
\eq \widetilde{CH}^{\tilde{\e}'}_{2n-(n-3),2} (Y, \xi, i_1(V); \fk{r}')  \simeq  \ov{HC}_{2n,2}(\widetilde{\mc{L}}^{\tilde{\e}'_0}(Y_0, \xi_0, i_1(V), \L'; \fk{r}') ). \eeq
This proves the claim.
\epf

\begin{lemma-star} \label{lemma:injective}
The natural map 
\eq \widetilde{CH}^{\tilde{\e}}_{2n-(n-3),2} (Y, \xi, V; \fk{r})  \to \widetilde{CH}^{\tilde{\e}_0}_{2n-(n-3),2}(Y_0, \xi_0, V; \fk{r}) \eeq is injective.
\end{lemma-star}

\pf
Since $\L'$ is loose in $Y_0-V$, it follows by \Cref{proposition:loose-acyclic} and \Cref{lemma:reduced-cyclic-acyclic} that 
\eq \ov{HC}_{2k}(\mc{L}^{\tilde{\e}_0}(Y_0, \xi_0, V, \L'; \fk{r}))=0 \eeq 
for all $k \in \Z$. The lemma thus follows from Theorem* \ref{theorem-star:reduced-attaching-triangles} and \Cref{remark:exact-sequence-bigrading}. 
\epf

\rmk \label{remark:casals-etnyre-coincide}
One can slightly tweak \Cref{construction:tight-surgery-example} so that $\L, \L'$ are disjoint and $\L \cup \L'$ is a loose Legendrian link. One can then upgrade \Cref{lemma:formal-contact} to require that $f(\L)=\L'$ and $f(\L')=\L$ in  \eqref{enumerate:loose-to-loose} of \Cref{construction:tight-surgery-example}.  In particular, this means that $\L'$ is a stabilization of $\L$ and $\L$ is a stabilization of $\L'$. 

Let us apply this tweaked construction to the setup considered in \Cref{construction:tight-surgery-example}, where $(Y_0, \xi_0)= \op{ob}(T^*S^{n-1}, \op{id})$, for $n \geq 4$, $V \sub (Y_0, \xi_0)$ is the binding and $\L \sub (Y_0, \xi_0)$ be the zero section of a page. It is well known that the zero section of a page in $\op{ob}(T^*S^{n-1}, \op{id})$ is the standard Legendrian unknot. Hence \Cref{lemma:i1-obd} implies that $i_1(V)$ is the pushoff of the standard unknot. By construction, it now also follows that $i_0(V)$ is the contact pushoff of a stabilization of the unknot. \Cref{theorem:tight-examples} thus provides an alternative way to distinguish (for $n \geq 4$ even) the basic example considered by Casals and Etnyre in \cite[Sec.\ 5]{casals-etnyre}. 
\ermk

\subsection{Relative symplectic and Lagrangian cobordisms}

In this final section, we exhibit some constraints on relative symplectic and Lagrangian cobordisms.  In particular, we prove the results which were advertised in \Cref{subsection:intro-cob-applications} of the introduction. 

\pf[Proof of \Cref{theorem:intro-symplectic-cobordism}] 
Suppose for contradiction that such a relative symplectic cobordism exists. According to \Cref{theorem:openbook}, we have $\widetilde{CH}_\bu(Y, \xi, V; \fk{r}) \neq 0$ for some $\fk{r}=(\a, \tau, r) \in \fk{R}(Y, \xi, V)$ which we now view as fixed. According to \Cref{theorem:loosevanishing}, we also have $\widetilde{CH}_\bu(Y, \xi, V'; \fk{r}') = 0$ for all $\fk{r}'= (\a', \tau', r') \in \fk{R}(Y, \xi, V')$. Choose $\fk{r}'$ depending on our previous choice of $\fk{r}$ so that $r' \geq e^{\mc{E}((H, \l_H)^{\a'}_{\a})} r$. Then \Cref{corollary:cobordism-map} along with our topological assumptions on $H$ furnishes a unital $\Q$-algebra map 
\eq \widetilde{CH}_\bu(Y, \xi, V'; \fk{r}') \to \widetilde{CH}_\bu(Y, \xi, V; \fk{r}). \eeq This gives the desired contradiction. 
\epf

In contrast to \Cref{theorem:intro-symplectic-cobordism}, we expect that one could prove that $V$ is concordant to $V'$ by adapting work of Eliashberg and Murphy \cite{eliashberg-murphy-caps}, but we do not pursue this here. Note that we could also have proved \Cref{theorem:intro-symplectic-cobordism} using the full invariant $CH_\bu(-;-)$ instead of its reduced counterpart. 

For Lagrangian cobordisms, we have the following result.

\prop \label{proposition:lag-concordance}
Let $(Y, \xi)$ be a contact manifold. Let $\L, \L'$ be Legendrian knots such that $H^1(\tau(\L'); \Z)= H^2(\tau(\L'); \Z)=0$. Suppose that $(SY, \l_Y, L)$ is a Lagrangian concordance from $\L'$ to $\L$. Given $\fk{r}= (\a, \tau, r) \in \fk{R}(Y, \xi, \tau(\L))$, there is a map of $\Q$-algebras
\eq  \label{equation:leg-map} \widetilde{CH}_{\bullet}(Y, \xi, \tau(\L'); \fk{r}') \to \widetilde{CH}_{\bullet}(Y, \xi, \tau(\L); \fk{r}) \eeq 
for some $\fk{r}'= (\a', \tau', r') \in \fk{R}(Y, \xi, \tau(\L'))$.  (A similar statement holds for the non-reduced invariants $CH_{\bullet}(-)$). 
\eprop

\pf
Observe that the trivial Lagrangian cobordism $L= \R \tms \L \sub \R \tms Y$ admits a ``symplectic push-off" $\tau(L):= \R \tms \tau(\L) \sub \R \tms Y$. It follows by the Lagrangian neighborhood theorem that any Lagrangian concordance $(SY, \l_Y, L)$ also admits a symplectic push-off $(SY, \l_Y, H)$, which is a relative symplectic cobordism from $(Y, \xi, \tau(\L'))$ to $(Y, \xi, \tau(\L))$. Fix $\a'$ arbitrarily and choose $r'$ so that $r' \geq e^{\mc{E}((H, \l_H)^{\a'}_{\a})} r$ (note that $\tau'$ is unique since $H^1(\tau(\L'); \Z)=0$).  The claim now follows from \Cref{corollary:cobordism-map}. 	
\epf	

\pf[Proof of \Cref{theorem:intro-lag-concordance}] 
Suppose for contradiction that $\L'$ is concordant to $\L$. As in the proof of \Cref{corollary:leg-distinguish-infinite-ot}, we have 
\eq \widetilde{CH}_{\bullet}(Y, \xi, \tau(\L); \fk{r}) \neq 0 \eeq 
for a suitable choice $\fk{r} \in \fk{R}(Y, \xi, \tau(\L))$. On the other hand, we have $\widetilde{CH}_{\bullet}(Y, \xi, \tau(\L'); \fk{r}')=0$ for all $\fk{r}' \in \fk{R}(Y, \xi, \tau(\L'))$. This gives a contradiction in view of \Cref{proposition:lag-concordance}. 
\epf

We end by exhibiting examples of Lagrangian cobordisms which cannot be displaced from a codimension $2$ symplectic submanifold. 

\begin{construction} \label{construction:obd-intersection-rigidity}
Let $(Y_0, \xi_0)= \op{obd}(T^*S^{n-1}, \op{id})$ for $n \geq 3$. Let $V \sub (Y_0, \xi_0)$ be the binding and let $\L \sub (Y_0, \xi_0)$ be the zero section of a page, which is a loose Legendrian by \cite[Prop.\ 2.9]{casals-murphy}. Let $\mc{U}_1 \sub Y_0-V$ be a small ball which intersects $\L$ in an $(n-1)$-ball and let $\L'$ be obtained by stabilizing $\L$ inside $\mc{U}_1$. 
	
Let $\mc{U}_2 \sub Y_0- (V \cup \L \cup \mc{U}_1)$ be an open subdomain. Let $(Y, \xi)$ be obtained by attaching a sequence of handles along isotropics contained in $\mc{U}_2$. Observe that $V, \L, \L'$ can be viewed as submanifolds of both $Y_0$ and $Y$; we will not distinguish these embeddings in our notation. We let $(\hat X, \hat \l, \hat{V})$ be the associated relative symplectic cobordism from $(Y, \xi, V)$ to $(Y_0, \xi_0, V)$.
\end{construction}

\pf[Proof of \Cref{theorem:introduction-non-displaceability}] 
Note that we can identify $\fk{R}(Y_0, \xi_0, V)= \fk{R}(Y, \xi, V)$. According to \Cref{theorem-star:legendrian-nonvanishing}, $\mc{L}(Y_0, \xi_0, V, \L; \fk{r}) \neq 0$ for suitable $\fk{r} \in \fk{R}(Y_0, \xi_0, V)$. In constrast, \Cref{proposition:loose-acyclic} implies that $\mc{L}(Y, \xi, V, \L'; \fk{r}) = 0$ since $\L'$ is loose in $Y-V$ by construction. By \Cref{corollary:cobordism-map-legendrian}, the existence of a concordance from $\L'$ to $\L$ which doesn't intersect $\hat{V}$ would imply that there is a unital map of $\Q[U]$-algebras 
\eq \mc{L}(Y, \xi, V, \L'; \fk{r}) \to \mc{L}(Y_0, \xi_0, V, \L; \fk{r}). \eeq 
This gives a contradiction.
\epf

We remark that \Cref{construction:obd-intersection-rigidity} could be generalized in various directions without affecting the validity of \Cref{theorem:introduction-non-displaceability}, but we do not pursue this here.

\appendix

\section{Connected sums of almost-contact manifolds} \label{section:contact-connected-sums}

Let $G$ be a connected\footnote{This assumption is used in the proof of \Cref{proposition-connected-sum-well-defined}.} subgroup of $\SO(n)$. An \emph{almost $G$-structure} on a smooth oriented manifold $M$ is a homotopy class of maps $M \to BG$ lifting the classifying map of the tangent bundle of $M$:
\begin{center}
\begin{tikzcd}
	& BG \arrow[d] \\
	M \arrow[r,"TM" swap]\arrow[ru] & B\SO(n)
\end{tikzcd}
\end{center}
An \emph{almost $G$-manifold} is a manifold equipped with an almost $G$-structure.
\begin{example}
	Taking $G = \mathrm{U}(n) \subset \SO(2n)$ yields the usual notion of an almost complex manifold. Almost-contact manifolds correspond to $G = \mathrm{U}(n) \subset \SO(2n+1)$.
\end{example}

If the $n$-dimensional sphere $S^n$ admits an almost $G$-structure, a result of Kahn \cite[Theorem 2]{kahn-obstr} implies that for any two $n$-dimensional almost $G$-manifolds $M$ and $N$, there exists an almost $G$-structure on $M \# N$ which is compatible with the given ones on $M$ and $N$ in the complement of the disks used to form the connected sum. In general, this structure is not unique, so the connected sum $M \# N$ is not well-defined as an almost $G$-manifold. However, we will show in \ref{section-connect-sum} that a choice of almost $G$-structure $\beta$ on $S^n$ induces a canonical almost $G$-structure on the connected sum of any two almost $G$-manifolds. Hence, any such $\beta$ gives rise to a connected sum operation $(M,N) \mapsto M \#_\beta N$ for almost $G$-manifolds. Moreover, the set of almost $G$-structures on $S^n$ forms a group under this operation (with identity $\beta$). In \cref{section-group-structure}, we will show that this group acts on the set of almost $G$-structures of any $n$-dimensional almost $G$-manifold.

\subsection{Connected sums of almost $G$-manifolds} \label{section-connect-sum}

Let $S^n$ be the unit sphere in $\R^{n+1}$, equipped with its standard orientation as the boundary of the unit disk $D^{n+1}$. We will write its points as pairs $(x,z) \in \R^n \times \R$. Define
\begin{align*}
	D_- &= \{ (x,z) \in S^n \mid z < 1/2 \}, \\
	D_+ &= \{ (x,z) \in S^n \mid z > -1/2 \}, \\
	A &= D_- \cap D_+,\\
	C_\pm &= D_\pm \setminus A.
\end{align*}
Note that $D_-$ and $D_+$ are open disks, $C_-$ and $C_+$ are closed disks, $A$ is an open annulus, and $S^n = D_- \cup D_+ = C_- \sqcup A \sqcup C_+$.

Let $M$ and $N$ be smooth connected oriented $n$-dimensional manifolds. Choose orientation preserving embeddings $i_+ : \overline{D}_+ \to M$ and $i_- : \overline{D}_- \to N$. We define the connected sum $M \# N = M \#_{i_+,i_-} N$ by
$$ M \# N = \bigr( M \setminus i_+(C_+) \sqcup N \setminus i_-(C_-) \bigl)/{\sim} $$
where $i_+(x) \sim i_-(x)$ for every $x \in A$.

We will now explain how to construct a classifying map for the tangent bundle of $M \# N$. The following elementary fact from topology will be useful.
\begin{proposition} \label{proposition-cofibration}
	Let $i : A \to X$ be a cofibration. Assume $A$ is contractible. Then for any connected space $Y$ and continuous maps $F : X \to Y$, $f : A \to Y$, there exists a map $F' : X \to Y$ homotopic to $F$ such that $F' \circ i = f$.
\end{proposition}
Let $\tau_S : S^n \to B\SO(n)$ be a classifying map for $TS^n$. Let $\tau_M$ and $\tau_N$ be classifying maps for $TM$ and $TN$ such that $\tau_M \circ i_+ = \tau_S|_{\overline{D}_+}$ and $\tau_N \circ i_- = \tau_S|_{\overline{D}_-}$ (such maps always exist by \Cref{proposition-cofibration}). Define $\tau_{M \# N}$ to be the unique map $M \# N \to B\SO(n)$ which coincides with $\tau_M$ on $M \setminus i_+(C_+)$ and with $\tau_N$ on $N \setminus i_-(C_-)$.
\begin{proposition} \label{proposition-classifying-tangent-bundle}
	$\tau_{M\# N}$ is a classifying map for $T(M \# N)$.
\end{proposition}
We start with an easy topological lemma.
\begin{lemma}\label{lemma:extension-bundle-automorphism}
	Let $E$ be an oriented vector bundle over a manifold $M^n$ and let $i : D^n \to M^n$ be an embedding. Then any automorphism of $i^*E$ can be extended to an automorphism of $E$.
\end{lemma}
\begin{proof}
	Let $\phi$ be an automorphism of $i^*E$. Since $D^n$ is contractible, we can trivialize $i^*E$ and think of $\phi$ as a map $D^n \to \mathrm{GL}^+(n)$. Clearly $\phi|_{\partial D^n}$ is nullhomotopic, and since $\mathrm{GL}^+(n)$ is connected, we can extend $\phi$ to a map $\tilde{\phi} : D_2^n \to \mathrm{GL}^+(n)$ which is constant with value $\mathrm{Id} \in \mathrm{GL}^+(n)$ near $\partial D_2^n$. Using a tubular neighborhood of $i(\partial D^n) \subset M$, we can also extend $i$ to an embedding $\tilde{i} : D_2^n \to M^n$. Then $\tilde{\phi}$ gives us an automorphism of $\tilde{i}^*E$ which is equal to the identity over a neighborhood of $\partial D_2^n \subset D_2^n$ and hence extends trivially to an automorphism of $E$.
\end{proof}

\begin{proof}[Proof of \Cref{proposition-classifying-tangent-bundle}]
	Let $\tilde{\gamma}_n \to B\SO(n)$ be the universal bundle over $B\SO(n)$. We want to show that $T(M \# N)$ is isomorphic to $\tau_{M \# N}^*\tilde{\gamma}_n$.
	
	The tangent bundle $T(M \# N)$ of the connected sum is obtained by gluing $T(M \setminus i_+(C_+))$ and $T(N \setminus i_-(C_-))$ along the maps $di_+ : TA \to T(M \setminus i_+(C_+))$ and $di_- : TA \to T(N \setminus i_-(C_-))$. Because of our assumption that $\tau_M \circ i_+ = \tau_S|_{\overline{D}_+}$ and $\tau_N \circ i_- = \tau_S|_{\overline{D}_-}$, we also have that $\tau_{M \# N}^*\tilde{\gamma}_n$ is obtained by gluing $(\tau_M|_{M \setminus i_+(C_+)})^*\tilde{\gamma}_n$ and $(\tau_N|_{N \setminus i_-(C_-)})^*\tilde{\gamma}_n$ along bundle maps $(\tau_S|_A)^*\tilde{\gamma}_n \to (\tau_M|_{M \setminus i_+(C_+)})^*\tilde{\gamma}_n$ and $(\tau_S|_A)^*\tilde{\gamma}_n \to (\tau_N|_{N \setminus i_-(C_-)})^*\tilde{\gamma}_n$ covering $i_+ : A \to M \setminus i_+(C_+)$ and $i_- : A \to N \setminus i_-(C_-)$ respectively.
	Hence, in order to show that $T(M \# N)$ is isomorphic to $\tau_{M \# N}^*\tilde{\gamma}_n$, it suffices to construct a commutative diagram
	\begin{center}
	\begin{tikzcd}
		T(M \setminus i_+(C_+)) \ar[r,dashed] & (\tau_M|_{M \setminus i_+(C_+)})^*\tilde{\gamma}_n \\
		TA \ar[u,"di_+"]\ar[d,swap,"di_-"]\ar[r,dashed] & (\tau_S|_A)^*\tilde{\gamma}_n \ar[u] \ar[d] \\
		T(N \setminus i_-(C_-)) \ar[r,dashed] & (\tau_N|_{N\setminus i_-(C_-)})^*\tilde{\gamma}_n
	\end{tikzcd}
	\end{center}
	where the horizontal arrows are bundle isomorphisms.

	Start by fixing an isomorphism $\phi : TS^n \to \tau_S^*\tilde{\gamma}_n$, and let the middle arrow of the diagram be the restriction of $\phi$ to $TA$. To get the top and bottom arrows, it suffices to find bundle isomorphisms completing the following commutative squares:
	\begin{center}
	\begin{tikzcd}
		TM \ar[r,dashed] & \tau_M^*\tilde{\gamma}_n \\
		T\overline{D}_+ \ar[u,"di_+"]\ar[r,swap,"\phi"] & (\tau_S|_{\overline{D}_+})^*\tilde{\gamma}_n \ar[u]
	\end{tikzcd}
	\begin{tikzcd}
		T\overline{D}_- \ar[d,swap,"di_-"]\ar[r,"\phi"] & (\tau_S|_{\overline{D}_-})^*\tilde{\gamma}_n \ar[d] \\
		TN \ar[r,dashed] & \tau_N^*\tilde{\gamma}_n 
	\end{tikzcd}
	\end{center}
	This is possible by \Cref{lemma:extension-bundle-automorphism}.
\end{proof}

We are now ready to define the connected sum of two almost $G$-manifolds.
\begin{definition}\label{definition:almost-connected-sum}
	Suppose that $S^n$ admits an almost $G$-structure, and fix a choice $\beta$ of one such structure. Let $\beta_M$ and $\beta_N$ be almost $G$-structures on $M$ and $N$ respectively. We define an almost $G$-structure $\beta_M \#_\beta \beta_N$ on $M \# N$ as follows.

	Pick maps $\tilde{\tau}_S : S^n \to BG$, $\tilde{\tau}_M : M \to BG$ and $\tilde{\tau}_N : N \to BG$ representing $\beta$, $\beta_M$ and $\beta_N$ respectively. By \Cref{proposition-cofibration}, we can assume that $\tilde{\tau}_M \circ i_+ = \tilde{\tau}_S|_{\overline{D}_+}$ and $\tilde{\tau}_N \circ i_- = \tilde{\tau}_S|_{\overline{D}_-}$. Hence, there is a unique map
	\[
		\tilde{\tau}_{M \# N} = \tilde{\tau}_M \#_{\tilde{\tau}_S} \tilde{\tau}_N : M \# N \to BG
	\]
	 which coincides with $\tilde{\tau}_M$ on $M \setminus i_+(C_+)$ and with $\tilde{\tau}_N$ on $N \setminus i_-(C_-)$. By \Cref{proposition-classifying-tangent-bundle}, the composition
	\begin{center}
	\begin{tikzcd}
		M \# N \arrow[r,"\tilde{\tau}_{M\# N}"] & BG \arrow[r] & B\SO(n)
	\end{tikzcd}
	\end{center}
	is a classifying map for $T(M \# N)$. Hence, we can (and do) define $\beta_M \#_\beta \beta_N$ to be the homotopy class of $\tilde{\tau}_{M \# N}$.
\end{definition}

\begin{proposition}\label{proposition-connected-sum-well-defined}
	The almost $G$-structure $\beta_M \#_\beta \beta_N$ is well-defined, i.e. independent of the choice of $\tilde{\tau}_S$, $\tilde{\tau}_M$ and $\tilde{\tau}_N$.
\end{proposition}
\begin{proof}
	Let $\tilde{\tau}_S^j$, $\tilde{\tau}_M^j$, and $\tilde{\tau}_N^j$ represent $\beta$, $\beta_M$ and $\beta_N$ respectively, where $j \in \{0,1\}$. As in \Cref{definition:almost-connected-sum}, we assume that $\tilde{\tau}_M^j \circ i_+ = \tilde{\tau}_S^j|_{\overline{D}_+}$ and $\tilde{\tau}_N^j \circ i_- = \tilde{\tau}_S^j|_{\overline{D}_-}$.

	Fix a homotopy $\tilde{\tau}_S^t$ between $\tilde{\tau}_S^0$ and $\tilde{\tau}_S^1$. We will show that there exist homotopies $\tilde{\tau}_M^t$ and $\tilde{\tau}_N^t$ such that $\tilde{\tau}_M^t \circ i_+ = \tilde{\tau}_S^t|_{\overline{D}_+}$ and $\tilde{\tau}_N^t \circ i_- = \tilde{\tau}_S^t|_{\overline{D}_-}$. This implies that $\tilde{\tau}_{M\# N}^0$ is homotopic to $\tilde{\tau}_{M\# N}^1$ and hence that $\beta_M \#_\beta \beta_N$ is well-defined.

	Pick an arbitrary homotopy $h : M \times I \to BG$ between $\tilde{\tau}_M^0$ and $\tilde{\tau}_M^1$ and define a map
	\[ g : \overline{D}_+ \times \partial I^2 \to BG \]
	by $g(x,t,0) = h(i_+(x),t)$, $g(x,0,s) = \tilde{\tau}_S^0(x)$, $g(x,1,s) = \tilde{\tau}_S^1(x)$ and $g(x,t,1) = \tilde{\tau}_S^t(x)$. We can extend $g$ to a map $\hat{g} : \overline{D}_+ \times I^2 \to BG$ since the obstruction to doing so lies in
	\[ H^2(\overline{D}_+ \times I^2, \overline{D}_+ \times \partial I^2 ; \pi_1(BG)) \cong \pi_1(BG) \cong \pi_0(G), \]
	which is trivial by our assumption that $G$ is connected.

	Let
	\[
		f : \bigl( M \times (I \times \{0\} \cup \{0\} \times I \cup \{1\} \times I) \bigr) \cup \bigl( i_+(\overline{D}_+) \times I^2 \bigr) \to BG
	\]
	be defined by
	\begin{itemize}
		\item $f(x,t,0) = h(x,t)$, $f(x,0,s) = \tilde{\tau}_M^0(x)$ and $f(x,1,s) = \tilde{\tau}_M^1(x)$ for $x \in M$;
		\item $f(x,t,s) = \hat{g}(i_+^{-1}(x),t,s)$ for $x \in i_+(\overline{D}_+)$.
	\end{itemize}
	Since $i_+ : \overline{D}_+ \to M$ is a cofibration, the domain of $f$ is a retract of $M \times I^2$. We can therefore extend $f$ to a map $\hat{f} : M \times I^2 \to BG$. Restricting $\hat{f}$ to $M \times I \times \{1\}$ then provides us with a homotopy $\tilde{\tau}_M^t$ such that $\tilde{\tau}_M^t \circ i_+ = \tilde{\tau}_S^t|_{\overline{D}_+}$.

	The same argument gives us a homotopy $\tilde{\tau}_N^t$ such that $\tilde{\tau}_N^t \circ i_- = \tilde{\tau}_S^t|_{\overline{D}_-}$, so this completes the proof.
\end{proof}

\begin{definition} \label{definition:connected-sum-almost-g}
	If $M = (M,\beta_M)$ and $N = (N,\beta_N)$ are almost $G$-manifolds, their connected sum (with respect to $\beta$) is the almost $G$-manifold $M \#_\beta N := (M \# N, \beta_M \#_\beta \beta_N)$.
\end{definition}
As usual, there is an ambiguity in the notation $M \#_\beta N$ since the construction of the connected sum involves a choice of embeddings $i_+ : \overline{D}_+ \to M$, $i_- : \overline{D}_- \to N$. However, the result is independent of these choices up to the appropriate notion of equivalence, as one would expect.
\begin{definition}
	A \emph{diffeomorphism of almost $G$-manifolds} $f : (M,\beta_M) \to (N,\beta_N)$ consists of a smooth diffeomorphism $f : M \to N$ such that $f^*\beta_N = \beta_M$.
\end{definition}
\begin{proposition} \label{proposition-diffeomorphism-almost-manifolds}
	The connected sum $M \#_\beta N$ is well-defined up to diffeomorphism of almost $G$-manifolds. More precisely, given any orientation preserving embeddings $i_+, j_+ : \overline{D}_+ \to M$ and $i_-,j_- : \overline{D}_- \to N$, there exists an orientation preserving diffeomorphism $\phi : M \#_{i_+,i_-} N \to M \#_{j_+,j_-} N$ such that
	\[
		\phi^*(\beta_M \#_{j_+,j_-,\beta} \beta_N) = \beta_M \#_{i_+,i_-,\beta} \beta_N
	\]
	for any almost $G$-structures $\beta_M$, $\beta_N$ on $M$, $N$.
\end{proposition}
\begin{proof}
	This follows from the isotopy extension theorem as in the smooth case.
\end{proof}

\rmk[Connected sums of contact manifolds] \label{remark:connected-sum-genuine}
	Suppose that $(M_1, \a_1), (M_2, \a_2)$ are contact manifolds. Then one can form the connected sum $(M_1 \# M_2, \a_1 \# \a_2)$, which is also a contact manifold. The connected sum is obtained by choosing Darboux balls in $M_1, M_2$ and connecting them by a ``neck". This operation can also be understood as a contact surgery along a $0$-sphere. We refer to \cite[Sec.\ 6.2]{bvk} and \cite[Sec.\ 3]{van-koert} for more details.

	Let $\beta \in \alm_{U(n)}(S^{2n-1})$ be the almost-contact structure induced by the standard contact structure on the sphere. Then the operation of connected sum (with respect to $\beta$) of almost $\mathrm{U}(n-1)$-manifolds defined in \Cref{definition:connected-sum-almost-g}, and the operation of connected sum of contact manifolds described above, commute with the forgetful map from contact manifolds to almost-contact manifolds. This can be shown as in the proof of \Cref{proposition-classifying-tangent-bundle}, replacing $\mathrm{BSO(n)}$ with $\mathrm{BU(n-1)}$. 
\ermk

The main properties of the connected sum in the smooth case have analogues for almost $G$-manifolds:
\begin{proposition} \label{proposition:monoid}
	Let $M$, $N$, $P$ be connected almost $G$-manifolds of dimension $n$ and let $\beta$, $\beta'$ be almost $G$-structures on $S^n$.
	\begin{enumerate}
		\item $M \#_\beta (S^n,\beta) \cong M$.
		\item $(M \#_\beta N) \#_{\beta'} P \cong M \#_\beta (N \#_{\beta'} P)$.
	\end{enumerate}
\end{proposition}
\begin{proof}
	If one takes $i_- : \overline{D}_- \to S^n$ to be the inclusion map, then the connected sum $M \# S^n$ is canonically identified with $M$ as a smooth manifold. If one further takes $\tilde{\tau}_N = \tilde{\tau}_S$ in \Cref{definition:almost-connected-sum}, then this identification is compatible with the almost $G$-structures on $M \# S^n$ and $M$. This proves that $M \#_\beta (S^n,\beta) \cong M$.

	To prove that $(M \#_\beta N) \#_\beta P \cong M \#_\beta (N \#_\beta P)$, choose embeddings $i_+ : \overline{D}_+ \to M$, $i_- : \overline{D}_- \to N$, $j_+ : \overline{D}_+ \to N$ and $j_- : \overline{D}_- \to P$. If we assume that $i_-$ and $j_+$ have disjoint images, then $i_-$ induces an embedding $\overline{D}_- \to N \#_{j_+,j_-} P$, $j_+$ induces an embedding $\overline{D}_+ \to M \#_{i_+,i_-} N$, and there is a canonical identification of smooth manifolds
	\[
		(M \#_{i_+,i_-} N) \#_{j_+,j_-} P \cong M \#_{i_+,i_-} (N \#_{j_+,j_-} P).
	\]
	Moreover, this identification is compatible with the almost $G$-structures in the sense that for any choice of maps $\tilde{\tau}_S$, $\tilde{\tau}'_S$, $\tilde{\tau}_M$, $\tilde{\tau}_N$, $\tilde{\tau}_P$, the following diagram commutes:
	\begin{center}
	\begin{tikzcd}[row sep = small]
		(M \#_{i_+,i_-} N) \#_{j_+,j_-} P \ar[dd,"\cong"]\ar[rrd,"(\tilde{\tau}_M \#_{\tilde{\tau}_S} \tilde{\tau}_N) \#_{\tilde{\tau}'_S} \tilde{\tau}_P"] & & \\
		& & BG \\
		M \#_{i_+,i_-} (N \#_{j_+,j_-} P) \ar[rru,swap,"\tilde{\tau}_M \#_{\tilde{\tau}_S} (\tilde{\tau}_N \#_{\tilde{\tau}'_S} \tilde{\tau}_P)"] & &
	\end{tikzcd}
	\end{center}
\end{proof}

\subsection{The group of almost $G$-structures on the sphere} \label{section-group-structure}

We will denote the set of almost $G$-structures on a manifold $M$ by $\alm_G(M)$. More generally, if $A \subset M$ is a closed subset and $\beta_0$ is an almost $G$-structure on some open neighborhood of $A$, then $\alm_G(M,A;\beta_0)$ will denote the set of almost $G$-structures on $M$ which agree with $\beta_0$ near $A$. 

In this section, we will show that $\#_\beta$ is a group operation on $\alm_G(S^n)$, with $\beta$ as identity element. The resulting group will be denoted by $\alm_G^\beta(S^n)$. We will then show that $\alm_G^\beta(S^n)$ acts on $\alm_G(M)$, and more generally on $\alm_G(M,A;\beta_0)$ if $M \setminus A$ is connected.

\begin{proposition}\label{proposition:existence-inverses}
	Given any $\beta_1 \in \alm_G(S^n)$, there exists a $\beta_2 \in \alm_G(S^n)$ such that $\beta_1 \#_\beta \beta_2 = \beta$.
\end{proposition}
\begin{proof}
	Recall the decomposition $S^n = C_- \cup A \cup C_+$ introduced at the beginning of section~\ref{section-connect-sum}. We will use the notation $\langle \tau^-, \tau^A, \tau^+ \rangle$ to denote the unique (assuming it exists) map $S^n \to BG$ which coincides with the given maps $\tau^- : C_- \to BG$, $\tau^A : \overline{A} \to BG$ and $\tau^+ : C_+ \to BG$ on $C_-$, $\overline{A}$ and $C_+$ respectively.

	Let $\tilde{\tau}_S = \langle \tau_S^-, \tau_S^A, \tau_S^+ \rangle$ be a representative for $\beta$. Given $\beta_1$ and $\beta_2$ in $\alm_G(S^n)$, we can choose representatives of the form $\langle \tau_1^-, \tau_S^A, \tau_S^+ \rangle$ and $\langle \tau_S^-, \tau_S^A, \tau_2^+ \rangle$ by \Cref{proposition-cofibration}. Then $\beta_1 \#_\beta \beta_2$ is represented by $\langle \tau_1^-, \tau_S^A, \tau_2^+ \rangle$. Hence, all we need to show is that for any $\tau_1^- : C_- \to BG$, there exists $\tau_2^+ : C_+ \to BG$ such that $\langle \tau_1^-, \tau_S^A, \tau_2^+ \rangle$ is homotopic to $\tilde{\tau}_S$. This again follows from \Cref{proposition-cofibration}.
\end{proof}

\begin{corollary} \label{corollary:group-structure-sphere}
	$(\alm_G(S^n),\#_\beta)$ is a group with identity $\beta$.
\end{corollary}
\begin{proof}
	This follows from \Cref{proposition:monoid} and \Cref{proposition:existence-inverses}.
\end{proof}
\begin{remark} \label{remark:indep-beta-somorphism}
	The group $(\alm_G(S^n),\#_\beta)$ is independent of $\beta$ up to isomorphism. Indeed, given any $x,y, \beta, \beta' \in \alm_G(S^n)$, it follows from \Cref{proposition:monoid} that
	\[
		(x \#_\beta \beta') \#_{\beta'} (y \#_\beta \beta') = (x \#_\beta (\beta' \#_{\beta'} y)) \#_\beta \beta' = (x \#_\beta y) \#_\beta \beta',
	\]
	which implies that the map
	\begin{align*}
		(\alm_G(S^n),\#_\beta) &\to (\alm_G(S^n),\#_{\beta'}) \\
		x &\mapsto x \#_\beta \beta'
	\end{align*}
	is a group isomorphism.
\end{remark}

Given orientation preserving embeddings $i_+ : \overline{D}_+ \to M$ and $i_- : \overline{D}_- \to S^n$, the results of section~\ref{section-connect-sum} give us a well-defined map
\[
	\alm_G(M) \times \alm_G^\beta(S^n) \to \alm_G(M \#_{i_+,i_-} S^n).
\]
For the remainder of this section, we will take $i_-$ to be the inclusion map $\overline{D}_- \hookrightarrow S^n$. Then $M \#_{i_+,i_-} S^n$ is canonically identified with $M$ (regardless of what $i_+$ is) and we get a map
\begin{equation} \label{connected-sum-action}
	\alm_G(M) \times \alm_G^\beta(S^n) \to \alm_G(M).
\end{equation}
By \Cref{proposition:monoid}, this is a group action. Note that the diffeomorphism $\phi : M \to M$ appearing in the statement of \Cref{proposition-diffeomorphism-almost-manifolds} (applied to $N = S^n$) can be chosen to be isotopic to the identity, which implies that the map \eqref{connected-sum-action} is independent of $i_+$.

If we assume that the image of the embedding $i_+ : \overline{D}_+ \to M$ is disjoint from $A$, then it follows directly from \Cref{definition:almost-connected-sum} that the subset $\alm_G(M,A;\beta_0) \subset \alm_G(M)$ is invariant under the map \eqref{connected-sum-action}. If $M \setminus A$ is connected, then the resulting action on $\alm_G(M,A;\beta_0)$ doesn't depend on the choice of $i_+$.

\end{document}